\documentclass[12pt]{amsart}
\usepackage{graphicx}
\usepackage{verbatim}
\usepackage{textcomp}
\usepackage{amssymb}
\usepackage{cite}
\usepackage{color}
\usepackage[all]{xy}
\usepackage{graphicx}
\usepackage{color}
\usepackage{hyperref}
\usepackage{soul}
\hypersetup{
    colorlinks,
    citecolor=black,
    filecolor=black,
    linkcolor=black,
    urlcolor=black
}

\usepackage[normalem]{ulem}
\usepackage{cancel}
\usepackage{tikz}
\usepackage[all]{xy}
\usepackage{graphicx}
\usetikzlibrary{backgrounds}
\newtheorem{theorem}{Theorem}[section]
\newtheorem{lemma}[theorem]{Lemma}
\newtheorem{proposition}[theorem]{Proposition}
\newtheorem{corollary}[theorem]{Corollary}   

\newtheorem{claim}[theorem]{Claim}
\allowdisplaybreaks

\newtheorem{keylemma}[theorem]{Iterative Lemma}

\theoremstyle{definition}
\newtheorem{definition}[theorem]{Definition}
\newtheorem{remark}[theorem]{Remark}

 \newenvironment{proofRT}%
{{\sc Proof of {Theorem~\ref{RegularityTheorem} }and Theorem~\ref{goto0}.}}
{{\sc q.e.d.} \\}

 \newenvironment{proofofimplication}
 {{\sc Proof that (\ref{assumptionGS}) implies (\ref{conclusionGS}).}}
{{\sc q.e.d.} \\}

 \newenvironment{george}%
{{\sc Proof of Theorem~\ref{rigiditytheorem}.}}%
{{\sc q.e.d.} \\}
{{\sc Proof of Lemma~\ref{uscord}.}}%
{{\sc q.e.d.} \\}
{{\sc Proof of Lemma~\ref{federeroutermeasure}.}}%
{{\sc q.e.d.} \\}
{{\sc Proof of Lemma~\ref{cd2general}.}}%
{{\sc q.e.d.} \\}

{{\sc Proof of Lemma~\ref{usordv}.}}%
{{\sc q.e.d.} \\}
  \newenvironment{chika}%
{{\sc Proof of Corollary~\ref{symTeich}.}}%
{{\sc q.e.d.} \\}
 \newenvironment{bach}%
{{\sc Proof of Theorem~\ref{surf}.}}%
{{\sc q.e.d.} \\}
{{\sc Proof of Theorem~\ref{hmco2}.}}%
{{\sc q.e.d.} \\}
{{\sc Proof of claim~(\ref{rickclaim}).}}%
{{\sc q.e.d.} \\}

  \newenvironment{proofofRTMS}%
{{\sc Proof of Theorem~(\ref{RegularityTheoremModelSpace}).}}%
{{\sc q.e.d.} \\}
\numberwithin{equation}{section}

\newtheorem*{theorem*}{Theorem}
{{\sc Proof of Lemma~\ref{tri1}.}}%
{{\qed} \\}
{{\sc Proof of Theorem~\ref{regularity}.}}%
{{\qed} \\}
{{\sc Proof of Theorem~\ref{main}.}}%
{{\qed} \\}
\newenvironment{proofof(i)}%
    {{\sc Proof of $(i)$.}}%
  {{\qed} \\}  
  
  \newenvironment{proofof(iv)}%
    {{\sc Proof of $(iv)$.}}%
  {{\qed} \\}  
\newcommand{\R}{\mathbb R}

\newcommand{\C}{\mathbb C}
\newcommand{\Sp}{\mathbb S}

\newcommand{\vt}{\vartheta}
\newcommand{\du}{d\mu}
\newcommand{\s}{\Sigma}
\newcommand{\D}{\mathbb D}
\title{  Rigidity of Teichm\"{u}ller space}

\author[Daskalopoulos]{Georgios Daskalopoulos}
\address{Department of Mathematics \\
                 Brown Univeristy \\
                 Providence, RI}%02912}
\email{daskal@math.brown.edu}
\author[Mese]{Chikako Mese}
\address{Johns Hopkins University\\
Department of Mathematics\\
Baltimore, MD}%  21218}
\email{cmese@math.jhu.edu}

\begin{document}
\maketitle
\begin{abstract}
We prove the \emph{holomorphic rigidity conjecture of Teichm\"{u}ller space} which loosely speaking states that the action of the mapping class group uniquely determines the Teichm\"{u}ller space as a complex manifold. 
The method of proof is through harmonic maps.  
We prove that the singular set of a harmonic map from a smooth $n$-dimensional Riemannian domain to the Weil-Petersson completion $\overline{\mathcal T}$ of Teichm\"{u}ller space
 has Hausdorff dimension at most $n-2$, and moreover,   $u$ has certain decay near the singular set.  Combining this with the earlier work of  Schumacher, Siu and Jost-Yau,  
 we provide a proof of the holomorphic rigidity of Teichm\"{u}ller space. 
 In addition, our results  provide as a byproduct a harmonic maps proof of both the  high rank and the rank one  superrigidity of the mapping class group proved via other methods by Farb-Masur and Yeung. 
 \end{abstract}

\section{Introduction} 
\subsection{Statement of  Results and brief history}
The main result of  the paper is the following statement.

\begin{theorem}[{\bf Holomorphic Rigidity of Teichm\"{u}ller Space}] \label{rigiditytheorem}
Let $ \Gamma$ denote the mapping class group of an oriented  surface  $S$ of genus $g \geq 2$.   Assume that $ \Gamma$ acts (as a discrete automorphism group) on a contractible K\"{a}hler manifold $\tilde{M}$ such that there is a finite index subgroup  $ \Gamma'$ of  $ \Gamma$ satisfying the properties: 
\begin{itemize}
\item[(i)]    $M:=\tilde{M}/{\Gamma'}$ is a smooth quasiprojective variety.% that is  homotopy equivalent to $\mathcal M':={\mathcal T}/{\Gamma'}$ where  ${\mathcal T}$ is the Teichm\"{u}ller space of $S$. 
\item[(ii)]  $M$ admits a compactification $\overline M$ as an algebraic variety such that   the codimension  of $\overline M \backslash M$ is $\geq$ 3. 
\end{itemize}
Then $\tilde{M}$ is equivariantly biholomorphic or conjugate biholomorphic to  the Teichm\"{u}ller space ${\mathcal T}$ of $S$ 
where $\Gamma$ acts on ${\mathcal T}$ as the mapping class group.
\end{theorem}

We will derive Theorem~\ref{rigiditytheorem} as a consequence of the following  more general holomorphic rigidity Theorem and its Corollary.

\begin{theorem}\label{rigiditytheorem0}
Let $M$ be a complete, finite volume K\"ahler manifold  with universal cover $\tilde M$ and $\pi_1(M)$ finitely generated. Let $ \Gamma$ be the mapping class group of an oriented  surface  $S$ of genus $g$  and  $p$ marked points such that $k=3g-3+p>0$, $\overline {\mathcal T}$  the Weil-Petersson completion of the Teichm\"{u}ller space ${\mathcal T}$ of $S$ and $\rho: \pi_1(M) \rightarrow \Gamma$ a homomorphism.   If there exists a  finite energy
 $\rho$-equivariant  harmonic map $ u: \tilde M \rightarrow \overline {\mathcal T}$,  
then there  exists a stratum ${\mathcal T}'$ of $\overline{\mathcal T}$ such that $ u$ defines a pluriharmonic map into ${\mathcal T}'$.
 Furthermore,
\[
 \sum_{i,j,k,l}R_{ijkl}d''u_i \wedge d'u_j \wedge d'u_k \wedge d''u_l \equiv 0
\]
where $R_{ijkl}$ denotes the Weil-Petersson curvature tensor. In particular,
if additionally the (real) rank of $u$ is $\geq  3$ at some point, then  $ u$ is holomorphic or conjugate holomorphic. 
\end{theorem}

The assumption about the existence of a finite energy
 $\rho$-equivariant  harmonic map to the Weil-Petersson completion $\overline{\mathcal T}$ of Teichm\"{u}ller space holds in many important cases. For example, if $M$ is compact and
   $\rho$  is sufficiently large (see definition below), then harmonic maps exist.  More generally, this is also true if we replace the assumption that $M$ is compact by the assumption $M$ is complete, satisfies the assumptions of Theorem~\ref{rigiditytheorem0} and admits a finite energy  map to $\overline{\mathcal T}$.
   
 Recall from \cite[p.142]{papa} or \cite[Definition 2.1]{daskal-wentworth}  that   two
  pseudo-Anosov elements of the mapping class group are called {\it independent} if their fixed point sets in the space of projective measured foliations  do not coincide.
 A subgroup of the mapping class group $\Gamma$ is called {\it {sufficiently large}} if it contains two independent pseudo-Anosov elements. A homomorphism $\rho$ into the mapping class group is called sufficiently large if its image is sufficiently large.

   By combining Theorem~\ref{rigiditytheorem0} above with   \cite[Corollary 1.3]{daskal-wentworth} we obtain the following.
     
   \begin{corollary}\label{rigiditytheorem1}Let $M$ be a complete, finite volume K\"ahler manifold  with universal cover $\tilde M$ and $\pi_1(M)$  finitely generated. Let $ \Gamma$ be the mapping class group of an oriented  surface  $S$ of genus $g$  and  $p$ marked points such that $k=3g-3+p>0$ and $\rho: \pi_1(M) \rightarrow \Gamma$ a homomorphism that is sufficiently large.   If there exists a  finite energy
 $\rho$-equivariant   map $ \tilde M \rightarrow  {\mathcal T}$,  
then there  exists a $\rho$-equivariant  pluriharmonic map $ u: \tilde M \rightarrow  {\mathcal T}$.
 Furthermore,
\[
 \sum_{i,j,k,l}R_{ijkl}d''u_i \wedge d'u_j \wedge d'u_k \wedge d''u_l \equiv 0
\]
where $R_{ijkl}$ denotes the Weil-Petersson curvature tensor. In particular,
if additionally the (real) rank of $u$ is $\geq  3$ at some point, then  $ u$ is holomorphic or conjugate holomorphic. 
\end{corollary}

   The rank condition also holds in many important applications, for example in Theorem~\ref{rigiditytheorem}. This is usually verified by showing that certain nontrivial homology classes in $M$ of degree $\geq 3$ are mapped  nontrivially under $u$ (see for example \cite{siu}).
   
   The following  theorem, due to  Farb-Masur and Yeung, also follows as a byproduct of our methods.
 
\begin{corollary}[{\bf{Superrigidity of the MCG, cf. \cite{farb-masur},  \cite{yeung}}}] \label{symTeich}
 Let $\tilde M= G \slash K$ be an irreducible symmetric space of noncompact type other than
 $SO_0(p,1) \slash SO(p) \times SO(1)$, $SU_0(p,1) \slash S(U(p) \times U(1))$.
  Let 
 $\Lambda$ be a discrete subgroup of $G$ with finite volume quotient and let  $\Gamma$ denote the mapping class group of an oriented  surface of genus $g$  and  $p$ marked points such that $k=3g-3+p>0$.  If the rank of $\tilde M$ is $\geq 2$, we  assume additionally that $\Lambda$ is cocompact. Then there exists no sufficiently large  homomorphism $\rho: \Lambda \rightarrow \Gamma$. 
  \end{corollary}
  
  The  phenomenon of strong rigidity was discovered by Mostow for a large class of locally symmetric spaces of nonpositive curvature. The famous Mostow rigidity theorem of 1968 \cite{mostow} states that if two compact hyperbolic manifolds of dimension greater than two  have the same fundamental group, then they are isometric. In particular, Mostow's result says that for compact hyperbolic manifolds, the metric structure is rigidly determined by the topology.   This statement  was later extended to other locally symmetric spaces of nonpositive curvature, not necessarily compact but satisfying a finite volume assumption   (cf.~\cite{mostow2}, \cite{prasad}). 
  
A natural question is whether structures other than metric structures are also rigidly determined by the topology. One such case is holomorphic rigidity within the class of K\"ahler manifolds. In fact, a weak form of holomorphic rigidity was discovered earlier in the 1960 work of Calabi-Vesentini.  \cite{calabi-vesentini}.  They showed that 
compact quotients of bounded complex symmetric domains of complex dimension at least two do not admit any nontrivial infinitesimal holomorphic deformations. In the late 1970's, Yau conjectured that strong  rigidity holds for compact K\"{a}hler manifolds of complex dimension at least two and negative sectional curvature.  This  was subsequently proved in 1980 using harmonic maps by Siu  \cite{siu} in the case when one of the manifolds has strong negative curvature. 

Siu's work inspired an outburst of  important results in geometric superrigidity including the work of Corlette \cite{corlette}, Mok-Siu-Yeung \cite{mok-siu-yeung}, Jost-Yau (cf. \cite{jost} and the references therein) and Gromov-Schoen\cite{gromov-schoen} among others. The proofs of all the aforementioned results  use harmonic maps.
Indeed, one starts with the work of Eells-Sampson \cite{eells}  which asserts that 
 if two Riemannian manifolds are homotopy equivalent and if one of them is non-positively  curved, then there  there exists a harmonic map from the manifold without the curvature condition to the other manifold which   is also a homotopy equivalence. Then  a Bochner-type formula leads to the conclusion that the harmonic map must preserve either the metric or the holomorphic structure. The passage through harmonic maps is necessary because the system of equations which determines that a map is either totally geodesic  or holomorphic  are overdetermined whereas the system of harmonic map equations is not.

Siu \cite{siu2} and Jost-Yau \cite{jost-yau1} extended Siu's result to a class of non-compact symmetric domains  with appropriate metric properties at infinity. Given that Teichm\"{u}ller space resembles a complex symmetric domain and admits a metric of strong negative curvature (as we will see in the next paragraph),  Jost and Yau also attempted to prove   holomorphic rigidity of Teichm\"{u}ller space  \cite{jost-yau}.  Their proof was incorrect.  
 
Before we continue, we  briefly review some important properties of the Teichm\"{u}ller space  $\mathcal T$ (of an oriented  surface  $S$ of genus $g$  and  $p$ marked points such that $k=3g-3+p>0$) that are relevant to this article. First recall that 
  $\mathcal T$    endowed with the Weil-Petersson metric $G_{wp}$ is a   K\"ahler  manifold  \cite{ahlfors}  whose sectional curvature  is negative \cite{tromba} and \cite{wolpert0}.  Moreover, the curvature tensor of $G_{wp}$ is strongly negative  in the sense of Siu \cite{schumacher}, which makes it plausible that $\mathcal T$ is holomorphicaly rigid. However, the Weil-Petersson metric is incomplete \cite{wolpertPJ} and \cite{chu}, and this causes major difficulties in pursuing Siu's approach. 
  
  Let $(\overline{\mathcal T}, d_{\overline{\mathcal T}})$ denote the metric completion of $({\mathcal T},G_{wp})$.
  The metric space $(\overline{\mathcal T}, d_{\overline{\mathcal T}})$ is a complete NPC space; i.e.~a geodesic space with non-positive curvature in the sense of Alexandrov  \cite{daskal-wentworth}, \cite{wolpert} and \cite{yamada}.   
 Set theoretically, $\overline{\mathcal T}$ is  nothing but the augmented Teichm\"uller space  \cite{masur}, \cite{abikoff}. 
Its boundary $\partial {\mathcal T}$  can be stratified by smooth open strata corresponding to deformations of nodal surfaces formed by pinching a finite set of nontrivial, nonperipheral, simple closed curves   \cite{masur} and \cite{wolpert}. In other words, $\overline{\mathcal T}$ is a stratified space (with the original Teichm\"{u}ller space ${\mathcal T}$ being the top dimensional open stratum). 

Given the incompleteness of Teichm\"uller space, one is tempted to replace ${\mathcal T}$ by $\overline{\mathcal T}$ and study harmonic maps to the NPC metric space $\overline{\mathcal T}$. Harmonic maps to metric spaces was initiated in the seminal paper of Gromov and Schoen \cite{gromov-schoen} where they study  harmonic maps  to Euclidean buildings (a special type of Riemannian polyhedra with non-positive curvature in the sense of Alexandrov). Their work was subsequently extended for harmonic maps  into  general NPC spaces by Korevaar-Schoen and Jost  \cite{korevaar-schoen1},  \cite{korevaar-schoen2} and \cite{jost}.   For other work on harmonic maps to singular spaces relevant to this paper, we refer to   \cite{daskal-meseDM} and \cite{daskal-meseSR}.

In \cite{gromov-schoen} (as  well as in  \cite{daskal-meseDM} and \cite{daskal-meseSR}), the main technical point is how to handle  the singularities of the harmonic map.  To do this, one  gains control of the map near the set of points that do not map to  smooth points in the target.
We do the same  in this paper, but  there are  additional difficulties stemming from the non-local compactness of $\overline{\mathcal T}$. By contrast, 
the spaces studied by  \cite{gromov-schoen}  were locally compact. The most important technical challenge tackled in this paper is to overcome the difficulty presented by the non-local compactness of $\overline{\mathcal T}$.  

Before attempting to study harmonic maps, one needs to get a good understanding of the geometry of $\overline{\mathcal T}$ near its boundary.
In \cite{masur}, Masur initiated the study of the Weil-Petersson metric near the boundary of $\overline{\mathcal T}$.
In recent years,  many authors have extended Masur's work to establish stronger asymptotic properties of the Weil-Petersson geometry.   See for example,  \cite{schumacher}, \cite{daskal-wentworth}, \cite{yamada}, \cite{wolpert}, \cite{wo4}, \cite{liusunyau1}, \cite{liusunyau2} and \cite{huang} among many others.
In \cite{daskal-meseC1},  we  proved stronger $C^1$-estimates which will be used in this paper.   These estimates differ from the  previously known derivative estimates  because  they estimate the asymptotic difference of the Weil-Petersson metric and a product metric given on the product of the boundary strata and its normal space (which will be described in more detail below, cf.~Section~\ref{intro-part2}).

We end this summary by stating the two main technical theorems that allow us to control the harmonic map near its singular set. Below we denote by ${\mathcal R}(u)$ to be the set of points in the domain that possess a neighborhood mapping into a single stratum in $\overline{\mathcal T}$ and ${\mathcal S}(u)$ to be its complement.

\begin{theorem}\label{RegularityTheorem}
Let  $({\mathcal T}, G_{wp})$ denote the Teichm\"uller space of an oriented surface  of genus $g$  and $p$ marked points such that $k=3g-3+p>0$ with the Weil-Petersson metric and  let $(\overline{\mathcal T}, d_{ \overline{\mathcal T}})$ be its metric completion.  If $(\Omega,g)$ is an  $n$-dimensional Lipschitz Riemannian domain and $u:(\Omega,g) \rightarrow( \overline{\mathcal T}, d_{ \overline{\mathcal T}})$ is a harmonic map, then 
\[
\dim_{\mathcal H} \Big({\mathcal S}(u)\Big) \leq n-2.
\]
\end{theorem}

\begin{theorem}\label{goto0}
Let   $u:(\Omega,g) \rightarrow( \overline{\mathcal T}, d_{ \overline{\mathcal T}})$ be as in Theorem~\ref{RegularityTheorem}.  For any compact subdomain $\Omega_1$ of $\Omega$, there exists a sequence of smooth functions $\psi_i$  with $\psi_i \equiv 0$ in a neighborhood of ${\mathcal S}(u) \cap \overline{\Omega}_1$, $0 \leq \psi_i \leq 1$ and $\psi_i(x) \rightarrow 1$ for all $x \in \Omega_1 \backslash {\mathcal S}(u)$ such that
\[
\lim_{i \rightarrow \infty} \int_{\Omega} |\nabla \nabla u| |\nabla \psi_i| \ d\mu =0.
\]
\end{theorem}
Theorem~\ref{goto0} should be viewed as  an \emph{estimate on the growth of the norm of the gradient  $\nabla u$ of $u$ near its singular set}.  The existence of  the sequence  $\psi_i$ allows us to  justify Stoke's Theorem, a crucial  step  in  applying the Bochner technique to rigidity.

\subsection{Description of the main technical points}
\label{intro-part2}
 As mentioned before, all the above theorems are proved by using the theory of harmonic maps to metric spaces. The proof takes advantage of the
important special feature of  the metric space $\overline{\mathcal T}$  near a boundary point ---  it is  asymptotically isometric to the product of  a smooth open stratum  ${\mathcal T}' \subset \partial {\mathcal T}$ (which has the structure of a  smooth K\"ahler manifold)  and a simpler metric space $\overline{\bf H}$  or its product $\overline{\bf H} \times \dots \times \overline{\bf H}$  (cf. \cite{daskal-wentworth}, \cite{yamada}, \cite{wolpert}, \cite{wo4}, \cite{liusunyau1}, \cite{liusunyau2} and \cite{daskal-meseC1}).  The metric space $\overline{\bf H}$ is  called  the  \emph{model space}.
Thus, for a harmonic map $u:\Omega \rightarrow \overline{\mathcal T}$ near  a singular point $x \in {\mathcal S}(u)$, we can write
 $u=(V,v)$   where $V$ is the \emph{regular component} that  maps  into the smooth manifold ${\mathcal T}'$  and $v$ is the \emph{singular component} mapping into $\overline{\bf H}$ or  $\overline{\bf H} \times \dots \times \overline{\bf H}$.

The difficulty in analyzing $u=(V,v)$   is that the component maps $V$ and $v$ are not necessary harmonic. 
This situation is further complicated by the fact that the singular component $v$  may be the non-dominant component  (i.e.~the higher order term) of $u$.   Moreover, one cannot use tools from   elliptic PDE's~(as one would  for maps into Riemannian manifolds) because the harmonic maps  may  a priori have a  large singular set.
Nonetheless, in this paper we will push forward the harmonic map theory  by overcoming two major obstacles.  The \emph{first obstacle}  is that the Weil-Petersson metric near the boundary of  ${\mathcal T}$ is not a product, but only asymptotically a product.  
 The \emph{second obstacle} is  the non-local compactness and  degenerating geometry of  $\overline{\mathcal T}$.     
The techniques that we will introduce  to handle these issues are the main accomplishments of this paper  and  the crux of the proofs of the Regularity Theorems~\ref{RegularityTheorem} and \ref{goto0}. \\

{\sc Overcoming obstacle 1:  Monotonicity Formula and the Order Function.}
A key technical tool in analyzing the structure of
a harmonic map $u:(\Omega,g) \rightarrow (X,d)$ from a Riemannian domain into an NPC space is the order function $Ord^u$ of $u$.
If $u$ is a harmonic function,  then $Ord^u(x_0)$ is the order with which $u$ attains its value $u(x_0)$ at $x_0$. 
  In its simplest form, the order is the limit as $r \rightarrow 0$ of the  scale invariant ratio 
\begin{equation} \label{monotonicityofratio}
\frac{rE(r)}{I(r)}= \frac{r  \displaystyle{\int_{B_r(x_0)} |\nabla u|^2 \ d\mu}}{\displaystyle{\int_{\partial B_r(x_0)} d^2(u,P_0) d\Sigma}}
\end{equation}
where  the numerator is $r$ times the energy $E(r)$ 
of $u$ in a geodesic ball $B_r(x_0)$ of radius $r$ centered at $x_0 \in \Omega$ and the denominator $I(r)$ is  the  $L^2$-distance between  $P_0 \in X$ and $u$ on the boundary  $\partial B_r(x_0)$. 
A ratio of this type had been previously used in the study of various elliptic PDE problems (e.g.~\cite{agmon}, \cite{almgren}, \cite{garofolo-lin}, \cite{landis1}, \cite{landis2},\cite{lin}, \cite{miller}), but Gromov and Schoen \cite{gromov-schoen} were the first to introduce this idea in the context of harmonic maps to NPC metric spaces.

 The existence of the order function  is due to the monotonicity (in the parameter $r$) of the ratio~(\ref{monotonicityofratio}) which in turn  
follows from the domain and target variations of harmonic maps. 
The idea for  the domain variation is as follows.   Let $B_r(x_0)$ be a geodesic neighborhood  of $x_0$ with  normal coordinates $x=(x^1, \dots, x^n)$ centered at $x_0=0$ and consider a diffeomorphism of the form
$F_t(x)=(1+\tau \eta(x)) x$  where $\eta$ has compact support in $B_r(x_0)$ (hence $F_t$ is the identity  outside  $B_r(x_0)$).  A domain variation of $u$ is the one-parameter family $u_t=u \circ F_t$ with $u_0=u$.  Since the total energy  function 
\begin{equation}\label{minprop}
t \mapsto E^{u_t}=\int_\Omega |\nabla u_t|^2 d\mu
\end{equation}
  has a minimum at $t=0$, we can differentiate the above equation in $t$ and obtain the domain variation formula 
\begin{equation} \label{DomainVariationHM}
0 = \int_{B_r(x_0)} |\nabla u|^2 (2-n) \eta- |\nabla u|^2 \sum_i x^i \frac{\partial \eta}{\partial x^i} + 2 \sum_{i,j,k} g^{ik} \frac{\partial \eta}{\partial x^i} x^j \frac{\partial u}{\partial x^j} \cdot \frac{\partial u}{\partial x^k} d\mu.
\end{equation}
  For harmonic maps between smooth Riemannian manifolds, the domain variation formula  yields the well known monotonicity of the scale-invariant (with respect to  dilation of the domain) energy,
    \[
  r^{2-n} E(r) = r^{2-n} \int_{B_r(x_0)} |\nabla u|^2 d\mu.
  \]
This has  played  an important role in the regularity theory of harmonic maps between smooth Riemannian manifolds (notably in the Schoen-Uhlenbeck $\epsilon$-regularity theorem [SU]).   Using a generalization of the notion of energy,  for harmonic maps to NPC spaces (cf.~\cite{gromov-schoen} and \cite{korevaar-schoen1}),  the domain variation formula readily generalizes  to the case of  NPC targets.  

Gromov and Schoen's innovation in \cite{gromov-schoen} was to improve the classical monotonicity formula  to  obtain a more sophisticated tool for studying harmonic maps into NPC spaces.    The idea is  to combine the domain variation formula with the convexity of the distance function $d$ on the target NPC space $X$. Indeed, they consider  target variations of  $u$ by   pulling it back  along a geodesic to a fixed point.  More precisely, fix $x_0 \in \Omega$, $P_0 \in X$ and a non-negative function $\zeta$ with  compact support in a neighborhood $B_r(x_0)$ of $x_0$.  Consider an one-parameter family of maps $u_t$, for $t>0$ sufficiently small, by setting $u_t(x)$ to be the point on a geodesic between  $P_0$ and $u(x)$ at a distance $(1-t\zeta(x))d(P_0,u(x))$ from $P_0$.  The  minimizing property  of the energy of $u$
yields the   subharmonicity  of the function $d(u,P_0)$; more precisely,   $d(u,P_0)$ satisfies in the weak sense the differential inequality (cf.~\cite[Proposition 2.2]{gromov-schoen})
\begin{equation} \label{weaksubharmonicity}
\triangle d^2(u(x),P_0) \geq 2 |\nabla u|^2.
\end{equation}
Combining   the domain variation formula (\ref{DomainVariationHM}) with the target variation formula (\ref{weaksubharmonicity}), they obtain  the monotonicity formula (cf.~\cite[proof of (2.5)]{gromov-schoen})
\[
\frac{1}{r} + \frac{E'(r)}{E(r)} - \frac{I'(r)}{I(r)} \geq O(r)
\]
where $O(r)$ measures how far away the domain metric $g$ is from being Euclidean.  The monotonicity of the ratio (\ref{monotonicityofratio}) follows immediately from this differential inequality if  $O(r)$ is identically equal to $0$. If $O(r)$ not equal to 0, one simply adjusts the ratio (\ref{monotonicityofratio}) by multiplying it by $e^{cr^2}$  for an appropriate choice of $c>0$.   The limit of (\ref{monotonicityofratio}) at each point on the domain  defines  the order function $Ord^u:\Omega \rightarrow [1,\infty)$. 

In  \cite{daskal-meseDM} and in the present paper, we extend the notion of order to a wider class of maps.  To movivate this generalization,    recall that 
a harmonic map  $u=(u^1, \dots, u^m): \Omega \rightarrow {\mathbb R}^m$ into the Euclidean space can be viewed as $n$-independent harmonic functions. 
Assuming continuity,  a harmonic map between Riemannian manifolds can also be expressed as a set of component functions $u=(u^1, \dots, u^m)$ by using local coordinates;  but if the target metric is non-Euclidean, the component functions are not independent of each  other.  Indeed, the harmonic map equations 
\[
\triangle u^i + \sum_{\alpha, \beta} \sum_{j,k} g^{\alpha \beta} \Gamma^i_{jk} \circ u \frac{\partial u^j}{\partial x^\alpha} \frac{\partial u^k}{\partial x^\beta}=0, \ i=1, \dots, m
\]
show that the behavior of each component function is influenced by the behavior of the other component functions via the Christoffel symbols $\Gamma^i_{jk}$ of the target metric.  On the other hand, Riemannian manifolds are locally asymptotic to Euclidean space.  Namely, normal coordinates centered at a point show that  a smooth Riemannian manifold is Euclidean up to second order at  that point.  We can interpret  this to mean that Riemannian manifolds are asymptotically a product of $m$-copies of ${\mathbb R}$.

Analogously to harmonic maps into $\R^m$, a harmonic map $u$ into a Euclidean building can be   expressed  by component maps which are themselves harmonic.  Indeed,  we can locally write  $u=(V,v)$ where $V$ is a  harmonic map into a Euclidean space  and $v$ is a harmonic map into a lower dimensional Euclidean building. 
It is a serious technical issue that many of the techniques developed by Gromov and Schoen cannot be directly applied to  NPC spaces that don't decompose locally as a product.  

In this paper, building upon earlier work in  \cite{daskal-meseDM}, we  develop a  technique to study  harmonic maps into spaces that are  only  \emph{asymptotically} a product  of NPC spaces. In many ways, the step from  harmonic  maps into a  product  of  NPC spaces  to harmonic maps into  a space that is asymptotically a product is analogous to the passage  from harmonic functions to harmonic maps into Riemannian  manifolds.  As indicated above, a harmonic map into $\overline{\mathcal T}$ is given by $u=(V,v)$ where $V$ maps into a smooth Riemannian manifold and $v$  maps into an NPC space.
Since $v$ is not a harmonic map, we will have to modify  (\ref{minprop}). In fact,  we will derive analogues of the domain and target variation formulas   (\ref{DomainVariationHM}) and (\ref{weaksubharmonicity}) with correction terms.  Combining these formulas, we will obtain the monotonicity formula
\[
\frac{1}{r} + \frac{E'(r)}{E(r)} - \frac{I'(r)}{I(r)} \geq -C
\]
where $C$ is a constant that not only depends on  how far away the domain metric $g$ is from the Euclidean metric but also on how far the target metric is from being a product metric.  The conclusion  is that we can associate an order function $Ord^v:\Omega \rightarrow [1,\infty)$ to the singular component map $v$ of $u$ and use it to analyze its behavior.  
\\

{\sc Overcoming obstacle 2:  Inductive argument and  regularity.}
  The second obstacle is  the non-local compactness and  degenerating geometry of $(\overline{\mathcal T}, d_{ \overline{\mathcal T}})$ near the boundary.     In order to explain how we deal with this issue, we will first introduce two fundamental concepts  from the work of Gromov and Schoen \cite{gromov-schoen}. 
Let $X$ be an NPC space, let's say a  Euclidean building for the sake of concreteness, and  $X_0$ a totally geodesic subspace of $X$, for example an apartment of $X$.   
The first fundamental concept is the notion  of   a homogeneous degree 1 map  $l: \R^n \rightarrow X_0 \subset X$ being \emph{effectively contained} in $X_0$.  This loosely means that a sufficiently  small neighborhood of the image of $l$ is contained in $X_0$ except for a set of small measure.    
The second is the notion of $X_0$ being \emph{essentially regular}. Loosely, this means that a harmonic map into $X_0$ has an approximation by a homogeneous map that is better than first order.   
To illuminate these notions, we give the following example.\\

\begin{figure}[h]\label{effcont}
    \centering
    \includegraphics[width=\textwidth]{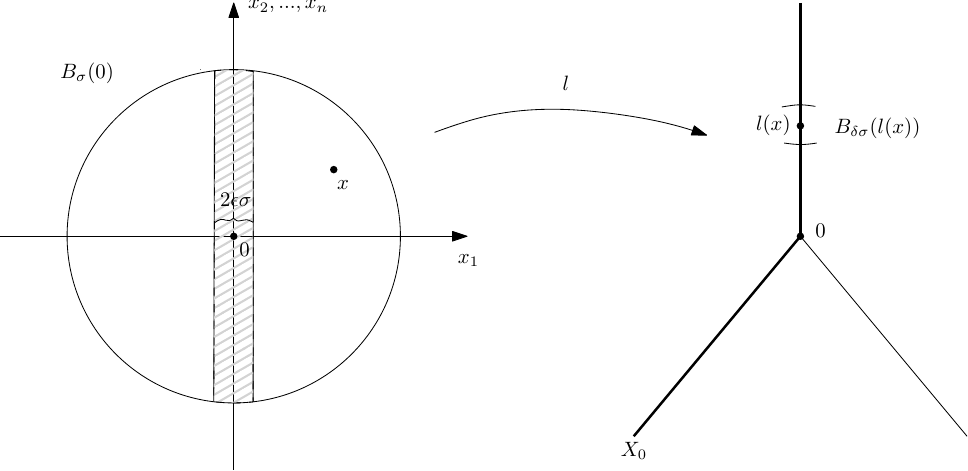}
    \caption{The non-shaded area shows the set of points whose image admits a $\delta \sigma$ neighborhood that does not intersect 
   $X \backslash X_0$.}
\end{figure}
\emph{Example 1.}  Let $X$ be a $k$-pod formed by  $k$ distinct copies $E_1, \dots,  E_k$ of the half-line  $[0,\infty)$   identified at 0 (called the juncture of the $k$-pod).  The distance $d_X(p,q)$ between two points $p \in E_i$ and $q \in E_j$  is defined to be  $|p-q|$ if $i=j$ and $p+q$ if $i \neq j$. Then $(X, d_X)$ is an NPC space.

We identify $X_0=E_1 \cup E_2 $ as a totally geodesic subspace  of $X$ isometric to $\R$ and let $l$ be an affine function (a special case of a homogeneous degree 1 discussed in \cite[Proposition 3.1]{gromov-schoen}), i.e.
\begin{equation}  \label{alphaxb}
l: \R^n \rightarrow \R  \simeq X_0 \subset X, \ \ 
l(x)=\vec{A} \cdot x +b,
\end{equation}
for some $\vec{A} \in \R^n$ and $b \in \R$. In the above, we can assume  $b=0$; otherwise, 
$l$ maps a neighborhood of
 $0 \in \R^n$ into a subset of $X_0 \simeq \R$, away from the juncture. Also by rotating our coordinates if necessary we may assume that $\vec{A}=(a,0,...0)$. Note
that in this case  $l(x)=a x^1 $ and
\[
  {\bf B}_{\delta \sigma}(l(x)) \cap X \backslash X_0 \neq \emptyset  \ \ \ \Leftrightarrow \ \ \ |l(x) |<\delta \sigma  \ \ \ \Leftrightarrow \ \ \  |x^1 |<\frac{\delta}{ |a|} \sigma.
\]
Hence, given $\epsilon>0$, there exists $\delta$ (for example, we can take $\delta=\frac{\epsilon |a|}{2v_{n-1}}$ where $v_{n-1}$ denotes the Euclidean volume of the unit $(n-1)$-dimensional ball) such that
\begin{equation}\label{exeffct}
\mbox{Vol}\{ x \in B_{\sigma}(0):  {\bf B}_{\delta \sigma}(l(x)) \cap (X \backslash X_0) \neq \emptyset \} <  \epsilon \sigma^n.
\end{equation}
See {\sc Figure} 1.   This  defines  the notion of \emph{a linear map   effectively contained in a totally geodesic subspace} in the sense of \cite[page 211]{gromov-schoen}.

We now come to the notion of \emph{essentially regular}. In this example, the totally geodesic subspace $X_0=E_1 \cup E_2 \simeq \R$ is essentially regular in the sense of \cite[page 210]{gromov-schoen}.  More precisely, for a harmonic function $f: (B_1(0), g)   \rightarrow  X_0 \simeq \R$,  the Taylor approximation implies
\[
d(f(x),l(x)) \leq C|x|^2 
\]
where  $l(x)=\nabla f(0) \cdot x+f(0)$ and  the constant $C$ depends only on the geometry of the domain and the  total energy of $f$.  
Thus, $X_0 \simeq \R$ is \emph{essentially regular}; namely there exists $\alpha>0$ (we can take  $\alpha=1$ in this example) and $C>0$ such that
\begin{equation}\label{exessreg}
\sup_{x \in B_{\sigma}(0)} d(f(x),l(x)) \leq C \sigma^{1+\alpha} \sup_{x \in B_1(0)} d(f(x),L(x)), \ \ \ \forall \sigma \in (0,\frac{1}{2}]
\end{equation}
for any affine function $L(x)=\vec{A}\cdot x+b$.
 The important feature of essential regularity is that the parameters $\alpha$ and  $C$ are independent of the subspace $X_0$ and depend only the geometry of the domain and the  total energy $E^f$ of $f$. 
The case of Euclidean buildings is a higher dimensional generalization of the above example with its apartments playing the role of essentially regular subspaces.
  
For the sake of this introduction and in order to illustrate the main ideas, we will briefly discuss the Gromov-Schoen argument adapted to the simple case where $X$  is a $k$-pod as in Example 1.  A more technical discussion will be presented at the beginning of Section~\ref{sec:  proofofkeylemma}.  We start with a  harmonic map $u:B_1(0) \rightarrow X$,  where $B_1(0) \subset \R^n$ is the unit ball,  and a homogeneous degree 1 map
$l:B_1(0) \rightarrow X$  as in (\ref{alphaxb}) is effectively contained in an essentially regular totally geodesic subspace $X_0 \simeq \R$. We also assume that $u(0)=l(0)$ and that $u$ and $l$ are $D$-close, i.e 
\begin{equation} \label{setup}
\sup_{x \in B_1(0)} d(u(x), l(x)) < D.
\end{equation}
From the initial data, $u$ and $l$,  the goal is to produce \emph{a linear scale approximation}; i.e.
\begin{equation} \label{ssa}
\sup_{x \in B_{\sigma}(0)} d(u(x), l(x)) < c\sigma, \ \ c>0.
\end{equation}
The idea of proving regularity by the use of  a linear scale approximation is well known.  Examples include the $\epsilon$-regularity theorem of Schoen-Uhlenbeck \cite{su} and other work concerning the uniqueness of tangent maps  \cite[Chapter 3] {simon}. Estimate (\ref{ssa})  is usually achieved by an inductive process, where at each stage one improves the estimate by a fixed amount. In the example above, 
 the idea is to show that there exists  $\theta \in (0,\frac{1}{2}]$ such that if  an affine map 
 \[
_il:B_{\theta^i}(0) \rightarrow X_0
\]
at the $i^{th}$ stage is  ``close" to $u$ in a  ball of radius $\theta^i$,  then one can find a  new affine map
 \[
 _{i+1}l:B_{\theta^{i+1}}(0) \rightarrow X_0
 \]
that is ``closer" to   $u$  in a smaller ball radius $\theta^{i+1}$ for the $(i+1)^{th}$ stage.

To find $_{i+1}l$, consider the  harmonic function $v:  B_{\theta^i}(0) \rightarrow X_0 \approx \R$ with boundary condition $\pi \circ u$ where $
\pi:   X \rightarrow X_0
$ is the closest point projection map.  Since $X_0$ is essentially regular, $v$ has a ``good" linear approximation  $_{i+1}l$.  Since $l$ is effectively contained in $X_0$ and approximates $u$, then maps $u$ and $v$ are ``close."    One can show that indeed $_{i+1}l$ is the desired linear map for the $(i+1)^{th}$ stage.   For the convenience of the reader, we will sketch this simpler version of the inductive argument in subsection~\ref{simple g-s} before our main regularity results.

In this paper, we will apply a variation of the  Gromov-Schoen argument with the completion of Teichm\"uller space $\overline{\mathcal T}$ playing the role of a Euclidean building. Since all the degenerating geometry of $\overline{\mathcal T}$ comes from the model space $\overline{\bf H}$,  we will limit our discussion to  $\overline{\bf H}$  in this introduction.  This case  was treated  in our previous paper \cite{daskal-meseER}, and what is outlined below can also serve as its summary. In this paper, we will further extend these ideas to handle the case of $\overline{\mathcal  T}$. 
 
We first define  
$\overline{\bf H}$ precisely.  Consider  the Riemannian surface $({\bf H}, g_{\bf H})$
consisting of the upper half plane
 \[
 {\bf H}=\{(\rho,\phi) \in \R^2: \rho >0, \phi \in \R\}\]
endowed with the Riemannian metric 
\[
g_{\bf H}=d\rho^2+ \rho^6 d\phi^2.
\]
The NPC space $\overline{\bf H}$ is the  metric completion of ${\bf H}$ constructed by adding the boundary line $\{\rho=0\}$ and identifying this line as a single point $P_0$.  We call $P_0$ the \emph{singular point of $\overline{\bf H}$}.  The difficulty in  analyzing the behavior of a harmonic map into $\overline{\bf H}$ is caused by the degenerating geometry and the non-local compactness of $\overline{\bf H}$.  

The first step  is to find essentially regular totally geodesic subspaces of  $\overline{\bf H}$.  
 The difficulty is that, because of the degenerating geometry of ${\bf H}$ near $P_0$ (the Gaussian curvature approaches $-\infty$ near $P_0$),  the only totally geodesic subspaces of $\overline{\bf H}$  that resemble Euclidean spaces and contain  $P_0$ are the point $P_0$ itself and geodesics emanating from $P_0$. (These geodesics are given by curves $\rho \mapsto (\rho,\phi_0)$ for a fixed $\phi_0$.) 
The degenerating geometry of ${\bf H}$    is highlighted by  
 the harmonic map equations in $\overline{\bf H}$,  
\begin{equation} \label{hmequphi'}
u_{\rho} \triangle u_{\rho} = 3  u_{\rho}^6 |\nabla u_{\phi}|^2 \ \mbox{ and }  \ u_{\rho}^4 \triangle u_{\phi} = -6 \nabla u_{\rho} \cdot u_{\rho}^3 \nabla u_{\phi}.
\end{equation}
Notice that the right hand side of each equation  is bounded since $u$ is Lipschitz.  The left hand side of  each equation, though, involves $u_{\rho}$.  Thus,  the \emph{harmonic map equations are degenerate} since $u_{\rho}(x) \rightarrow 0$ as $x \rightarrow {\mathcal S}(u)$.  The following example provides a hint on how to proceed.\\

\emph{Example 2.}
Consider the 2-dimensional space $(H^+,g_0)$ where 
\[
H^+=\{(r,\theta) \in {\bf R}^2:  r>0\} \ \mbox{ and }  \ 
g_0(r,\theta)
= dr^2+ r^2 d\theta^2.
\]
The  Christoffel symbols with respect to the polar  coordinates $(r,\theta)$ are
\begin{eqnarray*}
\begin{array}{lcl}
\Gamma_{r r}^{r}=0 & & \Gamma_{\theta \theta}^{\theta} =  0 \\
\Gamma_{r \theta}^{r}=0 & & \displaystyle{ \Gamma_{r \theta}^{\theta}=\frac{1}{r}}
\\
\Gamma_{\theta \theta}^{r}=-r & &  \Gamma_{r r}^{\theta}=0.
\end{array}
\end{eqnarray*}
For a map $h$ into $(H^+,g_0)$, write $h=(h_{r},h_{\theta})$ with respect to the polar coordinates $(r,\theta)$.
Then the harmonic map equations are  
\begin{eqnarray} \label{he1}
h_{r} \triangle h_{r} =   h_{r}^2 |\nabla h_{\theta}|^2 \ \mbox{ and }  \ h_{r}^2 \triangle h_{\theta} = -2 \nabla h_{r} \cdot h_{r} \nabla h_{\theta}.
\end{eqnarray}
This set of equations looks very similar to the harmonic map equations (\ref{hmequphi'}) in  the sense that they are both degenerate. 
Now assume that
 the value of $h_\theta$ is contained in  $[0, 2\pi)$ which allows us to apply the change of variables to Euclidean coordinates 
\begin{equation} \label{ec1}
(r,\theta) \mapsto (x=r \cos \theta,y=r \sin \theta).
\end{equation} 
This change of variables converts equation  (\ref{he1}) to the standard  harmonic map equations with respect to the Eucledian metric, i.e.~$
 \triangle h_x=  0 \ \mbox{ and } \  \triangle h_y=  0.
$
In this form, the  smoothness of $h_x$ and $h_y$ can be immediately deduced from  the  theory of elliptic partial differential equations. \\

Example 2 illustrates  the following key points:  
\begin{itemize}
\item[(i)]  The polar coordinates $(r,\theta)$ in $\R^2$ are ill-suited for the regularity theory of  harmonic maps.   
\item[(ii)]
 A bound on the angular component of a harmonic map implies regularity results. 
  \end{itemize}
By the same token as (i), the standard coordinates $(\rho,\phi)$ of $({\bf H},g_{\bf H})$ are ill-suited to study harmonic maps (although they are convenient when studying the behavior of the degenerating Riemann surfaces corresponding to points of ${\mathcal T}$  approaching its boundary).  Furthermore, (ii) hints that one should look to  bound the ``angular" coordinates in order to find essentially regular subspaces.

The idea of choosing the right  coordinates and finding essentially regular subspaces to study the harmonic maps led to our paper \cite{daskal-meseER}.  There,  we introduced a change of variables which takes the coordinates $(\rho,\phi)$ to new coordinates   analogous to the change of variables (\ref{ec1}) from polar coordinates $(r,\theta)$ to Euclidean  coordinates $(x,y)$. 
 In essence, we introduced a new coordinate system for $({\bf H},g_{\bf H})$ that can be used to study harmonic maps.

Before we describe the new coordinates of ${\bf H}$, we will  first discuss the difficulty caused by the degenerating geometry and non-compactness of $\overline{\bf H}$ in relation to the key point (ii) above.
For a harmonic map $u:\Omega \rightarrow \overline{\bf H}$ and $x_0 \in \Omega$, a consequence of having a well-defined order $Ord^u(x_0)$ is that there exists a sequence of  blow-up maps of $u$ at $x_0$.  (Loosely speaking, these are maps constructed by concentrating in on the point $x_0$ and scaling up  $u$ restricted to small  geodesic balls centered at  $x_0$.)  Because of the non-local compactness of $\overline{\bf H}$ near $P_0$,  if $u(x_0)=P_0$, then this sequence of blow up maps does not converge as a map into $\overline{\bf H}$ since there exists no uniform bound  on the angular component for the sequence.  
 In short, we cannot expect to approximate  $u$ by a homogeneous degree 1 map $l$ with a good bound on the angular component map $l_{\phi}$.  This poses a problem in setting up the  Gromov-Schoen inductive argument since the heart of this argument is to use an essentially  regular subspace that effectively contains a homogeneous degree 1 map approximating $u$ (cf.~(\ref{setup})).

The problem described in the paragraph above led us to consider  the NPC space 
\[
\overline{\bf H}_2=\overline{\bf H}^+ \sqcup \overline{\bf H}^-/\sim.
\]
Here, $\overline{\bf H}^+$ and $\overline{\bf H}^-$ denote two distinct  copies of $\overline{\bf H}$ and  $\sim$ indicates that the singular point $P_0$ from each copy  is identified as a single point.   The induced distance function   $d_{{\bf H}_2}$ on $\overline{\bf H}_2$ is given by 
\begin{eqnarray*}
d_{{\bf H}_2}((\rho_1,\phi_1), (\rho_2,\phi_2)) = 
\left\{
\begin{array}{ll}
d_{\bf H}((|\rho_1|,\phi_1), (|\rho_2|,\phi_2)) & \mbox{ if $\rho_1 \rho_2 \geq 0$}\\
 |\rho_1|+|\rho_2| & \mbox{ if $\rho_1\rho_2<0$}.  
\end{array}
\right.
\end{eqnarray*}
By using the identification $(\rho,\phi) \mapsto (-\rho,\phi)$ in $\overline{\bf H}^-$, we obtain 
 ``coordinates" on $\overline{\bf H}_2$  where  
 \begin{eqnarray*}
 \overline{\bf H}^+ & = & \{(\rho,\phi) \in \R^2:  \rho>0\},\\
 \overline{\bf H}^- & = & \{(\rho,\phi) \in \R^2:  \rho<0\},\\
P_0 & = & \{(\rho,\phi) \in \R^2:  \rho=0\}.
\end{eqnarray*}
(Calling $(\rho,\phi)$ coordinates is a slight misnomer as they are not coordinates in the traditional sense near $P_0$.)
 The importance of $\overline{\bf H}_2$ can be explained by    the observation that 
harmonic maps  into $\overline{\bf H}_2$ exhibit a completely different behavior  than the one described in the previous paragraph.    Indeed, at an order 1 singular point, a harmonic map into $\overline{\bf H}_2$  can locally be approximated by a single homogeneous degree 1 map $l:B_1(0) \rightarrow \overline{\bf H}_2$  given by $l(x)=(l_\rho(x),l_\phi(x))=(Ax^1,0)$ for some constant $A>0$ (after a rotation of the domain $B_1(0)$ and a translation $(\rho,\phi) \mapsto (\rho, \phi-c)$ of the target $\overline{\bf H}_2$ for an appropriate constant $c \in \R$).  Here, the key point is that the angular component map of $l$ is identically constant.
The map $l$ is effectively contained in  the  subspace $\overline{\bf H}_2[\phi_0]$
for $\phi_0>0$. (This assertion follows from essentially the same argument as the proof of  Lemma~\ref{geomcor} below.)
Since $s \mapsto (s,\phi_0)$ and $s \mapsto (s,-\phi_0)$ are geodesics, $\overline{\bf H}_2[\phi_0]$ is   geodesically convex  in $\overline{\bf H}_2$.  A harmonic map  whose image lies in $\overline{\bf H}_2[\phi_0]$ has the property that its angular component function $v_{\phi}$ is bounded. 
The change of coordinates
\begin{equation} \label{coch2}
(\rho,\phi) \mapsto (\rho \cos \sqrt 3 \rho^2 \phi , \rho \sin \sqrt 3 \rho^2 \phi )
\end{equation}
in $\overline{\bf H}_2$ is 
analogous to the change of coordinates in $\R^2$ from polar coordinates $(r,\theta)$ to  the standard  coordinates $(x,y)$.
% If we apply  change of coordinates   \begin{equation} \label{polar}
%\Upsilon:=\rho-\frac{3}{2} \rho^5 \phi^2
% \ \mbox{ and } \ \Phi:= \rho^3 \phi,
%\end{equation}
By applying elliptic theory after the change of variables,  we prove $\overline{\bf H}_2[\phi_0]$ is essentially regular. 
%(in a sense slightly different than in \cite[page 210]{gromov-schoen}, cf.~\cite[Section 3]{daskal-meseER} where we apply a  change of coordinates (cf.~\ref{polar} below) slightly different from  (\ref{coch2}) but which agree up to first order).  
%In fact,  $\Upsilon$ and $\sqrt 3 \Phi$ agree up to the first order with $\rho \cos \sqrt  3 \rho^2 \phi$  and $\rho \sin \sqrt  3 \rho^2 \phi$ respectively. Using the coordinates $(\Upsilon, \Phi)$, we show that    $\overline{\bf H}_2[\phi_0]$  
%     
%
%In Section~\ref{jf},  

The key to showing  regularity of harmonic maps into $\overline{\bf H}$ is the close relationship between the geometries of $\overline{\bf H}$ and $\overline{\bf H}_2$ near $P_0$ which we now describe.
%we introduce new coordinates $(s,t)$ such that we can associate a  constant $t_*$ to any symmetric homogeneous degree 1 map $l(x)$  such that with respect to  coordinates $(s,t)$
%\begin{equation} \label{pouty}
%l(x)=(Ax^1,t_*).  
%\end{equation}
%We refer to the number $t_*$
%as the \emph{address of $l$}.
%Once we fix a particular symmetric homogeneous degree 1 map $l(x)$, then we apply a simple translation in the $t$-coordinate which results in new coordinates $(\varrho,\varphi)$ with respect to which  $l$ is expressed by (\ref{pout}).
First, observe that 
the  curve $\hat \gamma(\tau)=(\tau, \phi_*)$, with $\phi_*$ fixed, in $\overline{\bf H}_2$    is a geodesic line.  In $\overline{\bf H}$, there are no geodesic lines through $P_0$, only  geodesic rays  with  $P_0$ as an endpoint.  On the other hand, since $\overline{\bf H}_2$ is a union of two copies of $\overline{\bf H}$,  $\hat \gamma$ resembles the curve $\sigma$ constructed by joining  two geodesic rays in $\overline{\bf H}$.  More specifically, let 
\[
\sigma^{\phi_0}(\tau)=
\left\{
\begin{array}{ll} (\tau,\phi_0) & \mbox{for $\tau \in (0,\infty)$}\\
P_0 & \mbox{for $\tau=0$}\\
(-\tau,-\phi_0) & \mbox{for $\tau \in (-\infty,0)$}.
\end{array}
\right.
\]
Moreover, let $\gamma^{\phi_0}$ be the geodesic segment in $\overline{\bf H}$  from  $(1,\phi_0)$ to  $(1,-\phi_0)$.   Then since   
\[
\lim_{\phi_0 \rightarrow \infty} d((1,-\phi_0), (1,\phi_0)) = 2= \mbox{length}(\sigma^{\phi_0}\big|_{[-1,1]}),
\]
the geodesic  $\gamma^{\phi_0}$ resembles the  broken geodesic $\sigma^{\phi_0}\big|_{[-1,1]}$ 
 for $\phi_0>0$  large. 
(Details of this phenomenon is given in Section~\ref{sssec:sg}; specifically, see Lemma~\ref{symapp}.)
Therefore, the geodesic $\hat \gamma$ of $\overline{\bf H}_2$ resembles the geodesic  $\gamma^{\phi_0}$  of $\overline{\bf H}$ for $\phi_0>0$ large.  We use this property of geodesics to identify $\overline{\bf H}_2$ with $\overline{\bf H}$ as follows.  
%$\gamma\big|_{(\rho_*,\phi_*-c)}^{(\rho_*,\phi_*+c)}$ is a geodesic segment that starts at $(\rho_*,\phi_*-c)$ and terminates at $(\rho_*,\phi_*+c)$, 
% This hints that the set
%\[
%\overline{\bf H}_2[\phi_0]=\{(\rho,\phi) \in \overline{\bf H}_2:  |\phi|\leq \phi_0\}
%\]

Observe that  $\overline{\bf H}_2$ is foliated by an one-parameter family of geodesic lines $\{\rho \mapsto (\rho,\phi)\}$ (whose images are the horizontal lines in the left diagram of Figure~\ref{figure:cocER}).  
  Motivated by this, we also foliate ${\bf H}$ by a  family of geodesics (see in the right diagram of Figure~\ref{figure:cocER}).  We define a map which associates the family of geodesics in $\overline{\bf H}_2$ to the family of geodesics in $\overline{\bf H}$.  Indeed, let
 \begin{equation} \label{defofc}
 c=(c_{\rho},c_{\phi}):(-\infty,\infty) \times (-\infty,\frac{3}{2})\rightarrow {\bf H}
 \end{equation}
satisfying the following:  
\begin{itemize}
\item[(i)]
$s \mapsto c(s,t)=(c_{\rho}(s,t), c_{\phi}(s,t))$ is a unit speed geodesic such that 
\[
c_{\rho}(s,t)=c_{\rho}(-s,t), \ \ c_{\phi}(s,t)=-c_{\phi}(s,t).
\] 
\item[(ii)]   $t \mapsto  c_{\rho}(0,t)$  satisfies the equation \[
\frac{\partial c_{\rho}}{\partial t}(0,t)=c_{\rho}^3(0,t).
\]
\item[(iii)] $c_{\rho}(0,1)=1$ and $c_{\phi}(0,t)=0$ for all $t \in (-\infty,\frac{3}{2}).
$\end{itemize}

The parameters $s$ and $t$ define coordinates of ${\bf H}$ via the map
\[
(s,t) \mapsto c(s,t).
\]
Given a homogeneous degree 1 map $l(x)$ of the form $l(x)=c(Ax^1,t_*)$, we apply a translation by $t_*$ to construct coordinates $(\varrho,\varphi)$. 
More precisely, since 
\begin{equation} \label{atstar}
l(0)=(0,t_*) \ \mbox{in the coordinates $(s,t)$},
\end{equation} 
we  define  coordinates $(\varrho,\varphi)$ by setting 
\begin{equation} \label{cov}
(\varrho, \varphi)=(s,t-t_*).
\end{equation}
This results in
\[
l(0)=(0,0) \ \mbox{in coordinates $(\varrho,\varphi)$}.
\]
%Thus,  the construction of the coordinates  $(\varrho,\varphi)$ depends on $t_*$, and we will say that the coordinates $(\varrho,\varphi)$ are \emph{anchored} at $t_*$.
Using the new coordinates $(\varrho,\varphi)$ \emph{anchored at $t_*$}, we introduce a family of  totally geodesic subspaces of ${\bf H}$ which will play a central role in the proof of the key technical Lemma.

\begin{figure}[h] \label{figure:cocER}
    \centering
    \includegraphics[width=\textwidth]{g2.pdf}
    \caption{$\overline{\bf H}_2$ on the left and  $\overline{\bf H}$ on the right} \label{pic:varrhovarphi}
\end{figure}

Here, we emphasize that  the coordinates $(\varrho,\varphi)$ not only depend on the family of geodesics $\{s \mapsto c(s,t)\}$ but also on the parameter $t_*$.  We are interested in the asymptotics as $t_* \rightarrow -\infty$.

%More precisely, we want the metric expression of $g_{\bf H}$ with respect to $(\varrho,\varphi)$ as $t_* \rightarrow -\infty$ to resemble to the metric expression of $g_{{\bf H}_2}$  with respect to the coordinates $(\rho,\phi)$ which is
%\[
%g_{{\bf H}_2}(\rho,\phi) =\left(
%\begin{array}{cc}
%1 & 0\\
%0 & \rho^6
%\end{array}
%\right).
%\]
%In other words, we want as $t_* \rightarrow -\infty$
%\begin{eqnarray*}
%g_{\bf H}(\varrho,\varphi) 
% \approx   \left(
%\begin{array}{cc}
%1 & 0\\
%0 &  \varrho^6
%\end{array}
%\right).
%\end{eqnarray*}

%\vspace{1cm}
%\begin{figure}[h]\label{pic:effcont}
%    \centering
%    \includegraphics[width=\textwidth]{g3.pdf}
%    \caption{The non-shaded area shows the set of points whose image admits a $\delta \sigma$ neighborhood that does not intersect 
%   $X \backslash X_0$.}
%\end{figure}

%
%\begin{figure}[h]\label{pic:effcont} \label{figure:coc}
%\includegraphics [scale=0.5]{daskal6.pdf} 
%\caption{Geodesics of ${\bf H}_2$ (vertical lines on the diagram on the right-hand side) gets mapped to geodesics of ${\bf H}$.}
%\end{figure}
The expression of the metric $g_{\bf H}$ in the  coordinates $(\varrho,\varphi)$  is
\begin{eqnarray*}
g_{\bf H}(\varrho,\varphi) & = & \left(
\begin{array}{cc}
\left| \frac{\partial c}{\partial s}(\varrho,\varphi+t_*) \right|^2 & <\frac{\partial c}{\partial s}(\varrho,\varphi+t_*), \frac{\partial c}{\partial t}(\varrho,\varphi+t_*)>\\
<\frac{\partial c}{\partial s}(\varrho,\varphi+t_*), \frac{\partial c}{\partial t}(\varrho,\varphi+t_*)>& \left| \frac{\partial c}{\partial t}(\varrho,\varphi+t_*) \right|^2
\end{array}
\right) \\
& = &  \left(
\begin{array}{cc}
1 & 0\\
0 &  \left| \frac{\partial c}{\partial t}(\varrho,\varphi+t_*) \right|^2
\end{array}
\right).
\end{eqnarray*}
The top diagonal entry is equal to 1 because $s \mapsto u(s,t)$ is unit speed (cf.~(i)).
%which says $s\mapsto c_{\rho}(s,t)$ is a unit speed geodesic.  
The off-diagonal terms are equal to 0 because of the following reason:  First, note that the curve $t \mapsto c(0,t)$ parametrizes  the line $\phi=0$ by (ii) and (iii).  Next, since the  geodesic $s \mapsto c_{\rho}(s,t)$ is symmetric  in the variable $s$ by (i), its   minimum value is achieved at $s=0$.  In particular, $\frac{\partial c_{\rho}}{\partial s}(0,t)=0$ which in turn implies  $\frac{\partial c}{\partial s}(0,t)$   is parallel to the line $\rho=0$.  Therefore, we conclude that the Jacobi field $\frac{\partial c}{\partial t}$ is perpendicular to the velocity vector $\frac{\partial c}{\partial s}$ of the geodesic at $s=0$, and  they must be  perpendicular for all $s$ by a standard property of Jacobi fields.  This justifies that the off-diagonal entries are equal to 0.  
%
%The examination of the bottom diagonal entry is more subtle than that of the other entries and  explains the role of the differential equation given in (i). 
%First, note that a geodesic 
%$s \mapsto c(s, t)$  for $-t$  large (which amounts to considering geodesic in figure ??  toward  the bottom of the stack of geodesics in the diagram on the left) converges to the piecewise geodesic $s \mapsto \sigma(s,t)$ made up of  two geodesic curves $s \mapsto (-s,-\phi_0(t))$ for $s \in (-\infty,0]$ and $s \mapsto (s,\phi_0(t))$ for $s \in [0,\infty)$  as $t\rightarrow -\infty$.   In particular, since  $\sigma_{\rho}(s,t)=|s|$, we conclude that 
%$c_\rho(s,t) \rightarrow |s|$ and also 
%\[
%{\color{red} \frac{\partial c_{\rho}}{\partial t}(s,t) \rightarrow 0   \mbox{ uniformly for $s \in I$
%as $t \rightarrow -\infty$} }
%\]
% Thus, bottom diagonal entry  is $\left| \frac{\partial c}{\partial t}(s,t) \right|^2
% \approx \cancel{\left( \frac{\partial c_{\rho}}{\partial s}(s,t) \right)^2} {\color{red} \left( \frac{\partial c_{\phi}}{\partial t}(s,t) \right)^2}$ for $s \in I$ and $-t$ large.  If we assume 
% \begin{equation} \label{od}
% \cancel{ \frac{\partial c_{\rho}}{\partial t}(s,t)=c_{\rho}^3(s,t), \ \ \ }  {\color{red} \frac{\partial c_{\phi}}{\partial t}(s,t)=c_{\rho}^3(s,t),}
%  \end{equation}
%   then we can conclude $\left| \frac{\partial c}{\partial t}(s,t) \right|^2
% \approx  c_{\rho}^6(s,t)$ as  $t \rightarrow -\infty$.  Since  $c_{\rho}^6(s,t)
% \approx  s^6$ by (\ref{csr}), 
 The bottom diagonal term $\left| \frac{\partial c}{\partial t}\right|$ quantifies  how the family of geodesics $\{c_t(s)\}=\{c(s,t)\}$ are spread apart.  The differential equation $\frac{\partial c_{\rho}}{\partial t}(0,t)=c_{\rho}^3(0,t)$ of (ii)  gives the  initial spread (i.e.~the spread at $s=0$).  In \cite[Section 4]{daskal-meseER}, we have shown that this is enough to prove  $\left|\frac{\partial c}{\partial t}(s,t) \right| - c_{\rho}^3(s,t) \rightarrow 0$ uniformly for $s$ in a compact set away from $s=0$ as $t_* \rightarrow -\infty$. 
In summary,   in the coordinates $(\varrho,\varphi)$, $g_{\bf H}$ has the property that
  \begin{eqnarray}\label{C0exp}
\ \ g_{\bf H}(\varrho,\varphi) 
=   \left(
\begin{array}{cc}
1 & 0\\
0 &   \left| \frac{\partial c}{\partial t}(\varrho,\varphi+t_*) \right|^2\end{array}
\right) \approx   \left(
\begin{array}{cc}
1 & 0\\
0 &  \varrho^6
\end{array}
\right) \ \mbox{ as $t_* \rightarrow -\infty$}.
\end{eqnarray}
In an analogy with $\overline{\bf H}_2$, we showed in \cite{daskal-meseER} that the
totally geodesic subspaces 
\[
{\bf H}[\varphi_0] = \{(\varrho, \varphi) \in {\bf H}:  |\varphi| \leq \varphi_0\}
\] 
(pictured in the right diagram of  {\sc Figure}~2) are essentially regular, and we can set-up the  Gromov-Schoen inductive argument with $\overline{\bf H}[\varphi_0]$ as the totally geodesic set effectively containing the homogeneous degree 1 map $l(x)=(l_\varrho(x),l_\varphi)=(Ax^1,0)$.  With this, we can prove the regularity of harmonic maps into $\overline{\bf H}$ (cf.~\cite[Theorem 35]{daskal-meseER}).

As explained above, $\overline{\mathcal T}$ near a boundary point is asymptotically isometric to the product of a smooth K\"ahler manifold and the product of a finite number of copies of the model space. 
In this paper, we use the  strategy described above but also incorporating this almost product structure,  to prove the regularity  of harmonic maps into $\overline{\mathcal T}$ (cf.~Theorem~\ref{RegularityTheorem}).
%On the other hand, imposing assumption~(\ref{od}) along with (iii)  results in an over-determined set of conditions for $c$.  Thus, we impose (i) which controls how spread apart the geodesics are, and we show that $g_{\bf H}$ with respect to $(\varrho,\varphi)$ behaves as in (\ref{C0exp}).  
%Note that in (\ref{C0exp})  above, we only gave evidence for $C^0$-estimates of the Jacobi field $s \mapsto \frac{\partial c}{\partial t}(s,t)$. However in order to study harmonic maps we need  higher derivative estimates of the metric. Obtaining higher derivative estimates for the asymptotic behavior of the Jacobi field $s \mapsto \frac{\partial c}{\partial t}(s,t)$ is much more subtle and constitutes the majority of Section~\ref{jf}. 
%The discussion above is illustrated in the picture in the previous page.

\subsection{Summary of the Paper}     In the following paragraphs, we outline the  organization of the paper and explain the main ideas:\\

In Section~\ref{WPcomptionofTspace}, we discuss the asymptotic geometry of the Weil-Petersson completion $\overline{\mathcal T}$ of Teichm\"uller space. 
According to \cite{yamada}, \cite{daskal-wentworth}, \cite{wolpert} and \cite{daskal-meseC1}, the Weil-Petersson completion of a Teichm\"{u}ller space near a boundary point is  asymptotically isometric to the product of  a  boundary stratum ${\mathcal T}'$  and a normal space $\overline{\bf H}^{k-j}=\overline{\bf H} \times \dots \times \overline{\bf H}$.  We refer to  Section~\ref{sec:modelspace} for a precise definition of  the metric space $(\overline{\bf H},d_{\overline{\bf H}})$ given as a metric completion of the incomplete Riemann surface $({\bf H},g_{\bf H})$.  Since each open  boundary stratum ${\mathcal T}'$ can be identified with a  product of lower dimensional Teichm\"{u}ller spaces hence a  smooth Hermitian manifold, the singular behavior of the Weil-Petersson geometry  is completely captured by the model space $\overline{\bf H}$.    For one, the Gauss curvature of $g_{\bf H}$ approaches $-\infty$ near its boundary  reflecting the sectional curvature  blow-up of $G_{wp}$ near $\partial {\mathcal T}$.  Moreover, the non-local compactness of $\overline{\mathcal T}$  is also captured by $\overline{\bf H}$.  Indeed, a geodesic ball in $(\overline{\bf H},d_{\bf H})$ centered at a boundary point is not compact. The degenerating geometry and the lack of compactness \emph{imposes  severe challenges in the theory of harmonic maps} and the core of this paper is to deal with these phenomena.   In Section~\ref{sec:stratification}, we define a stratification preserving homeomorphism between a neighborhood ${\mathcal N} \subset \overline{\mathcal T}$ of a point  on a boundary stratum  and a neighborhood  in $\C^j \times \overline{\bf H}^{k-j}$.   In   Section~\ref{secC1}, we detail the precise way in which the Weil-Petersson metric  in ${\mathcal N}$  is asymptotically a product metric.  

In Section~\ref{MintoMS}, we prove the Regularity Theorem~\ref{RegularityTheoremModelSpace} for harmonic maps into the model space $(\overline{\bf H},d_{\overline{\bf H}})$.   The importance of this section is that, by considering $(\overline{\bf H},d_{\overline{\bf H}})$ as the target space, we isolate  the main difficulties (namely, the non-compactness and degenerating geometry)  that we will need to deal with  when the target space is  $\overline{\mathcal T}$.  Central to the proof is the notion of order of a harmonic map into an NPC space introduced in \cite{gromov-schoen}.  The order and other relevant notions from the theory of harmonic maps are summarized in Section~\ref{basic1}.   We remark that the order of a harmonic function is the order with which it attains its value;   equivalently, it is the degree of the monomial that best approximates it.  

The strategy of the proof of the Regularity Theorem~\ref{RegularityTheoremModelSpace} is to first prove that the set of higher order points  (i.e.~the set of points of order $>1$) is of Hausdorff codimension at least 2.   We then complete the proof by showing that \emph{no order 1 singular points exist}.  We do this  in Section~\ref{sec:into hbar1}  by applying the key technical Lemma for the Model Space (cf. Lemma~\ref{keylemma}), a special case of the key technical Lemma~\ref{keylemma'}.   This lemma  gives sufficient conditions for a map into $(\overline{\bf H},d_{\overline{\bf H}})$ not to hit the boundary point $P_0$.  The ideas surrounding the key technical Lemma is the lynchpin of the proof of the regularity theorem as we address the degeneration and non-compactness of the model space at $P_0$.  We note that the most technically difficult part, the proof the of the key technical Lemma~\ref{keylemma'} is postponed until Section~\ref{sec:  proofofkeylemma}.

In Section~\ref{intoWPcompletion}, we prove the Regularity Theorems~\ref{RegularityTheorem} and \ref{goto0} for harmonic maps into $\overline{\mathcal T}$.  The proof follows the similar strategy as for the proof of  Regularity Theorem~\ref{RegularityTheoremModelSpace} for the model space.  The first step of showing that the set of higher order points is of Hausdorff codimension at least 2 is done in  much the same way as in Section~\ref{MintoMS}.  On the other hand, the second step of dealing with the order 1 singular points is more difficult  because of the  complicated structure of the stratification for $\overline{\mathcal T}$.   
Nonetheless, the main issue  is  the same  for both  $\overline{\mathcal T}$ and  $\overline{\bf H}$, namely,  the  non-compactness and the degenerating geometry near the boundary.   We will again invoke the key technical Lemma~\ref{keylemma'}.  The idea is to use the asymptotic product structure of $(\overline{\mathcal T}, d_{\overline{\mathcal T}})$ near its boundary to decompose the given harmonic map into two maps, one of which maps into a boundary stratum (which is a smooth K\"{a}hler manifold) and the other   into the normal space $\overline{\bf H}^{k-j}$.  
These two maps are not harmonic because of the lack of product structure, but  the latter map is \emph{asymptotically harmonic} in an appropriate sense.   
We thus  adjust the arguments of Section~\ref{MintoMS} so that they work for   asymptotically harmonic maps.     For the reader's convenience, we will give a detailed  outline of this argument  at the beginning of Section~\ref{intoWPcompletion}.    

In Section~\ref{sec:  proofofkeylemma}, we  prove the key technical Lemma.  This can be thought of as the core of the paper and  the most technically challenging part of this work.

In Section~\ref{sec:dim2}, we specialize to the case when the domain dimension is 2.  In fact, we prove that there are no singular points in this case (cf. Theorem~\ref{surf} below).   
 
 In Section~\ref{proofrigiditytheorem0}, we prove our Theorem~\ref{rigiditytheorem0} and Corollary~\ref{rigiditytheorem1}. This follows fairly easily from  Theorem~\ref{RegularityTheorem} and  Theorem~\ref{goto0} by applying the result of \cite{siu}.
 
 In Section~\ref{applications1}, we deduce Theorem~\ref{rigiditytheorem} from our main Theorem~\ref{rigiditytheorem0} and Corollary~\ref{rigiditytheorem1}.
 Additionally, as a by-product,  we  provide a harmonic maps proof of 
Corollary~\ref{symTeich}.
 
  We would like to point out that for a  harmonic map $u$ defined on a general Riemannian domain Theorem~\ref{RegularityTheorem}  only asserts that the singular set ${\mathcal S}(u)$ of $u$ is of codimension at least 2 (or more precisely that $u$ maps a connected domain into a single stratum up to codimension at 2) and  does not necessarily imply  that $u$ maps into  of ${\mathcal T}$ (or even a single stratum). Our main theorem asserts that the stronger statement is true {\it{only}} when the harmonic map $u$ is  holomorphic. However, we show that for two dimensional domains, this assertion is always true. Namely,
\begin{theorem}\label{surf}
If  $u:(\Omega,g) \rightarrow( \overline{\mathcal T}, d_{ \overline{\mathcal T}})$ is a harmonic map from a connected Lipschitz domain  $\Omega$  in a Riemann  surface,  then there exists a single stratum ${\mathcal T}'$ of $\overline{\mathcal T}$ such that $u(\Omega) \subset  {\mathcal T}'$. 
\end{theorem}

It is reasonable to conjecture that this assertion holds  for higher dimensions; however, this is not needed for the applications discussed in this article. 

As a final comment, we would like to point out  that due to the length of this paper, we have omitted  several important topics that will be presented elsewhere. First is the connection with symplectic Lefschetz fibrations which, by Theorem~\ref{surf}, induce harmonic maps and in some cases even minimal surfaces  into the Teichm\"uller space.  More generally, our results imply  a classification theorem for surface fibrations over quasi projective varieties. Indeed,  Theorem~\ref{rigiditytheorem0} and Corollary~\ref{rigiditytheorem1} imply that, under mild  non degeneracy conditions on the rank of the harmonic map (which can be checked by topological considerations), any smooth fibration on a quasiprojective variety with quasiperiodic monodromy at infinity is isomorphic to a holomorphic fibration. (We would like to thank J. Jost for originally pointing this out to us.) In another direction, we would like to  remove Assumption (ii) on the codimension of the singular set of $\bar M$ from Theorem~\ref{rigiditytheorem}.  This Assumption was added by Jost and Yau in \cite{jost-yau}, in order to guarantee  the existence of  a finite energy map from $M$ to ${\mathcal T}$. The existence of  a finite energy map is also one of our Assumptions in Theorem~\ref{rigiditytheorem0} and Corollary~\ref{rigiditytheorem1}. It is possible that a more careful analysis would yield a finite energy map in general. However, in an upcoming article, we will circumvent this issue by considering infinite energy maps.

{\bf Acknowledgements.}  In the special case when the domain  
is a region in a Riemann surface,   it was first proved by R. Wentworth that the singular set of a harmonic map into the model space 
of Teichm\"{u}ller space is empty. 
 Although his method is strictly \emph{two dimensional} (using the Hopf differential associated with the harmonic map) and 
cannot be generalized to arbitrary dimensions, 
some of the preliminary results used in this paper (for example, the structure of limits of harmonic maps to the model space)
 have their origin in \cite{wentworth}. We would like to thank R. Wentworth for sharing his unpublished manuscript with us, W. Minicozzi for his continuous support of this project, B. Shiffman for the reference \cite{bernie} and R. Schoen, K. Uhlenbeck, S. Wolpert and S. T. Yau for sharing their insights on the subject with us. Additionally, we would also like to thank Victoria Gras Andreu for drawing Figures 1-3. Last, but not least, we would like to thank the referee for carefully reading the manuscript and making several suggestions that improved the exposition.

\section{The Weil-Petersson completion of Teichm\"uller space} \label{WPcomptionofTspace}
In this Section, we discuss the asymptotic geometry of the Weil-Petersson completion $\overline{\mathcal T}$ of Teichm\"uller space.  Near its boundary, $\overline{\mathcal T}$ is  asymptotically isometric to the product of  the Weil-Petersson metric on a stratum ${\mathcal T}'$ which is a product of  lower dimensional Teichm\"{u}ller spaces  and the normal space $\overline{\bf H}^{k-j}$ which is a product  of the model space $\overline{\bf H}$.  Moreover, the singular behavior of the Weil-Petersson geometry  is completely captured by the model space $\overline{\bf H}$.   In Section~\ref{sec:modelspace}, we collect several properties of the model space  $\overline{\bf H}$ that we will need later. 
%Some of its properties, for example its non-local compactness and its homogeneous structure, reflect the flavor of the proofs in later sections. 
 In Section~\ref{sec:stratification}, we define a stratification preserving homeomorphism between a neighborhood ${\mathcal N} \subset \overline{\mathcal T}$ of a point  on the boundary stratum  and a neighborhood  in $\C^j \times \overline{\bf H}^{k-j}$.  This homeomorphism will  be used to define local coordinates in ${\mathcal T}$.   In Section~\ref{secC1}, we give a precise description of the asymptotic product structure of the Teichm\"{u}ller space near its boundary.   Indeed, %Theorem~\ref{WP} 
Theorem~\ref{WP2} states the  $C^1$-estimates of the Weil-Petersson metric proved in \cite{daskal-meseC1}.   These estimates  
%generalizes the well-known $C^0$-estimates of \cite{yamada}, \cite{dw} and \cite{wolpert} and 
improve other $C^1$-estimates existing in the literature, for example  \cite{liusunyau1} and  \cite{liusunyau2}; more precisely, we show that the $C^1$-error term  of the Weil-Petersson metric is the derivative of the error appearing in the  well known $C^0$-estimates of \cite{yamada}, \cite{daskal-wentworth} and \cite{wolpert}.  
In Theorem~\ref{A2}, the $C^1$-estimates reformulated in   the precise way needed to apply the techniques developed in \cite{daskal-meseDM}.  (The other well-known estimates in the literature, for example \cite{wo4}, are to our knowledge insufficient for this purpose.)
 
\subsection{The Model Space} \label {sec:modelspace}
Consider the smooth Riemannian manifold $({\bf H}, g_{\bf H})$ consisting of the upper half plane
 \[
 {\bf H}=\{(\rho,\phi) \in \R^2: \rho >0, \phi \in \R\}\]
endowed with the Riemannian metric 
\[
g_{\bf H}=d\rho^2+ \rho^6 d\phi^2.
\]
(Note that in most literature on Weil-Petersson geometry, one considers the slightly different metric $4dr^2+ r^6 d\theta^2$ which is clearly isometric to $g_{\bf H}$  via the change of coordinates $ \rho=2r , \phi=\frac{\theta}{8}$.)   We call $(\rho,\phi)$ the \emph{standard model space coordinates} and $g_{\bf H}$   the \emph{model space metric}.
The Christoffel symbols of $g_{\bf H}$ are given by
\begin{equation} \label{christoffel}
\begin{array}{lcl}
\Gamma_{\rho \rho}^{\rho}=0 & & \Gamma_{\phi \phi}^{\phi} =  0 \\
\Gamma_{\rho \phi}^{\rho}=0 & & \displaystyle{ \Gamma_{\rho \phi}^{\phi}=\frac{3}{\rho}}
\\
\Gamma_{\phi \phi}^{\rho}=-3\rho^5 & &  \Gamma_{\rho \rho}^{\phi}=0.
\end{array}
\end{equation}
The Gauss curvature is
\[
K=-\frac{6}{\rho^2}.
\] 
The geodesic equations for $\gamma=(\gamma_{\rho},\gamma_{\phi})$  are given by the equations
\begin{eqnarray} \label{gequ}
\gamma^{\phi_0}_{\rho} \frac{d^2 \gamma^{\phi_0}_{\rho}}{ds^2} = 3(\gamma^{\phi_0}_{\rho})^6 \left( \frac{d\gamma^{\phi_0}_{\phi}}{ds} \right)^2 \mbox{ and } (\gamma^{\phi_0}_{\rho})^4 \frac{d^2\gamma^{\phi_0}_{\phi}}{ds^2} =-6 (\gamma^{\phi_0}_{\rho})^3 \frac{d\gamma^{\phi_0}_{\rho}}{ds} \frac{d\gamma^{\phi_0}_{\phi}}{ds}.
\end{eqnarray} 

Let $d_{\bf H}$ be the distance function of $\bf H$ induced by the metric $g_{\bf H}$; i.e.~for $P=(\rho,\phi), P =(\rho',\phi') \in {\bf H}$, 
let 
\[
d_{\bf H}(P,P') =\inf_{{\mathcal G}_{P,P'} }\mbox{length}(\gamma)
\]
where ${\mathcal G}_{P,P'} $ is the set of  all piecewise $C^1$ curves $\gamma:[a,b] \rightarrow {\bf H}$ with $\gamma(0)=P$ and  $\gamma(1)=P'$.
The metric space $({\bf H},g_{\bf H})$   is incomplete since  for any fixed $\phi_0 \in \R$, the geodesic 
\[
\gamma=(\gamma_{\rho},\gamma_{\phi}):[0,1) \rightarrow ({\bf H},g_{\bf H}), \ \ \gamma_{\rho}(t)=1-t, \gamma_{\phi}(t)=\phi_0
\]
 leaves every compact subset of ${\bf H}$ and is of length 1. 
On the other hand
\begin{lemma} 
$({\bf H},d_{\bf H})$ is geodesic; i.e.~for any $P,P' \in {\bf H}$, there exists a curve $\gamma \in {\mathcal G}_{P,P'} $ such that $d_{\bf H}(P,P')= \mbox{length}(\gamma)$.  
\end{lemma}
\begin{proof}
Suppose not. Then, there exist a  sequence $\gamma_i \in {\mathcal G}_{P,P'} $ and $t_i \in [0,1]$ such that length$(\gamma_i) \rightarrow d_{\bf H}(P,P')$ and $\hat{\rho}_i \rightarrow 0$ for $(\hat{\rho}_i, \hat{\phi}_i):=\gamma_i(t_i)$.  Since $\mbox{length}(\gamma_i) =\mbox{length}(\gamma_i\big|_{[0,t_i]}) +\mbox{length}(\gamma_i\big|_{[t_i,1]}) \geq (\rho_1-\hat{\rho_i})+ (\rho_2-\hat{\rho_i})$, this  implies $d_{\bf H}(P,P') \geq \rho_0+\rho_1$.  This is a contradiction; indeed,   if we let $\gamma_{\epsilon}$ be the join of the straight line from $P=(\rho,\phi)$ to $(\epsilon, \phi)$, followed by the straight line from $(\epsilon,\phi)$ to $(\epsilon,\phi')$ followed by the straight line from $(\epsilon,\phi')$ to $P'=(\rho',\phi')$, then 
$
\mbox{length}(\gamma_{\epsilon})= \rho-\epsilon+ \epsilon^3|\phi-\phi'|+\rho'-\epsilon
< \rho+\rho' \leq d_{\bf H}(P,P')$  for $\epsilon>0$ sufficiently small.
\end{proof}
The metric completion of $({\bf H},d_{\bf H})$ is denoted by $(\overline{\bf H}, d_{\overline{\bf H}})$.  Here, 
\[
\overline{\bf H} = {\bf H} \cup \{P_0\}
\] 
where we can think of  the entire axis $\rho=0$ is identified to a single point $P_0$.   The distance function $d_{\overline{\bf H}}: \overline{\bf H} \times \overline{\bf H} \rightarrow [0,\infty)$ is given by
$$
d_{\overline{\bf H}}(P,Q)=
\left\{
\begin{array}{cl}
d_{\bf H}(P,Q) & \mbox{if }P,Q \in {\bf H}\\
\rho & \mbox{if } P=P_0 \mbox{ and } Q=(\rho,\phi) \in {\bf H}.
\end{array}
\right.
$$ Since every neighborhood of $P_0$ contains points with arbitrary large $\phi$-coordinate, it follows that
the space  $(\overline{\bf H},d_{\overline{\bf H}})$ is not locally compact. \emph{This is the source of many technical hurdles in this paper.} However, $(\overline{\bf H}, d_{\overline{\bf H}})$ is an NPC space since it is a metric completion of a geodesically convex negatively curved surface.  
We also record the following two simple lemmas.
\begin{lemma} \label{distequiveuc}
If  $P_1, P_2 \in {\bf H}$ are given  as $P_1=(\rho_1,\phi_1)$ and $P_2=(\rho_2,\phi_2)$, then
\[
|\rho_1-\rho_2|  \leq d_{\overline{\bf H}}(P_1,P_2).
\]   
\end{lemma}

\begin{proof} 
Let $\gamma$ be the geodesic from $P_1$ to $P_2$.
Let $\pi$ be the projection map onto the geodesic $\{\phi=\phi_1\} \cup \{P_0\}$.   Then $\pi$ is distance decreasing and 
$
d(P_1,P_2)= \mbox{length}(\gamma) \geq \mbox{length}(\pi \circ \gamma_1) \geq |\rho_1-\rho_2|.
$
\end{proof}

\begin{lemma}  \label{tangentcone}
The tangent cone $T_{P_0}\overline{\bf H}$ is isometric to $[0,\infty)$.
\end{lemma}

\begin{proof}  First, note that any unit speed geodesic emanating from $P_0$ is of the form
\[
\gamma_{\phi_0}:  [0,\infty) \rightarrow \overline{\bf H}, \ \ \gamma_{\phi_0}(t) =
\left\{
\begin{array}{ll}
P_0 & \mbox{for }t=0\\
(t, \phi_0) & \mbox{for }t>0
\end{array}
\right.
\]
for some fixed $\phi_0 \in (-\infty,\infty)$.  Comparing the length of the geodesic from $(\epsilon,\phi_0)$ to $(\epsilon,\phi_0')$ to the length of the vertical line  from $(\epsilon,\phi_0)$ to $(\epsilon,\phi_0')$,
%$t \mapsto (\epsilon, (1-t)\phi_1+t\phi_2)$ for $0\leq t\leq 1$, 
we obtain
\[
d(\gamma_{\phi_0}(\epsilon), \gamma_{\phi_0'}(\epsilon)) \leq \epsilon^3 |\phi_0-\phi_0'|.
\]
Thus, the angle   between the two geodesics  $\gamma_{\phi_0}$ and $\gamma_{\phi_0'}$  at $P_0$ is given by 
\[
\angle_{P_0}(\gamma_{\phi_0}, \gamma_{\phi_0'})=\lim_{\epsilon \rightarrow 0} \cos^{-1} \frac{2\epsilon^2 - d^2(\gamma_{\phi_0}(\epsilon), \gamma_{\phi_0'}(\epsilon))}{2\epsilon^2} \leq \lim_{\epsilon \rightarrow 0} \cos^{-1} \frac{2\epsilon^2 - \epsilon^6|\phi_0-\phi_0'|^2}{2\epsilon^2}=0.
\]
It follows that the space of directions at $P_0$ (i.e.~the space of equivalence classes of geodesics emanating from $P_0$) contains exactly one element.  Since the tangent cone is the metric cone over the space of directions, it is isometric to $[0,\infty)$.
\end{proof}

Another important feature of the space $({\bf H}, g_{\bf H})$ is that it possesses a {\it homogeneous structure}.
More precisely, we can define new coordinates  $(\rho,\Phi)$ of ${\bf H}$ where the first coordinate function $\rho$ is the same as that of the original coordinates, but  the second coordinate function $\Phi$  defined  by setting
 \[
 \Phi=\rho^3 \phi.
\]
We call $(\rho,\Phi)$ {\it the homogeneous coordinates} and in these coordinates
  the metric is given by
\begin{eqnarray}\label{equatinhomcord}
g_{\bf H} & = & 
\left( \begin{array}{cc}
1+9\Phi^2 \rho^{-2}  &  -3\rho^{-1} \Phi  \\
& \\
 - 3\rho^{-1} \Phi  & 1
\end{array}
\right).
\end{eqnarray}
For $\lambda >0$ consider the \emph{dilation map}
\[
\lambda: \overline{\bf H} \rightarrow \overline{\bf H} 
\]
given in homogeneous coordinates by
\[
 P=(\rho, \Phi) \mapsto \lambda P=(\lambda\rho, \lambda \Phi) \mbox{ and } P_0 \mapsto P_0.
\]
It follows immediately from (\ref{equatinhomcord}) that the local expression of $g_{\bf H}$ is invariant under dilations. This implies that if we extend the dilation map $\lambda$ to $\overline {\bf H}$ by 
\begin{equation} \label{hcs}
 \lambda P= \left\{ \begin{array}{cl} \lambda P &  \mbox{ if $P \neq P_0$} \\ 
P_0 & \mbox{ if $P=P_0$}
\end{array}
\right.
\end{equation}
then the   distance function $d_{\overline{\bf H}}$   is homogeneous of degree 1; i.e.
\begin{equation} \label{dfih}
d_{\overline{\bf H}}(\lambda P, \lambda Q) =\lambda d_{\overline{\bf H}}(P,Q), \ \ \ \forall P, Q \in \overline{\bf H}.
\end{equation}
%We note that in the original coordinates $(\rho,\phi)$ of $Y$, we have
%\begin{equation} \label{xscalar}
%\lambda (\rho,\phi) = (\lambda \rho, \lambda^{-2} \phi).
%\end{equation} 

The stratification of  $\overline{\bf H}={\bf H} \cup \{P_0\}$ induces a stratification on the product space $  \overline{\bf H}^l$ for any positive integer $l$. The metric 
$g_{\bf H}$ defines a metric $h$ on the stratified space $  \overline{\bf H}^l$ so that $(\overline{\bf H}^l,h)$ becomes a stratified Hermitian space. The distance function $d_h$ induced by $h$ coincides with the completion of the distance function on ${\bf H}^l$ induced from the metric $g_{\bf H}$.
\begin{definition} \label{defprodh}
For a positive integer $l$, we refer to the stratified Hermitian space $(\overline{\bf H}^l,h)$ and the NPC metric space $(\overline{\bf H}^l, d_h)$
as 
\emph{the normal space} (to the boundary of Techm\"{u}ller space).  This terminology will be justified in Theorem~\ref{A2} below.
\end{definition}

In the following, we summarize the properties of the normal space.  
%In particular, Assumption 1 of \cite{daskal-meseDM} holds.

\begin{proposition}[{\bf{Homogeneous structure of  the Model Space}}]  \label{A1}
The metric space  $(\overline{\bf H}^{k-j},d_h)$ is an NPC space with  a homogeneous structure with respect to ${\mathbb P}_0 = (P_0, \dots, P_0) \in \overline{\bf H}^{k-j}$. In other words, there is a continuous map 
  \[
  \R_{> 0} \times \overline{\bf H}^{k-j} \rightarrow \overline{\bf H}^{k-j},\ \ \  (\lambda, P) \mapsto \lambda P
  \]
 such that $\lambda {\mathbb P}_0={\mathbb P}_0$ for every $\lambda>0$ 
and the distance function $d_h$ is homogeneous of degree 1 with respect to this map, i.e. 
\[
d_h(\lambda P, \lambda Q) =\lambda d_h(P,Q), \ \ \ \forall P, Q \in \overline{\bf H}^{k-j}, \  \lambda \in (0,\infty).
\]
\end{proposition}

\begin{proof}
Indeed, using the homogeneous structure on $\overline{\bf H}$ defined by (\ref{hcs}), we can define a continuous map $\R_{> 0} \times \overline{\bf H}^{k-j} \rightarrow \overline{\bf H}^{k-j}$ by setting
 \[
(\lambda, (P^1, \dots, P^{k-j})) \rightarrow (\lambda P^1, \dots, \lambda P^{k-j}).
\]
\end{proof}

\subsection{Local coordinates of ${\mathcal T}$ near $\partial {\mathcal T}$}

%The stratification of $\overline{\mathcal T}$ } 
\label{sec:stratification}
Let  ${\mathcal T}$ denote  the Teichm\"uller space of  an oriented   surface  of genus $g$  with $p$ marked points  such that $k=3g-3+p>0$.  Endowed with the Weil-Petersson metric $G_{wp}$, $({\mathcal T}, G_{wp})$ is a smooth K\"ahler manifold   of complex dimension $k=3g-3-p$ (cf. \cite{ahlfors}) and has negative sectional curvature (cf. \cite{tromba} and \cite{wolpert0}). However, $({\mathcal T}, G_{wp})$  is   \emph{incomplete} (cf. \cite{chu} and \cite{wolpertPJ}).  Let $(\overline{\mathcal T}, d_{\overline{\mathcal T}})$  denote its metric completion. The metric space $(\overline{\mathcal T}, d_{\overline{\mathcal T}})$ is no longer a smooth manifold, but it is an NPC metric space (cf. \cite{daskal-wentworth}, \cite{wolpert} and \cite{yamada}). Furthermore,
 $\overline{\mathcal T}$ is a stratified space (cf. \cite{masur}), sometimes called \emph{the augmented Teichm\"uller} space; more precisely, we can write 
\begin{equation}\label{strat'}
\overline{\mathcal T}= \bigcup   {\mathcal T}'
\end{equation}
where ${\mathcal T}'={\mathcal T}$ or ${\mathcal T}'$ is the space  parametrizing nodal surfaces  obtained from the original  surface  with a number of  (mutually disjoint) simple closed curves pinched.  (One can show that ${\mathcal T}'$ is a product of lower dimensional Teichm\"uller spaces.)
We call ${\mathcal T}'$ an \emph{open stratum} of $\overline{\mathcal T}$. Recall  that all the strata are totally geodesic with respect to the Weil-Petersson distance (cf. \cite{daskal-wentworth}, \cite{wolpert} and \cite{yamada}). 

Define $\#:  \overline{\mathcal T} \rightarrow \{0,\dots, k\}$ by setting
\begin{equation} \label{numba}
\#P=j
\end{equation}
if $P \in {\mathcal T}'$ where ${\mathcal T}'$ is a $j$-dimensional open stratum. 
Consider $P \in \overline{\mathcal T}$ with $\#P=j$  corresponding to a nodal surface $R_0$.  Let  $s=(s_1,\dots, s_j) \in \C^j  \mapsto R_s$ be a parameterization of the neighborhood of $R_0$ in ${\mathcal T}'$.  We can regularize each node of $R_s$ by the  plumbing construction, and let  $t=(t_1, \dots, t_{k-j}) \in \C^{k-j}$ denote  the  {\it plumbing coordinates}.    Thus,  provided that all the $t_i's$ are nonzero, we can construct an analytic family of Riemann surfaces $R_{s,t}$ of genus $g$  with $p$ marked points degenerating as $(t_1, \dots, t_{k-j}) \rightarrow (0, \dots, 0)$ to the nodal surface $R_s$. The parameters $s$ and $t$ together define a set of coordinates  on ${\mathcal T}$ near $P$ (see \cite{abikoff}, \cite{masur}, \cite{yamada}, \cite{daskal-wentworth} and \cite{wolpert} for further details).

The parameter $t$ gives rise to the normal space ${\bf H}^{k-j}$.  Indeed, we define
\begin{equation} \label{tonH}
t=(t_1, \dots, t_{k-j}) \in  \C^{k-j} \mapsto ((\rho_1, \phi_1),..., (\rho_{k-j}, \phi_{k-j}))\in {\bf H}^{k-j}
\end{equation}
 by setting  
\[
\rho_i=2(-\log |t_i|)^{-\frac{1}{2}} \ \mbox{and} \  \phi_i=\frac{1}{8} \arg t_i.
\]
The stratification of the space $\C^j \times  \overline{\bf H}^{k-j}$ induced from the stratification of $\overline{\bf H}={\bf H} \cup \{P_0\}$  is compatible with the stratification on $\overline{\mathcal T}$ given in (\ref{strat'}). More precisely, given $P \in \overline {\mathcal T}$ with $\#P=j$ (cf. (\ref{numba})),  there exists a  neighborhood ${\mathcal N} \subset \overline{\mathcal T}$ of $P$, a neighborhood ${\mathcal U} \subset \C^j$ of ${\mathcal O}=(0,\dots,0)$, a neighborhood ${\mathcal V} \subset \overline{\bf H}^{k-j}$ of ${\mathbb P}_0=(P_0, \dots, P_0)$ such that the map  
 \begin{equation} \label{Psihomeo}
\Psi: {\mathcal N} \rightarrow  {\mathcal U} \times {\mathcal V} \subset \C^j \times  \overline{\bf H}^{k-j}
 \end{equation}
 given in terms of the parameters described above as
    \[  
  Q  \mapsto \Psi(Q)=(s_1,...s_j, (\rho_1, \phi_1),...,(\rho_{k-j}, \phi_{k-j}))
  \]
has the following properties:
 \begin{itemize}
 \item[$(i)$] $ \Psi(P)=({\mathcal O},{\mathbb P}_0)=(0, \dots, 0,P_0, \dots, P_0) \in \C^j \times  {\bf H}^{k-j}$. 
 \item[$(ii)$] $\Psi$ is a stratification preserving homeomorphism and when restricted to each open stratum is a  biholomorphism, hence:
  \item[$(iii)$]  If $G$ denotes the pullback of the Weil-Petersson metric $G_{WP}$ under $ \Psi^{-1}$,
 then  $G$ is a Hermitian metric along each stratum of ${\mathcal U} \times {\mathcal V}$
 such that 
\[
 \Psi: ({\mathcal N}, G_{WP})\rightarrow  ({\mathcal U} \times {\mathcal V}, G)
 \] 
 is a Hermitian isometry between stratified spaces.
 In particular $ \Psi$ induces an isometry $ \Psi: ({\mathcal N}, d_{G_{WP}})\rightarrow  ({\mathcal U} \times {\mathcal V}, d_{G})$, where   $d_{G_{WP}}$ and $d_G$
 denote the distance functions defined by $G_{WP}$ and $G$ respectively.
 \end{itemize}

Throughout the rest of the paper, we will use the map $ \Psi$ as local coordinates near a $j$-dimensional stratum and express the Weil-Petersson metric in terms of $\Psi$.
 Using the natural identification  ${\mathcal U}={\mathcal U} \times \{P_0 \}$, let $H$ denote any smooth extension of  $G \Big|_{{\mathcal U} \times \{P_0 \}}$ from ${\mathcal U}$  to $\C^j$. Let $V=(V^1, \dots, V^j)$ be normal coordinates of $H$ near 0 and assume without loss of generality that they are the restriction of the standard coordinates on $\C^j$. Let $h$ denote the metric on the statified space $\overline{\bf H}^{k-j}$ as in Definition~\ref{defprodh}.

It is a straightforward computation to show that in terms of the complex parameter $t=(t_1, \dots, t_{k-j})$ of ${\bf H}^{k-j}$ given by (\ref{tonH}), the Hermitian metric $h$ has the expression 
\[
h=(h_{i\bar{j}}) \mbox{ where }   h_{i\bar{j}} =
\left\{
\begin{array}{cll}
  \pi^3 |t_i|^{-2} (-\log |t_i|)^{-3} & \mbox{ for } & i=j\\
0 &\mbox{ for } & i \neq j.
 \end{array}
 \right.
\]
The co-metric is
\[
h=(h^{i\bar{j}}) \mbox{ where }   h^{i\bar{j}} =
\left\{
\begin{array}{cll}
 \displaystyle{  \pi^{-3} |t_i|^{2} (-\log |t_i|)^{3}}   & \mbox{ for } & i=j\\
0 &\mbox{ for }  & i \neq j.
 \end{array}
 \right.
\]
 \begin{definition} \label{productmetric}
The metric  $G$ above will again be called the \emph{Weil-Petersson metric}.  Additionally, with  $H$ and $h$ as above, metric $H \oplus h$ will be called the \emph{product metric}.
 \end{definition}

\subsection{$C^0$ and $C^1$-estimates of Weil-Petersson metric}
\label{secC1}
The $C^0$-asymptotic behavior of the Weil-Petersson metric near the boundary of Teichm\"uller space is given by the well-known estimates below. Notice  that we use the {\it upper case $I, J$  to index the $s$-coordinates and  lower case $i, j,k$ for the $t$-coordinates}.  

\begin{theorem}[{\bf \cite{daskal-wentworth}, \cite{masur},  {\cite{yamada}, \cite{wolpert}}}]\label{dwyw}
The Weil-Petersson co-metric $G_{WP}^{-1}=(G^{**})$ 
  satisfies the following estimates (assuming $i,j,k$ are all distinct):
 \begin{eqnarray*} 
(i') \ \ \ \ G^{i\overline i}  &=&   h^{i\overline i} \left( 1+  O\left(\sum_{l=1}^{k-j} (-\log|t_l|)^{-2}\right)\right)\\
(ii')  \ \ \ \ G^{j\overline k}&=& O(|t_j||t_k|)\ \ \\
(iii') \ \ \ \  G^{I\overline j} &=& O(|t_j|) \\
(iv') \ \ \   G^{I\overline J} & = & G^{I\bar{J}}(0)+  O\left(\sum_{l=1}^{k-j} (-\log|t_l|)^{-2}\right). 
 \end{eqnarray*} 
The Weil-Petersson metric  $G_{WP}=(G_{**})$  satisfies the following estimates  (assuming  $i,j,k$ are all distinct):
 \begin{eqnarray*} 
(i) \ \ \ \ G_{i\overline i}  &=&   h_{i\overline i} \left( 1+  O\left(\sum_{l=1}^{k-j} (-\log|t_l|)^{-2}\right)\right)\\
(ii)  \ \ \ \ G_{j\overline k}&=& O\left( (-\log|t_j|)^{-3}(-\log|t_k|)^{-3} |t_j|^{-1} |t_k|^{-1}\right) \\
(iii) \ \ \ \  G_{I\overline j} &=& O\left(|t_j|^{-1} (-\log|t_j|)^{-3} \right) \\
(iv) \ \ \   G_{I\overline J} & = & G_{I\bar{J}}(0)+  O\left(\sum_{l=1}^{k-j} (-\log|t_l|)^{-2}\right). 
\end{eqnarray*}
\end {theorem}

The $C^0$-estimates above are not strong enough for the proof of Theorem~\ref{rigiditytheorem0}.  Indeed, in \cite{daskal-meseDM},  we developed  a general harmonic map theory in the setting where the target space has   a $C^1$-asymptotic product structure.  Subsequently, in  \cite{daskal-meseC1}  we proved the  asymptotic  $C^1$-estimates for the Weil-Petersson metric  suited  for the  techniques of \cite{daskal-meseDM}.  These estimates (cf.~Theorem~\ref{WP} and Theorem~\ref{WP2} below) give a more precise description of the asymptotic product structure than the  ones given in  \cite{schumacher}, \cite{huang},  \cite{liusunyau1}, \cite{liusunyau2} and \cite{liusunyau3}.  In particular,  our results  estimate the derivatives of the difference between  the Weil-Petersson metric $G_{WP}$ and the model metric $h$  and can be summarized as follows:  
\begin{quote}
\emph{The $C^1$-error terms of the co-metric is of the same order as the derivative of $C^0$-error terms.}
 \end{quote}
Our results in \cite{daskal-meseC1} also differ from the ones in \cite{wo4}  in the sense that they are expressed in terms of local coordinates on ${\mathcal T}$. Notice that
Wolpert expresses his asymptotic estimates in terms of a certain frame given by gradients of geodesic length functions, but unfortunately this frame  does not come from a set of local coordinates on Teichm\"{u}ller space. It is not clear to the authors how to use Wolpert's estimates in conjunction with harmonic maps.   In the estimates below, we again use the {\it upper case $I, J, K$  to index the $s$-coordinates and  lower case $i, j, k$ for the $t$-coordinates}. 

\begin{theorem}[ {\bf{\cite{daskal-meseC1} Theorem 2}}] \label{WP} The Weil-Petersson co-metric $G_{WP}^{-1}=(G^{**})$ satisfies the following estimates (assuming $i,  j, k$ are all distinct):   %\neq i$ and $I \neq J \neq K \neq I$): 
\begin{eqnarray*}
(i) \ \ \ \frac{\partial }{\partial t_i} G^{i \overline i}& = &  \frac{\partial }{\partial t_i}{h}^{ii}+O\left(|t_i| (-\log|t_i|) \right)\\
(ii)  \ \ \ \frac{\partial }{\partial t_i} G^{j \overline j}&=& O\left(|t_i|^{-1}(-\log |t_i|)^{-3} |t_j|^{2}(- \log|t_j|)^{3}\right)\\
(iii)  \ \ \ \frac{\partial }{\partial t_i} G^{i \overline j}&=&O\left(|t_j|\right) \\
(iv)  \ \ \ \frac{\partial }{\partial t_i} G^{j \overline k}&=& O\left(|t_i|^{-1} (-\log|t_i|)^{-3} |t_j| |t_k|\right) \\
(v) \ \ \  \frac{\partial }{\partial t_i}  G^{I \overline J} & = & O\left(|t_i|^{-1} (-\log|t_i|)^{-3}\right) \\ 
(vi) \ \ \ \frac{\partial }{\partial t_i} G^{j \overline I} &=& O\left(|t_i|^{-1} (-\log|t_i|)^{-3} |t_j|\right) \\
(vii) \ \ \ \frac{\partial }{\partial t_i} G^{I \overline j} &=& O\left(|t_i|^{-1} (-\log|t_i|)^{-3} |t_j|\right) \\
(viii) \ \ \ \frac{\partial }{\partial t_i} G^{I \overline i} &=& O(1). 
 \end{eqnarray*}
\end{theorem}
We also record the following estimates of Liu, Sun and Yau, to get the complete picture of the $C^1$-asymptotic behavior of the Weil-Petersson co-metric.
\begin{theorem}[ {\bf{\cite{liusunyau3}, formula (3.16)}}]\label{WP3} The Weil-Petersson co-metric satisfies the following estimates (assuming $i,j$ are distinct):
\begin{eqnarray*}
(i) \ \ \ \frac{\partial }{\partial s_I} G^{i \overline i}& = &  O\left(|t_i|^2 (-\log|t_i|)^3 \right)\\
(ii)  \ \ \ \frac{\partial }{\partial s_I} G^{i \overline j}&=&O\left(|t_i| |t_j|\right)\\
(iii) \ \ \  \frac{\partial }{\partial s_I}  G^{J \overline i} & = & O\left(|t_i|\right) \\ 
(vi) \ \ \ \frac{\partial }{\partial s_I} G^{i \overline J} &=& O\left(|t_i| \right) \\
(v) \ \ \ \frac{\partial }{\partial s_I} G^{J \overline K} &=& O(1). 
 \end{eqnarray*}
\end{theorem}
By inverting the matrix $G^{i\overline j}$ and combining the above three theorems,  we obtain the $C^1$-estimates of the Weil-Petersson metric.

\begin{theorem}[ {\bf{\cite{daskal-meseC1}, Theorem 3}}]\label{WP2} The Weil-Petersson metric satisfies the following $t$-derivative  estimates (assuming $i,  j, k, $ are all distinct): 
\begin{eqnarray*}
(i) \ \ \ \frac{\partial }{\partial t_i} G_{i\overline i}& = &  \frac{\partial }{\partial t_i}{h}_{i\overline i}+O\left(|t_i|^{-3}(-\log |t_i|)^{-5} \right)\\
(ii)  \ \ \ \frac{\partial }{\partial t_i} G_{j\overline j}&=& O\left(|t_i|^{-1}(-\log |t_i|)^{-3} |t_j|^{-2}(- \log|t_j|)^{-3}\right) \\
(iii)  \ \ \ \frac{\partial }{\partial t_i} G_{i \overline j}&=&O\left(|t_i|^{-2} (-\log|t_i|)^{-3}(|t_j|^{-1} (-\log|t_j|)^{-3}\right) \\
(iv)  \ \ \ \frac{\partial }{\partial t_i} G_{j\overline k}&=& O\left(|t_i|^{-1} (-\log|t_i|^{-3})(|t_j|^{-1} (-\log|t_j|)^{-3}(|t_k|^{-1} (-\log|t_k|)^{-3} \right)  \\
(v) \ \ \  \frac{\partial }{\partial t_i}  G_{I\overline J} & = & O\left(|t_i|^{-1} (-\log|t_i|)^{-3}\right) \\ 
(vi) \ \ \ \frac{\partial }{\partial t_i} G_{I\overline j} &=& O\left(|t_i|^{-1} (-\log|t_i|)^{-3} |t_j|^{-1} (-\log|t_j|)^{-3} \right)  \\
(vii) \ \ \ \frac{\partial }{\partial t_i} G_{j\overline I} &=& O\left(|t_i|^{-1} (-\log|t_i|)^{-3} |t_j|^{-1} (-\log|t_j|)^{-3} \right)  \\
(viii) \ \ \ \frac{\partial }{\partial t_i} G_{I\overline i} &=& O\left(|t_i|^{-2} (-\log|t_i|)^{-3} \right)
\end{eqnarray*}
\end{theorem}

%\begin{proof}
%We use the fact that
%\[
%h_{i\bar{i}} = O(|t_i|^{-2} (-\log |t_i|)^{-3})
%\]
%and the standard formula for the derivative of the metric in terms of the derivative of the co-metric.  To prove formula (i) for example, we have
%\begin{eqnarray*}
%\frac{\partial }{\partial t_i} G_{i\bar{i}} & = &  
%G_{i\bar{i}}\frac{\partial }{\partial t_i} G^{i \overline i} \overline{G_{i\bar{i}}}
%+
%\sum_{j \neq i} G_{i\bar{j}}
%\frac{\partial }{\partial t_i} G^{j \overline j}
%\overline{G_{j\bar{i}}}
%+
%\sum_{j \neq i} G_{i\bar{i}}
%\frac{\partial }{\partial t_i} G^{i \overline j} 
%\overline{G_{j\bar{i}}}
%\\
%& & \ + 
%\sum_{j,k \neq i} G_{i\bar{j}}\frac{\partial }{\partial t_i} G^{j \overline k} \overline{G_{k\bar{i}}}
%+
%\sum_{I,J}G_{i\bar{I}}\frac{\partial }{\partial t_i}  G^{I \overline J} \overline{G_{J\bar{i}}}
%+
%\sum_{j \neq i, I} G_{i\bar{j}} \frac{\partial }{\partial t_i} G^{j \overline I} 
%\overline{G_{I \bar{i}}}
%\\
%& & \ \ +
%\sum_{I, j \neq i} G_{i\bar{I}} \frac{\partial }{\partial t_i} G^{I \overline j} 
%\overline{G_{j\bar{i}}}
%+
%\sum_I G_{i\bar{I}}
%\frac{\partial }{\partial t_i} G^{I \overline i}
%\overline{G_{i\bar{i}}},
%\end{eqnarray*}
%and the terms on the right hand side  above can be estimated by Theorem~\ref{dwyw} formulas (i)-(iv) and Theorem~\ref{WP} formulas (i)-(v).
%\end{proof}

\begin{theorem}\label{WP4} The Weil-Petersson metric satisfies the following $s_I$-derivative estimates (we are not assuming  $i,  j$ are distinct): 
\begin{eqnarray*}
%(i)  \ \ \ \frac{\partial }{\partial s_I} G_{i\overline i}&=& O\left( |t_i|^{-2}(- \log|t_i|)^{-3}\right) \\
(i)  \ \ \ \frac{\partial }{\partial s_I} G_{i\overline j}&=& O\left((|t_i|^{-1} (-\log|t_i|)^{-3}(|t_j|^{-1} (-\log|t_j|)^{-3} \right)  \\
(ii) \ \ \  \frac{\partial }{\partial s_I}  G_{J\overline K} & = & O\left(1\right) \\ 
(iii) \ \ \ \frac{\partial }{\partial s_I} G_{J\overline j} &=& O\left( |t_j|^{-1} (-\log|t_j|)^{-3} \right)  \\
(iv) \ \ \ \frac{\partial }{\partial s_I} G_{j\overline J} &=& O\left( |t_j|^{-1} (-\log|t_j|)^{-3} \right)  \\
 \end{eqnarray*}
\end{theorem}

%\begin{proof}
%We proceed similarly to the proof of  Theorem~\ref{WP2},  To prove (i) for example, we have
%\begin{eqnarray*}
%\frac{\partial }{\partial s_I} G_{i\bar{j}} & = & \sum_{k} G_{i\bar{k}}
%\frac{\partial }{\partial s_I}G^{k \overline k}
%\overline{G_{k\bar{j}}}
%+
%\sum_{k,l \neq k} G_{i\bar{k}}
%\frac{\partial }{\partial s_I}G^{k \overline l} 
%\overline{G_{l\bar{j}}}
%+
%\sum_{J,k}G_{i\bar{J}}\frac{\partial }{\partial s_I} G^{J \overline k} \overline{G_{k\bar{j}}}
%\\
%& & \ 
%+
%\sum_{k, J} G_{i\bar{k}} \frac{\partial }{\partial s_I}G^{k \overline J} 
%\overline{G_{J \bar{j}}}
% + \sum_{J, K} G_{i\bar{J}}
%\frac{\partial }{\partial s_I}G^{J \overline K} 
%\overline{G_{K\bar{j}}},
%\end{eqnarray*}
%and the terms on the right hand side  above can be estimated by Theorem~\ref{dwyw} formulas (i)-(iv) and Theorem~\ref{WP} formulas (i)-(viii).
%\end{proof}
%
%
%***
%\begin{eqnarray*}
%\frac{\partial }{\partial s_I} G_{J\bar{j}} & = & \sum_{k} G_{J\bar{k}}
%\frac{\partial }{\partial s_I}G^{k \overline k}
%\overline{G_{k\bar{j}}}
%+
%\sum_{k,l \neq k} G_{J\bar{k}}
%\frac{\partial }{\partial s_I}G^{k \overline l} 
%\overline{G_{l\bar{j}}}
%+
%\sum_{K,k}G_{J\bar{K}}\frac{\partial }{\partial s_I} G^{K \overline k} \overline{G_{k\bar{j}}}
%\\
%& & \ 
%+
%\sum_{k, J} G_{J\bar{k}} \frac{\partial }{\partial s_I}G^{k \overline K} 
%\overline{G_{K \bar{j}}}
% + \sum_{K, J} G_{J\bar{K}}
%\frac{\partial }{\partial s_I}G^{K \overline L} 
%\overline{G_{L\bar{j}}},
%\end{eqnarray*}
%***

In the next corollary, we reformulate the estimates in Theorem~\ref{dwyw},  Theorem~\ref{WP}, Theorem~\ref{WP2} and Theorem~\ref{WP4} in terms of the metrics $G$, $H$ and $h$. Again, we use the {\it upper case $I, J, K$  to index the $V=(V^1, \dots, V^j)$-coordinates of $\C^j$ and  lower case $i,l,m$ to index the $v=(v^1, \dots, v^{k-j})$ coordinates of $\overline{\bf H}^{k-j}$.}

\begin{proposition} [{\bf{$C^1$-Asymptotic Product structure of the WP-Metric}}] \label{A2}  The Weil-Petersson metric $G$  is asymptotically the product metric
$
H \oplus h
$
of Definition~\ref{productmetric} in the following sense:\\    
\\
Let $V=(V^1, \dots,V^j)$ be coordinates for $\C^j$  and $v=(v^1, \dots, v^{k-j})$ be  coordinates for  
${\bf H}^{k-j}$.
 There exists a constant  $C>0$  such that if we write
\begin{eqnarray*}
H(V) = (H_{I \overline J}(V)), & & H^{-1}(V) = (H^{I \overline J}(V)), \\
h(v) = (h_{i \overline l}(v)), & & h^{-1}(v) = (h^{i \overline l}(v)),\\
G(V,v)  =  \left( 
\begin{array}{cc}
G_{I \overline J}(V,v) & G_{I \overline l}(V,v)\\
G_{l \overline J}(V,v) & G_{i \overline l}(V,v)
\end{array}
\right),
& &
G^{-1}(V,v)  =  \left( 
\begin{array}{cc}
G^{I \overline J}(V,v) & G^{I \overline l}(V,v)\\
G^{l \overline J}(V,v) & G^{i \overline l}(V,v)
\end{array}
\right) 
\end{eqnarray*}
with respect to coordinates $(V, v)$ of  $\C^j \times {\bf H}^{k-j}$,  
then the following estimates hold near $(0,{\mathbb P}_0)$ with $I,J, K=1, \dots, j$ and $i,l,m=j+1, \dots, k$:\\
\\
\emph{$C^0$-estimates:}
\begin{equation} \label{metricproperties}
\begin{array}{rcl}
|G_{I \overline J}(V,v)-H_{I \overline J}(V)|& \leq & C  H_{I \overline I}(V)^{\frac{1}{2}}H_{J \overline J}(V)^{\frac{1}{2}} d^2(v,{\mathbb P}_0)\\
|G_{I \overline l}(V,v)| &\leq &C H_{I \overline I}(V)^{\frac{1}{2}} h_{l \overline l}(v)^{\frac{1}{2}}\ d^2(v,{\mathbb P}_0)\\
|G_{i \overline l}(V,v)-h_{i \overline l}(v)|& \leq & C  h_{i \overline i}(v)^{\frac{1}{2}}   h_{l \overline l}(v)^{\frac{1}{2}} \ d^2(v,{\mathbb P}_0)
\end{array}
\end{equation}
\noindent \emph{$C^0$-estimates of the inverse:}
\begin{equation} \label{C0inverse}
\begin{array}{rcl}
|G^{I \overline J}(V,v)-H^{I \overline J}(V)|& \leq & C H^{I \overline I}(V)^{\frac{1}{2}}H^{J \overline J}(V)^{\frac{1}{2}} d^2(v,{\mathbb P}_0)\\
|G^{I \overline l}(V,v)| &\leq &C H^{I \overline I}(V)^{\frac{1}{2}} h^{l \overline l}(v)^{\frac{1}{2}} \ d^2(v,P_0) \\
|G^{i \overline l}(V,v)-h^{i \overline l}(v)|& \leq & C h^{i \overline i}(v)^{\frac{1}{2}}   h^{l \overline l}(v)^{\frac{1}{2}} \ d^2(v,{\mathbb P}_0) 
\end{array}
\end{equation}
\emph{$C^1$-estimates:}\\
\begin{equation}  \label{metricproperties'}
\begin{array}{rcl}
 | \frac{\partial}{\partial V^I}{G}_{J \overline K}(V,v)| & \leq & C  H_{I \overline I}(V)^{\frac{1}{2}} H_{J \overline J}(V)^{\frac{1}{2}}H_{K \overline K}(V)^{\frac{1}{2}} \\
 |\frac{\partial}{\partial V^I}{G}_{J \overline l}(V,v)| &\leq & C  H_{I \overline I}(V)^{\frac{1}{2}}H_{J \overline J}(V)^{\frac{1}{2}}h_{l \overline l}(v)^{\frac{1}{2}}  \ d(v,{\mathbb P}_0)\\
|\frac{\partial}{\partial V^I}{G}_{i \overline j}(V,v)|& \leq & C H_{I \overline I}(V)^{\frac{1}{2}} h_{i \overline i}(v)^{\frac{1}{2}} h_{j \overline j}(v)^{\frac{1}{2}}
 \end{array} 
\end{equation}
\begin{equation} \label{metricproperties''}
\begin{array}{rcl}
 | \frac{\partial}{\partial v^i}{G}_{I \overline J}(V,v)| & \leq & C  h_{i \overline i}(v)^{\frac{1}{2}} H_{I \overline I}(V)^{\frac{1}{2}}H_{J \overline J}(V)^{\frac{1}{2}} d(v,{\mathbb P}_0)\\
 |\frac{\partial}{\partial v^i}{G}_{I\overline j}(V,v)|& \leq & C H_{I \overline I}(V)^{\frac{1}{2}} h_{i \overline i}(v)^{\frac{1}{2}} h_{j \overline j}(v)^{\frac{1}{2}} \\
  |\frac{\partial}{\partial v^m}\left( {G}_{i \overline j}(V,v)-h_{i \overline j}(v)\right)| & \leq & C  h_{m \overline m}(v)^{\frac{1}{2}} h_{i \overline i}(v)^{\frac{1}{2}}h_{j \overline j}(v)^{\frac{1}{2}}
  \end{array} 
\end{equation}
\end{proposition}

 \begin{proof}
The estimates we need to prove are coordinate independent.  Thus, we can assume  that  $V$ are normal coordinates centered at $0$ for the metric $H$ and $v=t$.     With this, we have
\[
H_{I \overline I}(V)=O(1), \ \   h_{i \overline i}(v)= O( |t_i|^{-2} (-\log |t_i|)^{-3})
\]
 and 
\[
(-\log |t_i|)^{-1} \leq \sum_{l=1}^{k-j}  (-\log |t_l|)^{-1}  \leq Cd^2(v, {\mathbb P}_0).
\]
As an example, we check the first and the second  estimate in (\ref{metricproperties''}). For the first,  we use Theorem~\ref{WP2}(iv) and the fact that $H_{I \overline I}(V)^{\frac{1}{2}}H_{J \overline J}(V)^{\frac{1}{2}} =O(1)$ to obtain (for $|t_i|$ sufficiently small such that $(-\log |t_i|)^{-1}<1$ and a generic constant $C$) \begin{eqnarray*}
\left|  \frac{\partial }{\partial v^i}  G_{I\overline J} \right| & \leq  & C\left(|t_i|^{-1} (-\log|t_i|)^{-3}\right) \\ 
 &\leq& C\left(|t_i|^{-1} (-\log|t_i|)^{-\frac{3}{2}}\right)  \left( (-\log|t_i|)^{-\frac{1}{2}}\right)\\
 & \leq & C  h_{i \overline i}(v)^{\frac{1}{2}} H_{I \overline I}(V)^{\frac{1}{2}}H_{J \overline J}(V)^{\frac{1}{2}} d(v,{\mathbb P}_0).
 \end{eqnarray*} 
For the second estimate with $ i\neq j$ , we use Theorem~\ref{WP2}(vi) to obtain
\begin{eqnarray*}
\left| \frac{\partial }{\partial v^i} G_{I\overline j} \right|& \leq & C\left(|t_i|^{-1} (-\log|t_i|)^{-3} |t_j|^{-1} (-\log|t_j|)^{-3} \right)\\
&\leq & C \left(|t_i|^{-1} (-\log|t_i|)^{-\frac{3}{2}}\right)\left( |t_j|^{-1} (-\log|t_j|)^{-\frac{3}{2}} \right)\\
& \leq & C H_{I \overline I}(V)^{\frac{1}{2}} h_{i \overline i}(v)^{\frac{1}{2}} h_{j \overline j}(v)^{\frac{1}{2}}\
\end{eqnarray*}  
If $ i= j$  we use Theorem~\ref{WP2}(viii) to obtain
\begin{eqnarray*}
\left| \frac{\partial }{\partial v^i} G_{I\overline i} \right| & \leq & C\left(|t_i|^{-2} (-\log|t_i|)^{-3} \right)\\
%&\leq& C\left(|t_i|^{-2} (-\log|t_i|)^{-3} \right)\\
& \leq & C H_{I \overline I}(V)^{\frac{1}{2}}h_{i \overline i}(v).
\end{eqnarray*}
The other estimates can be justified the same way. 
 \end{proof}

 \section{Maps into the Model Space $\overline {\bf H}$}
 \label{MintoMS}
Given a map $u$ into the model space $\overline {\bf H}$, a \emph{regular point} is a point of the domain of $u$ that maps  into ${\bf H}$ and a \emph{singular point} is a point of the domain of $u$ that maps to $P_0$.   The \emph{regular set} ${\mathcal R}(u)$ is the set of regular points and the \emph{singular set} ${\mathcal S}(u)$ is the set of singular points of $u$.  The goal of this section is to prove the following slightly easier version of the main theorem; more specifically, we prove a regularity theorem for harmonic maps into the model space of the Weil-Petersson metric.   
\begin{theorem}[{\bf{Regularity Theorem for Harmonic Maps into the Model Space}}] \label{RegularityTheoremModelSpace}
  If   $u:(\Omega,g) \rightarrow (\overline{\bf H},d_{\overline{\bf H}})$ is a harmonic map from an  $n$-dimensional Lipschitz Riemannian domain, then 
\[
\dim_{\mathcal H} \Big({\mathcal S}(u)\Big) \leq n-2.
\]
\end{theorem}

The strategy is to first show that the 
set ${\mathcal S}^{>1}(u)$ of  singular points of order $>1$ (for the definition of the \emph{order} see (\ref{orderdef})), is of   Hausdorff codimension at least 2 (cf. Subsection~\ref{sec:into hbar0}, Proposition~\ref{hmintoms}), then to prove that there exist no order 1 singular points  (cf. Subsection~\ref{sec:into hbar1},  Proposition~\ref{order1pointsmodelspace}). Note that the order is always $\geq 1$ by Lipschitz continuity (cf. \cite{gromov-schoen}, Theorem 2.3).  The proof of Proposition~\ref{order1pointsmodelspace}  relies heavily on the \emph{key technical Lemma for the Model Space} (cf. Subsection~\ref{sec:into hbar1}, Lemma~\ref{keylemma}) which gives sufficient conditions when a harmonic map into  $(\overline{\bf H},d_{\overline{\bf H}})$ does not hit the boundary point $P_0$.   This is a special case of the \emph{key technical Lemma} stated  in Section~\ref{intoWPcompletion} that is used  to address the regularity theorem for harmonic maps into $(\overline{\mathcal T}, d_{\overline{\mathcal T}})$.  The \emph{key technical Lemma} is the most challenging aspect of this paper as it introduces new techniques to address the non-local compactness  and degenerating geometry of the target space $(\overline{\bf H},d_{\overline{\bf H}})$ or $(\overline{\mathcal T}, d_{\overline{\mathcal T}})$ near the boundary.

\subsection{Harmonic maps into NPC spaces}\label{basic1}

In this subsection, we recall some basic  facts regarding harmonic maps into general NPC spaces. The standard references are  \cite{gromov-schoen}, \cite{korevaar-schoen1} and \cite{korevaar-schoen2}.  Additionally,  \cite{daskal-meseDM}  discusses harmonic map theory in a setting most relevant of this paper.    
 
 Let $(\Omega, g)$ be an $n$-dimensional Riemannian domain and let $(Y,d)$ an NPC space. 
For a finite energy map $u: (\Omega,g) \rightarrow (Y,d)$, let  $|\nabla u|^2$ denote the energy density as defined in \cite{korevaar-schoen1} (1.10v).  
A map $u$ is said to be \emph{harmonic}  if it is energy minimizing amongst all  finite energy maps with the same boundary conditions on every bounded  Lipschitz subdomain $\Omega' \subset \Omega$   (cf. \cite{korevaar-schoen1}). 
We record the following important result. 
\begin{theorem}[{\bf Lipschitz continuity: \cite{gromov-schoen}, \cite{korevaar-schoen1}, \cite{serbinowski}}] \label{KSlip}
A harmonic map $u: (\Omega,g) \rightarrow (Y,d)$ into an NPC space is 
   locally Lipschitz continuous with the local Lipschitz constant depending on the geometry of $ (\Omega,g) $,  the total energy of $u$ and the distance to the boundary.  If the boundary of $\Omega$ is smooth and the boundary data are $C^{\alpha}$ $(0<\alpha<1)$, the map $u$ extends up to the boundary with the $C^{\alpha}$ norm depending on the boundary data and on the total energy.
   \end{theorem} 
   
 Next, we recall the notion of \emph{order}.   Let $v:(\Omega,g) \rightarrow (Y,d)$ be a map (not necessarily harmonic).  For $x_0 \in \Omega$, define
\begin{eqnarray*} \label{nott}
E^v_{x_0}(\sigma) :=\int_{B_{\sigma}(x_0)} |\nabla v|^2 d\mu \ \mbox{ and } \
I^v_{x_0}(\sigma) := \int_{\partial B_{\sigma}(x_0)} d^2(v,v(x_0)) d\Sigma.
\end{eqnarray*} 
When the dependence of point $x_0$ is understood, we will omit it from the notation above and write $E^v(\sigma)$ and $I^v(\sigma)$  instead.
The \emph{order of the map $v$ at $x_0$} is defined by
  \begin{eqnarray}\label{orderdef}
 Ord^v(x_0):=\lim_{\sigma \rightarrow 0} \frac{\sigma \ E^v(\sigma) }{I^v(\sigma)} \mbox{ provided the limit exists.}
 \end{eqnarray}
 
 \begin{definition}
The set
 \[
 {\mathcal S}^{>1}(v):=\{x_0 \in \Omega:  \ Ord^v(x_0) \mbox{ exists and is $>1$}\}
 \]
 is the \emph{higher order points} of $v$.
 \end{definition}

 \begin{remark}
 For a harmonic function $u:(\Omega,g) \rightarrow \R$ and $x_0 \in \Omega$,  $Ord^u(x_0)$ is the order with which $u$ attains its value $u(x_0)$ at $x_0$.
 \end{remark}

\begin{theorem}[{\bf{Existence of the order function: \cite{gromov-schoen}, \cite{korevaar-schoen1}}}] \label{orderforharmonic}
For a harmonic map $u:(\Omega,g) \rightarrow (Y,d)$ into an NPC space and a compact subset $K$ of $\Omega$, there exist  constants $c>0$ and $\sigma_0>0$ depending only on the  domain metric (with $c=0$ when $g$ is a Euclidean metric) such that for any $x_0 \in K$,
\[
\sigma \mapsto e^{c \sigma^2}\frac{\sigma \ E^u(\sigma) }{I^u(\sigma)} \mbox{ 
 is non-decreasing for $\sigma  \in (0,\sigma_0)$}.
 \]
Thus, $Ord^u(x_0)$ exists  for all $x_0 \in K$.   Furthermore, 
\[
\sigma \mapsto e^{c \sigma^2}\frac{E^u(\sigma) }{\sigma^{n-2+2\alpha}}, \ \  \mbox{ and }  \sigma \mapsto e^{c \sigma^2}\frac{I^u(\sigma) }{\sigma^{n-1+2\alpha}}
\mbox{ 
 are non-decreasing for $\sigma \in (0,\sigma_0)$}.
 \]
 \end{theorem}
 \begin{proof}
The statements  above follow from Section 1.2 of \cite{gromov-schoen}  combined with \cite{korevaar-schoen1}, \cite{korevaar-schoen2}.
\end{proof}

   We record the following important result of \cite{korevaar-schoen2} Proposition~3.7 and Theorem 3.11.
   \begin{theorem}[{\bf{Compactness Theorem \cite{korevaar-schoen2}}}]
   \label{KScpct}
Assume the following:
\begin{itemize}
\item[(i)] The sequence   of smooth metrics $g_i$ on $B_R(0)$ converges in $C^\infty$ to the Euclidean metric  $g_0$.
 \item[(ii)] $(Y_i,d_i)$ is a sequence of NPC spaces.
\item[(iii)] The sequence   $v_i:(B_R(0),g_i) \rightarrow (Y_i,d_i)$ of maps  has a  uniform Lipschitz constant on compact subsets of $B_R(0)$.
\end{itemize}  
Then there exists a subsequence $v_{i'}$ of  $v_i$ converging locally uniformly in the pullback sense (cf. \cite{korevaar-schoen2} Definition 3.3) to a map $v_0:(B_1(0),g_0) \rightarrow (Y_0,d_0)$ into an NPC space, and $v_0$ has the  same local Lipschitz constant as $v_i$. 
Furthermore, if  $v_{i}$ is harmonic,
then  $v_0$ is also a harmonic  and
\[
E_0^{v_0}(r)=\lim_{i \rightarrow \infty} E_0^{v_{i'}}(r), \ \ \ \forall r \in (0,R).
\]
\end{theorem}

\begin{remark}
 The first assertion of the Compactness Theorem~\ref{KScpct} statement  can be viewed as a generalized version of the Arzela-Ascoli theorem for maps into different target spaces.   Note that (by an application of the usual Arzela-Ascoli Theorem to the sequence of pullback distance functions)
  \begin{equation} \label{pbdfuc}
  d(v_{i'}(\cdot), v_{i'}(\cdot)) \mbox{ converges locally uniformly to }d_0(v_0(\cdot), v_0(\cdot)).   
  \end{equation}
  This latter property will be important in the application of Theorem~\ref{KScpct}.
  \\
  \end{remark}

We now define the notion of \emph{blow-up maps} of a map $v: (\Omega,g) \rightarrow (Y,d)$ (not necessarily harmonic) at $x_0 \in \Omega$.  Throughout the paper we will define different cases of blow-up maps so it is important to deal with the general case first. Below,  $g_0$ denotes the standard Euclidean metric.  
 We identify $x_0=0$ via  normal coordinates $(x^1, \dots, x^n)$ centered at $x_0$ and let
  $v:(B_R(0),g) \rightarrow (Y,d)$ be a Lipschitz map.
For  $\sigma_0>0$ sufficiently small,  a function 
 \[
\nu: (0,\sigma_0) \rightarrow \R_{>0} \ \ \mbox{ with } \  \lim_{\sigma \rightarrow 0^+} \nu(\sigma)= 0
 \]
is called a  \emph{scaling factor}. Let $g_{\sigma}$ denote the rescaled metric on $B_R(0)$ given in terms of the  coordinates $(x^1, \dots, x^n)$ as
 \begin{equation} \label{defineblowupmap'}
 {g_{\sigma}}_{ij}(x) =g_{ij}(\sigma x)
  \end{equation} 
  and
  \[
  d_{\sigma}(P,Q)=\nu(\sigma)^{-1}d(P,Q).
 \]
 The \emph{blow-up map} of $v$ at $x_0=0$ with scaling factor $\nu(\sigma)$  is the map defined by 
 \begin{equation} \label{defineblowupmap}
 v_{\sigma}: (B_1(0) \subset B_{\sigma^{-1}R}(0) ,g_{\sigma}) \rightarrow (Y,d_{\sigma}), \ \ 
 v_{\sigma}(x)=v(\sigma x).
 \end{equation}
 \begin{remark} \label{blowupsforhm}
For a harmonic map $u: \Omega \rightarrow (Y,d)$  and $x_0 \in \Omega$, we 
 make a special choice of the scaling factor.  More specifically, we let $\nu(\sigma)$ be equal to 
 \[
  \mu(\sigma):=\sqrt{\frac{I^u(\sigma)}{\sigma^{n-1}}}.
 \]
With this choice,   the blow-up map 
   \[
u_{\sigma}:(B_1(0),g_{\sigma}) \rightarrow (Y,d_{\sigma})
 \]
 satisfies 
\[
 I^{u_{\sigma}}(1)=1 \mbox{ and } E^{u_{\sigma}}(1) \leq 2Ord^u(x_0) \mbox{ for $\sigma>0$ sufficiently small.}
 \]
In particular, the sequence  $u_{\sigma}$ has  uniformly bounded energy.  Again applying  the monotonicity properties of Theorem~\ref{orderforharmonic}, we have
 \begin{equation} \label{endec}
 E^{u_{\sigma}}(\theta) \leq \theta^{n-2+2\alpha} E_0
 \end{equation}
 where the constant $E_0$ can be chosen independently of $\sigma$ and $\alpha=Ord^u(x_0) \geq 1$.
Moreover, $u_\sigma$ is a harmonic map and Theorem~\ref{KSlip} and (\ref{endec}) imply
\begin{equation} \label{uLip}
u_{\sigma}
\mbox{ has uniform local Lipschitz bound}.
\end{equation}
Thus, applying the Compactness Theorem~\ref{KScpct},   we can find a sequence $\sigma_i \rightarrow 0$ such that  $u_{\sigma_i}$   converges locally uniformly in the pullback sense  to a non-constant harmonic map 
 \[
 u_*:(B_1(0),g_0) \rightarrow (Y_*,d_*)
 \]
  from a Euclidean ball into an NPC space.  
By following the argument of \cite{gromov-schoen} Proposition~3.3, we conclude that the map $u_*$ is  \emph{homogeneous degree $\alpha$}; more precisely,  for any $\xi \in \partial B_1(0)$, 
the image $\{u_*(t\xi): \ t \in (0,1)\}$ is a geodesic
 and
\[
d_*(u_*(t\xi ),u_*(0)) = t^{\alpha} d_*(u_*(\xi),u_*(0)), \ \ \forall t \in (0,1).
\]
Since $u_*$ is Lipschitz continuous by Theorem~\ref{KSlip}, it follows that
\[
Ord^u(x_0)=Ord^{u_*}(0)=\alpha \geq 1.
\]

In the case of a harmonic map $u$ into a smooth Riemannian manifold $M$, the target space of a  tangent map $u_*$  at $x$  is the tangent space $T_{u(x)}M$.  On the other hand, the target space  $(Y_*,g_*)$ of a tangent  map $u_*$  at $x$  may be quite different from the tangent  cone $T_{u(x)}Y$.   
For example,  if $u:(\Omega,g) \rightarrow (\overline{\bf H},d_{\overline{\bf H}})$ is a harmonic map with $u(x)=P_0$, then a tangent map $u_*$ cannot  map to the tangent cone $T_{P_0}\overline{\bf H}$ at $P_0$.  Indeed, if the image of $u_*$ is  $T_{P_0}\overline{\bf H}$,  then we  have a violation of the minimum principle  since $T_{P_0}\overline{\bf H}$ is isometric to $[0,\infty)$ (cf. Lemma~\ref{tangentcone}). This is one of the technical issues dealt in this paper. 
 \end{remark}

\begin{definition}
The homogeneous harmonic map $u_*$ defined in Remark~\ref{blowupsforhm} is called a \emph{tangent map} of $u$ at $x_0$.
\end{definition}

We will record  three  lemmas  about the upper semicontinuity of the order and Hausdorff dimension. Some version of the lemmas  are more or less known to the experts, but we will include their proofs here for completeness since the exact version stated below is hard to find in literature. 

 \begin{lemma} \label{uscord} 
Let $u:(\Omega,g) \rightarrow (Y,d)$ be a harmonic map.  Let $x_0 \in \Omega$ and  $u_{\sigma_i}$ be a sequence of  blow-up maps of $u$ at $x_0$  converging locally unifomrly in the pullback sense to a tangent map $u_*$.  If $x_i \rightarrow x_*$, then
\[
\liminf_{\sigma_i \rightarrow 0} Ord^{u_{\sigma_i}}(x_i) \leq Ord^{u_*}(x_*).
\]
\end{lemma}
\begin{proof}
Fix $r \in (0,1)$.  By  Theorem~\ref{KScpct}  and (\ref{uLip}), we have 
\[
E^{u_*}_{x_*}(r) =  \lim_{\sigma_i \rightarrow 0} E^{u_{\sigma_i}}_{x_*}(r).%, 
\]  
Furthermore, we claim
\begin{equation} \label{thisuseslip}
  \lim_{\sigma_i \rightarrow 0}|E^{u_{\sigma_i}}_{x_*}(r) -E^{u_{\sigma_i}}_{x_i}(r)|=0. 
\end{equation}
To prove (\ref{thisuseslip}), for $\epsilon >0$ choose $i$ large so that $|x_i-x_*|<\epsilon$. By the uniform Lipschitz assumption (\ref{uLip}) there exists $C >0$ such that
\[
E^{u_{\sigma_i}}_{x_*}(r)-C\epsilon \leq E^{u_{\sigma_i}}_{x_*}(r-\epsilon) \leq E^{u_{\sigma_i}}_{x_i}(r) \leq E^{u_{\sigma_i}}_{x_*}(r+\epsilon) \leq E^{u_{\sigma_i}}_{x_*}(r)+C\epsilon
\] 
which immediately implies the desired equality.
Combining the above, we have
 \[
 E^{u_*}_{x_*}(r) =  \lim_{\sigma_i \rightarrow 0} E^{u_{\sigma_i}}_{x_i}(r).
 \]
 Furthermore, by  the local uniform convergence of  the pullback distance functions (\ref{pbdfuc}) 
 \[
 I^{u_*}_{x_*}(r)=\lim_{\sigma_i \rightarrow 0} I^{u_{\sigma_i}}_{x_i}(r).
 \]
   Combining the above two equalities, we obtain
 \[
 \lim_{\sigma_i \rightarrow 0}\frac{r E^{u_{\sigma_i}}_{x_i}(r)}{I^{u_{\sigma_i}}_{x_i}(r)}=\frac{r E^{u_*}_{x_*}(r)}{I^{u_*}_{x_*}(r)}.
\] 
 Now we apply the monotonicity property of Theorem~\ref{orderforharmonic}, namely
 \[
Ord^{u_{\sigma_i}}(x_i) \leq
 e^{cr} \frac{r E^{u_{\sigma_i}}_{x_i}(r)}{I^{u_{\sigma_i}}_{x_i}(r)}.
\]
Taking liminf as $i \rightarrow \infty$  in the above inequality, we obtain
\[
\liminf_{\sigma_i \rightarrow 0} Ord^{u_{\sigma_i}}(x_i) \leq e^{cr} \frac{r E^{u_*}_{x_*}(r)}{I^{v_0}_{x_*}(r)} .% \ \ \ r_0 \in [\frac{R}{2},R].
\]
Finally, letting $r \rightarrow 0$ yields
\[
\liminf_{\sigma_i \rightarrow 0} Ord^{u_{\sigma_i}}(x_i) \leq Ord^{u_*}(x_*).
\]
\end{proof}

\begin{lemma} \label{federeroutermeasure}
Let $E_i$ be a sequence of compact subsets of $\R^n$  and let $E_0 \subset \R^n$ be a compact set.  Assume that if $x_i$  is a sequence such that $x_i \in E_i$ and $x_i \rightarrow x_*$, then $x_* \in E_0$.  Then
\[
 \limsup_{i \rightarrow \infty} \dim_{\mathcal H}(E_i) \leq \dim_{\mathcal H}(E_0).
 \] 
\end{lemma}

\begin{proof}
First, for any subset $E \subset \R^n$ and any real number $s \in [0,n]$, recall that  \cite{federer2}
\begin{equation} \label{hath}
\dim_{\mathcal H}(E)= \inf\{s :  \hat{\mathcal H}^s(E)=0\}
\end{equation}
where
\[
\hat{\mathcal H}^s(E) = \inf \left\{ \sum_{i=1}^{\infty} r_i^s:  \mbox{ all coverings $\{B_{r_i}(x_i)\}_{i=1}^{\infty}$ of $E$ by open balls. }\right\}.
\]
Since $E_0$ is compact, we may consider finite coverings $E_0 \subset \bigcup_{i=1}^N B_{r_i}(x_i)$.  Fix $\epsilon>0$.  By the assumption,  we have that for $i$ sufficiently large
\[
E_i \subset \{x :  |x-E_0|<\epsilon\}
\] 
where $|x-E_0|=\inf \{|x-y|:  y \in E_0\}$.  Thus, if 
$\hat{\mathcal H}^s(E_0)=0$ for some $s \in [0,n]$, then $\hat{\mathcal H}^s(E_i)=0$ for $i$ sufficiently large.   The assertion follows from (\ref{hath}).  
\end{proof}

\begin{lemma} \label{cd2general}
\label{hmintoms}
Let $u:(\Omega,g) \rightarrow (Y,d)$ be a  map. For any $x_0 \in {\mathcal S}^{>1}(u)$, assume that there exists a sequence  $u_{\sigma_i}$  of  blow-up maps of $u$ at $x_0$  converging locally uniformly in the pullback sense to a homogeneous harmonic map $u_*$ with the following properties:

\begin{itemize}
\item[(i)] If a sequence $\{x_i\} \subset B_1(0)$ is such that $x_i \in {\mathcal S}^{>1}(u_{\sigma_i})$ with $x_i \rightarrow x_*$, then $x_* \in {\mathcal S}^{>1}(u_*)$.\item [(ii)]  $
\dim_{\mathcal {H}}({\mathcal S}^{>1}(u_*)) \leq n-2.
$
\end{itemize}
Then 
\[
\dim_{\mathcal {H}}({\mathcal S}^{>1}(u)) \leq n-2.
\]  
\end{lemma}
\begin{proof}
Assume on the contrary that $
\dim_{\mathcal H}({\mathcal S}^{>1}(u)) > n-2$; i.e., there exists  $s >n-2$ such that ${\mathcal H}^s({\mathcal S}^{>1}(u)) >0$.  By \cite{federer} 2.10.19,  there exists $x_0 \in {\mathcal S}^{>1}(u)$ such that  
\[
\lim_{\sigma \rightarrow 0} {\mathcal H}^s({\mathcal S}^{>1}(u_{\sigma})) =\lim_{\sigma \rightarrow 0} \frac{ {\mathcal H}^s({\mathcal S}^{>1}(u)  \cap B_{\sigma}(x_00))}{\sigma^{s}} \geq 2^{-s}.
\]
%where $u_{\sigma}$ is the blow-up map of $u$ at $x_0$. 
Thus,  $\dim_{\mathcal H}({\mathcal S}^{>1}(u_{\sigma_i}))  \geq s$ for $\sigma_i$ sufficiently small. From (i)  and  Lemma~\ref{federeroutermeasure}, we conclude that 
\[
n-2<s \leq  \limsup_{\sigma_i \rightarrow 0} \dim_{\mathcal H}({\mathcal S}^{>1}(u_{\sigma_i})) \leq \dim_{\mathcal H}({\mathcal S}^{>1}(u_*)).
\]
This contradicts (ii).
\end{proof}

\subsection{The Order Gap} \label{sec:ordergap}  
In this subsection, we  prove an  order gap theorem for the limit map of a sequence of harmonic maps into $(\overline{\bf H},d_{\overline{\bf H}}).$  (We note that an important example of such a limit is a tangent map of harmonic map into $(\overline{\bf H},d_{\overline{\bf H}})$ as we shall see in Section~\ref{sec:into hbar0}.)  For two dimensional domains, many of the ideas in this subsection first appeared in \cite{wentworth}  in a slightly different language.  

We will use the following properties of a  map $u:(\Omega,g) \rightarrow (\overline{\bf H},d_{\overline{\bf H}})$.  Given an open set $ U \subset \Omega$ such that  $u(U)$ is contained in the  smooth Riemannian manifold $({\bf H},g_{\bf H})$, we can write 
\[
u=(u_{\rho},u_{\phi})
\]
in $U$ with respect to the standard coordinates $(\rho,\phi)$ of ${\bf H}$. 
If $u$ is Lipschitz continuous in $U$, then there exists a constant $C>0$ such that 
\begin{equation} \label{lip}
|\nabla u_{\rho}| \leq C \ \mbox{ and } \ |u_{\rho}^3 \nabla u_{\phi}|\leq C.
\end{equation}
If $u: (\Omega,g) \rightarrow (\overline{\bf H},d_{\overline{\bf H}})$ is a harmonic map, then $u_\rho$ and $u_{\phi}$ satisfy  the  equations   in $\Omega \backslash {\mathcal S}(u)$ 
\begin{eqnarray} \label{hmequphi}
u_{\rho} \triangle u_{\rho} = 3  u_{\rho}^6 |\nabla u_{\phi}|^2 \ \mbox{ and }  \ u_{\rho}^4 \triangle u_{\phi} = -6 \nabla u_{\rho} \cdot u_{\rho}^3 \nabla u_{\phi}.
\end{eqnarray}
In the above $\nabla$ and $\triangle$ denote the gradient and the Laplacian with respect to the metric $g$.
 
%In particular it follows from (\ref{hmequphi}) and (\ref{lip}) that there is a constant $C$ dependent only on $r$ and the total energy of $u$ such that
%\begin{eqnarray} \label{orc}
%| \triangle u_{\rho}| \leq \frac{C}{ u_{\rho}} \ \mbox{ and } | \triangle u_{\phi}|  \leq \frac{C}{ u_{\rho}^4}  \ \ \mbox{ in } \ \ B_r(x_0) \backslash {\mathcal S}(u).
%\end{eqnarray}

\begin{lemma} \label{proveitsfl}
Let $B_R(0) \subset {\mathbb R}^n$. There exists $\epsilon_0 >0$ depending only on the domain dimension $n$ such that if  $w_i=(w_i^{\rho},w_i^{\phi}):(B_R(0),g_i) \rightarrow (\overline{\bf H},d_{\overline{\bf H}})$ is a sequence of harmonic maps with uniformly bounded energy converging   locally uniformly in the pullback sense 
(cf. Theorem~\ref{KScpct}) to a   homogeneous harmonic map $v_0:(B_R(0),g_0) \rightarrow (Y_0,d_0)$  into an NPC space, $\lim_{i\rightarrow \infty} w_i(0)=P_0=\overline{\bf H}\backslash \bf H$ and  the metric $g_i$ converging in $C^\infty$ to the standard Euclidean metric  $g_0$,  then
 \[
Ord^{v_0}(0)=1 \ \mbox{ or } \ Ord^{v_0}(0) \geq  1+\epsilon_0.
\]
If $Ord^{v_0}(0)=1$, then $v_0$ maps into a geodesic.
Furthermore, the set of higher order points of $v_0$  has Hausdorff codimension at least 2; i.e.
\[
\dim_{\mathcal H}({\mathcal S}^{>1}(v_0)) \leq  n-2.
\]
  \end{lemma}

\begin{remark}
The notion of  local uniform convergence in the pullback sense that appears in Lemma~\ref{proveitsfl} was discussed in Theorem~\ref{KScpct}.   On the other hand, in the proof of Lemma~\ref{proveitsfl}, we only need the fact that the sequence of pullback distance functions $d_{\overline{\bf H}}(w_i(\cdot), w_i(\cdot))$ converges locally uniformly to the pullback distance function $d_0(v_0(\cdot),v_0(\cdot))$.   In particular, a sequence of blow-up maps of a harmonic map has a subsequence satisfying this property (cf. (\ref{pbdfuc}) and Remark~\ref{blowupsforhm}).
\end{remark}

\begin{proof}  
For the sake of  simplicity, we will assume that $g_i$ is the standard Euclidean metric $g_0$.  
%We will also denote $d=d_{\overline{\bf H}}$. 
We break up the proof of Lemma~\ref{proveitsfl} into four claims.\\
\\
{\sc Claim 1}. \emph{If  $\Omega_0$ is a connected component of the open set $\{x \in B_R(0):  v_0(x) \neq v_0(0)\}$, then $v_0$ maps $\overline{\Omega}_0$  into a geodesic ray starting at $v_0(0)$.}\\  

{\sc Proof of Claim 1}.  Since $E^{w_i}(R)$ is uniformly bounded, Theorem~\ref{KSlip} and (\ref{lip}) imply that, for any $r \in (0,R)$, there exists $C>0$ such that for all $i$ and $x \in B_r(0) \backslash \{x: w_i(x) \neq P_0\}$
\begin{eqnarray}\label{lipcont}
 |\nabla w_i^{\rho}|(x) \leq C, \ \ (w_i^{\rho}(x))^3|\nabla w_i^{\phi}|(x) \leq C.
  \end{eqnarray}  

Fix $x_{\Omega_0} \in \Omega_0$ and let $K$ be a compact set contained  in $\Omega_0$ and containing  $x_{\Omega_0}$.  
 The local uniform convergence in the pullback sense of $w_i$ to $v_0$ and the fact that $\lim_{i \rightarrow \infty} w_i(0)=P_0$, imply 
 \[
  \lim_{i \rightarrow \infty}w_i^{\rho}(x)=  \lim_{i \rightarrow \infty}d_{\overline{\bf H}}(w_i(x),P_0)=\lim_{i \rightarrow \infty}d_{\overline{\bf H}}(w_i(x),w_i(0))=d_0(v_0(x),v_0(0))=:f(x)
 \]
for $x \in B_R(0)$. Since $K$ is  compactly contained in $\Omega_0$,  the convergence $w_i^{\rho}(x) \rightarrow f(x)$ is uniform in $K$, and  it follows that  the function $w_i^{\rho}$ is bounded away from 0 in $K$ for $i$ sufficiently large.
Therefore (\ref{lipcont}) implies that $w_i^{\phi}$ is uniformly  Lipschitz in $K$,  and there exists a subsequence of  $w_i^{\phi}-w_i^{\phi}(x_{\Omega_0})$ 
 (which we shall still denote by $w_i^{\phi}-w_i^{\phi}(x_{\Omega_0})$ by an abuse of notation)  that converges uniformly in $K$.  By taking a compact exhaustion of $\Omega_0$  and applying a diagonalization procedure, we can assume (by taking a subsequence if necessary) that  $w_i^{\phi}-w_i^{\phi}(x_{\Omega_0})$ converges locally uniformly to some function $g$ in $\Omega_0$.  Thus,  $(w_i^{\rho}, w_i^{\phi}-w_i^{\phi}(x_{\Omega_0}))$  converges locally uniformly in $\Omega_0$ to the map $(f, g):\Omega_0 \rightarrow {\bf H}$.  Since $(w_i^{\rho}, w_i^{\phi}-w_i^{\phi}(x_{\Omega_0}))$  is a sequence of harmonic maps into a smooth Riemannian manifold $({\bf H},g_{\bf H})$,  this convergence is actually  locally  $C^\infty$.  The map $w_i$ is harmonic which implies that the functions $w_i^{\rho}$ and $w_i^{\phi}$ satisfy  
 \[
 w_i^{\rho} \triangle  w_i^{\rho} =3 ( w_i^{\rho} )^6 |\nabla  w_i^{\phi}|^2 \ \mbox{ in } \Omega_0.
 \] 
Thus, the functions $f$ and $g$ also satisfy
\begin{equation} \label{hmeq}
f \triangle f = 3 {f}^6 |\nabla g |^2 \ \mbox{ in } \Omega_0.
\end{equation}
Furthermore, the homogeneity of $v_0$ implies  the homogeneity of $f$. Thus, extend the domain of  $f$ is $\R^n$, assume $\Omega_0$ is a cone and write \[
f(r,\theta)=r^{\alpha} F(\theta) \ \mbox{ in } \Omega_0
\]
where  $\alpha=Ord^{v_0}(0)$,
 \[
F:\Omega_0 \cap \partial B_1(0) =:\Lambda  \subset \Sp^{n-1} \rightarrow \overline\R_+
\]
 and $\theta=(\theta^1,\dots,\theta^{n-1})$ are the coordinates of $\Sp^{n-1}$.  
 The above two equations imply 
 that 
\[
\alpha(\alpha+n-1)F+  \triangle_{\theta} F=  r^{4\alpha+2}F^6(\theta)|\nabla g |^2.
\]
Since this inequality holds for any $r>0$,  we can conclude that $
|\nabla g |^2=0$.  Hence $f$  is a harmonic function by (\ref{hmeq}).
Since $w_i^{\phi}-w_i^{\phi}(x_{\Omega_0})=0$ at $x=x_{\Omega_0}$, we see that $g(x_{\Omega_0})=0$.   Hence  $g=0$ in $\Omega_0$ and $(w_i^{\rho},w_i^{\phi}-w_i^{\phi}(x_{\Omega_0}))$ converges locally uniformly to $(f,0)$ in $\Omega_0$.  This in turn  implies 
$(w_i^{\rho},w_i^{\phi})$ converges locally uniformly in the pullback sense  to $(f,0)$ in $\Omega_0$.  Since $(f,0)$ maps into the geodesic ray $\{(\rho,0) \in {\bf H}\} \cup \{P_0\}$, $w_i(0) \rightarrow P_0$ and  $(w_i^{\rho},w_i^{\phi})$  also converges locally uniformly in the pullback sense  to $v_0$ in $\Omega_0$,  $v_0$ also maps $\overline{\Omega}_0$ into a geodesic ray starting at $v_0(0)$.   \\
\\
{\sc Claim 2}.  \emph{There exists $\epsilon_0 \in (0,1)$ sufficiently small  depending only on the domain dimension $n$ such that if $Ord^{v_0}(0) < 1+\epsilon_0$, then there exists a geodesic $\gamma$ such that $v_0(B_1(0)) \subset \gamma$.  
}\\

{\sc Proof of Claim 2.}  This argument essentially goes back to \cite{gromov-schoen}, but we include it here for the sake of completeness.  Let $\Omega_0$, $f$, $F$, $\Lambda$ and  $\theta=(\theta^1,\dots,\theta^{n-1})$ be as in the proof of Claim 1; i.e. $\Omega_0$ is extended to a cone in $\R^n$, $F(\theta)$ is defined in $\Lambda=\Omega_0 \cap \Sp^{n-1}$ and 
\begin{equation} \label{eig}
 \triangle f = 0 
\
 \mbox{ and }
\
f(r,\theta)=r^{\alpha} F(\theta) \mbox{ in } \Omega_0  \mbox{ with }\alpha=Ord^{v_0}(0).
\end{equation}
Combining the above two equations, we conclude that $F$ is a Dirichlet eigenfunction with eigenvalue $\alpha+n-1$; i.e. $F$ satisfies 
\[
\left\{
\begin{array}{ll}
\alpha(\alpha+n-1)F+  \triangle_{\theta} F=0 & \mbox{in } \Lambda
\\
F\big|_{\Omega}=0 &
\end{array}
\right.
\]
in the domain $\Lambda$.

Now assume that  there exists at least three distinct connected components of $\{x \in B_R(0):  v_0(x) \neq v_0(0)\}$. Then at least one of the components, which we will call $\Omega_0$, has the property that  $\mbox{Vol}(\Lambda) \leq \frac{1}{3} \mbox{Vol}(\Sp^{n-1})$ for $\Lambda=\Omega_0 \cap \Sp^{n-1}$ (after extending $\Omega_0$ as a cone in $\R^n$ as in the proof of Claim 1). Recall that the Faber-Krahn theorem implies that the first Dirichlet eigenvalue $\lambda_1(\Lambda)$ of $\Lambda$ is bounded from below by the first eigenvalue $\lambda_1({\mathcal B})$ of a geodesic ball $\mathcal B$ in $\Sp^{n-1}$ with volume equal to $\frac{1}{3} \mbox{Vol}(\Sp^{n-1})$.  Since $\lambda_1({\mathcal B}) \geq n-1+\delta_n$ for some number $\delta_n>0$ depending only on $n$, it follows that 
\[
\alpha(\alpha+n-1) \geq \lambda_1(\Lambda) \geq \lambda_1({\mathcal B}) \geq n-1+\delta_n.
\]
Therefore, there exists $\epsilon_0>0$ depending only on $n$ such that 
 $\alpha \geq 1+ \epsilon_0$. Consequently,  if $\alpha < 1+\epsilon_0$, then  $\{x \in B_R(0):  v_0(x) \neq v_0(0)\}$ has at most two components. 
The maximum principle applied to the subharmonic function $f=d_0(v_0,v_0(0))$  implies that there cannot be only one component.   Therefore, $\alpha < 1+\epsilon_0$ implies that there exist exactly two connected components $\Omega_+$ and $\Omega_-$  of $\{x \in B_R(0):  v_0(x) \neq v_0(0)\}$.  
Let $\gamma_+$ and $\gamma_-$ be geodesic rays starting at $v_0(0)$  such that $v_0(\Omega_+) \subset \gamma_+$ and  $v_0(\Omega_+) \subset \gamma_-$.  Since $v_0$ is harmonic, $\gamma:=\gamma_+ \cup \gamma_-$ is a geodesic.     \\
\\

{\sc Claim 3.} Either $Ord^{v_0}(0)=1 \ \mbox{ or } \ Ord^{v_0}(0) \geq  1+\epsilon_0$. \emph{If $Ord^{v_0}(0)=1$, then $v_0$ maps into a geodesic.}   \\

{\sc Proof of Claim 3.}  Let $\epsilon_0 \in (0,1)$ be as in Claim 2 and assume $Ord^{v_0}(0) < 1+\epsilon_0$. By Claim 2, the image of $v_0$ is contained in a geodesic $
\gamma$.  Thus, we can identify $\gamma$ with $\R$ and assume $v_0$ is a harmonic  function.  Since $Ord^{v_0}(0)  < 1+\epsilon_0<2$ and the order of a harmonic function is integer valued, we conclude  $Ord^{v_0}(0)=1$.   In this case,  $v_0$ is a degree 1 harmonic function, hence linear.  \\
\\
{\sc Claim 4.}  
$
\dim_{\mathcal H}({\mathcal S}^{>1}(v_0)) \leq  n-2$.
\\

{\sc Proof of Claim 4}.
We will apply Federer's dimension reduction argument.
Assume on the contrary that $
\dim_{\mathcal H}({\mathcal S}^{>1}(v_0)) > n-2$;~i.e. there exists  $s >n-2$ such that ${\mathcal H}^s({\mathcal S}^{>1}(v_0)) >0$.  By \cite{federer} 2.10.19,   there exists $x_0 \in {\mathcal S}^{>1}(v_0)$ such that $x_0 \neq 0$  and
\[
\lim_{\sigma \rightarrow 0} {\mathcal H}^s({\mathcal S}^{>1}(v_{0, \sigma})) =\lim_{\sigma \rightarrow 0} \frac{ {\mathcal H}^s({\mathcal S}^{>1}(v_0) \cap B_{\sigma}(x_0)) }{\sigma^{s}} \geq 2^{-s}
\]
where $v_{0, \sigma}$ is the blow-up map of $v_0$ at $x_0$. We claim that  
\begin{equation} \label{ogv}
%Ord^{v_{0, \sigma}}(0)=
Ord^{v_0}(x_0) \geq 1+\epsilon_0
\end{equation} 
for the same $\epsilon_0>0$ as {\sc Claim} 2. Indeed, since  $v_0$ maps into a union of geodesics,  the function $f(x)=d_{\overline{\bf H}}(v_0(x),v_0(0))$ is a homogeneous harmonic functions in each component $\Omega_0$ of  $B_1(0) \backslash \{v_0(x)=v_0(0)\}$. In particular, $f$  in $\Omega_0$ satisfies (\ref{eig}).  Thus, we can apply the same argument as in Claim 2 to show an order gap for $v_0$ with the same $\epsilon_0$.  Since $Ord^{v_0}(x_0)\neq 1$ (because $x_0 \in {\mathcal S}^{>1}(v_0)$), the claim follows.

By rotating if necessary, we can assume $x_0=(0,\dots, 0, |x_0|)$.   
The homogeneity of $v_0$ implies that  $Ord^{v_0}(0,\dots, 0, t) \geq 1+\epsilon_0$ for  $0 <t<1$.  This in turn implies that $Ord^{v_{0, \sigma}}(0,\dots, 0,t) \geq 1+\epsilon_0$  for  $-1<t <1$.  By the upper semicontinuity of order (cf. Lemma~\ref{uscord}), this  implies $Ord^{v_{0, *}}(0, \dots, 0,t) \geq 1+\epsilon_0$ for  $-1<t <1$ where $v_{0, *}$ is a tangent map of $v_0$ at $x_0$.  
Thus,  if $\vec{e}_1, \dots, \vec{e}_n$ are the standard basis vectors of $\R^n$, then $v_{0, *}$  is independent of the $\vec{e}_n$-direction and its restriction to $\R^{n-1}$ spanned by $\vec{e}_1, \dots, \vec{e}_{n-1}$ denoted $\tilde{v}_{0,*}$, is a homogeneous map of degree $\alpha_{0, *} \geq 1+\epsilon_0$.  We then have
\[
 {\mathcal S}^{>1}(v_{0, *}) = {\mathcal S}^{>1}(\tilde{v}_{0, *})   \times \R
\ \mbox{ 
and }
\
\dim_{\mathcal H}({\mathcal S}^{>1}(\tilde{v}_{0,*}))  \geq s-1.
\]
Since $s>n-2$, we can repeat this argument inductively to produce a geodesic with order not equal to 1 at some point, which is contradiction. (This part of the argument is essentially the same as in \cite{gromov-schoen} Lemma 6.5 where we refer the reader for further details).   
\end{proof}

\subsection{Higher order points}  \label{sec:into hbar0}
The goal of this subsection is to prove that the set of higher order points of a harmonic map into $(\overline{\bf H},d_{\overline{\bf H}})$ is of Hausdorff codimension 2 (cf. Proposition~\ref{hmintoms}).  To do this, we apply Lemma~\ref{proveitsfl} to a sequence of blow-up maps of a harmonic map into $(\overline{\bf H},d_{\overline{\bf H}})$.  Generally speaking, note that  the blow-up maps  of a map into an NPC space do not necessarily map into the same NPC space as the original map  (because the distance function $d_{\sigma}$ is different from the original distance function $d$).  On the other hand, for a map into $(\overline{\bf H},d_{\overline{\bf H}})$, we can use the homogeneous structure   of $(\overline{\bf H},d_{\overline{\bf H}})$  discussed in Section~\ref{sec:modelspace} to define its blow-up maps   as a map again into $(\overline{\bf H},d_{\overline{\bf H}})$.
Indeed, given a harmonic map $u=(u^{\rho},u^{\phi}):(B_R(0),g) \rightarrow  (\overline{\bf H},d_{\overline{\bf H}})$ with $u(0)=P_0$, we can define \begin{equation} \label{defofblowups}
u_{\sigma}: (B_1(0), g_{\sigma}) \rightarrow  (\overline{\bf H},d_{\overline{\bf H}}), \ \ u_{\sigma}(x)=\nu^{-1}(\sigma)u(\sigma x).
 \end{equation}
In other words, if we write  
$
u=(u_{\rho}, u_{\Phi})
$ and $u_{\sigma}=(u_{\sigma \rho}, u_{\sigma \Phi})$ in the homogeneous coordinates, then
\[
 u_{\sigma \rho}(x)=\nu^{-1}(\sigma)u_{\rho}(\sigma x) \ \mbox{ and } \  u_{\sigma \Phi}(x) =\nu^{-1}(\sigma) u_{\Phi}(\sigma x).
\]
Because of  the homogeneity of the distance function under the dilation map, this is equivalent to the construction of blow-up maps given by (\ref{defineblowupmap}).
By Remark~\ref{blowupsforhm}, there exists a sequence $\sigma_i \rightarrow 0$ such that  $u_{\sigma_i}=(u_{\sigma_i}^{\rho}, u_{\sigma_i}^{\phi}) $  converges locally uniformly in the pullback sense to a tangent map  $u_*$  of $u$. The next is a corollary of Lemma~\ref{proveitsfl}.

\begin{corollary} \label{msproperties}
If $u=(u^{\rho},u^{\phi}):B_1(0) \rightarrow (\overline{\bf H},d_{\overline{\bf H}})$,  $x_0 \in B_1(0)$ and $u_*$ is a tangent map of $u$ at $x_0$, then
\[
Ord^{u}(x_0)=1 \ \mbox{ or } \ Ord^{u}(x_0) \geq  1+\epsilon_0, \ \mbox{ and } \ \dim_{\mathcal {H}}({\mathcal S}^{>1}(u_*)) \leq n-2.
\]
Moreover, if $u(x_0)=P_0$ and $Ord^u(x_0)=1$, then  $u_*$   maps into a geodesic.   
\end{corollary}

\begin{proof}  
%Let $x \in B_1(0)$.  Let  $\{u_{\sigma_i}=(u_{\sigma_i}^{\rho}, u_{\sigma_i}^{\phi}) \}$ and $u_*$ be the blow-up maps and a tangent map of $u$ at $x$ such that $u_{\sigma_i}$ converges to $u_*$ locally uniformly in the pullback sense (cf. Remark~\ref{blowupsforhm}).  
First, assume $u(x) \neq P_0$.  Then $u$ maps a neighborhood of $x$ into a smooth Riemannian manifold, and  $u_*$ maps into $T_{u(x)}{\bf H}=\R^2$ and the lemma holds trivially with $\epsilon_0=1$.  Next, assume $u(x)=P_0$ which then implies  $u_{\sigma_i}(0)=P_0$.  The lemma follows by applying  Lemma~\ref{proveitsfl} with $w_i=u_{\sigma_i}$ and $v_0=u_*$.  
\end{proof}

We now arrive at the following.

\begin{proposition}
\label{hmintoms}
If $u=(u^{\rho},u^{\phi}):B_1(0) \rightarrow (\overline{\bf H}, d_{\overline{\bf H}})$ is a harmonic map,  then the set of points such that $Ord^u >1$ is of Hausdorff codimension at least 2, i.e.
\[
\dim_{\mathcal {H}}({\mathcal S}^{>1}(u)) \leq n-2.
\]  
\end{proposition}

\begin{proof}
Given $x_0 \in B_1(0)$, there exists a  sequence $\{u_{\sigma_i}\}$ of blow-up maps that converges locally uniformly in the pullback sense to a tangent map $u_*$ (cf. Remark~\ref{blowupsforhm}).  It suffices to check  assumptions (i) and  (ii) of Lemma~\ref{cd2general}.  
First, assume  $x_i \in {\mathcal S}^{>1}(u_{\sigma_i})$ with  $x_i \rightarrow x_*$.  By the order gap property of Corollary~\ref{msproperties}, we have
$Ord^{u_{\sigma_i}}(x_i) \geq 1+\epsilon_0$.  
The upper semicontinuity of order (cf. Lemma~\ref{uscord}) implies $Ord^{u_*}(x_*) \geq 1+\epsilon_0$ which in turn implies $x_* \in {\mathcal S}(u_*)$.  This verifies (i).    By Corollary~\ref{msproperties}, we have $\dim_{\mathcal {H}}({\mathcal S}^{>1}(u_*)) \leq n-2$.  This verifies  (ii).
\end{proof}

\subsection{Order 1 points}    \label{sec:into hbar1}
We continue with the proof of Regularity  Theorem~\ref{RegularityTheoremModelSpace}.  In view of Proposition~\ref{hmintoms},  it suffices to show that there exists no order 1 singular points of a harmonic map.  In this subsection, we analyze the order 1 points.  
An important tool for this analysis is a    global coordinate system  of ${\bf H}$ %introduced in  \cite{daskal-meseER} 
that are   constructed 
 by foliating ${\bf H}$ by  \emph{symmetric geodesics}.
We introduce %in Section~\ref{ssec:sg} 
  \emph{symmetric geodeiscs} in Section~\ref{sssec:sg} and study their  properties. The important observation (cf. Lemma~\ref{linearapproxblowupsym}) is that  blow-up maps at an order 1 point behave  like symmetric geodesics.  
% In Section~\ref{sssec:gc}, we define the global coordinates $(\varrho,\varphi)$.  
 In Section~\ref{sssec:prt}, we will complete  the proof of Theorem~\ref{RegularityTheoremModelSpace} by showing that there exists no order 1 singular points.
 
 \subsubsection{Symmetric Geodesics} \label{sssec:sg}
 A  \emph{symmetric geodesic} is an arclength parameterized geodesic  $\gamma:(-\infty,\infty) \rightarrow ({\bf H},g_{\bf H})$
such that
if we write
$
\gamma=(\gamma_{\rho},\gamma_{\phi})
$
 with respect to the original coordinates $(\rho,\phi)$ of ${\bf H}$, then
\[
\gamma_{\rho}(s)=\gamma_{\rho}(-s) \ \mbox{ and } \ \gamma_{\phi}(s)=-\gamma_{\phi}(-s).
\] 
The behavior of symmetric geodesics is explained by the following lemma. See also Figure \ref{pic:long geodesics}.
\vspace{.5cm}
\begin{figure}[h]
    \centering
    \includegraphics[width=0.9\textwidth]{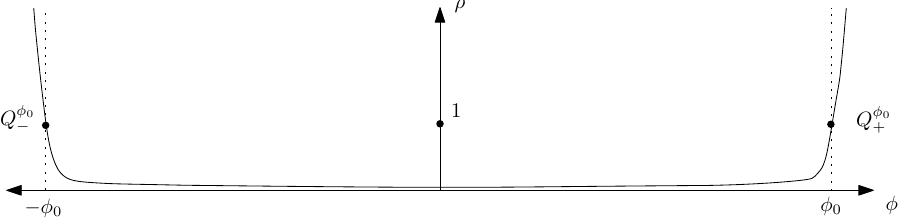}
    \caption{As $\phi_0 \rightarrow \infty$ the geodesic becomes almost vertical}\label{pic:long geodesics}
\end{figure}
\begin{lemma} \label{symapp}
Let $\phi_0>0$ and   $\sigma^{\phi_0}=(\sigma^{\phi_0}_{\rho},\sigma^{\phi_0}_{\phi}):(-\infty,\infty) \rightarrow (\overline{\bf H},d_{\overline{\bf H}})$ be a piecewise geodesic  defined by 
\[
\sigma^{\phi_0}(s)=
\left\{
\begin{array}{ll} (s,\phi_0) & \mbox{for $s \in (0,\infty)$}\\
P_0 & \mbox{for $s=0$}\\
(-s,-\phi_0) & \mbox{for $s \in (-\infty,0)$}.
\end{array}
\right.
\]
  Let $\gamma^{\phi_0}=(\gamma^{\phi_0}_{\rho},\gamma^{\phi_0}_{\phi}):(-\infty,\infty) \rightarrow (\overline{\bf H},d_{\overline{\bf H}})$ be the unit speed  symmetric geodesic passing through the points
\[
Q^{\phi_0}_-=(1,-\phi_0) \mbox{ and } Q^{\phi_0}_+=(1,\phi_0).
\]
  Then 
  \[
  d_{\overline{\bf H}}(\gamma^{\phi_0},\sigma^{\phi_0}) \rightarrow 0 \mbox{  as  } \phi_0 \rightarrow \infty \mbox{  uniformly on  the interval $[-1,1]$}.
  \]
\end{lemma}

\begin{proof}   
We break up the proof into  three claims.\\
\\
{\sc Claim 1}. For any $\phi_0$,  $\gamma^{\phi_0}_{\rho}(0) \leq \gamma^{\phi_0}_{\rho}(s)$ for all $s \in (-\infty, \infty)$.\\
\\
{\sc Proof of Claim 1}.   The first of the geodesic equations (\ref{gequ})
 implies that  $\gamma^{\phi_0}_{\rho}$ is convex.  Combining this with the symmetry of $\gamma^{\phi_0}_{\rho}$,  {\sc Claim 1} follows.\\
\\
{\sc Claim 2.}  $\gamma_{\rho}^{\phi_0}(0) \rightarrow 0$ as $\phi_0 \rightarrow \infty$.\\
\\
{\sc Proof of Claim 2}. 
If  $\gamma_{\rho}^{\phi_0}(0) \geq c>0$, then $\gamma_{\rho}^{\phi_0}(s) \geq c$ for all $s$ by claim~$(i)$ and hence
 \[
1=\left| \frac{d\gamma^{\phi_0}}{ds} \right|^2 =\left( \frac{d\gamma^{\phi_0}_{\rho}}{ds} \right)^2+\gamma_{\rho}^6(s) \left( \frac{d\gamma^{\phi_0}_{\phi}}{ds} \right)^2 \geq c^6 \left( \frac{d\gamma^{\phi_0}_{\phi}}{ds} \right)^2.
\]
Thus,
\[
\phi_0^2 \leq |\gamma_{\phi}^{\phi_0}(1) |^2 \leq  \left( \int_0^1 \left| \frac{d\gamma^{\phi_0}_{\phi}}{ds} \right| ds \right)^2 \leq \int_0^1 \left| \frac{d\gamma^{\phi_0}_{\phi}}{ds} \right|^2  ds \leq c^{-6}.
\]
Since this impossible for large $\phi_0$, we have  {\sc Claim}~2. \\
\\
{\sc Claim 3}.   $d_{\overline{\bf H}}(Q^{\phi_0}_-, \gamma^{\phi_0}(-1))=d_{\overline{\bf H}}(Q^{\phi_0}_+, \gamma^{\phi_0}(1)) \rightarrow 0$ as $\phi_0 \rightarrow  \infty$. \\
\\
{\sc Proof of Claim 3.}  This assertion follows immediately from the fact that $\gamma^{\phi}$ is a unit speed geodesic passing through points $Q^{\phi_0}_-$ and $Q^{\phi_0}_+$ and that $d(Q^{\phi_0}_-,Q^{\phi_0}_+) \rightarrow 2$ as $\phi_0 \rightarrow \infty$.  This proves {\sc Claim 3}.\\
\\
Claims~2 and 3 assert 
\[
d_{\overline{\bf H}}(\sigma^{\phi_0}(0), \gamma^{\phi_0}(0))=d_{\overline{\bf H}}(P_0, \gamma^{\phi_0}(0)) =\gamma_{\rho}^{\phi_0}(0) \rightarrow 0 
\]
and 
\[ d_{\overline{\bf H}}(\sigma^{\phi_0}(1), \gamma^{\phi_0}(1))= d_{\overline{\bf H}}(Q^{\phi_0}_+, \gamma^{\phi_0}(1)) \rightarrow 0.
  \]   
  Since $\gamma^{\phi_0}$ and $\sigma^{\phi_0}$ are geodesics on the interval $[0,1]$, the assertion  follows from the convexity of the function $t \mapsto d(\gamma^{\phi_0}(t),\sigma^{\phi_0}(t))$ (which follows from the NPC condition).
\end{proof}

\begin{definition} \label{defstretchy}
A   map 
$l=(l_{\rho},l_{\phi}):B_1(0) \rightarrow ({\bf H},g_{\bf H})$ is said to be a \emph{symmetric homogeneous degree 1 map}  if 
 \begin{equation} \label{stretchy}
l  (x)=\gamma(Ax^1)
 \end{equation}
where  $A>0$ and $\gamma$ is a symmetric geodesic.  We call $A$ the \emph{stretch factor} or simply the  \emph{stretch} of $l$.
\end{definition}
\begin{definition} \label{defstretchy}
A map
\[
T:\overline{\bf H} \rightarrow \overline{\bf H}
\]
 given by 
\[
 T(P_0)=P_0  \ \ \mbox{ and } \ \ T(\rho,\phi)=(\rho,\phi+c), \ (\rho,\phi) \in {\bf H} 
\]
for some $c \in \R$ is called a   \emph{translation isometry}. Notice that  
if $T$ is a translation isometry,  $l$ is a symmetric homogeneous degree 1 map and $R:\R^n \rightarrow \R^n$ is a rotation, then 
\[
L=T \circ l \circ R:  B_1(0) \rightarrow ({\bf H},g_{\bf H})
\]
 is a homogeneous degree 1 in the sense of Remark~\ref{blowupsforhm}. 
\end{definition}

Let $u_{\sigma_i}$ be a sequence of blow-up maps  of a harmonic map $u:(B_R(0),g) \rightarrow (\overline{\bf H}, d_{\overline{\bf H}})$ at $x_0=0 \in B_R(0)$ converging locally uniformly in the pullback sense a tangent map  $u_*:(B_1(0),g_0) \rightarrow (Y_*,d_*)$.  If $Ord^u(0)=1$, then  $u_*$ maps into a geodesic by Lemma~\ref{proveitsfl}.  Let $A$ be the norm of the gradient of $u_*$.  The following lemma
%, proven in \cite{daskal-meseER}, 
gives more precise information of $u_*$ by embedding this geodesic in $({\bf H},g_{\bf H})$ as  a (translation of) a sequence of symmetric geodesics.   This also explains why symmetric homogeneous degree 1 maps  given in Definition~\ref{defstretchy} naturally arise in the study of harmonic maps into $(\overline{\bf H},d_{\overline{\bf H}})$.  

 \begin{lemma} \label{linearapproxblowupsym}
For a harmonic map $u:(B_R(0),g) \rightarrow (\overline{\bf H}, d_{\overline{\bf H}})$, let   $u_{\sigma_i}$, $u_*$ and $A$ be as above.      If $Ord^u(0)=1$,   then there exist a sequence of translation isometries $T_i:{\bf H} \rightarrow {\bf H}$,  a rotation  $R: \R^n \rightarrow \R^n$ and a sequence of symmetric homogeneous degree 1 maps $l_i:B_1(0) \rightarrow ({\bf H},g_{\bf H})$ with $d_{\overline{\bf H}}(P_0,l_i(0)) \rightarrow 0$  and stretch $A$ 
such that
\[
 \lim_{\sigma_i \rightarrow 0} \sup_{B_r(0)} d_{\overline{\bf H}}(u_{\sigma_i}, T_i \circ l_i \circ R) = 0, \ \forall r \in (0,1).
\]   
 \end{lemma}
 
\begin{proof} 
%For the sake of notational simplicity we denote $d=d_{\overline{\bf H}}$.
By Corollary~\ref{msproperties},  $u_*$ maps onto a geodesic.  By identifying this geodesic with $\R$, we can assume for the rest of the of the proof that $u_*$ is a linear function.  %Thus, the set $\{x \in B_r(0):  u_*(x) \neq u_*(0)\}$ has exactly two connected components $\Omega_+$ and $\Omega_-$.   
After rotating the domain $B_1(0) \subset \R^n$ if necessary, we may assume  
\begin{equation} \label{ax1}
u_*:  B_1(0) \rightarrow \R, \ \ \ \ u_*(x)=Ax^1
\end{equation}
 for some constant $A$ and 
\[
\Omega_{\pm}=\{x=(x^1, \dots, x^n) \in B_1(0):  \pm x^1>0 \}.
\]
 Following the proof of Lemma~\ref{proveitsfl}, Claim 1, we obtain that 
 \begin{equation} \label{ucL}
\mbox{$d_{\overline{\bf H}}(u_{\sigma_i}, L_{\pm,i}) \rightarrow 0$   uniformly on compact sets of $\Omega_{\pm}$}
\end{equation}
 where
 \[
L_{\pm,i}(x)=(Ax^1,\phi_{\pm,i}), \ \  x \in \Omega_{\pm}
 \]
 with  $\phi_{\pm,i}$ equal to the $\phi$-coordinate of  $u_{\sigma_i}(x_{\Omega_{\pm}})$ and $x_{\Omega_{\pm}}= (\pm \frac{1}{2}, 0, \dots, 0) \in \Omega_{\pm}$.   (Here, $x_{\Omega_{\pm}}$ and $\Omega_{\pm}$ replace  $x_{\Omega_0}$ and $\Omega_0$ in Lemma~\ref{proveitsfl}, Claim 1).  Define the map $
 L_i:  B_1(0) \rightarrow \overline{\bf H}
 $
by setting 
 \[
 L_i(x)= 
 \left\{
 \begin{array}{ll}
L_{+,i}(x) & \mbox{if }x^1>0 \\
P_0 & \mbox{if } x^1=0\\
L_{-,i}(x) & \mbox{if } x^1<0\\
\end{array}
 \right.
\]
Since $u_{\sigma_i}$ converges locally uniformly to $u_*(x)=Ax^1$ and $d_{\overline{\bf H}}(L_i(\cdot), L_i(\cdot))=d_{\R}(u_*(\cdot), u_*(\cdot))$ we have,
\begin{equation} \label{ucL"}
d_{\overline{\bf H}}(u_{\sigma_i}(\cdot), u_{\sigma_i}(\cdot))-d_{\overline{\bf H}}(L_i(\cdot), L_i(\cdot)) \rightarrow 0
\end{equation}
uniformly on compact sets of $B_1(0)$.
We claim that 
\begin{equation} \label{ucL'}
\mbox{$d_{\overline{\bf H}}(u_{\sigma_i},L_i) \rightarrow 0$ uniformly on compact subsets of $B_1(0)$.}
\end{equation}
  Indeed, let $K \subset B_1(0)$ be a compact set and $\epsilon>0$ be given.
 For all $i$, we can choose $\delta>0$ such that 
\[
|x^1|<\delta \ \Rightarrow \ d_{\overline{\bf H}}(L_i(x), P_0)<\epsilon.
\]
By (\ref{ucL"})
\[
\lim_{\sigma_i \rightarrow 0}d_{\overline{\bf H}}(u_{\sigma_i}(x), P_0)=\lim_{\sigma_i \rightarrow 0}d_{\overline{\bf H}}(u_{\sigma_i}(x), u_{\sigma_i}(0))=\lim_{i \rightarrow \infty}d_{\overline{\bf H}}(L_i(x), L_i(0))=\lim_{i \rightarrow \infty}d_{\overline{\bf H}}(L_i(x), P_0)
\]
hence for $i$ sufficiently large, 
\[
x \in K \mbox{ and } |x^1|<\delta \ \Rightarrow \ d_{\overline{\bf H}}(u_{\sigma_i}(x), P_0)<\epsilon.
\]
Thus, for $i$ sufficiently large,  $x \in K$ and $|x^1|<\delta$ imply
\begin{eqnarray*}
\lefteqn{d_{\overline{\bf H}}(u_i(x), L_i(x))}\\
 & \leq & d_{\overline{\bf H}}(u_i(x), P_0)+d_{\overline{\bf H}}(L_i(x),P_0) \\
& < & 2\epsilon 
\end{eqnarray*}
For $x \in K$ with $|x^1| \geq \delta$, we have $d_{\overline{\bf H}}(u_{\sigma_i}(x), L_i(x))<\epsilon$ for sufficiently large $i$ by (\ref{ucL}).  This proves (\ref{ucL'}).  

We now use Lemma~\ref{symapp} to replace $L_i$ (up to a translation isometry) with a symmetric homogeneous degree 1 map. Indeed,  
recall that by construction, the $\phi$-coordinate of the  point $L_i(x_{\Omega_{\pm}})$ is  $\phi_{\pm,i}$.  Define
\[
\phi_i:=\frac{|\phi_{+,i}-\phi_{-,i}|}{2},
\]
and let $l_i$ be a symmetric homogeneous 1 map with $l_i(x)=\gamma(Ax^1)$ where $\gamma$ is a geodesic passing through $(A,\phi_i)$ and $(A,-\phi_i)$ and $T_i$ be the translation isometry such that the $\phi$-coordinate of $T_i \circ l_i(x_{\Omega_{\pm,i}})$ is equal to $\phi_{\pm,i}$.    By  Lemma~\ref{symapp}, we conclude that 
\[
d_{\overline{\bf H}}(L_i,T_i \circ l_i) \rightarrow 0 \ \mbox{uniformly on compact subsets of $B_1(0)$}.
\]  
Combined with (\ref{ucL'}), we have proved the assertion. 
\end{proof}

\subsubsection{The completion of the proof of Regularity Theorem~\ref{RegularityTheoremModelSpace}}  \label{sssec:prt} 
The following lemma is the heart of the argument of Regularity Theorem~\ref{RegularityTheoremModelSpace}.      Due to its highly technical nature, we postpone the proof until Section~\ref{sec:  proofofkeylemma}.

\begin{lemma}[{\bf The Key Technical Lemma for the Model Space}] \label{keylemma}
Given $c_0$, $E_0,A>0$, there exists $D_0 \in (0,1)$ that give the following implication. \\
\\
{\bf Assumptions. }  The metric metric $g$ on $B_1(0)$ and the map $v:(B_1(0),g) \rightarrow (\overline{\bf H},d_{\overline{\bf H}})$ satisfy:
\vspace*{0.08in}
\begin{itemize}
 \item[{\bf (i)}]{\bf{(almost Euclidean domain metric)}}  
The metric  $g$ is  close to the Euclidean metric in the sense of (\ref{volumewrtmetric'}).  
\item[{\bf (ii)}]{\bf{(energy decay)}}
The energy of the map $v$
% \[
% v=(v^1, \dots, v^{k-j}):(B_1(0),g) \rightarrow (\overline{\bf H}^{k-j},d_h)
% \]
satisfies
    \[
  E^v(\vartheta) \leq \vartheta^n E_0, \ \ \ \forall \vartheta \in (0,\frac{1}{2}).
\]
\item[{\bf (iii)}]{\bf{(close to a symmetric homogeneous degree 1 map)}} There exists a symmetric homogeneous degree 1 map 
\[
l:B_1(0) \rightarrow ({\bf H},g_{\bf H})
\]
 with stretch $A$ such that
    \[
\sup_{B_{\frac{1}{2}}(0)} d_{\overline{\bf H}}(v,l) < D_0.
\]

\item[{\bf (iv)}]{\bf{(subharmonicity of the distance)}}
For   $\vartheta \in (0,\frac{1}{24})$, $R \in [\frac{5}{8},\frac{7}{8}]$  and a harmonic map $w:(B_{\vartheta R}(0),g)  \rightarrow (\overline{\bf H},d_{\overline{\bf H}})$, 
 \begin{equation} \label{shofhm}
 \sup_{B_{\frac{15\vartheta R}{16}}(0)} d_{\overline{\bf H}}^2(v,w)  \leq  \frac{c_0}{(\vartheta R)^{n-1}} \int_{\partial B_{\vartheta R}(0)} d_{\overline{\bf H}}^2(v,w)d\Sigma.
\end{equation} 
 \end{itemize}
 {\bf Conclusion. }  Then  $v(0) \neq P_0$
\end{lemma}

\begin{proof}
This is a special case of the key technical Lemma~\ref{keylemma'}. The proof  is given in  subsection~\ref{sec:theproof}.
\end{proof}

By combining  Lemma~\ref{linearapproxblowupsym} with Lemma~\ref{keylemma}, we obtain the following 
\begin{proposition}\label{order1pointsmodelspace}
 If $u=(u^{\rho},u^{\phi}):B_1(0) \rightarrow (\overline{\bf H}, d_{\overline{\bf H}})$ is a harmonic map,  then  there exist no order 1 singular points of $u$.
\end{proposition}

\begin{proof}
For $x_0 \in {\mathcal S}(u)$,  let $u_{\sigma_i}$ be a sequence of blow-up maps of $u$ at $x_0$ converging to a tangent map $u_*$.  We want to show $Ord^u(x_0)>1$.  On the contrary,  assume $Ord^u(x_0)=1$.  As in the proof of Lemma~\ref{linearapproxblowupsym}, we assume that $u_*(x)=Ax^1$ (cf. (\ref{ax1}).    For sufficiently small $\sigma_i>0$,  assumption (i) of Lemma~\ref{keylemma} is satisfied with $g$ replaced by $g_{\sigma_i}$.   

Next, since $g_{\sigma_i}$ converges to $g_0$ in $C^k$ (for any $k=1, 2, \dots$) as $\sigma_i \rightarrow 0$, there exists $c_0>0$ (independent of  $\sigma_i$ for $\sigma_i>0$ sufficiently small)  such that for any subharmonic function $f:B_1(0) \rightarrow \R$  with respect to the metric $g_{\sigma_i}$, we have  
\[
 \sup_{B_{\frac{15\vartheta R}{16}}(0)} f  \leq  \frac{c_0}{(\vartheta R)^{n-1}} \int_{\partial B_{\vartheta R}(0)} f d\Sigma.
\]
Furthermore, by (\ref{endec}),  there exists $E_0>0$ such that
 \[
  E^{u_{\sigma_i}}(\vartheta) \leq \vartheta^n E_0, \ \ \forall \vartheta \in (0,\frac{1}{2}).
\]

Choose $D_0>0$ as in Lemma~\ref{keylemma} depending on $E_0,A$ and $c_0$ above.   By Lemma~\ref{linearapproxblowupsym}   there exists $\sigma_i>0$ sufficiently small, a rotation $R:\R^n \rightarrow \R^n$, a translation isometry $T$ and a symmetric homogeneous degree 1 map $l$ with stretch $A$  such that 
\[
\sup_{B_{\frac{1}{2}}(0)} d_{\overline{\bf H}} (T \circ u_{\sigma_i} \circ R, l)<D_0.
\]
In other words, assumption (iii) of Lemma~\ref{keylemma} is satisfied with $v$ replaced by $T \circ u_{\sigma_i} \circ R$.  Since $T$ and $R$ are isometries,  assumption (ii) of Lemma~\ref{keylemma} is satisfied with $v$ replaced by $T \circ u_{\sigma_i} \circ R$.
Furthremore, since $T \circ u_{\sigma_i} \circ R$ is a harmonic map, hence $d_{\overline{\bf H}}^2(T \circ u_{\sigma_i} \circ R,w)$ is a subharmonic function for any harmonic map $w$.  Thus, assumption (iv) of Lemma~\ref{keylemma} is satisfied with $v=T \circ u_{\sigma_i} \circ R$. 
 Lemma~\ref{keylemma} implies that $T \circ u_{\sigma_i} \circ R(0) \neq P_0$ which in turn implies $u(x_0) \neq P_0$, a contradiction.
\end{proof}

\begin{proofofRTMS}
Combine Proposition~\ref{hmintoms} and Proposition~\ref{order1pointsmodelspace}.
\end{proofofRTMS}

\section{Harmonic maps into $\overline{\mathcal T}$}  \label{intoWPcompletion} 

In this section, we prove the  Regularity Theorems~\ref{RegularityTheorem} and~\ref{goto0} for harmonic maps into the  Weil-Petersson completion $(\overline{\mathcal T}, d_{\overline{\mathcal T}})$ of Teichm\"{u}ller space ${\mathcal T}$.   Recall that $\overline{\mathcal T}$ is a stratified space 
\begin{equation}\label{strat}
\overline{\mathcal T}= \bigcup   {\mathcal T}'
\end{equation}
where ${\mathcal T}'$ is either the original Teichm\"{u}ller space ${\mathcal T}$ or  an open stratum of $\overline{\mathcal T}$ (cf. Section~\ref{secC1}).  
For a map $u:(\Omega,g) \rightarrow( \overline{\mathcal T}, d_{ \overline{\mathcal T}})$, we say $x\in \Omega$ is  a \emph{regular point} if $u$  maps a neighborhood of $x$  into an open stratum ${\mathcal T}'$ of $\overline{\mathcal T}$. A  point $x \in \Omega$ is a  \emph{singular point} of $u$ if it is not a regular point.   We define the  \emph{regular set}   ${\mathcal R}(u)$ as the set  of regular points and the \emph{singular set} ${\mathcal S}(u)$ as a set of singular points.
%\[
%{\mathcal R}(u)=\{x \in \Omega: \exists r>0 \mbox{ such that } B_r(u(x)) \subset  {\mathcal T}' \mbox{ for } {\mathcal T}' \mbox{ as in  } (\ref{strat}) \}
%\] 
%and its  \emph{singular set} as
%\[
%{\mathcal S}(u)=\Omega \backslash {\mathcal R}(u).
%\]   

Analogously to the proof of the Regularity Theorem~\ref{RegularityTheoremModelSpace} for  harmonic maps into $(\overline{\bf H},d_{\overline{\bf H}})$, the strategy is to  first show that the set of higher order points is of Hausdorff codimension at least 2, then to study the order 1 points. %(cf. Subsection~\ref{sec:into hbar0}, Proposition~\ref{hmintoms}), 
% that \emph{there exist no order 1 singular points}.
However, this  is  more involved  for harmonic maps into $(\overline{\mathcal T}, d_{\overline{\mathcal T}})$.  More precisely,  because of the  more complicated structure of the stratification for $\overline{\mathcal T}$ as compared to $\overline{\bf H}$, we will  use  an induction based on the codimension of the boundary stratum.  
Nonetheless, the main issue  is  the same  for both  $\overline{\mathcal T}$ and  $\overline{\bf H}$, namely,  the  non-compactness and  degenerating geometry near the boundary.

We will now give the outline of this section.  
According to Section~\ref{secC1},  a neighborhood of a point on an $j$-dimensional open stratum is asymptotically the product of a smooth K\"{a}hler manifold of dimension $j$ and  a neighborhood of ${\mathbb P}_0=(P_0, \dots, P_0)$ in the normal space $\overline{\bf H}^{k-j}$.  
In Section~\ref{deflocalmodel}, we define a local representation $u=(V,v)$ with respect to  this asymptotic product structure;  
more specifically,  $V$ maps into a  smooth  K\"{a}hler manifold of dimension $j$ (hence is referred to as the  \emph{regular component map}) and $v$ maps into  $(\overline{\bf H}^{k-j}, d_h)$ (hence is referred to as the \emph{singular component map}).   
We will prove that the set of higher order points of $u$ is of Hausdorff codimension at least 2 in Section~\ref{higordpts}.  

The rest of Section~\ref{intoWPcompletion} is devoted to studying the  order 1 points of $u$.
For this, we will rely on the   \emph{key technical Lemma} (whose proof is deferred until Section~\ref{sec:  proofofkeylemma})
% to prove that $v$ cannot hit the boundary point ${\mathbb P}_0  \in \overline{\bf H}^{k-j}$ at order 1 points.   
 which gives sufficient conditions for  a map into  $(\overline{\bf H}^{k-j}, d_h)$ to not hit the boundary point ${\mathbb P_0}$.  This is the most challenging aspect of this paper as it introduces techniques to address the non-compactness  and  degenerating geometry of $(\overline{\mathcal T}, d_{\overline{\mathcal T}})$.  
 The difference between this section and Section~\ref{MintoMS} (where we prove the regularity theorem for harmonic maps into the model space)
%Theorem~\ref{RegularityTheoremModelSpace}  
is that the singular component map $v$ \emph{is not a harmonic map}. Again, this is because the Weil-Petersson metric is not a product near the boundary.   (Recall that in the proof of Theorem~\ref{RegularityTheoremModelSpace}, the \emph{key technical Lemma  for the Model Space} is applied to a sequence of blow-up maps of a  harmonic map into $(\overline{\bf H}, d_{\overline{\bf H}}$).)
For this reason,  we introduce in Section~\ref{apprham} the notion of a \emph{sequence of approximately harmonic maps into $(\overline{\bf H}^{k-j}, d_h)$} and prove that if $v_i$ is such a sequence converging to a harmonic map $v_0$ with $Ord^{v_0}(0)=1$, then  $v_0(0) \neq {\mathbb P}_0$.  (This is the generalization of the  result contained in Proposition~\ref{order1pointsmodelspace} for the case of a sequence of harmonic maps).

In Section~\ref{indstepp}, we begin the proof  that the set of order 1 singular points of $u$ is of Hausdorff codimension at least 2 by  setting up an induction  on the codimension of the stratum intersecting the image of $u$. Notice that if the codimension is 0, then $u$ maps into the interior of Teichm\"uller space. Hence $u$ is regular and there is nothing to prove.  The method follows closely our paper 
 \cite{daskal-meseDM}, where we developed a theory of harmonic maps  $u=(V,v)$ into  spaces with an asymptotic product structure  with the map $v$  not necessarily harmonic.  
More specifically, in  \cite{daskal-meseDM} we developed the tools such  as  \emph{monotonicity},  \emph{order function}  and  \emph{tangent map} for almost harmonic maps  to study the singular component map $v$.  
%In particular, in  \cite{daskal-meseDM} we explicitly stated  a list of sufficient conditions  for  a harmonic map $u=(V,v)$ and its target space  that allow the theory to work in great generality.
% \emph{These sufficient conditions  are listed in  \cite{daskal-meseDM}  Section 5  as ``Assumptions".} 
 The purpose of Section~\ref{sec:DM} is to introduce the results from \cite{daskal-meseDM}  needed in this paper and sketch the main ideas of their proof adapted to the case of maps to the Weil-Petesson completion of Teichm\"uller space. In particular, we define the order of the singular component $v$.  Analogously to the case of a harmonic map, we first  show that   the set of  higher order points   of $v$ is of codimension at least 2 and then show  that there are no singular points of $u$ that are also order 1 points of $v$.   
Finally, in Section~\ref{conclusionofproof}, we finish the proof of the Regularity Theorems~\ref{RegularityTheorem} and~\ref{goto0} by completing the inductive step of the argument.

 \subsection{A Local Represention for Maps into  $\overline{\mathcal T}$}\label{deflocalmodel}
For a map  $u:(\Omega,g) \rightarrow (\overline{\mathcal T}, d_{\overline{\mathcal T}})$, recall that  ${\mathcal R}(u)$ is the set of points in $\Omega$ that possess a neighborhood mapping into a single stratum in $\overline{\mathcal T}$ and ${\mathcal S}(u)$ is its complement. 
We decompose the singular set ${\mathcal S}(u)$ as a disjoint union of sets
\begin{eqnarray}\label{sinset}
{\mathcal S}(u)=\bigcup_{j=0}^{k}\hat{\mathcal S}_j(u)
\end{eqnarray}
where  
\[
\hat{\mathcal S}_j(u) =\{x \in {\mathcal S}(u):  \# u(x)=j \}, \ \ j=0,\dots,k.
\]
In other words, $x \in \hat{\mathcal S}_j(u)$ implies that $u(x)$  is a point in a $j$-dimensional stratum.
If $\#u(x)=k$, then $u(x) \in {\mathcal T}$, and hence $u(x) \in {\mathcal R}(u)$.  Thus,   
\begin{equation} \label{fis}
\hat{\mathcal S}_k(u)=\emptyset.
\end{equation}  
For $x_{\star} \in \hat{\mathcal S}_j(u)$, consider the composition $\Psi \circ u$ in $B_{\sigma_{\star}}(x_{\star})$ for a sufficiently small $\sigma_{\star}>0$ where 
\[
\Psi: {\mathcal N} \subset( \overline{\mathcal T}, d_{ \overline{\mathcal T}}) \rightarrow  {\mathcal U} \times {\mathcal V} \subset \C^j \times  \overline{\bf H}^{k-j}
\]
is the coordinate chart defined in Section~\ref{sec:stratification}.  

\begin{definition} \label{locrep}
For $x_{\star} \in \hat{\mathcal S}_j(u)$, we will write  the composition $\Psi \circ u$  in $B_{\sigma_{\star}}(x_{\star})$ as 
\begin{equation} \label{eqforlr}
u=(V,v):(B_{\sigma_{\star}}(x_{\star}),g) \rightarrow  ({\mathcal U} \times {\mathcal V}, d_G)
\end{equation}
where $d_G$ is the distance function induced from  the Weil-Petersson metric $G$ (cf.~Definition~\ref{productmetric}) 
and refer to it as a \emph{local representation} of $u$ at $x_{\star}$.
\end{definition}

Let $H$ and $h$ be as in Corollary~\ref{A2}. The \emph{regular  component map} of $u$ is the map   
\begin{equation} \label{regcpt}
V: (B_{\sigma_{\star}}(x_{\star}),g) \rightarrow (\C^j,H)
\end{equation}
into the hermitian manifold $(\C^j,H)$.  The \emph{singular component} of $u$ is the map 
\begin{equation} \label{singcpt}
v=(v^1, \dots, v^{k-j}):(B_{\sigma_{\star}}(x_{\star}),g) \rightarrow (\overline{\bf H}^{k-j},h).
\end{equation}
%For $x_{\star} \in \hat{\mathcal S}_j(u)$ and  $u=(V,v)$ as in Definition~\ref{locrep}, the singular component map  $v$ maps into ${\mathbf H}^{k-j}$.  
In particular, we observe that (since $v$ can map into interior points of $\overline{\bf H}^{k-j}$), $\#u(x) \geq j$ for all $x \in B_{\sigma_{\star}}(x_{\star})$.  Therefore,   
%$
%Since $\Psi$ is a stratification preserving homeomorphism, if   $x \in B_{\sigma_{\star}} (x_{\star}) \cap \hat{\mathcal S}_l(u)$, then the cardinality of the set
%\[
%\{m \in \{1, \dots, k-j\}:  v^m(x)=P_0\}
%\]
%is $k-l$.
%In particular, we observe that 
\begin{equation} \label{regcpt00}
\hat{\mathcal S}_l(u) \cap B_{\sigma_{\star}}(x_{\star})=\emptyset, \ \ \forall l=0, \dots, j-1.
\end{equation}

\begin{remark} \label{regularityV}
Let $u$ as in (\ref{locrep}) be a harmonic map, $x_0 \in B_{\sigma_{\star}}(\frac{x_{\star}}{2})$, $\sigma_0 \in (0,\frac{\sigma_{\star}}{2})$ and $\phi \in C^\infty_c(B_{\sigma_0}(x_0))$.  By considering a  variation $u_t=(V+t\eta, v)$, where $V=(V^1, \dots, V^j)$ is as in (\ref{regcpt}) and $\eta =(\eta^1, \dots, \eta^j)$ with  $\eta^I= \sum_K G^{IK} (V,v) \varphi$, a straightforward computation implies 
\[
-\int_{B_{\sigma_0}(x_0)} g^{\alpha \beta} \frac{\partial V^I}{\partial x^\alpha} \frac{\partial \varphi}{\partial x^\beta} d\mu = \int_{B_{\sigma_0}(x_0)} \varphi \cdot f d\mu
\]
for a bounded function $f$. (For explicit details, see~\cite[Lemma 50]{daskal-meseDM}.) 
Thus, by elliptic regularity,  $V \in W^{2,p}(B_{\sigma_0}(x_0))$ and  $V$ is (H\"older) continuous. \end{remark}

\subsection{Higher Order Points of $u$}\label{higordpts}
%Given a harmonic map $u:\Omega \rightarrow \overline{\mathcal T}$, we will show 
The purpose of this subsection is to show that  the set of higher order points of a harmonic map $u:(\Omega,g) \rightarrow (\overline{\mathcal T}, d_{\overline{\mathcal T}})$ 
is of Hausdorff codimension at least 2.  Let  $x_{\star} \in \hat{\mathcal S}_j(u)$
and 
\[
u=(V,v):(B_{\sigma_{\star}}(x_{\star}),g) \rightarrow ({\mathcal U} \times {\mathcal V}, d_G)
\]
be  a local representation  (cf. Definition~\ref{locrep}).
For $x_0 \in \hat{\mathcal S}_j(u) \cap B_{\sigma_{\star}}(x_{\star})$,  
 identify $x_0=0$ via normal coordinates for the metric $g$  and  identify $V(x_0)=0$ via normal coordinates for the metric $H$.  
We consider the family of blow-up maps  $u_\sigma$  of the harmonic map $u$ described in Remark~\ref{blowupsforhm}; in other words, $u_\sigma$ is scaled with respect to the scaling factor 
\begin{equation} \label{standardscaling}
\mu(\sigma) = \sqrt{\frac{I^u(\sigma)}{\sigma^{n-1}}}.
\end{equation}
More precisely, we consider the maps 
\begin{equation} \label{buwrtu}
V_\sigma:(B_1(0),g_\sigma)  \rightarrow  (\C^j,H_{\mu(\sigma)}) \ \ \ v_\sigma:(B_1(0),g_\sigma)  \rightarrow  (\overline{\bf H}^{k-j},d_h)
\end{equation}
and
\begin{equation} \label{bups}
u_\sigma=(V_\sigma,v_\sigma): (B_1(0),g_\sigma) \rightarrow ({\mathcal U} \times {\mathcal V},d_{G_{\mu(\sigma)}})
\end{equation}
where
\[
V_\sigma(x)=\mu^{-1}(\sigma) V(\sigma x) \  \mbox{and} \ v_\sigma(x)=\mu^{-1}(\sigma) v(\sigma x).
\]
The metrics $g_\sigma$ and $H_{\mu(\sigma)}$ are defined in terms of 
 the normal coordinates of $g$ on $B_1(0)$ and the  coordinates $V^1,...,V^j$ on $\C^j$ by
 \[
{g_\sigma}_{ kl}(x)=g_{kl}(\sigma x), \ \mbox{and} \ {H_{\mu(\sigma)}}_{IJ}(y)=H_{IJ}(\mu(\sigma) y).
\]
The metric $G_{\mu(\sigma)}$ on the stratified space ${\mathcal U} \times {\mathcal V}$ is defined similarly by
\[
{G_{\mu(\sigma)}}_{ kl}(y,P)=G_{ kl}(\mu(\sigma)y, \mu(\sigma)P)
\] 
 in terms of the  coordinates $V^1,...,V^j$ on $\C^j$ and the homogeneous coordinates $(\rho^1, \Phi^1) ,...,(\rho^{k-j}, \Phi^{k-j})$ on $\overline{\bf H}^{k-j}$.  In the above, the dilation map on $\C^j$ is the standard multiplication map, whereas the dilation map on  $\overline{\bf H}^{k-j}$ is defined in (\ref{hcs}).
We denote by
\[
H\oplus h_{\mu(\sigma)}=H_{\mu(\sigma)} \oplus h_{\mu(\sigma)}
\]
the product metric on the stratified space $\C^j \times \overline{\bf H}^{k-j}$ and let $d_{H\oplus h_{\mu(\sigma)}}$, $d_{G_{\mu(\sigma)}}$ denote the corresponding distance functions. 
\begin{lemma} \label{disest2}
Let $u=(V,v)$, $u_\sigma=(V_\sigma,v_{
\sigma})$ be as above and  $d_{H \oplus h}$, $d_G$ be the distance functions on ${\mathcal U} \times {\mathcal V}$  induced by the metric $H \oplus h$ and $G$ respectively.
\begin{itemize}
\item[(i)]  There exists a constant $C>0$ such that for $P,Q \in {\mathcal U} \times {\mathcal V}$ at  distance at most $\lambda$ from $(0, {\mathbb P}_0)$,
\[
\left( 1-C\lambda^2  \right)\leq \frac{d_{H \oplus h}(P,Q)}{d_{G}(P,Q)} \leq \left(1+C \lambda^2 \right). 
\]
\item[(ii)] If $h=(W,w):B_1(0) \rightarrow ({\mathcal U} \times {\mathcal V},G)$ is Lipschitz continuous in $B_R(0)$, for some $R \in (0,1)$, then there exists $C>0$ such that
\[
\left|  |\nabla h|^2(x)- \left(|\nabla W|^2(x)+|\nabla w|^2(x) \right) \right| \leq C d^2_h(w(x),(0, {\mathbb P}_0))
\]
for almost every $x \in B_R(0)$.
% and every   $x \in {\mathcal R}(h) \cap B_R(0)$.  
Here, we view $W$ and $w$ as maps into $({\mathcal U}, H)$ and $({\mathcal V}, d_h)$ respectively.
%  {\color{red} {\bf !!Lip functions are only a.e.~differentiable!!  DO WE REALLY NEED PART (ii)????}}
\item[(ii)] Given $R \in (0,1)$, there exists $C>0$ such that for  almost every  $x \in B_R(0)$, every   $x \in {\mathcal R}(u) \cap B_R(0)$ and $\sigma>0$ sufficiently small, the  blow-up map 
\[
u_\sigma=(V_\sigma,v_\sigma): (B_1(0),g_\sigma) \rightarrow ({\mathcal U} \times {\mathcal V}, d_{G_{\mu(\sigma)}})
\]
of the harmonic map $u$ with scaling factor  (\ref{standardscaling}) 
satisfies
\[
(1+C  \sigma^2)^{-1}   |\nabla u_\sigma|^2(x) \leq   |\nabla V_\sigma|^2(x)+\sum_{i=1}^{k-j} |\nabla v_\sigma^i|^2(x)  \leq (1+C \sigma^2) |\nabla  u_\sigma|^2(x). 
\]
\end{itemize}
\end{lemma}

\begin{proof} 
Part (i) follows from the  $C^0$-estimates of $G$ contained in Theorem~\ref{dwyw}.  Indeed,  for any vector $\gamma' \in T_{P'} (\C^j \times {\bf H}^{k-j})$ with $P' \in B_{\lambda}({\mathbb P}_0)$, we have
\[
\left| <\gamma',\gamma'>_{H \oplus h} - <\gamma',\gamma'>_{G}\right| \leq C \lambda^2 <\gamma',\gamma'>_{H \oplus h}.
\]
Let
\[
\gamma:  [0,d_{G}(P,Q)] \rightarrow  \C^j \times {\bf H}^{k-j}
\]
 be the arclength parameterized geodesic with respect to $d_G$ between $P \in B_{\lambda}({\mathbb P}_0)$ and $Q \in B_{\lambda}({\mathbb P}_0)$.  Then
 \begin{eqnarray*} \label{asyoneway}
 d_{H\oplus h}^2(P,Q)
 &  \leq & \left( \int_0^{d_{G}(P,Q)}  <\gamma' , \gamma'>_{H \oplus h}^{\frac{1}{2}} dt \right)^2 %+ \left(  \int_0^{d_G(P,Q)} <\gamma_2' , \gamma_2'>_{h}^{\frac{1}{2}}  dt \right)^2 
 \nonumber \\
  &  \leq & d_{G}(P,Q)  \int_0^{d_{G}(P,Q)}  <\gamma' , \gamma'>_{H \oplus h} dt  \nonumber \\
%  & & \ +   d_G(P,Q) \int_0^{d_G(P,Q)} <\gamma_2' , \gamma_2'>_{h} dt  \nonumber \\
    &  \leq  & (1+C\lambda^2) d_{G}(P,Q) \int_0^{d_{G}(P,Q)}  <\gamma' , \gamma'>_{G} dt \nonumber \\
     %   &  = & (1+C\sigma^2)d_G(P,Q) \int_0^{d_G(P,Q)}  1 +C\sigma^2 dt  \nonumber \\
        &  \leq  &  d_{G}^2(P,Q) \left( 1+C\lambda^2  \right).
\end{eqnarray*}
Next, let 
\[
\gamma:  [0,d_{H \oplus h}^2(P,Q)] \rightarrow \C^j \times {\bf H}^{k-j}
\]
 be the  arclength parameterized geodesic with respect to $d_{H \oplus h}$ between $P$ and $Q$. 
Thus
\begin{eqnarray*} \label{asyotherway}
d_{G}^2(P,Q) 
& \leq & \left( \int_0^{d_{H \oplus h}(P,Q)} < \gamma', \gamma'>^{\frac{1}{2}}_{G}  dt \right)^2  \nonumber \\
& \leq & d_{H \oplus h}(P,Q) \int_0^{d_{H \oplus h}(P,Q)} < \gamma', \gamma'>_{G}  dt  \nonumber \\
& \leq  & (1+C\lambda^2)d_{H \oplus h}(P,Q)
\int_0^{d_{H \oplus h}(P,Q)}  <\gamma' , \gamma'>_{H \oplus h} dt  \nonumber \\
%& = & \sqrt{d_H^2(P_1,Q_2)+d^2(P_2,Q_2)}   \int_0^{\sqrt{d_H^2(P_1,Q_2)+d^2(P_2,Q_2)}}  1 + C\sigma^2dt  \nonumber \\
& \leq &  d_{H \oplus h}^2(P,Q)   \left( 1+C\lambda^2 \right).
\end{eqnarray*}
This completes the proof of (i). 
The inequalities of (ii) hold for almost every $x \in B_R(0)$ by the definition of  energy density (cf.  \cite{korevaar-schoen1}) and (i). 
%  {\color{red}  THIS IS NOT RIGHT! $\Rightarrow$ The smoothness  imply the the inequalities of (ii) hold for  every   $x \in {\mathcal R}(u) \cap B_R(0)$.}  
Finally, since $u_\sigma$ is uniformly Lipschitz continuous in $B_R(0)$ (cf. (\ref{uLip})),  (iii) follows from (ii).
\end{proof}

\begin{lemma} \label{blowupcomponent}
%Let $u:\Omega \rightarrow \overline{\mathcal T}$ be a harmonic map,  $x_{\star} \in \Omega$ and $u=(V,v)$ be a local representation at $x_{\star}$.   
Let $u=(V,v)$ a local representation at $x_{\star} \in \hat{\mathcal S}_j(u)$.  For $x_0 \in \hat{\mathcal S}_j(u) \cap B_{\sigma_{\star}}(x_{\star})$,  there exists a sequence $\sigma_i \rightarrow 0$  such that the  blow-up maps $u_{\sigma_i}=(V_{\sigma_i},v_{\sigma_i})$ of $u$ at $x_0$ converge locally uniformly in the pullback sense to a nonconstant map   
\[
u_*=(V_*,v_*)=(V_*,v_*^1, \dots, v^{k-j}_*):B_1(0) \rightarrow \C^j \times Y_{1*} \times \dots \times Y_{k-j*}
\] 
where  $(Y_{1*},d_{1*}), \dots,   (Y_{k-j*},d_{k-j*})$ are NPC spaces and the sequences    $V_{\sigma_i}$, $v^1_{\sigma_i}$, $\dots$, $v^{k-j}_{\sigma_i}$ converge locally uniformly in the pullback sense to   homogeneous degree $\alpha$ harmonic maps  $V_*$, $v_*^1$, $ \dots$, $v^{k-j}_*$ respectively. 
\end{lemma}

\begin{proof}
For any $r \in (0,1)$, Lemma~\ref{disest2} and (\ref{uLip})  imply that there exists $C>0$ such that
 \begin{equation} \label{locallipboundforblowup}
 |\nabla V_\sigma|^2, \ |\nabla v_\sigma^1|^2, \dots, |\nabla v_\sigma^{k-j}|^2 \leq C \mbox{ in } B_r(0)
 \end{equation}
for sufficiently small $\sigma$ (with respect to the metric $g(0)$ on the domain which is uniformly equivalent to $g_\sigma$ for $\sigma$ small). 
Let $\sigma_i \rightarrow 0$ be such that $u_{\sigma_{i}}$ converges to a tangent map  $u_*$ locally uniformly in the pullback sense
(cf.~Remark \ref{blowupsforhm}). 
Applying the Compactness Theorem~\ref{KScpct} and a diagonalization argument, we also have that there exist a subsequence of $\sigma_i$   (which we call again $\sigma_{i} $  for the sake of simplicity),  NPC spaces  $(Y_{1*},d_{1*}), \dots,   (Y_{k-j*},d_{k-j*})$  and maps $V_*: {\bf R}^n  \rightarrow  (\C^j, H(0))$, $v_*^1: {\bf R}^n  \rightarrow (Y_{1*},d_{1*}), \dots, v_*^{k-j}: {\bf R}^n  \rightarrow (Y_{k-j*},d_{k-j*})$   
such that  $V_{\sigma_{i}}$, $v_{\sigma_{i}}^1$, $\dots$, $v_{\sigma_{i}}^{k-j}$  converge  locally uniformly in the pull-back sense to  $V_*, v_*^1, \dots, v_*^{k-j}$  respectively.  
Furthermore, Lemma~\ref{disest2} implies that for $x',x'' \in B_1(0)$, 
\[
d_{G_{\sigma_{i}}}^2(u_{\sigma_{i}}(x'), u_{\sigma_{i}}(x''))= d_{H_{\sigma_{i}}}^2(V_{\sigma_{i}}(x'), V_{\sigma_{i}}(x''))+\sum_{\mu=1}^{k-j} d^2_{\bf H}(v_{\sigma_{i}}^\mu(x'), v_{\sigma_{i}}^\mu(x''))+O(\sigma_{i}^2).
\] 
Thus, we conclude that $u_{\sigma_{i}}$ converges locally uniformly in the pullback sense to 
\[
(V_*,v_*^1, \dots, v^{k-j}_*):B_1(0) \rightarrow \C^j \times Y_{1*} \times \dots \times Y_{k-j*}
\]  
and
\[
d_*^2(u_*(x'), u_*(x''))= |V_*(x') - V_*(x'')|^2+\sum_{m=1}^{k-j}  d_{m*}^2(v_*^m(x'), v_*^m(x'')).
\]    
In particular, we can now assume that  $u_*$ is the map $(V_*,v_*^1, \dots, v^{k-j}_*)$.  The harmonicity of $V_*$, $v_*^1, \dots, v^{k-j}_*$ is implied by the harmonicity of the tangent map $u_*$.   Furthermore, the homogeneity of tangent map $u_*$ implies the homogeneity of $V_*$, $v_*^1,\dots,v_*^{k-j}$.
\end{proof}

\begin{lemma}\label{ordergap'}
Let $u:(\Omega,g) \rightarrow  (\overline{\mathcal T}, d_{\overline{\mathcal T}})$ be a harmonic map.     
There exists  $\epsilon_0>0$ depending only on the dimension $n$  of $\Omega$ such that for $x_0 \in \hat{\mathcal S}_j(u)$ and  a tangent map $u_*$ of $u$ at $x_0$,  we have
\begin{equation} \label{fo}
Ord^u(x_0)=1 \ \mbox{ or } \ Ord^u(x_0) \geq  1+\epsilon_0
\end{equation}
and 
\begin{equation} \label{so}
\dim_{\mathcal H}({\mathcal S}^{>1}(u_*)) \leq  n-2.
\end{equation}
\end{lemma}

\begin{proof}
For $x_0 \in {\mathcal R}(u)$, statements (\ref{fo}) and (\ref{so}) obviously hold (with $\epsilon_0=1$) since all the strata of $\overline{\mathcal T}$ are smooth  manifolds.  Thus, now consider 
 $x_0 \in \hat{\mathcal S}_j(u)$.  By Lemma~\ref{blowupcomponent}, there exists a sequence of blow-up $u_{\sigma_i}=(V_{\sigma_i},v_{\sigma_i})$ at $x_0$ that converges locally uniformly in the pullback sense to a map   
\[
u_*=(V_*,v_*^1, \dots, v^{k-j}_*):B_1(0) \rightarrow \C^j \times Y_{1*} \times \dots \times Y_{k-j*}
\] 
with $V_*, v_*^1, \dots, v^{k-j}_*$ homogeneous  harmonic maps and $V_{\sigma_i}$, $v_{\sigma_i}=(v_{\sigma_i}^1, \dots, v^{k-j}_{\sigma_i})$ converging locally uniformly in the pullback sense to $V_*$, $v_*=(v_*^1, \dots, v^{k-j}_*)$ respectively.
First, assume $V_*$ is  non-constant.  Then   $Ord^{u_*}(0)=Ord^{V_*}(0)$, and since $V_*$ is a harmonic map into Euclidean space, statements (\ref{fo}) and (\ref{so}) obviously hold (again with $\epsilon_0=1$).
Alternatively, assume that $V_*$ is  a constant map.   In this case,
\begin{equation} \label{crackers}
\lim_{\sigma_i \rightarrow 0} \sup_{\partial B_r(0)} d_{H_{\mu(\sigma_i)}}(V_{\sigma_i}(0),V_{\sigma_i}) =0, \ \ \ \forall r \in (0,1).
\end{equation}
Define
\[
\hat{u}_{\sigma_i}: B_{\frac{1}{2}}(0) \rightarrow (\C^j \times \overline{\bf H}^{k-j},d_{H\oplus h_{\mu(\sigma_i)}}),  \ \ \ \hat{u}_{\sigma_i} =(V_{\sigma_i}(0),v_{\sigma_i})
\]
and let 
 \[
\phi_{\sigma_i}: B_{\frac{1}{2}}(0) \rightarrow (\C^j \times \overline{\bf H}^{k-j},d_{H \oplus h_{\mu(\sigma_i)}}),  \ \ \ \phi_{\sigma_i} =(W_{\sigma_i},w_{\sigma_i})
\]
be the  harmonic map with boundary values equal to $\hat{u}_{\sigma_i}$.  
 Since $\phi_{\sigma_i}$ and $u_{\sigma_i}$ are harmonic maps, $d_{H\oplus h_{\mu(\sigma_i)}}^2(\phi_{\sigma_i}, u_{\sigma_i}, )$ is a weakly subharmonic function by \cite{korevaar-schoen1} Lemma 2.4.2.  Thus, there exists a constant $c_0>0$ such that
\[
d^2_{H\oplus h_{\mu(\sigma_i)}}(\phi_{\sigma_i}(x), u_{\sigma_i}(x)) \leq c_0 \int_{\partial B_{\frac{1}{2}}(0)} d^2_{H\oplus h_{\mu(\sigma_i)}}(\phi_{\sigma_i},u_{\sigma_i}) d\Sigma.
\] 
By Lemma~\ref{disest2} and noting that  $\phi_{\sigma_i}=\hat{u}_{\sigma_i}$  on $B_{\frac{1}{2}}(0)$, we have
 \begin{eqnarray*}
\lim_{\sigma_i \rightarrow 0} \sup_{B_{\frac{1}{4}}(0)} d^2_h(w_{\sigma_i}(x), v_{\sigma_i}(x))
 & \leq & C  \lim_{\sigma_i \rightarrow 0}  \sup_{B_{\frac{1}{4}}(0)} d^2_{H\oplus h_{\mu(\sigma_i)}}(\phi_{\sigma_i}(x), u_{\sigma_i}(x))
 \\
 & \leq &  Cc_0 \lim_{\sigma_i \rightarrow 0}\int_{\partial B_{\frac{1}{2}}(0)} d^2_{H\oplus h_{\mu(\sigma_i)}}(\phi_{\sigma_i},u_{\sigma_i}) d\Sigma\\
 & = &  Cc_0 \lim_{\sigma_i \rightarrow 0}\int_{\partial B_{\frac{1}{2}}(0)} d^2_{H\oplus h_{\mu(\sigma_i)}}(\hat{u}_{\sigma_i},u_{\sigma_i}) d\Sigma\\
% &  \leq & C \lim_{\sigma_i \rightarrow 0}\int_{\partial B_{\frac{1}{2}}(0)} d^2(W_{\sigma_i},V_{\sigma_i}) d\Sigma\\
  &  \leq & C\lim_{\sigma_i \rightarrow 0} \int_{\partial B_{\frac{1}{2}}(0)} d^2_{H_{\mu(\sigma_i})}(V_{\sigma_i}(0),V_{\sigma_i}) d\Sigma\\
  & = & 0 \ \ \ \ \mbox{(by (\ref{crackers}))}.
 \end{eqnarray*}
 Thus, the sequence $w_{\sigma_i}$ of harmonic maps into $\overline{\bf H}^{k-j}$ converges locally uniformly in the pullback sense to $v_*$ and $w_{\sigma_i}(0) \rightarrow {\mathbb P}_0$.  Applying 
Lemma~\ref{proveitsfl} with $w_i =w_{\sigma_i}$ and $v_0=v_*$, we conclude  that there exists $\epsilon_0 \in (0,1]$  satisfying (\ref{fo}) and also that (\ref{so}) is valid.
 \end{proof}

The following is  the main result of this subsection.

\begin{proposition}\label{highorderpts}
If $u: (\Omega,g) \rightarrow( \overline{\mathcal T}, d_{ \overline{\mathcal T}})$ is a harmonic map from an $n$-Riemannian domain, then the set ${\mathcal S}^{>1}(u)$ of higher order points  is of Hausdorff co-dimension 2; i.e.~
\[
\dim_{\mathcal H} ({\mathcal S}^{>1}(u)) \leq n-2.
\]
\end{proposition}

\begin{proof}
%We need only to prove the assertion for a local representation $u=(V,v)$ at $x_{\star} \in \Omega$.   
Given $x_0 \in \Omega$, 
%\in B_{\sigma_{\star}}(x_{\star})$, 
let  $u_{\sigma_i}$ be a sequence of blow-up maps that converges locally uniformly in the pullback sense to a tangent map $u_*$.   It  suffices to check  assumptions (i) and  (ii) of Lemma~\ref{cd2general}.  
To check (i), assume  $x_i \in {\mathcal S}^{>1}(u_{\sigma_i})$ with  $x_i \rightarrow x_*$.  By the order gap property of Lemma~\ref{ordergap'}, we have
$Ord^{u_{\sigma_i}}(x_i) \geq 1+\epsilon_0$.  
The upper semicontinuity of order (cf. Lemma~\ref{uscord}) implies $Ord^{u_*}(x_*) \geq 1+\epsilon_0$ which in turn implies $x_* \in {\mathcal S}^{>1}(u_*)$.  This verifies (i).    By Corollary~\ref{msproperties}, we have $\dim_{\mathcal {H}}({\mathcal S}^{>1}(u_*)) \leq n-2$.  This verifies  (ii).
\end{proof}

In view of Proposition~\ref{highorderpts}, it makes sense to disregard the higher order  points of the singular set of $u$. Thus, with the notation as in Section~\ref{deflocalmodel}, we set
\begin{equation} \label{sj}
{\mathcal S}_j(u) =\hat{\mathcal S}_j(u) \backslash {\mathcal S}^{>1}(u).
\end{equation}
In other words, ${\mathcal S}_j(u)$ is the set  singular points of $\hat{\mathcal S}_j(u)$ of order 1.

 \subsection{A Sequence of Asymptotically Harmonic Maps}\label{apprham}
 
Let $u:(\Omega,g) \rightarrow( \overline{\mathcal T}, d_{ \overline{\mathcal T}})$ be a harmonic map, $x_{\star} \in S_j(u)$ and $u=(V,v)$ be a local representation at $x_{\star}$.   Since the Weil-Petersson metric is not a product near the boundary of $ \overline{\mathcal T}$, neither the regular component  map $V$ nor  the singular component map $v$ is a harmonic map.  We will see later (cf. Lemma~\ref{compblowupu} and Lemma~\ref{compblowupv}) that the singular component $v$  is  asymptotically harmonic in the sense that a sequence of  blow-up maps  of $v$ at $x_0 \in {\mathcal S}_j(u) \cap B_{\frac{\sigma_{\star}}{2}}(x_{\star})$ is a   \emph{sequence of asymptotically harmonic maps}.  We now define this notion.  

\begin{definition}  \label{asymptoticallyharmonic}
We say  that a sequence of maps $v_i:(B_1(0),g_i) \rightarrow (\overline{\bf H}^{k-j},d_h)$  with $v_i(0)={\mathbb P}_0$ is  a \emph{sequence of asymptotically harmonic maps} if the following conditions are satisfied: 
\begin{itemize}
    \item[(i)]  The sequence of metrics  $g_i$  on $B_1(0) \subset {\bf R}^n$ converges in $C^\infty$  to the Euclidean metric $g_0$ on $B_1(0) \subset {\bf R}^n$.
\item[(ii)] There exists a constant $E_0>0$ such that  $E^{v_i}(\vartheta) \leq \vartheta^n E_0$ for $\vartheta \in (0,\frac{3}{4}]$ where $n$ is the dimension of the domain $B_1(0)$. 

%\item[$(iii)$]  \emph{The total energies of $\{v_i\}$ are uniformly bounded.}
  
\item[(iii)]  The sequence $v_i \big|_{B_{\frac{1}{2}}(0)}$ converges locally uniformly in the pullback sense to a  homogeneous harmonic map $v_0:(B_{\frac{1}{2}}(0),g_0) \rightarrow (Y_0,d_0)$  into an NPC space. 
(Note that we also allow $v_0$ to be the constant map  for technical purposes.)

\item[(iv)]  For any fixed  $R \in (0,1)$, $r \in (0,1)$  and $c>0$, there exist $c_0>0$ and a sequence $c_i \rightarrow 0$ such that for any harmonic map  $w:(B_R(0),g_i)  \rightarrow  \overline{\bf H}^{k-j}$  with 
\[
\sup_{B_R(0) }d_h(w,{\mathbb P}_0) \leq c,
\]
we have
%$w\big|_{\partial B_R(0)} =v_\sigma\big|_{\partial B_R(0)}$, we have
 \begin{equation} \label{abc}
 \sup_{B_{r \vartheta}(0)} d_h^2(v_i,w)  \leq  \frac{c_0}{\vartheta^{n-1}}  \int_{\partial B_{\vartheta}(0)} d_h^2(v_i,w)d\Sigma_{g_i} + c_i \vartheta^3, \ \forall  \vartheta \in (0,R]
 %c s^{2-n}  \int_{ B_{s}(0)} d^2(v_i,v_i(0))d\mu_{g_i}
 \end{equation} 
 where $\Sigma_{g_i}$ is the volume form on $\partial B_{\vartheta}(0)$ with respect to the metric $g_i$.
\end{itemize}
\end{definition}
 
 \begin{remark}\label{harmisasharm}
The sequence of  blow-up maps of a harmonic map $u:(B_1(0),g) \rightarrow (\overline{\bf H}^{k-j},d_h)$ with $u(0)= {\mathbb P}_0$ as in Remark~\ref{blowupsforhm} is a sequence of asymptotically harmonic maps.  In particular, since $u_{\sigma_i}$ is harmonic  for each $i$, inequality (\ref{abc}) is satisfied with $v_i=u_{\sigma_i}$ and $c_i=0$ (cf. \cite{korevaar-schoen1} Lemma 2.4.2; replace $\eta$ by $t\eta$ and take the limit $t \rightarrow 0$).  
 \end{remark}

\begin{remark} \label{thepropertyoftheharmonicmapw}
The theory we developed for a sequence of asymptotically harmonic maps in \cite{daskal-meseDM} only requires that inequality (\ref{abc})  holds for the following two types of harmonic maps: (i) 
% used in this paper.  Indeed, in this paper, we only need that  $w$ i
the Dirichlet solution  with $w\big|_{\partial B_{\frac{3}{4}}(0)}=v_i\big|_{\partial B_{\frac{3}{4}}(0)}$ when $v_i$ is uniformly Lipschitz continuous in $B_{\frac{3}{4}}(0)$ and  (ii) 
$w$ is identically equal to ${\mathbb P}_0$.   
\end{remark}

The importance of a sequence of asymptotic harmonic map   is that the limit map $v_0$ satisfies the following property.  This should be compared to the result about 
the limit of harmonic maps in  Lemma~\ref{proveitsfl}.

\begin{lemma} \label{ahmresult}
 Let $v_{i}:(B_1(0),g_i) \rightarrow  (\overline{\bf H}^{k-j},d_h)$   be
a sequence of asymptotic harmonic maps with $v_i(0) ={\mathbb P}_0$.
Then $v_0$ (cf.~Definition~\ref{asymptoticallyharmonic} (iii)) 
 maps into a product of NPC spaces; i.e.  
\[
v_0=(v_0^1, \dots, v_0^{k-j}): B_{\frac{1}{2}}(0)  \rightarrow Y_0=Y_0^1 \times \dots \times Y_0^{k-j}
\]
where $v_0^{\mu}: B_{\frac{1}{2}}(0) \rightarrow (Y_0^{\mu},d_0^{\mu})$ (for $\mu=1, \dots, k-j$) is a homogeneous harmonic map  into an NPC space. 
If $v_0$ is non-constant, then
there exists $\epsilon_0>0$ such that 
 \begin{equation} \label{ks}
Ord^{v_0}(0)=1 \ \mbox{ or } \ Ord^{v_0}(0) \geq  1+\epsilon_0.
\end{equation}
If $Ord^{v_0}(0)=1$, then either $v_0^{\mu}$ maps into a geodesic or $v_0^{\mu}(x)=P_0$ for all $x \in B_{\frac{1}{2}}(0)$.
Furthermore, set of higher order points of $v_0$  has codimension at least 2; i.e.
\begin{equation} \label{em}
\dim_{\mathcal H}({\mathcal S}^{>1}(v_0)) \leq  n-2.
\end{equation}
\end{lemma}

\begin{proof}  
Let $w_{i}:B_{\frac{3}{4}}(0) \rightarrow ({\bf H}^{k-j},d_h)$ be the harmonic map whose boundary values agree with that of $v_{i}\big|_{B_{\frac{3}{4}}(0)}$.  
Let 
\begin{equation} \label{nout}
R=\vartheta=\frac{3}{4}, \ r=\frac{2}{3}.
\end{equation}
%and $w=w_i$  in (iv) of Definition~\ref{asymptoticallyharmonic}. 
By  Definition~\ref{asymptoticallyharmonic} (ii)
\begin{equation}\label{ebr}
E^{w_i}(\frac{3}{4}) \leq E^{v_i}(\frac{3}{4}) \leq (\frac{3}{4})^n E_0,
\end{equation}
 hence Theorem~\ref{KSlip} implies that for a fixed $z_0 \in \partial B_{\frac{3}{4}}(0)$ and any 
 $x \in  B_{\frac{3}{4}}(0)$,
 \[ 
 d( w_i(x), w_i(z_0))
 \]
is uniformly bounded. This, combined with Definition~\ref{asymptoticallyharmonic} (iii), implies for any 
 $x \in  B_{\frac{3}{4}}(0)$,
\begin{eqnarray*}
d( w_i(x), {\mathbb P}_0) &\leq& d( w_i(x), w_i(z_0))+ d( w_i(z_0), {\mathbb P}_0) \\
&=& d( w_i(x), w_i(z_0))+ d( v_i(z_0), v_i(0))\\
&\leq& c,
\end{eqnarray*}
hence by (iv) of Definition~\ref{asymptoticallyharmonic}
we obtain
\begin{equation} \label{vwconverg}
\lim_{i \rightarrow \infty} \sup_{B_{\frac{1}{2}}(0)} d^2_h(v_{i},w_{i})  =0.
\end{equation}
Again by  (\ref{ebr}), Theorem~\ref{KSlip} implies that $\{w_i\}$  has  uniform local  Lipschitz estimates which in turn implies that  $\{w_i^{\mu}\}$ has  uniform local  Lipschitz estimates for each $\mu=1, \dots, k-j$.  Thus, by Compactness Theorem~\ref{KScpct},  there exists  a subsequence of $w_i^{\mu}\big|_{B_{\frac{1}{2}}(0) } $ (which we shall still denote again by the same notation for simplicty) that converges locally uniformly in the pullback sense to a limit map  $v_0^{\mu}:B_{\frac{1}{2}}(0) \rightarrow (Y_0^{\mu},d_0^{\mu})$ into an NPC space.  By (\ref{vwconverg}), the sequence $v_{i}^{\mu}$ also converges locally uniformly in the pullback sense to $v_0^{\mu}$.  
Thus, combining this with Definition~\ref{asymptoticallyharmonic} (iii), we can write $v_0=  (v_0^1, \dots, v_0^{k-j})$.  Furthermore, (\ref{vwconverg}) also implies
% we can assume without the loss of generality that  
%\[
%v_0\big|_{B_{\frac{1}{2}}(0)}=(v_0^1, \dots v_0^{k-j})
%\] and
that  $\lim_{i \rightarrow \infty} (w_i^{\mu}(0),P_0) = 0$.
Thus, the assertions (\ref{ks}) and (\ref{em}) follow from Lemma~\ref{proveitsfl}.
\end{proof}

Lemma~\ref{ahmresult} leaves the possibility that $Ord^{v_0}(0)=1$.
We will next eliminate this case  in Proposition~\ref{linearapproxblowupsym'} below.  For this purpose, we need the following key technical Lemma~\ref{keylemma'} (that generalizes  Lemma~\ref{keylemma})  which is the lynchpin to  the regularity theorem.  We reiterate that this  lemma handles the  difficulties stemming from the non-compactness and the degenerating geometry of the target space $\overline {\mathcal T}$
(like Lemma~\ref{keylemma}) with the additional complication that $v$ is not necessarily harmonic but only approximately harmonic in a certain sense (unlike Lemma~\ref{keylemma}).   The proof is deferred until  Section~\ref{sec:  proofofkeylemma} because of its highly technical nature.  

\begin{lemma}[{\bf Key Technical Lemma}] \label{keylemma'}
Given $c_0 \geq 1$, $E_0, A^1, \dots, A^m>0$,
there exist  $D_0 \in (0,\frac{1}{\sqrt{8}})$ and  $c>0$ that give the following implication.\\
\\
{\bf Assumptions. }  The metric metric $g$ on $B_1(0)$ and the map $v=(v^1, \dots, v^{k-j}):(B_1(0),g) \rightarrow (\overline{\bf H}^{k-j},d_h)$ satisfy:
\vspace*{0.08in}
\begin{itemize}
 \item[{\bf (i)}]{\bf{(almost Euclidean domain metric)}}  
The metric  $g$ is  $C^{\infty}$-close to the Euclidean metric.  
\item[{\bf (ii)}]{\bf{(energy decay)}}
The energy of the map $v$
% \[
% v=(v^1, \dots, v^{k-j}):(B_1(0),g) \rightarrow (\overline{\bf H}^{k-j},d_h)
% \]
satisfies
    \[
  E^v(\vartheta) \leq \vartheta^n E_0, \ \ \ \forall \vartheta \in (0,\frac{1}{2}).
\]
\item[{\bf (iii)}]{\bf{(close to a symmetric homogeneous degree 1 map)}} There exists a map
\[
l=(T^1 \circ l^1 \circ R^1, \dots, T^m \circ l^m \circ R^m, l^{m+1},...,l^{k-j} ):B_{\theta^i}(0) \rightarrow ({\bf H}^{k-j},h)
\]
where  for $\mu=1, \dots, m$,
\[
\begin{array}{rl}
R^{\mu}:B_{\theta^i}(0) \rightarrow B_{\theta^i}(0) & \mbox{ is a rotation}, \\
T^{\mu}:{\bf H} \rightarrow {\bf H} &\mbox{ is a translation isometry},\\ 
l^{\mu}:B_1(0) \rightarrow {\bf H} & \mbox{ is a symmetric homogeneous degree 1 map}
\end{array}
\]
  with stretch $A^{\mu}$      such that
\[
\sup_{B_{\frac{1}{2}}(0)} d_h(v,l) < D_0,
\]
\[
d_{\overline{\bf H}}(P_0, l^{\mu}(0))<\frac{1}{\sqrt{8}}, \ \ \ \forall {\mu} \in \{1, \dots, m\}
\]
and 
\[
l^{\mu} \mbox{ is identically equal to }P_0
\]
for $\mu=m+1, \dots, k-j$.
\item[{\bf (iv)}]{\bf{(almost subharmonicity of the distance)}}
There exists $c_0 \geq 1$ such that for   $\vartheta \in (0,\frac{1}{24})$, $R \in [\frac{5}{8},\frac{7}{8}]$  and a harmonic map $w:(B_{\vartheta R}(0),g)  \rightarrow (\overline{\bf H}^{k-j},d_h)$, 
 \begin{equation} \label{abc'}
 \sup_{B_{\frac{15\vartheta R}{16}}(0)} d_h^2(v,w)  \leq  \frac{c_0}{(\vartheta R)^{n-1}} \int_{\partial B_{\vartheta R}(0)} d_h^2(v,w)d\Sigma + c \vartheta^{3}.
\end{equation}
 \end{itemize}
 {\bf Conclusion. } Then $v(0) \neq {\mathbb P}_0$.
\end{lemma}

\begin{proof}
See Subection~\ref{sec:theproof}.
\end{proof}

 \begin{proposition} \label{linearapproxblowupsym'}
 Let $v_{i}:(B_1(0),g_i) \rightarrow  (\overline{\bf H}^{k-j},d_h)$   be
a sequence of asymptotic harmonic maps.   If $v_0$ (defined in Definition~\ref{asymptoticallyharmonic} (iii)) is non-constant, then $Ord^{v_0}(0) \neq 1$.     
 \end{proposition}

\begin{proof}
On the contrary, assume   $Ord^{v_0}(0)=1$.
By Lemma~\ref{ahmresult}, we can assume that $v_0$ maps into a product of NPC spaces; i.e.  
\[
v_0=(v_0^1, \dots, v_0^{k-j}): B_{\frac{1}{2}}(0)  \rightarrow Y_0=Y_0^1 \times \dots \times Y_0^{k-j}
\]
where $(Y_0^{\mu},d_0^{\mu})$ is an NPC space for $\mu=1, \dots, k-j$.
Since $v_0$ is a homogeneous harmonic map, each component map  $v_0^{\mu}: B_{\frac{1}{2}}(0) \rightarrow (Y_0^{\mu},d_0^{\mu})$  is  a homogeneous harmonic map.     By reordering if necessary, we can assume $v_0^1, \dots, v_0^m$ are non-constant maps and $v_0^{m+1}(x)= P_0, \dots, v_0^{k-j}(x) = P_0$ for all $x \in B_{\frac{1}{2}}(0)$.

 % By   Lemma~\ref{proveitsfl}, it  suffices to prove that $Ord^{v_0}(0) \neq 1$.  Thus, 
% With the intent of arriving at a contradiction, we assume $Ord^{v_0}(0)=1$ which implies  that  $Ord^{v_0^{\mu}}(0)=1$ for $\mu=1, \dots, m$.  
Let  $w_{i}=(w_i^1, \dots, w^{\mu}):B_{\frac{3}{4}}(0) \rightarrow ({\bf H}^{k-j},d_h)$ be the harmonic map whose boundary values agrees with that of $v_{i}\big|_{B_{\frac{3}{4}}(0)}$ as in the proof of Lemma~\ref{ahmresult}.   Then 
$w_i\big|_{B_{\frac{1}{2}}(0) } $ converges locally uniformly in the pullback sense to 
%the  homogeneous harmonic map
 $v_0$ by Definition~\ref{asymptoticallyharmonic} (iv) which in turn
implies that $w_i^{\mu}\big|_{B_{\frac{1}{2}}(0) } $ converges locally uniformly in the pullback sense to 
%the  homogeneous harmonic map
 $v_0^{\mu}$ for each $\mu=1, \dots, m$.   Therefore, by Lemma~\ref{proveitsfl},  $v_0^{\mu}$ maps into a geodesic since we are assuming $Ord^{v_0}(0)=1$.
Therefore, we can apply the same argument as the proof of   Lemma~\ref{linearapproxblowupsym} with $w_i$ replacing  $u_{\sigma_i}$ and $v_0$ replacing $u_*$  to   conclude that there exists a sequence of translation isometries $T_i^{\mu}$, a rotation $ R^{\mu}:B_{\theta^i}(0) \rightarrow B_{\theta^i}(0)$ and   a sequence of  symmetric  homogeneous degree 1 maps  $l_i^{\mu}$ with 
\begin{equation} \label{red}
d_{\bf H}(P_0,l_i^{\mu}(0)) \rightarrow 0
\end{equation}
 and  stretch $A^{\mu}$ such  that
\[
\lim_{i \rightarrow \infty} \sup_{B_{\frac{1}{2}}(0)} d(w_i^{\mu}, T_{i}^{\mu} \circ l_i^{\mu} \circ R^{\mu}) = 0.
\]
This defines the constant $A^{\mu}$.  Combined with (\ref{vwconverg}), we see that
\begin{equation}\label{frank}
\lim_{i \rightarrow \infty} \sup_{B_{\frac{1}{2}}(0)} d(v_i^{\mu}, T_{i}^{\mu} \circ l_i^{\mu} \circ R^{\mu}) = 0.
\end{equation}
Let  $E_0$ and $c_0$ be the constants in Definition~\ref{asymptoticallyharmonic}  (ii) and (iv) respectively.  Let    $D_0 \in (0,\frac{1}{\sqrt{8}})$ and $c$ be constants   in Lemma~\ref{keylemma'} corresponding to  $c_0$, $E_0$, $A^1, \dots, A^{m}$.   By  Definition~\ref{asymptoticallyharmonic} (iv), (\ref{red}) and (\ref{frank}), we can fix $i$ sufficiently large such that $c_i \leq c$, 
and
\[
d(P_0, l_i^{\mu}(0))<\frac{1}{\sqrt{8}}, \ \ \ \forall {\mu} \in \{1, \dots, m\}
\]
and
\[
\sup_{B_{\frac{1}{2}}(0)} d(v_{i}^{\mu}, T_{i}^{\mu} \circ l_i^{\mu} \circ R^{\mu}) <\frac{D_0}{m}, \ \ \ \forall {\mu} \in \{1, \dots, m\}.
\]
Define $l=(l^1, \dots, l^{k-j}):B_1(0) \rightarrow \overline{\bf H}^{k-j}$ by setting 
\[
l^{\mu}=T_{i}^{\mu} \circ l_i^{\mu} \circ R^{\mu}, \ \forall \mu=1, \dots, m \ \mbox{ and } \ l^{\mu}\equiv P_0, \ \forall \mu=m+1, \dots, k-j
\]
which gives us
\[
\sup_{B_{\frac{1}{2}}(0)} d(v_{i}, l) <D_0.
\]
Applying Lemma~\ref{keylemma'}, we obtain 
$v_i(0) \neq {\mathbb P}_0$.  This contradiction proves $Ord^{v_0}(0) \neq 1$. 
\end{proof}

\subsection{The Inductive Hypothesis }\label{indstepp}
In this subsection, we begin the proof of Theorems~\ref{RegularityTheorem} and~\ref{goto0} 
by starting a backwards induction  on $j$. 
We need the following two statements for a harmonic map $u:(\Omega,g) \rightarrow( \overline{\mathcal T}, d_{ \overline{\mathcal T}})$:
\\
\\
{\sc Statement 1}$[j]$:  For any $x_{\star} \in {\mathcal S}_j(u)$ and a local representation $u=(V,v):(B_{\sigma_{\star}}(x_{\star}),g) \rightarrow ({\mathcal U} \times {\mathcal V}, d_G)$ at $x_{\star}$ (cf. Definition~\ref{locrep}), we have  
\[
\dim_{\mathcal H}\left({\mathcal S}(u) \cap B_{\frac{\sigma_{\star}}{2}}(x_{\star})\right) \leq n-2.
\]
\\
{\sc Statement 2}$[j]$:  For  $x_{\star} \in {\mathcal S}_j(u)$, a local representation $u=(V,v):(B_{\sigma_{\star}}(x_{\star}),g) \rightarrow ({\mathcal U} \times {\mathcal V}, d_G)$ at $x_{\star}$, $q \in [1,2)$ sufficiently close to 2 and any subdomain $\Omega$  compactly contained in 
\[
B_{\frac{\sigma_{\star}}{2}}(x_{\star}) \backslash  \left( {\mathcal S}(u) \cap v^{-1}({\mathbb P}_0) \right),
\] 
 there exists a sequence of smooth  functions $\psi_i$ and a neighborhood of $\mathcal N_i$ contained in an $\epsilon_i$-neighborhood of ${\mathcal S}(u)$ with   $\psi_i \equiv 0$ in a neighborhood of ${\mathcal S}(u) \cap \overline{\Omega}$, $\psi_i \equiv 1$ outside of ${\mathcal N}_i$, $\epsilon_i \rightarrow 0$, $0 \leq \psi_i \leq 1$, $\psi_i \rightarrow 1$ for all $x \in \Omega   \backslash {\mathcal S}(u) $ 
\[
\lim_{i \rightarrow \infty} \int_{B_{\frac{\sigma_{\star}}{2}}(x_{\star})} \ |\nabla u| |\nabla \psi_i| \ d\mu =0,
\]
\[
\lim_{i \rightarrow \infty} \int_{B_{\frac{\sigma_{\star}}{2}}(x_{\star})} \ |\nabla u| |\nabla \psi_i|^q \ d\mu =0
\]
and 
\[
\lim_{i \rightarrow \infty} \int_{B_{\frac{\sigma_{\star}}{2}}(x_{\star})} \ |\nabla \nabla u| |\nabla \psi_i| \ d\mu =0.
\]
\\

We will prove {\sc Statement 1}$[j]$ and {\sc Statement 2}$[j]$  for all $j \in \{0, \dots, k\}$ by a backwards induction on $j$ as follows: \\
\\
{\sc Initial Step}.  
 {\sc Statement 1}$[k]$ and {\sc Statement 2}$[k]$ hold  since   $\hat{\mathcal S}_{k}(u)=\emptyset$ (cf.~(\ref{fis})).    \\
\begin{quote}\emph{Inductive Hypothesis [j+1]}:  \ {\sc Statement 1}$[m]$ and {\sc Statement 2}$[m]$ hold for $m=j+1, j+2, \dots, k$. \\
\end{quote}

\noindent {\sc Inductive Step}.  The \emph{Inductive Hypothesis [j+1]} implies that {\sc Statement 1}$[j]$ and {\sc Statement 2}$[j]$ hold. \\

%\begin{quote}
%If the \emph{Inductive Hypothesis [j+1]   holds} for a harmonic map $u:(\Omega,g) \rightarrow( \overline{\mathcal T}, d_{ \overline{\mathcal T}})$, then the  Assumptions 1-6 of \cite{daskal-meseDM} are satisfied for any $x_{\star} \in {\mathcal S}_j(u)$ and any local representation 
%\[
%u=(V,v):(B_{\sigma_{\star}}(x_{\star}),g) \rightarrow ({\mathcal U} \times {\mathcal V}, d_G)
%\]
%at $x_{\star}$.
%\end{quote}
  
%\subsection{Implications of the Inductive Hypothesis.} \label{checkassumptions}

Before we prove the {\sc Inductive Step}  in Section~\ref{conclusionofproof},   we will need to further analyze the singular component $v$ of the harmonic map $u$ in the next section. 

\subsection{Order of the singular component map} \label{sec:DM}
In this Section, we  prove  existence of the order function for the singular component  $v$ of a harmonic map into $\overline{\mathcal T}$. The difference with Gromov-Schoen
is  that $v$ is not necessarily energy minimizing, but only \emph{almost} energy minimizing. However, the basic steps are the same as in Gromov-Schoen  with the additional complication of keeping track of the error terms coming from the  almost harmonic map $v$. As in  \cite{gromov-schoen}, before proving that the order function exists we  have to show a {\it{target variation formula}} and a {\it{domain variation formula}}. These theorems have been proved for approximate harmonic maps to a wide range of spaces  in  \cite{daskal-meseDM}. 

 For the sake of completeness we state these theorems and sketch their proof in  Proposition~\ref{targetvariation} and Corollary~\ref{cor:tv}  for the target variation, Proposition~\ref{firstvariationdomain} and Corollary~\ref{domvariation} for the domain variation and Proposition~\ref{orderofv} for the existence of the order function. These theorems parallel \cite[Proposition 2.2]{gromov-schoen}  for the target variation, \cite[Formula (2.3)]{gromov-schoen} for the domain variation and \cite[Formula (2.5)]{gromov-schoen} for the monotonicity and thus the existence of the order function.  

% The corresponding results in  \cite{daskal-meseDM} are in parentheses.
%\begin{itemize}
%\item Variation in the target, Proposition~\ref{targetvariation}, (cf.~\cite[Section 6, Proposition~37]{daskal-meseDM})
%\item Variation in the domain, Proposition~\ref{firstvariationdomain} (cf.~\cite[Section 8, Lemma~52]{daskal-meseDM}) 
%\item Existence of order, Proposition~\ref{orderofv} (cf.~\cite[Section 9, Proposition~54 and Corollary~60]{daskal-meseDM})
%\item[(iii)]  \cite{daskal-meseDM} Section 9,  Corollary~60
%\item[(iii)] \cite[Section 9, Estimate (132)]{daskal-meseDM}  \ \ (Lemma~\ref{energycomp})
%\item[(iv)] \cite[Section 11, Proof of Statetement 2??]{daskal-meseDM}  \ \ (Proposition~\ref{proof2j}).
%\end{itemize}
%The corresponding results for the local representation $u=(V,v)$ are Proposition~\ref{targetvariation}, Proposition~\ref{orderofv}, Lemma~\ref{energycomp} and Proposition~\ref{proof2j} below.

Throughout this subsection, we 
assume that   the \emph{Inductive Hypothesis [j+1]}   holds  
for a harmonic map $u:(\Omega,g) \rightarrow( \overline{\mathcal T}, d_{ \overline{\mathcal T}})$, and let 
   \[
u=(V,v):(B_{\sigma_{\star}}(x_{\star}),g) \rightarrow  ({\mathcal U} \times {\mathcal V}, d_G)
\]
 be a local representation of $u$ at $x_\star \in {\mathcal S}_j(u)$. We start with the following proposition which is a restatement of our inductive hypothesis.
 
 \begin{proposition} \label{A6}
For any $q \in [1,2)$ sufficiently close to 2 and any subdomain $\Omega_1$ compactly contained in 
\[
B_{\frac{x_\star}{2}}(x_{\star}) \backslash ( {\mathcal S}(u) \cap v^{-1}({\mathbb P}_0)), 
\]
there exists a sequence of smooth  functions $\psi_i$ and a neighborhood of $\mathcal N_i$ contained in an $\epsilon_i$-neighborhood of ${\mathcal S}(u)$ with  $0 \leq \psi_i \leq 1$, $\psi_i \equiv 0$ in a neighborhood of ${\mathcal S}(u) \cap \overline{\Omega_1}$, $\psi_i \equiv 1$ outside of ${\mathcal N}_i$, $\epsilon_i \rightarrow 0$,  $\psi_i \rightarrow 1$ for all $x \in \Omega_1   \backslash {\mathcal S}(u) $ such that 
%
%there exists a sequence of smooth functions $\psi_i \equiv 0$ in a neighborhood of ${\mathcal S}(u) \cap \overline{\Omega_1}$, $0 \leq \psi_i \leq 1$, $\psi_i \rightarrow 1$ for all $x \in \Omega_1 \backslash {\mathcal S}(u)$ such that
\[
\lim_{i \rightarrow \infty} \int_{\Omega_1} \ |\nabla u| |\nabla \psi_i| \ d\mu =0.
\]
\[
\lim_{i \rightarrow \infty} \int_{\Omega_1}\ |\nabla u| |\nabla \psi_i|^q \ d\mu =0.
\]
\[
\lim_{i \rightarrow \infty} \int_{\Omega_1} \ |\nabla \nabla u| |\nabla \psi_i| \ d\mu =0.
\]
\end{proposition}
\begin{proof} 
Since $ v(\Omega_1)$ does not contain (the most singular point) ${\mathbb P}_0$, this follows from the inductive hypothesis {\sc Statement 2}$[j+1]$, \dots, {\sc Statement 2}$[k]$ and a partition of unity argument.
\end{proof}

 Before we discuss the target variation we need two preliminary propositions. In the language of \cite{daskal-meseDM} these correspond to Assumptions 3 and 4.
 
  \begin{proposition} \label{A3}
%If the harmonic map $u:\Omega \rightarrow \overline{\mathcal T}$ satisfies Inductive Hypothesis [j+1], then for $x_{\star} \in {\mathcal S}_j(u)$ and a local representation $u=(V,v):(B_{\sigma_{\star}}(x_{\star}),g) \rightarrow  ({\mathcal U} \times {\mathcal V}, d_G)$ at $x_{\star}$,  
The set ${\mathcal S}_j(u)$ satisfies the following:
\begin{itemize}
\item[(i)]   $v(x)={\mathbb P}_0$ for $x \in {\mathcal S}_j(u) \cap B_{\sigma_{\star}}(x_{\star})$

\item[(ii)]  $\dim_{\mathcal H}(({\mathcal S}(u) \backslash {\mathcal S}_j(u)) \cap  B_{\frac{\sigma_{\star}}{2}}(x_{\star})) \leq n-2$.
\end{itemize}
\end{proposition}

\begin{proof}
Assertion (i) follows immediately from the definition of ${\mathcal S}_j(u)$.  The inductive assumption along with Proposition~\ref{highorderpts} implies the assertion (ii).  
\end{proof}

\begin{proposition} \label{A4}
%Assume the harmonic map $u:\Omega \rightarrow \overline{\mathcal T}$ satisfies Inductive Hypothesis [j+1].    For $x_{\star} \in {\mathcal S}_j(u)$, a local representation $u=(V,v):(B_{\sigma_{\star}}(x_{\star}),g) \rightarrow  ({\mathcal U} \times {\mathcal V}, d_G)$ at $x_{\star}$
For $B_R(x_0) \subset B_{\frac{\sigma_{\star}}{2}}(x_{\star})$  
and any harmonic map 
$
w:(B_R(x_0),g) \rightarrow (\overline{\bf H}^{k-j}, d_h)$, the set  ${\mathcal R}(u,w)$ 
 is of full measure in ${\mathcal R}(u) \cap B_R(x_0)$.  Here,  ${\mathcal R}(u,w)$ is  the set  of points $x \in {\mathcal R}(u) \cap B_R(x_0)$ with the property that there exists  $r>0$  such that $v(B_r(x)), w(B_r(x))$ map into the same stratum of $\overline{\bf H}^{k-j}$.
 \end{proposition}

\begin{proof}
By Theorem~\ref{RegularityTheoremModelSpace}, we have $\dim_{\mathcal H}({\mathcal S}(w)) \leq n-2$.  Thus, ${\mathcal R}(w)$ is of full measure in $B_\sigma(x_0)$ which immediately implies  ${\mathcal R}(u,w)$ is of full measure in ${\mathcal R}(u) \cap B_\sigma(x_0)$.   
\end{proof}

By \emph{target variation}, we mean the one-parameter family of maps
\[
v_{t\eta}:B_{\vt}(x_0) \rightarrow \overline{\bf H}^{k-j}, \ \ \ v_{t\eta}(x) = (1-t\eta(x)) v(x) + t\eta(x) w(x)
\]
where  $\eta \in C^{\infty}_c(B_{\vt}(x_0))$ and $w: B_{\vt}(x_0) \rightarrow \overline{\bf H}^{k-j}$ is a Lipschitz map.   Here, the sum indicates  geodesic interpolation; in other words, given  two points $P, Q \in \overline{\bf H}^{k-j}$,  $\tau \mapsto (1-\tau)P+\tau Q$ for $\tau \in [0,1]$  denotes a constant speed parameterization of the unique geodesic from $P$ to $Q$.

We start by observing that  if $v$ was energy minimizing, we would have  
\begin{equation} \label{uptoet}
E^v_{x_0}(\vt) -  E_{x_0}^{v_{t\eta}}(\vt) \leq 0.
% \ \ \  (\mbox{compare with }(\ref{lst}) \mbox{ below})
\end{equation}
However, since the singular component map $v$ is not necessarily energy minimizing, we don't expect (\ref{uptoet}) to hold.
On the other hand,  the full map  $u=(V,v)$ is energy minimizing and hence  for $u_{t\eta}=(V,v_{t\eta})$
\[
E^u_{x_0}(\vt) -  E_{x_0}^{u_{t\eta}}(\vt) \leq 0. 
\] 
Furthermore, since by  Proposition~\ref{A2}, the Weil-Petersson metric  is  asymptotically a  product,   we have 
\begin{eqnarray*}
E^u_{x_0}(\vt)  & \approx & E^V_{x_0}(\vt) + E^v_{x_0}(\vt)
\\
E^u_{x_0}(\vt)  & \approx & E^V_{x_0}(\vt) + E^{v_{t\eta}}_{x_0}(\vt).
\end{eqnarray*}
%The right hand side of (\ref{lst}) due to  the error in  these  approximations.
These approximations mean that the equalities are correct up to a small  error.   By precisely accounting for these errors, we obtain the following theorem and its corollary which is the 
{\it{target variation formula.}}

\begin{proposition} \label{targetvariation}
There exists $R>0$ and $C>0$ such that
\begin{equation} \label{lst}
\limsup_{t \rightarrow 0^+} \frac{E^v_{x_0}(\vt) - E_{x_0}^{v_{t\eta}}(\vt)}{t}  \leq C \int_{B_{\vt}(x_0)} \eta d_h(v,{\mathbb P}_0) d_h(v,w) d\mu
\end{equation}
for any $x_0\in  {\mathcal S}_j(u) \cap B_{\frac{\sigma_{\star}}{2}}(x_{\star}), \ \vt \in (0,R]$, $\eta \in C_0^{\infty}(B_{\vt}(x_0))$ and harmonic map 
\[
w:(B_R(x_0),g) \rightarrow  (\overline{\bf H}^{k-j},d_h)
\]
where   
\[
v_{t\eta}(x):=(1-t\eta(x)) v(x)+t\eta(x) w(x).
\]
\end{proposition}
\begin{proof}This is proved in~\cite[Proposition 37]{daskal-meseDM}  by a   straightforward computation by using the precise asymptotic estimates of the Weil-Petersson metric given in  Proposition~\ref{A2} and Propositions~\ref{A3} and~\ref{A4}. We omit the details here.
\end{proof}

If $v$ is harmonic, then  the target variation  formula (cf.~\cite[(2.2)]{gromov-schoen}) is
\begin{equation} \label{qwt}
\triangle d_h^2(v,Q) \geq 0,
\end{equation}
where $Q$ is any point on the target space. However, since $v$ is only approximately harmonic, we have to modify (\ref{qwt}) by adding an error term to obtain equation (\ref{vwco}). The precise estimate is 

\begin{corollary} \label{cor:tv}
There exists $R>0$ and $C>0$ 
such that 
\begin{equation} \label{tvgs}
 - C\int_{B_{\vt}(x_0)}  \eta d_h^2(v,P_0)  d\mu
+2 \int_{B_{\vt}(x_0)}  \eta |\nabla v|^2  d\mu    \leq  - \int_{B_{\vt}(x_0)}  \nabla \eta \cdot \nabla d_h^2(v,P_0) \ d\mu
\end{equation}
and
\begin{equation} \label{vwco}
0 \leq - \int_{B_{\vt}(x_0)} \nabla \eta \cdot \nabla d_h^2(v,w) d\mu +C \int_{B_{\vt}(x_0)} \eta d_h(v,{\mathbb P}_0) d_h(v,w) d\mu
\end{equation}
%\[
%\sup_{B_{r \vt}(x_0)} d_h^2(v,w) \leq \frac{C}{\vt^{n-1}} \int_{\partial B_{\vt}(x_0)} d_h^2(v,w) d\s + \frac{C}{\vt^{n-2}} \int_{B_{\vt}(x_0) }d_h^2(v,{\mathbb P_0})\du
%\]
for any $x_0\in  {\mathcal S}_j(u) \cap B_{\frac{\sigma_{\star}}{2}}(x_{\star})$, $\vt \in (0,R]$, $\eta \in C_0^{\infty}(B_{\vt}(x_0))$  and  harmonic map  
\[
w:(B_R(x_0),g) \rightarrow  \overline{\bf H}^{k-j}.
\]
\end{corollary}

\begin{proof}
For inequality~(\ref{tvgs}), we combine the computation of the target variation formula in  \cite[Section 2]{gromov-schoen} with Proposition~\ref{targetvariation}.  See also \cite[Corollary 39]{daskal-meseDM}.
We now prove (\ref{vwco}).
Let $R>0$ and $C>0$ be as in Proposition~\ref{targetvariation}.
By \cite[Lemma 2.4.2]{korevaar-schoen1} (with $u_0=v$, $u_1=w$, replacing $\eta$ by $t\eta$, integrating over $B_{\vt}(x_0)$ and noting that $w$ is an energy minimizing map)
\begin{equation} \label{ks242}
t \int_{B_{\vt}(x_0)} \nabla \eta \cdot \nabla d_h^2(v,w)  \du -O(t^2) \leq E^v_{x_0}(\vt) - E_{x_0}^{v_{t\eta}}(\vt)
\end{equation}
%Inequality (\ref{ks242}) follows from \cite{korevaar-schoen1} Lemma 2.4.2 
%but we will provide the details for the convenience of the reader.  First,  by \cite{korevaar-schoen1} inequality (2.4ix),  at a.e.~$x \in B_{\vt}(x_0)$
%\[
%|(v_{t\eta})_*(\omega)|^2+|(v_{1-t\eta})_*(\omega)|^2 \leq |v_*(\omega)|^2+|w_*(\omega)|^2 - t\eta_*(\omega) (d^2_h(v,w))_*(\omega) + Q(t\eta,\nabla (t\eta))
%\]
%where $Q(t\eta, \nabla (t\eta))$ consists of integrable terms which are quadratic in $t\eta$ and $\nabla (t\eta)$,  $\omega$ is a unit vector in the tangent space at $x$ and  $|(v_{\eta})_*(\omega)|^2$, $|(v_{1-\eta})_*(\omega)|^2$,  $|v_*(\omega)|^2$ and $|w_*(\omega)|$ denote the directional energy density functions (generalizing the norm squared of the directional derivative of a map between Riemannian manifolds, cf. \cite{korevaar-schoen1} Section 1).  Apply this inequality with $\omega$ equal to each element in an orthonormal basis of tangent vectors at $x$, sum the resulting inequalities and then integrate over $x \in B_{\vt}(x_0)$ to obtain
%\[
%E^{v_{t\eta}}_{x_0}(\vt) +E^{v_{1-t\eta}}_{x_0}(\vt)  \leq E^v_{x_0}(\vt) +E^w_{x_0}(\vt)- t \nabla \eta \cdot \nabla d_h^2(v,w)+O(t^2).
%\]
%Since $\eta$ has compact support in $B_{\vt}(x_0)$, we have that $w\big|_{\partial B_{\vt}(x_0)} =v_{1-t\eta}\big|_{\partial B_{\vt}(x_0)}$.  Thus, the harmonicity of $w$ implies 
%\[
%E^w_{x_0}(\vt) \leq E^{v_{1-t\eta}}_{x_0}(\vt).
%\]
%Combining the above two inequalities, we conclude (\ref{ks242}).
which combined with  Proposition~\ref{targetvariation}
implies the result.
\end{proof}

\begin{lemma}  \label{anothersh} 
For $x_0\in  {\mathcal S}_j(u) \cap B_{\frac{\sigma_{\star}}{2}}(x_{\star})$, 
let  $v_\sigma$ be the singular component of the blow-up map $u_\sigma=(V_\sigma,v_\sigma)$ of $u$ at $x_0$ (cf.~(\ref{buwrtu})).
For  a fixed $R \in (0,1)$ and $r \in (0, 1)$,  there exists a constant $C>0$ that can be chosen independently of $\sigma$
such that for  any harmonic map  
\[
w:(B_R(0),g_{ \sigma}) \rightarrow  \overline{\bf H}^{k-j} 
\
\mbox{  with }
 \sup_{B_R(0) }d_h(w,{\mathbb P}_0) \leq c,
%$w\big|_{\partial B_R(0)} =v_\sigma\big|_{\partial B_R(0)}$ },
%E^w(R) \leq E^{v_\sigma}(R)$},
\]
 we have
 \begin{equation} \label{instead}
\sup_{B_{r\vartheta}(0)} d_h^2(v_\sigma,w) \leq   \frac{C}{\vartheta^{n-1}}  \int_{\partial B_{\vartheta}(0)} d_h^2(v_\sigma,w)  d\Sigma_\sigma  + C\sigma^2  \vartheta^3, \ \ \forall \vartheta \in (0,R]
\end{equation}
where $d\s_\sigma$ is the volume form with respect to the metric $g_\sigma$.
\end{lemma}
\begin{proof}
Throughout this proof, we use $c$ to denote an arbitrary constant that is independent of $\sigma$ that may change from line to line.  In the estimate of Corollary~\ref{cor:tv}, identify $x_0=0$ and replace  $\vt$  by $\sigma \vt$.    Multiply the resulting inequality by $\mu_\sigma^{-1}$  and apply change of variables to obtain
\[
0 \leq - \int_{B_{\vt}(0)} \nabla \eta \cdot \nabla d_h^2(v_\sigma,w) d\mu_\sigma +c\sigma^2 \int_{B_{\vt}(0)} \eta d_h(v_\sigma,{\mathbb P}_0) d_h(v_\sigma,w) d\mu_\sigma
\]
where  $\du_\sigma$ is the volume form with respect to the metric $g_\sigma$ (cf.~(\ref{defineblowupmap'})).
Fix $r \in (0,1)$.  Let $x \in B_{r\vt}(0)$, $s \in (0,\vt(1-r))$ and $\eta$ approximate the characteristic function of $B_s(x) \subset B_{\vt}(0)$.  We then obtain
\begin{eqnarray*} 
0 & \leq & \int_{B_s(x)} \frac{\partial}{\partial s} d_h^2(v_\sigma,w) d\Sigma_\sigma +c \sigma^2 \int_{B_s(x)} d_h(v_\sigma,{\mathbb P}_0)d_h(v_\sigma,w) d\mu_\sigma \nonumber \\
& \leq & s^{n-1} \frac{d}{ds} \left( \frac{e^{cs}}{s^{n-1}} \int_{\partial B_s(x)} d_h^2(v_\sigma,w) d\Sigma_\sigma \right) \nonumber \\ 
& & \ \ +c\sigma^2 \int_{B_s(x)} d_h(v_\sigma,{\mathbb P}_0)(d_h(v_\sigma,{\mathbb P}_0) + d_h(w,{\mathbb P}_0)) d\mu_\sigma.
\end{eqnarray*} 
Multiply this by $s^{1-n}$ and apply Holder inequality to obtain
\begin{eqnarray} \label{di}
\lefteqn{0 \leq  \frac{d}{ds} \left( \frac{e^{cs}}{s^{n-1}} \int_{\partial B_s(x)} d_h^2(v_\sigma,w) d\Sigma_\sigma \right) } \nonumber \\
& &  \ +cs^{1-n} \sigma^2 \int_{B_s(x)} d_h^2(v_\sigma,{\mathbb P}_0) d\mu_\sigma \nonumber \\
& &  \ \ +cs^{1-n} \sigma^2 \left( \int_{B_s(x)} d_h^2(v_\sigma,{\mathbb P}_0) d\mu_\sigma \right)^{\frac{1}{2}} \left(   \int_{B_s(x)} d_h^2(w,{\mathbb P}_0) d\mu_\sigma \right)^{\frac{1}{2}}.
\end{eqnarray}
Since the blow up map $u_\sigma$ has Lipschitz bound that can be chosen independently of $\sigma$ in $B_R(0)$ (cf.~(\ref{uLip})), so does the singular component map  $v_\sigma$ by Lemma~\ref{disest2}.  Indeed, there exists a constant $c>0$  that can be chosen independently of $\sigma$ such that 
\[
 \sup_{B_s(0) } d_h(v_\sigma,{\mathbb P}_0) \leq c s.
 \]  
%The uniform Lipschitz continuity of $v_\sigma$ and Theorem~\ref{KSlip} imply that 
%\[
%\sup_{B_s(0) }d_h(w,{\mathbb P}_0) \leq c.
%\]   
Additionally, 
\[
\mbox{Vol}_{g_\sigma}(B_s(x_0)) \leq c s^n.
\]
Combining the above two inequalities with the assumption $\sup_{B_s(0) }d_h(w,{\mathbb P}_0) \leq c$, we conclude that
the last two terms of the right hand side of (\ref{di}) can be replaced by $c\sigma^2 s^2$; i.e.
  \begin{eqnarray} \label{pdi'}
0 & \leq & \frac{d}{ds} \left( \frac{e^{cs}}{s^{n-1}} \int_{\partial B_s(x)} d_h^2(v_\sigma,w) d\Sigma_\sigma \right) +c\sigma^2 s^2.\end{eqnarray}
Integrating this   over $s \in (0,t)$ for $t \leq \vt (1-r)$ yields 
\[
d_h^2(v_\sigma(x),w(x)) \leq \frac{c}{t^{n-1}} \int_{\partial B_t(x)} d_h^2(v_\sigma,w) d\Sigma_\sigma + c\sigma^2 t^3.
\]
Multiplying the above by $t^{n-1}$,  integrating over $t  \in (0,\vt(1-r))$ and noting that $B_t(x) \subset B_{\vt}(x_0)$, \begin{equation} \label{dotcirc}
\sup_{B_{r\vt} (x_0)} d_h^2(v_\sigma,w) \leq 
\frac{c}{\vt^n} \int_{B_{\vt}(x_0)}  d_h^2(v_\sigma,w) d\mu_\sigma+ c\sigma^2 \vt^{n+3}.
\end{equation}
Now we consider (\ref{pdi'}) with $x=x_0$ and integrate this over $s \in (t, \vt)$. We obtain
\[
\frac{1}{t^{n-1}} \int_{\partial B_t(x_0)} d_h^2(v_\sigma,w) d\Sigma_\sigma 
\leq
\frac{c}{\vt^{n-1}} \int_{\partial B_{\vt}(x_0)} d_h^2(v_\sigma,w) d\Sigma_\sigma + c\sigma^2 \vt^3.
\]
Multiplying by $t^{n-1}$ and  integrating this over $t \in(0,\vt)$, 
\begin{equation} \label{starcirc}
\frac{1}{\vt^n} \int_{B_{\vt}(x_0)} d_h^2(v_\sigma,w) d\Sigma_\sigma 
\leq
\frac{c}{\vt^{n-1}} \int_{\partial B_{\vt}(x_0)} d_h^2(v_\sigma,w) d\Sigma_\sigma + c\sigma^2 \vt^3.
\end{equation}
Combine (\ref{dotcirc}) and (\ref{starcirc}) yields the desired inequality.   
\end{proof}

%With Lemma~\ref{anothersh}, we can now prove (cf.~Lemma~\ref{compblowupu} below) that  the sequence $v_{\sigma_i}$ of the singular component maps  of blow-up maps $u_{\sigma_i}=(V_{\sigma_i},v_{\sigma_i})$ is a sequence of  asymptotically harmonic maps as in Definition~\ref{asymptoticallyharmonic}.  \st{(We note that in} Lemma~\ref{compblowupv} \st{below, we will also prove that the sequence of blow-up maps of $v$ is a sequence of asymptotically harmonic maps.) }  In turn,  Lemma~\ref{tucson} below asserts that the singular component map $v_*$ of a tangent map $u_*=(V_*,v_*)$ at $x_0\in  {\mathcal S}_j(u) \cap B_{\frac{\sigma_{\star}}{2}}(x_{\star})$ is a constant map.
 
 \begin{lemma} \label{compblowupu}  
If $x_0\in  {\mathcal S}_j(u) \cap B_{\frac{\sigma_{\star}}{2}}(x_{\star})$ and $u_{\sigma_i}=(V_{\sigma_i},v_{\sigma_i})$ is  a sequence of  blow-up maps of $u$ at $x_0$ as in Lemma~\ref{blowupcomponent}, then
 $v_{\sigma_i}$ is a sequence of asymptotically harmonic maps. 
\end{lemma}

\begin{proof}
By Lemma~\ref{blowupcomponent}, there exists a sequence of blow-up maps $u_{\sigma_i}=(V_{\sigma_i},v_{\sigma_i})$ that converges locally uniformly in the pullback sense to a map   $u_*=(V_*,v_*)$
%\[
%u_*=(V_*,v_*^1, \dots, v^{k-j}_*):B_1(0) \rightarrow {\C}^j \times Y_{1*} \times \dots \times Y_{k-j*}
%\] 
where $V_*, v_*$ %^1, \dots, v^{k-j}_*$
are homogeneous degree $\alpha$ harmonic maps and the sequences $V_{\sigma_i}$, $v_{\sigma_i}$ converge to $V_*$, $v_*$ %=(v_*^1, \dots, v^{k-j}_*)$
 respectively.
We check properties (i)-(iv) of Definition~\ref{asymptoticallyharmonic} for $v_i=v_{\sigma_i}$ and $g_i=g_{\sigma_i}$.   Property (i) regarding the metric $g_i$ follows immediately from the definition of $g_{\sigma_i}$.  Property (ii) follows from the fact that $u_{\sigma_i}$ and hence $v_{\sigma_i}$ is uniformly locally Lipschitz continuous by (\ref{uLip}) and Lemma~\ref{disest2}.   Since  $v_{\sigma_i}$ converges to $v_*$, we have Property (iii).   
Finally, Property (iv) follows from Lemma~\ref{anothersh} with $c_0=C$ and $c_i=C\sigma_i^2$.  
\end{proof}

\begin{lemma} \label{tucson}
If $x_0 \in  {\mathcal S}_j(u) \cap B_{\frac{\sigma_{\star}}{2}}(x_{\star})$ and   $u_*=(V_*,v_*)=(V_*,v_*^1, \dots, v_*^{k-j})$ is a tangent map of $u$ at $x_0$ as in Lemma~\ref{blowupcomponent}, then $v_*$ is a constant map.
\end{lemma}

\begin{proof}
Let  $u_{\sigma_i}=(V_{\sigma_i},v_{\sigma_i})$ be the sequence of  blow-up maps converging locally uniformly in the pullback sense to $u_*=(V_*,v_*)$ as in Lemma~\ref{blowupcomponent}.  By Lemma~\ref{compblowupu}, $v_{\sigma_i}$ is a sequence of asymptoticly harmonic maps converging to $v_*$. 
Since $x_0 \in {\mathcal S}_j(u)$, we have by definition that  $v_{\sigma_i}(0) ={\mathbb P}_0$ and  $Ord^{u_*}(0)=1$ (cf. (\ref{sj})).  
By Proposition~\ref{linearapproxblowupsym'}, $v_*$ is identically constant.
\end{proof}
%\subsection{Assumptions 5 and 6 of \cite{daskal-meseDM}}
%\label{5,6}
%Before we verify that  \cite{daskal-meseDM} Assumption 5 holds in Proposition~\ref{A5}, we need a preliminary lemma.

%We can now state two additional properties of $u=(V,v)$. 
\begin{proposition}\label{A5}
For a.e.~$x\in  {\mathcal S}_j(u) \cap B_{\frac{\sigma_{\star}}{2}}(x_{\star})$,
we have \[
 |\nabla v|^2(x)=0 
% \]
\ \mbox{ and } \
|\nabla V|^2(x)=|\nabla u|^2(x).
% \mbox{ for a.e. }x \in {\mathcal S}_j(u) \cap B_{\sigma_{\star}}(x_{\star}).
\]
%In particular, $u$ as in (\ref{localrep}) satisfies \cite{daskal-meseDM}  Section 5 Assumption 5.
\end{proposition}

\begin{proof}
Let $x_0 \in {\mathcal S}_j(u)  \cap B_{\sigma_{\star}}(x_{\star})$ and identify $x_0=0$ via normal coordinates.   By Lemma~\ref{tucson}, we can fix a sequence $u_{\sigma_i}=(V_{\sigma_i},v_{\sigma_i})$ of  blow-up maps of $u$ such that $u_{\sigma_i}$ and $V_{\sigma_i}$ converge to a tangent map $u_*=V_*:B_1(0) \rightarrow {\C}^j$ and $v_{\sigma_i}$ converges to a constant map.   
Lemma~\ref{disest2} implies
\begin{equation} \label{willprovelater'}
E^{u_{\sigma_i}}(r) =\left(E^{V_{\sigma_i}}(r) +E^{v_{\sigma_i}}(r) \right)  + O(\sigma_i^2).
\end{equation}
Therefore, 
\begin{eqnarray*}
\limsup_{\sigma_i \rightarrow 0} E^{V_{\sigma_i}}(r) 
&  \leq & \limsup_{\sigma_i \rightarrow 0} (E^{V_{\sigma_i}}(r) + E^{v_{\sigma_i}}(r)) \\ 
%& = & \limsup_{i \rightarrow \infty}  \ ^{G(u(0))}E^{u_{\sigma_i}}(r)   \ \ \ \mbox{(since $G(u(0))$ is a product metric)}\\
& = & \lim_{\sigma_i \rightarrow 0}  \ E^{u_{\sigma_i}}(r)   \ \ \ \mbox{(by (\ref{willprovelater'}))} \\
& = & E^{u_*}(r)  \ \ \ \mbox{ (by  Theorem~\ref{KScpct})}\\
%& = & E^{(V^*, 0)}(r) \\
&=   & E^{V_*}(r) \ \ \ \mbox{ (since $u_*=V_*$)} \\
& \leq &  \liminf_{\sigma_i \rightarrow 0} E^{V_{\sigma_i}}(r)  \ \ \ \mbox{ (by  the lower semicontinuity of energy)}.
\end{eqnarray*}
This immediately implies
\begin{equation} \label{heyitgoestozero}
\lim_{\sigma_i \rightarrow 0} E^{V_{\sigma_i}} (r)  = \lim_{\sigma_i \rightarrow 0}  E^{u_{\sigma_i}}(r)  \ \mbox{ and } \ \lim_{\sigma_i \rightarrow 0} E^{v_{\sigma_i}}(r)=0.
\end{equation}
Since $|\nabla v|^2$ is an integrable function, almost every point of  $B_{\sigma_{\star}}(x_{\star})$ is a Lebesgue point.  In particular, at almost every $x \in {\mathcal S}_j(u)  \cap B_{\sigma_{\star}}(x_{\star})$, 
\begin{eqnarray*}
|\nabla v|^2(x) & = & \lim_{\sigma_i \rightarrow 0} \frac{1}{Vol(B_{\sigma_i r}(x))} \int_{B_{\sigma_i r}(0)} |\nabla v|^2 d\mu \\
& = &  \lim_{\sigma_i \rightarrow 0} \frac{\mu_{\sigma_i}^2}{Vol(B_r(0))} \int_{B_r(0)} |\nabla v_{\sigma_i}|^2 d\mu_{\sigma_i}\\
& \leq  &  \lim_{\sigma_i \rightarrow 0} \frac{C^2}{Vol(B_r(0))} \int_{B_r(0)} |\nabla v_{\sigma_i}|^2 d\mu_{\sigma_i}\\
& = & 0  \ \ \ \ \mbox{(by (\ref{heyitgoestozero}))}.
\end{eqnarray*} 
This implies the first assertion.  The second follows immediately from the first.
\end{proof}

Next, we discuss the \emph{variation of the domain} which gives an estimate on how far $v$ is from being an energy minimizing map with respect to domain variations.  By a domain variation, we mean the one-parameter family of maps  
\[
v_t:B_\sigma(x_0) \rightarrow (\overline{\bf H}^{k-j}, d), \ \ \ 
v_t(x)=v \circ F_t(x)
\]
where $F_t$ is a diffeomorphism given by
\[
F_t(x)=(1+t\xi(x))x, \ \   \mbox{$\xi \in C^{\infty}_c(B_\sigma(x_0))$},\ \  0 \leq \xi \leq 1.
\]
%Define
%\[
%u_t: B_{r_0}(x_0) \rightarrow ({\bf R}^j \times Y_2^{k-j},d_G)
%\]
%by setting
%\[
%u_t:=(V, v_t).
%\]
%Since $u=u_t$ on $\partial B_\sigma(x_0)$,  $u_t$ is a competitor.

In the next Proposition, if $v$ was a minimizer as in \cite{gromov-schoen}, then the right hand side would be 0. The error terms come because $v$ is almost energy minimizing   with respect to the  asymptotic product structure of the Weil-Petersson metric.

\begin{proposition} \label{firstvariationdomain}
%Let  $u=(V,v):B_{\sigma_{\star}}(x_{\star}) \rightarrow ({\bf R}^j \times  {Y_2}^{k-j}, d_G)$  be a harmonic map as in (\ref{localrep}).
 There exists $C>0$  such that  for $x_0 \in {\mathcal S}_j(u) \cap B_{\frac{\sigma_{\star}}{2}}(x_{\star})$ and $\sigma \in (0,r_0)$,
 we have
\begin{eqnarray*}
\lim_{t \rightarrow 0} \frac{E^v_{x_0}(\sigma)-E^{v_t}_{x_0}(\sigma)}{t}
& \leq & C  \int_{B_\sigma(x_0) }  \xi d^2(v,P_0) d\mu  + C \sigma \int_{B_\sigma(x_0) }  \xi |\nabla v|^2 d\mu \nonumber
%& & \ +  C \sigma \int_{B_\sigma(x_0)}    d^2(v,P_0)  |\nabla \xi| d\mu.
\end{eqnarray*}
Furthermore,   $C$ depends only on the constant in the estimates for the target metric  $G$, the domain metric $g$ and the Lipschitz constant of $u$. \end{proposition}
\begin{proof}
Follows  from Proposition~\ref{targetvariation} and the  computation  of \cite[Chapter 7 and Chapter 8]{daskal-meseDM} (cf.~\cite[Lemma 52]{daskal-meseDM}).  
%{\color{blue} DELETE $\rightarrow$ We note that we are using  items (a) and (b) above as well as the following additional facts (cf.~\cite[Remark 42]{daskal-meseDM}):  
%%We have already checked that  Assumptions 2,3,4 of \cite{daskal-meseDM} Section 5, hold for the Weil-Petersson metric, the local model  $u=(V,v)$ and the singular component map $v:(B_{\sigma_{\star}}(x_{\star}),g)  \rightarrow   (\overline{\bf H}^{k-j},d_h)$.   
%%We now check that  Assumptions 1,5,6 of \cite{daskal-meseDM} Section 5 are also satisfied.
%\begin{itemize}
%\item[(c)] The model space $(\overline{\bf H}^{k-j},d_h)$ has a homogeneous structure as given by Proposition~\ref{A1} (cf.~\cite[Section 5, Assumption 1]{daskal-meseDM}).
%\item[(d)]  Proposition~\ref{A5} and Proposition~\ref{A6} (cf.~\cite[Section 5, Assumption 5 and Assumption 6]{daskal-meseDM}).
%\end{itemize} 
%}
\end{proof}
An important consequence of Proposition~\ref{firstvariationdomain}  is the \emph{domain variation formula}.  
\begin{corollary} \label{domvariation}
 There exist $R_0>0$ and $C>0$ such that  for $x_0 \in {\mathcal S}_j(u) \cap B_{\frac{\sigma_{\star}}{2}}(x_{\star})$ and $\sigma \in (0,R_0)$, we have
\begin{equation} \label{qw}
\sigma \frac{d}{d\sigma} E^v_{x_0}(\sigma) +  (2-n+C\sigma) E^v_{x_0}(\sigma)  - 2\sigma \int_{\partial B_\sigma(x_0)} \left| \frac{\partial {v}}{\partial r} \right|^2 d\Sigma \geq 0.
\end{equation}
Furthermore,  $C$ depends only on the constant in the estimates  for the Weil-Petersson metric  $G_{WP}$, the domain metric $g$ and the Lipschitz constant of $u$.
\end{corollary}

\begin{proof}
Combine the usual computation for harmonic maps (cf.~\cite[Chapter 2.4]{simon} or \cite[p.~192-193]{gromov-schoen}) with the  domain variation formula of Proposition~\ref{firstvariationdomain}.  The details are given in~\cite[Proposition 53]{daskal-meseDM}
\end{proof}

Finally, we discuss the existence of the  order of the singular component map $v$.  If $v$ is harmonic, then we have the monotonicity formula~\cite[(2.5)]{gromov-schoen}   
%the target variation  formula is
%\begin{equation} \label{qwt}
%\triangle d^2(v,Q) \geq 0 \mbox{ weakly  \ \ \ (cf.~\cite[(2.2)]{gromov-schoen})}
%\end{equation}
%where $Q$ is any point on the target space and  the domain variation formula is
%\begin{equation} \label{qw}
%\sigma \frac{d}{d\sigma} E^v(\sigma)+  (2-n+C\sigma^2)E^v(\sigma) - 2\sigma \int_{\partial B_\sigma(0)} \left| \frac{\partial {v}}{\partial r} \right|^2 d\Sigma \geq 0 \ \ \ \mbox{(cf.~\cite[(2.3)]{gromov-schoen}).}
%\end{equation}
%Combining the above two formulas, we obtain 
\[
\frac{d}{d\sigma} \left( e^{c\sigma^2} \frac{\sigma E^v(\sigma)}{I^v(\sigma)} \right) \geq 0
\]
  which immediately implies that the order exists.  
Furthermore,  by \cite[proof of Theorem 2.3]{gromov-schoen}, we also obtain
\[
 \frac{d}{d\sigma} \left( e^{c\sigma^2} \frac{E^v(\sigma)}{\sigma^{n-2+2\alpha}} \right)\geq 0 \mbox{ and } \frac{d}{d\sigma} \left( e^{c\sigma^2} \frac{I^v(\sigma)}{\sigma^{n-1+2\alpha}} \right)\geq 0.
  \]
Since $v$ is not necessarily harmonic, by applying  the  target variation formula (cf.~Corollary~\ref{cor:tv} with $\eta$ approximating the characteristic function of $B_\sigma(x_0)$) and the domain variation formula (cf.~Corollary~\ref{firstvariationdomain}), 
%Combining  In the final step, we note)  that the results of \cite{daskal-meseDM} Section 9  ``Order Function" are proven  under  Assumptions 1-6 listed in \cite{daskal-meseDM} Section 5.    
we  obtain 
\begin{proposition}  \label{orderofv}
The singular component map  $v$ has a well-defined order at any point of ${\mathcal S}_j(u) \cap B_{\frac{\sigma_{\star}}{2}}(x_{\star})$. In other words,   
\begin{eqnarray} \label{limitexists}
Ord^{v}(x_0):= \lim_{\sigma \rightarrow 0} \frac{\sigma E^v(\sigma)}{I^v(\sigma)}\ \mbox {exists for any $x_0\in  {\mathcal S}_j(u) \cap B_{\frac{\sigma_{\star}}{2}}(x_{\star})$.}
\end{eqnarray}
Furthermore,  there  exist constants $c>0$, $C_1>0$ and $R_0>0$  depending continuously on the point $x_0$ such that
\begin{eqnarray}
Ord^v(x_0) \leq e^{c\sigma} \frac{\sigma E^v(\sigma)}{I^v(\sigma)}, \ \ \ \forall \sigma \in (0,R_0)
\end{eqnarray}
and
\begin{eqnarray}\label{smnl1}
e^{-C_1 \sigma}\frac{\sigma E^v_x (\sigma)}{I^v_x (\sigma)} \leq e^{C_1 \rho}\frac{\rho E^v_x (\rho)}{I^v_x (\rho)}  \ \ \  \forall \sigma< \rho<R_0.
\end{eqnarray}
Finally,
\begin{equation} \label{monotonicity!}
\sigma \mapsto e^{c\sigma} \frac{I^v(\sigma)}{\sigma^{n-1+2\alpha}}, \ \  \ \sigma \mapsto e^{c\sigma} \frac{E^v(\sigma)}{\sigma^{n-2+2\alpha}}
\end{equation}
are   non-decreasing functions in $(0,R_0)$ where $\alpha=Ord^v(x_0) \geq 1$. 
\end{proposition}

\begin{proof}
See \cite[Proposition 54 and Corollary 60]{daskal-meseDM}.
\end{proof}
As the consequence of the existence of order, we will show that the sequence of blow up maps is sequence of approximately harmonic maps (cf.~Lemma~\ref{compblowupv} below).

\begin{definition} \label{blowingupv}
Let $x_0\in  {\mathcal S}_j(u) \cap B_{\frac{\sigma_{\star}}{2}}(x_{\star})$, identify $x_0=0$ via normal coordinates and let $g_\sigma$ as in (\ref{defineblowupmap'}).  For 
\begin{equation} \label{buwrtv0}
 \nu(\sigma) = \sqrt{\frac{I^v(\sigma)}{\sigma^{n-1}}},
\end{equation}
 define 
\begin{equation} \label{buwrtv}
v_\sigma: (B_1(0),g_\sigma) \rightarrow (\overline{\bf H}^{k-j},d_h),  \ \ v_\sigma(x)=\nu^{-1}(\sigma) v(\sigma x).
\end{equation}
We call $v_\sigma$ the \emph{blow-up map of $v$ at $x_0$}.
\end{definition}

\begin{remark} \label{bothblowups}
We emphasize that the scaling factor in the map $v_\sigma$ of (\ref{buwrtv}) is different from the one in the map $v_\sigma$ of  (\ref{buwrtu}) \emph{although we use the same notation}.  More specifically, $v_\sigma$ of (\ref{buwrtv}) (i.e.~the blow-up map of $v$) is scaled with respect to the scaling factor $\sqrt{\frac{I^v(\sigma)}{\sigma^{n-1}}}$, whereas the map $v_\sigma$ of (\ref{buwrtu})  (i.e.~the singular component map of the blow-up $u_\sigma$ of $u$) is scaled with respect to the scaling factor $\sqrt{\frac{I^u(\sigma)}{\sigma^{n-1}}}$.
For a singular component map of the blow-up $u_\sigma$ of $u$, the energy bound follows from  the energy bound of  $u_\sigma$ (cf. Lemma~\ref{disest2}).     On the other hand, the blow-up map $v_\sigma$ of $v$ is  rescaled by $\sqrt{\frac{I^v(\sigma)}{\sigma^{n-1}}}$, which may tend to 0 much quicker than $\sqrt{\frac{I^u(\sigma)}{\sigma^{n-1}}}$ and hence the energy bound for $u_\sigma$ does not help. On the other hand,  by  Proposition~\ref{orderofv}, we can now give a uniform energy bound for  the blow-up map $v_\sigma$ of $v$ at $x_0$.  
More precisely,  for $\sigma>0$ sufficiently small,  (\ref{limitexists}) implies
\begin{equation} \label{bdbyo}
E^{v_\sigma}(1) =\frac{\sigma E^v(\sigma)}{I^{v}(\sigma)} \leq E_0:=2Ord^{v}(x_0).
\end{equation}.  
\end{remark}

\begin{definition} \label{approxhm}
The harmonic map \[
w_\sigma:B_{\frac{3}{4}}(0) \rightarrow ({\bf H}^{k-j},d_h)
\]
 whose boundary values agree with that of $v_\sigma\big|_{B_{\frac{3}{4}}(0)}$ is called the \emph{approximating harmonic map} for $v_\sigma$. 
\end{definition}

%\begin{lemma}  \label{supbd}
%For  $v_\sigma$ and $w_\sigma$ be as in Definition~\ref{blowingupv} and Definition~\ref{approxhm} respectively, 
% \[
%\sup_{B_{\frac{1}{2}}(0)} d^2(v_\sigma, w_\sigma) \leq C\sigma^2
%\]
%for a constant $C>0$ independent of $\sigma$.  
%\end{lemma}
%\begin{proof}
%Define  $\hat u=(V_\sigma, w_\sigma)$ where $V_\sigma=\nu^{-1}(\sigma) V(\sigma x)$  Since $\hat u$ is a competitor of $u_\sigma$,
%\begin{eqnarray*}
%\int_{B_{\frac{3}{4}}(0)}  d^2(v_\sigma, w_\sigma) d\mu 
%& \leq &  
%C\int_{B_{\frac{3}{4}}(0)} |\nabla d(v_\sigma, w_\sigma)|^2 d\mu.
%\ \ \ \ \mbox{(by the Poincare inequality)} \\
%& \leq & C \left( E^{v_\sigma}(\frac{3}{4}) - E^{w_\sigma}(\frac{3}{4}) \right)
%\ \ \ \ \mbox{(by \cite[(2.2iv)]{korevaar-schoen2})} \\
%& \leq &  C \left( E^{u_\sigma}(\frac{3}{4}) - E^{\hat u}(\frac{3}{4})+ C\sigma^2  \right) 
%\ \ \ \  \mbox{(by Lemma~\ref{disest2})}
%\\
%& \leq & C\sigma^2 \ \  \ \ \mbox{(since $u_\sigma$ is harmonic)}.
%\end{eqnarray*}
%After letting $w=w_\sigma$ in (\ref{instead}) and integrating, we can  combine the resulting inequality with the above to conclude the desired inequality. 
%\end{proof}
%

 \begin{lemma}
 %[{\bf \cite{daskal-meseDM} Section 9, estimate (132)}] 
 \label{energycomp}
Let  $v_{\sigma}$ and $w_{\sigma}$ be as in Definition~\ref{blowingupv} and Definition~\ref{approxhm} respectively.  There exists a sequence $\sigma_i \rightarrow 0$ and a constant $C>0$ such that
\[
|E^{v_{\sigma_i}}(r) - E^{w_{\sigma_i}}(r)| \leq C\sigma_i, \ \forall  r \in (0,\frac{1}{2}).
\]
\end{lemma}

\begin{proof}
 The main issue  is that the map $v_\sigma$ is not a competitor to the harmonic map $w_\sigma$ in the domain $B_{r}(0)$ because their boundary values do not necessarily match.  Therefore, we ``bridge'' the gap between $v_\sigma$ and $w_\sigma$ using \cite[Lemma 3.12]{korevaar-schoen2}. This is estimate  \cite[estimate (132)]{daskal-meseDM} where we refer for complete details.
\end{proof}

We next prove the existence of  blow up maps of the singular component map converging to a tangent map.  We will need the following.

\begin{lemma}  \label{anothershv} 
For $x_0\in  {\mathcal S}_j(u) \cap B_{\frac{\sigma_{\star}}{2}}(x_{\star})$, 
let  $v_\sigma$ be  a blow-up map  of $v$ at $x_0$  (cf.~Definition~\ref{blowingupv}).
For  a fixed $R \in (0,1)$ and $r \in (0, 1)$,  there exists a constant 
$C>0$ that can be chosen independently of $\sigma$
such that for  any harmonic map  
\[
w:(B_R(0),g_{ \sigma}) \rightarrow  \overline{\bf H}^{k-j}
\
\mbox{  with }
 \sup_{B_R(0) }d_h(w,{\mathbb P}_0) \leq c,
\] 
we have
\[
\sup_{B_{r \vartheta}(0)} d_h^2(v_\sigma,w) \leq   \frac{C}{\vartheta^{n-1}}  \int_{\partial B_{\vartheta}(0)} d_h^2(v_\sigma,w)  d\Sigma_\sigma  + C\sigma^2  \vartheta^3, \ \ \forall \vartheta \in (0,R]
\]
where $d\s_\sigma$ is the volume form with respect to the metric $g_\sigma$.
\end{lemma}

\begin{proof}
We argue in a similar way as in the proof of Lemma~\ref{anothersh}.  The only difference is that we do not know $v_\sigma$ is Lipschitz continuous in this proof.  Instead, we use the monotonicity property of the singular component map $v$ given by Proposition~\ref{orderofv}.
Indeed, the first estimate of (\ref{monotonicity!}) implies
\begin{eqnarray*}
\frac{\displaystyle{\int_{\partial B_{\rho} (0)} d_h^2(v_\sigma, {\mathbb P}_0) d\s_\sigma}}{\rho^{n-1} } \leq c \rho^2,  \ \forall \rho \in (0,1).
\end{eqnarray*}
Multiplying by $\rho^{n-1}$ and integrating over $\rho \in (0,s)$, we obtain
\begin{equation} \label{dii}
\int_{B_s(0)} d_h^2(v_\sigma, {\mathbb P}_0) \du_\sigma  \leq c  s^{n+2}, \ \forall s \in (0,1).
\end{equation}
Arguing as in the proof of Lemma~\ref{anothersh}, we obtain the analogue of (\ref{di}). Combining it with (\ref{dii}), we obtain
\[
0  \leq  \frac{d}{ds} \left( \frac{e^{c_1s}}{s^{n-1}} \int_{\partial B_s(0)} d_h^2(v_\sigma,w) d\Sigma_\sigma \right) +c\sigma^2 s^2
\]
which should be compared to (\ref{pdi'}) in the proof of Lemma~\ref{anothersh}.
The rest of the proof follows  exactly as the proof of Lemma~\ref{anothersh}.  
\end{proof}

\begin{lemma} \label{existenceoftangentmap}
For $x_0\in  {\mathcal S}_j(u) \cap B_{\frac{\sigma_{\star}}{2}}(x_{\star})$,  there exists a sequence of blow-up maps $v_{\sigma_i}$ of $v$ at $x_0$ (cf.~Definition~\ref{blowingupv})  converging locally uniformly in the pullback sense  to a homogeneous harmonic map $v_0:B_1(0) \rightarrow (Y_0,d_0)$ into an NPC space with 
\[
Ord^{v_0}(0)
 = Ord^{v}(x_0).
\] 
%In particular,
%there exists $\epsilon_0>0$ such that 
%\[
%Ord^{v}(x_0)=1 \ \mbox{ or } \ Ord^{v}(x_0) \geq  1+\epsilon_0.
%\]
\end{lemma}

\begin{proof}
Let $w_\sigma:B_{\frac{3}{4}}(0) \rightarrow ({\bf H}^{k-j},d_h)$ be the approximating harmonic map for $v_\sigma$ (cf.~Definition~\ref{approxhm}).
Since  
\begin{equation} \label{wuniformenergy}
E^{w_\sigma}({\frac{3}{4}}) \leq E^{v_\sigma}({\frac{3}{4}}) \leq E_0
\end{equation} 
by (\ref{bdbyo}), the family of harmonic maps $w_\sigma$  has a uniform local  Lipschitz estimate by Theorem~\ref{KSlip}.   The Compactness Theorem~\ref{KScpct} implies that there exists  a sequence of $w_{\sigma_i} \big|_{B_{\frac{1}{2}}(0)} $ that converges locally uniformly in the pullback sense to a harmonic map  $v_0:B_{\frac{1}{2}}(0) \rightarrow (Y_0,d_0)$ into an NPC space and
\[
E^{v_0}(r) =\lim_{\sigma_i \rightarrow 0} E^{w_{\sigma_i}}(r), \ \forall r \in (0,\frac{1}{2}).
\] 
By Lemma~\ref{anothershv}, 
\[
\lim_{\sigma_i \rightarrow 0} \sup_{B_{\frac{1}{2}}(0)} d^2_h(v_{\sigma_i},w_{\sigma_i})  =0,
\]
and thus 
\begin{eqnarray}\label{cvgamap}
 v_{\sigma_i} \big|_{B_{\frac{1}{2}}(0)}  \ \mbox{converges locally uniformly in the pullback sense to  $v_0$. }
 \end{eqnarray} 
  In particular, we have
\[
I^{v_0}(r) =\lim_{\sigma_i \rightarrow 0} I^{v_{\sigma_i}}(r), \ \forall r \in (0,\frac{1}{2}).
\]
 Furthermore, by Lemma~\ref{energycomp},   there exists a constant $C>0$ such that
\begin{equation} \label{compafevwenergies'}
|E^{v_{\sigma_i}}(r) - E^{w_{\sigma_i}}(r)| \leq C\sigma_i, \ \forall r \in (0,\frac{1}{2}),
\end{equation}
and hence
\[
E^{v_0}(r) =\lim_{\sigma_i \rightarrow 0} E^{v_{\sigma_i}}(r), \ \forall r \in (0,\frac{1}{2}).
\]
Thus, for $r \in (0,\frac{1}{2})$,
\[
\frac{rE^{v_0}(r)}{I^{v_0}(r)}=\lim_{\sigma_i \rightarrow 0}\frac{rE^{v_{\sigma_i}}(r)}{I^{v_{\sigma_i}}(r)}=
 \lim_{\sigma_i \rightarrow 0}
 \frac{r\sigma_i E^v(r\sigma_i)}{I^v(r \sigma_i)} 
 = Ord^{v}(x_0).
 \]
Note that  the right hand side is independent of $r$.  Thus, by following the argument of \cite{gromov-schoen} Proposition~3.3, we conclude that the map $v_0$ is  \emph{homogeneous degree $\alpha=Ord^{v}(x_0)$}.  Furthermore, letting $r \rightarrow 0$ above, we obtain
\[
Ord^{v_0}(0)
 = Ord^{v}(x_0).
\]
%Finally, the last assertion of the lemma follows from Proposition~\ref{ahmresult}.
\end{proof}

\begin{definition} \label{tangentmapofv}
The homogeneous harmonic map $v_0$  of Lemma~\ref{existenceoftangentmap} will be referred to as a \emph{tangent map} of $v$ at $x_0\in  {\mathcal S}_j(u) \cap B_{\frac{\sigma_{\star}}{2}}(x_{\star})$.   Note that $v_0$ may be different from   $v_*$ of Lemma~\ref{blowupcomponent}, the singular component of a tangent map $u_*=(V_*,v_*)$.
\end{definition}

For the convenience of the reader bellow we summarize the different blow-up and tangent maps used in the paper:
\begin{itemize}
\item The blow-up maps $u_{\sigma_i}=(U_{\sigma_i}, v_{\sigma_i})$ of the harmonic map $u=(U,V)$ defined  in (\ref{bups}) with scaling factor $\mu_\sigma$ defined in (\ref{standardscaling}) and converging in the pullback sense to the tangent map $u_*=(V_*,v_*)=(V_*, v_*^1,..., v_*^k)$ (cf. Lemma~\ref{blowupcomponent}).
\item The blow-up maps $v_{\sigma_i}$  of the singular component $v$ of the harmonic map 
$u=(U,V)$ defined in (\ref{buwrtv}) with scaling factor $\nu_\sigma$ defined in (\ref{buwrtv0}) converging to the homogeneous harmonic map  $v_0$ (cf. Lemma~\ref{existenceoftangentmap}).
\item  The approximating harmonic map 
$w_\sigma:B_{\frac{3}{4}}(0) \rightarrow ({\bf H}^{k-j},d_h)$ for $v_\sigma$
 whose boundary values agree with that of $v_\sigma\big|_{B_{\frac{3}{4}}(0)}$ (cf. Definition~\ref{approxhm}). 
\end{itemize}

\begin{lemma} \label{compblowupv}
%Assume that the Inductive Hypothesis [j+1]   holds  for a harmonic map $u:\Omega \rightarrow \overline{\mathcal T}$, $x_{\star} \in {\mathcal S}_j(u)$ and  $u=(V,v)$ is  a local representation of a harmonic map at $x_{\star}$.  
If  $x_0 \in  {\mathcal S}_j(u) \cap B_{\frac{\sigma_{\star}}{2}}(x_{\star})$, then 
the sequence $v_{\sigma_i}$ of blow-up maps of $v$ at $x_0$ converging to a tangent map $v_0$ (cf.~Definition~\ref{tangentmapofv}) is a sequence of  asymptotically harmonic maps.
%More specifically, let $w_{i}:B_{\frac{3}{4}}(0) \rightarrow  \overline{\bf H}^{k-j}$ denote the harmonic map with $w_i \big|_{B_{\frac{3}{4}}(0)}=v_{\sigma_i}\big|_{B_{\frac{3}{4}}(0)}$.  
%The blow-up maps $\{v_{\sigma_i}\}$ and the approximating harmonic maps $\{w_i\}$ converge locally uniformly in the pullback sense to a non-constant homogeneous harmonic map $v_0: (B_1(0), \delta) \rightarrow (Y_0,d_0)$ for some NPC space and   
%for any $r \in (0,\frac{1}{2})$,
%\[
%\lim_{i \rightarrow \infty} \sup_{B_r(0)} d(v_{\sigma_i}, w_{\sigma_i}) =0.
%\]
%Furthermore,
%For $\sigma_i$ sufficiently small and any sequence $x_i \in \sigma_i^{-1} {\mathcal S}_j(u) \cap B_{\frac{1}{2}}(0)$, $R\in (0,\frac{1}{4})$, there exists  $\{r_i\} \subset  (\frac{R}{2},R)$  such that
%\[
%\lim_{i \rightarrow \infty}  \left| E^{v_{\sigma_i}}_{x_i}(r_i)- E^{w_{\sigma_i}}_{x_i}(r_i) \right| =0.
%\]
% 
% \[
% \lim_{i \rightarrow \infty}\frac{r\sigma_i E^v_{x_0}(r\sigma_i)}{I^v_{x_0}(r\sigma_i)}= \lim_{i \rightarrow \infty}\frac{r E^{v_{\sigma_i}}_{x_0}(r)}{I^{v_{\sigma_i}}_{x_0}(r)}=\frac{r E^{v_0}_{x_0}(r)}{I^{v_0}_{x_0}(r)}, \ \ \ r \in (0,R].
%\]
%Hence, by the homogeneity of $v_0$, 
\end{lemma} 
\begin{proof}
We check properties (i)-(iv) of Definition~\ref{asymptoticallyharmonic} for $v_i=v_{\sigma_i}$ and $g_i=g_{\sigma_i}$.  Property (i) regarding the metric $g_i$ follows immediately from the definition of $g_{\sigma_i}$. From   the monotonicity property (\ref{monotonicity!}) and the energy bound (\ref{bdbyo}), we obtain (ii). Property (iii) is about the convergence of $v_{\sigma_i}$ to a tangent map $v_0$ (cf.~(\ref{cvgamap})).  Finally, Property (iv) follows from Lemma~\ref{anothershv} with $c_0=C$ and $c_i=C\sigma_i^2$.  
\end{proof}

%First, we  show {\sc Statement 1}$[j]$ holds.
%%Conclusion of the Proof of the Regularity Theorems~\ref{RegularityTheorem} and~\ref{goto0}}
%\label{conclusionofproof}
%% \subsection{The Singular Set of $v$} 
%We divide  the proof in the following two steps, one dealing with order 1 points and the second with higher order pints of $v$.

This yields the following corollary.  
\begin{corollary}\label{ptsorder1}
%Assume that the Inductive Hypothesis [j+1]   holds  for a harmonic map $u:\Omega \rightarrow \overline{\mathcal T}$.  Let  $x_{\star} \in {\mathcal S}_j(u)$ and  $u=(V,v)$ be a local representation of a harmonic map at $x_{\star}$.     
% Let $v$ be the singular component map of Proposition~\ref{orderofv} and $\epsilon_0>0$ be as in Lemma~\ref{ahmresult}
 If $x_0\in B_{\frac{\sigma_{\star}}{2}}(x_{\star}) \cap {\mathcal S}_j(u)$, then  
 \[
 Ord^v(x_0) \geq 1+ \epsilon_0.
 \]
 \end{corollary}
 \begin{proof}
% Let $\{v_{\sigma_i}\}$ be as in Lemma~\ref{compblowupv} converging locally uniformly in the pullback sense to a non-constant homogeneous harmonic map $v_0: (B_{\frac{1}{2}}(0), g_0) \rightarrow (Y_0,d_0)$. 
 Since $v_{\sigma_i}$ of Lemma~\ref{compblowupv} is a sequence of asymptotically harmonic maps,  the assertion follows from Lemma~\ref{ahmresult}, Proposition~\ref{linearapproxblowupsym'} and Lemma~\ref{existenceoftangentmap}. 
 \end{proof}

The above discussion yields a slight variation of Lemma~\ref{uscord} on the upper semicontinuity of the order for $v$.  
 \begin{lemma} \label{usordv} 
Let $x_0 \in {\mathcal S}_j(u) \cap B_{\frac{\sigma_{\star}}{2}}(x_{\star})$ and $v_{\sigma_i}$ be the sequence of blow-up maps of $v$ at $x_0$ converging locally uniformly in the pullback sense to a tangent map $v_0$.   After identifying $x_0=0$ via normal coordinates,  let 
\[
\sigma_i^{-1}({\mathcal S}_j(u) \cap B_{\frac{\sigma_{\star}}{2}}(x_{\star})):=\{\sigma_i^{-1}x:  x\in {\mathcal S}_j(u) \cap B_{\frac{\sigma_{\star}}{2}}(x_{\star})\}.
\]
If  $x_i \in \sigma_i^{-1}({\mathcal S}_j(u) \cap B_{\frac{\sigma_{\star}}{2}}(x_{\star})) \cap B_{\frac{1}{2}}(0)$ 
and $x_i \rightarrow x_*$, then
\[
\liminf_{\sigma_i \rightarrow 0} Ord^{v_{\sigma_i}}(x_i) \leq Ord^{v_0}(x_*).
\]
\end{lemma}

 \begin{proof}
Since $\sigma_i x_i \in {\mathcal S}_j(u)$, we can apply Proposition~\ref{orderofv} to assert that 
\[
\alpha_i:= \lim_{r \rightarrow 0} \frac{r E_{x_i}^{v_{\sigma_i}}(r)}{I_{x_i}^{v_{\sigma_i}}(r)}
\]
exists.
%and there  exist constants $c>0$ and $R_0>0$ such that
%\[
%Ord^{v_{\sigma_i}}(x_i) \leq e^{cr} \frac{\sigma E_{x_i}^{v_{\sigma_i}}(r)}{I_{x_i}^{v_{\sigma_i}}(r)}, \ \ \ \forall r \in (0,R_0).
%\]
%Moreover, by  \cite{daskal-meseDM} inequalities (114) and (115) there exists $C>0$, 
%\begin{eqnarray}\label{smnl1}
%e^{-C_1 r}\frac{r E^v_x (r)}{I^v_x (r)} \leq e^{C_1 \rho}\frac{\rho E^v_x (\rho)}{I^v_x (\rho)},  \ \ \ \forall r \in (0,R_0).
%\end{eqnarray}
%The important point is that, according to  \cite{daskal-meseDM} Proposition 54,  the constants above depend continuously on the point $x$.
The proof follows  as in Lemma~\ref{uscord} with $u_{\sigma_i}$ replaced with $v_{\sigma_i}$ and $u_*$ replaced with $v_0$.  The only difference is that the equality (\ref{thisuseslip}) in the proof of  Lemma~\ref{uscord} uses the uniform local Lipschitz continuity of the sequence $u_{\sigma_i}$.  Although we know that, for each $i$,  $v_{\sigma_i}$ is Lipschitz continuous  by the Lipschitz continuity of $u$, we have not proven  any \emph{uniform} local Lipschitz estimates (i.e.~independent of $i$) of the sequence $v_{\sigma_i}$. (See again Remark~\ref{bothblowups}.)  On the other hand, the sequence of approximating harmonic maps $w_{\sigma_i}$  for $v_{\sigma_i}$ (cf.~Definition~\ref{approxhm}) is uniformly locally Lipschitz continuous. Indeed, according to Proposition~\ref{orderofv} there exist constants $c_1>0$, $c_2>0$ and  $R_0>0$  such that for $0<\sigma_i<\rho<R_0$ 
\begin{eqnarray}\label{smnl}
E_{x_i}^{w_{\sigma_i}}({\frac{3}{4}}) \leq E_{x_i}^{v_{\sigma_i}}({\frac{3}{4}})  \leq E_{x_i}^{v_{\sigma_i}}(1)= \frac{\sigma_i E^v_{\sigma_i x} (\sigma_i)}{I^v_{\sigma_i x} (\sigma_i)} \leq c_1 \frac{\rho E^v_{\sigma_i x} (\rho)}{I^v_{\sigma_i x} (\rho)}\leq c_2.
\end{eqnarray}
Here, the last inequality follows from the continuity of 
\[
x \mapsto \frac{\rho E_x^v(\rho)}{I_x^v(\rho)}
\]
 and the second to the last inequality follows from (\ref{smnl1}). Thus  (\ref{smnl}) and Theorem~\ref{KSlip} imply that $w_{\sigma_i}$ is uniformly Lipschitz. 
    
Therefore, repeating the proof of (\ref{thisuseslip}), we obtain
\[
\lim_{\sigma_i \rightarrow 0}|E^{w_{\sigma_i}}_{x_*}(r) -E^{w_{\sigma_i}}_{x_i}(r)|=0.
\]
Combining  (\ref{compafevwenergies'}) with  the estimate
\[
%  \lim_{\sigma_i \rightarrow 0}
|E^{v_{\sigma_i}}_{x_*}(r) -E^{v_{\sigma_i}}_{x_i}(r)|
\leq   |E^{v_{\sigma_i}}_{x_*}(r) -E^{w_{\sigma_i}}_{x_i}(r)|+  |E^{w_{\sigma_i}}_{x_*}(r) -E^{w_{\sigma_i}}_{x_i}(r)|+ |E^{w_{\sigma_i}}_{x_*}(r) -E^{v_{\sigma_i}}_{x_i}(r)|,
\]
we obtain
\[
\lim_{\sigma_i \rightarrow 0}|E^{v_{\sigma_i}}_{x_*}(r) -E^{v_{\sigma_i}}_{x_i}(r)|=0
\]
which is (\ref{thisuseslip}) with $v_{\sigma_i}$ replacing with $u_{\sigma_i}$.  The rest of the proof  is exactly as in Lemma~\ref{uscord}.
\end{proof}

\subsection{Inductive Step}
\label{conclusionofproof}
Throughout this subsection, we assume that the \emph{Inductive Hypothesis [j+1]}   holds  
for a harmonic map $u:(\Omega,g) \rightarrow( \overline{\mathcal T}, d_{ \overline{\mathcal T}})$.  
Let  $x_{\star} \in {\mathcal S}_j(u)$ 
 and
\[
u=(V,v):(B_{\sigma_{\star}}(x_{\star}),g) \rightarrow  ({\mathcal U} \times {\mathcal V}, d_G)
\]
be a local representation (cf. Definition~(\ref{locrep})) of $u$ at $x_{\star}$.
The goal is to show that both {\sc Statement 1}$[j]$ holds and {\sc Statement 2}$[j]$ holds. 

\begin{proposition}\label{hmco2}
%Assume that the Inductive Hypothesis [j+1]   holds  
%for a harmonic map $u:(\Omega,g) \rightarrow( \overline{\mathcal T}, d_{ \overline{\mathcal T}})$. If  $x_{\star} \in {\mathcal S}_j(u)$ and  $u=(V,v)$ is a local representation of a harmonic map at $x_{\star}$, then
The set ${\mathcal S}_j(u)$ is of Hausdorff codimension 2 in $B_{\frac{\sigma_{\star}}{2}}(x_{\star})$; i.e.
\[
\dim_{\mathcal H}({\mathcal S}_j(u) \cap B_{\frac{\sigma_{\star}}{2}}(x_{\star})) \leq n-2.
\]
\end{proposition}

\begin{proof}
The assertion holds trivially if $v$ is identically equal to ${\mathbb P}_0$ (in this case $u$ maps into a single stratum of $\overline{\mathcal T}$),  so assume that $v$ is a non-constant map.   Assume on the contrary that $
\dim_{\mathcal H}({\mathcal S}_j(u) \cap B_{\frac{\sigma_{\star}}{2}}(x_{\star})) > n-2$; thus, there exists  $s >n-2$ such that ${\mathcal H}^s({\mathcal S}_j(u) \cap B_{\frac{\sigma_{\star}}{2}}(x_{\star})) >0$.  By \cite[2.10.19]{federer} (also see the proof of \cite[Lemma 6.5]{gromov-schoen}),  there exists $x_0 \in {\mathcal S}_j(u) \cap B_{\frac{\sigma_{\star}}{2}}(x_{\star})$ such that  
\[
2^{-s} \leq \liminf_{\sigma \rightarrow 0} \frac{ {\mathcal H}^s({\mathcal S}_j(u) \cap B_{\frac{\sigma_{\star}}{2}}(x_{\star})) \cap B_{\frac\sigma{2}}(x_0)}{(\sigma/2)^{s}}.
\]
With  $\sigma_i^{-1}({\mathcal S}_j(u) \cap B_{\frac{\sigma_{\star}}{2}}(x_{\star}))$ defined as in Lemma~\ref{usordv} and after identifying $x_0=0$ via normal coordinates, we conclude  
\begin{equation} \label{lhs65}
n-2<s \leq  \limsup_{\sigma_i \rightarrow 0} \dim_{\mathcal H}(\sigma_i^{-1}({\mathcal S}_j(u) \cap B_{\frac{\sigma_{\star}}{2}}(x_{\star})) \cap B_{\frac{1}{2}}(0).
\end{equation}

We claim
\begin{equation} \label{gslemma65}
 \limsup_{\sigma_i \rightarrow 0} \dim_{\mathcal H}(\sigma_i^{-1}({\mathcal S}_j(u) \cap B_{\frac{\sigma_{\star}}{2}}(x_{\star})) \cap B_{\frac{1}{2}}(0) )\leq \dim_{\mathcal H}({\mathcal S}^{>1}(v_0)).
\end{equation}
 Combining (\ref{lhs65}) and (\ref{gslemma65}), we arrive at a  contradiction (cf.~Lemma~\ref{ahmresult} (\ref{em})) which finishes the proof.

 We are left to prove (\ref{gslemma65}).  Indeed, taking a subsequence if necessary, we can assume that the sequence of  blow up maps  $v_{\sigma_i}$ of  $v$ at $x_0$ converges to a tangent map $v_0$. Let  $R \in (0,\frac{1}{2})$, $x_i \in \sigma_i({\mathcal S}_j(u) \cap B_{\frac{\sigma_{\star}}{2}}(x_{\star})) \cap \overline{B_R(0)}$ and assume $x_i \rightarrow x_*$.
Corollary~\ref{ptsorder1} implies $Ord^{v_{\sigma_i}}(x_i) \geq 1+ \epsilon_0$ which in turn implies
 $Ord^{v_0}(x_*) \geq 1+\epsilon_0$ by the upper semicontinuity of order (cf.~Lemma~\ref{usordv}).  Thus, we conclude that if $x_i \rightarrow x_*$ for $x_i \in   {\mathcal S}_j(u) \cap B_{\frac{\sigma_{\star}}{2}}(x_{\star}) \cap \overline{B_R(0)}$, then $x_* \in {\mathcal S}^{>1}(v_0) \cap \overline{B_R(0)}$. Thus, (\ref{gslemma65}) follows from Lemma~\ref{federeroutermeasure}.
 \end{proof}
%}

\begin{lemma} \label{proof2j}
%Assume that the Inductive Hypothesis [j+1]   holds  
%for a local representation
%\[
%u=(V,v):(B_{\sigma_{\star}}(x_{\star}),g) \rightarrow  ({\mathcal U} \times {\mathcal V}, d_G)
%\]
%of a harmonic map $u$ at $x_{\star}$ (cf. Definition~(\ref{locrep})).
If {\sc Statement 1}$[j]$ holds, then {\sc Statement 2}$[j]$ also holds.
\end{lemma}

%{\color{blue} NOTE TO GEORGE:  Please read this proof carefully.  I simplified the proof from \cite{daskal-meseDM} quite a bit and corrected some typos.  Major changes are indicated in blue.}

%{\bf{Embedd into the proof}}
%{\color{blue} 
%\begin{proposition} \label{A6}
%For $q \in [1,2)$ sufficiently close to 2 and any subdomain $\Omega_1$ compactly contained in 
%\[
%B_{\frac{x_\star}{2}}(x_{\star}) \backslash ( {\mathcal S}(u) \cap v^{-1}({\mathbb P}_0)), 
%\]
%there exists a sequence of smooth  functions $\psi_i$ and a neighborhood of $\mathcal N_i$ contained in an $\epsilon_i$-neighborhood of ${\mathcal S}(u)$ with  $0 \leq \psi_i \leq 1$, $\psi_i \equiv 0$ in a neighborhood of ${\mathcal S}(u) \cap \overline{\Omega_1}$, $\psi_i \equiv 1$ outside of ${\mathcal N}_i$, $\epsilon_i \rightarrow 0$,  $\psi_i \rightarrow 1$ for all $x \in \Omega_1   \backslash {\mathcal S}(u) $ such that 
%%
%%there exists a sequence of smooth functions $\psi_i \equiv 0$ in a neighborhood of ${\mathcal S}(u) \cap \overline{\Omega_1}$, $0 \leq \psi_i \leq 1$, $\psi_i \rightarrow 1$ for all $x \in \Omega_1 \backslash {\mathcal S}(u)$ such that
%\[
%\lim_{i \rightarrow \infty} \int_{\Omega_1} \ |\nabla u| |\nabla \psi_i| \ d\mu =0.
%\]
%\[
%\lim_{i \rightarrow \infty} \int_{\Omega_1}\ |\nabla u| |\nabla \psi_i|^q \ d\mu =0.
%\]
%\[
%\lim_{i \rightarrow \infty} \int_{\Omega_1} \ |\nabla \nabla u| |\nabla \psi_i| \ d\mu =0.
%\]
%\end{proposition}
%\begin{proof} 
%This follows from the inductive hypothesis {\sc Statement 2}$[j+1]$, \dots, {\sc Statement 2}$[k]$ and a partition of unity argument.
%\end{proof}
%}
%
%

\begin{proof}  %This follows from  \cite[pp.~68-73]{daskal-meseDM}, but we will include the proof here for the sake of completeness.
Throughout this proof,   we will use $C$ (which may change line by line) to denote an arbitrary constant that depends only on  the dimension $n$ of the domain, the
Lipschitz constant of $u$ in $B_{\frac{\sigma_{\star}}{2}}(x_{\star})$ and the modulus of continuity of $V$ in $B_{\frac{\sigma_{\star}}{2}}(x_{\star})$ (cf.~(\ref{regularityV})).  

Let  $\epsilon_0>0$ be  smaller than either of the $\epsilon_0$ that appears in Proposition~\ref{ordergap'} and Corollary~\ref{ptsorder1}. Choose $q<2$, $p>2$, $\delta \in (0,1)$ and $D \in (0,1)$ such that  Proposition~\ref{A6} holds for $q$  and  
\[
\frac{1}{p}+\frac{1}{q}=1, \ \
D<\delta<\epsilon_0, \ \ D<\epsilon_0-\delta
\]
\begin{equation} \label{choiceD}
 -2+D<-q-q\delta, \ \ 
-2+D<-p-p\delta+\epsilon_0.
\end{equation}
To prove {\sc Statement }2[$j$], we show that, for a fixed  subdomain $\Omega \subset \subset B_{\frac{\sigma_{\star}}{2}}(x_{\star})$ and $\varepsilon>0$,  there exists an open set ${\mathcal N}$ contained in an  $\varepsilon$-neighborhood of ${\mathcal S}(u) \cap \overline{\Omega}$ and a smooth  function $\psi$  such that $0 \leq \psi \leq 1$,   $\psi \equiv 0$ in a neighborhood  of ${\mathcal S}(u) \cap \overline{\Omega}$,  $\psi  \equiv 1$ on  $\Omega   \backslash {\mathcal N}$ that satisfies
\begin{eqnarray}
 \label{qqeq}
\int_{B_{\frac{\sigma_{\star}}{2}}(x_{\star})} \ |\nabla u| |\nabla \psi| \ d\mu & < & C \varepsilon,
\\
 \label{qqqeq}
 \int_{B_{\frac{\sigma_{\star}}{2}}(x_{\star})} \ |\nabla u| |\nabla \psi|^q \ d\mu & < &  C\varepsilon
\\ \label{theharderone}
 \int_{B_{\frac{\sigma_{\star}}{2}}(x_{\star})} \ |\nabla \nabla u| |\nabla \psi| \ d\mu & < &  C\varepsilon^{\frac{1}{p}}.
\end{eqnarray}
  
 %Let   $\Omega_2$ be such that $\Omega \subset \subset \Omega_2 \subset \subset B_{\frac{\sigma_{\star}}{2}}(x_{\star})$.
Note that  $|\nabla u|(x) \neq 0$ for  $x \in {\mathcal S}_j(u)$ since $u$ is of order 1 at any point in ${\mathcal S}_j(u)$ (cf.~(\ref{sj})).  Thus, by Proposition~\ref{A5}, $|\nabla V|(x)  \neq  0$ for $x \in {\mathcal S}_j(u) \cap B_{\frac{\sigma_{\star}}{2}}(x_{\star})$.
Since $\nabla V$ is continuous (cf.~Remark~\ref{regularityV}),  
 there exists an open set ${\mathcal N} \subset \subset B_{\frac{\sigma_{\star}}{2}}(x_{\star})$ contained in an $\varepsilon$-neighborhood of ${\mathcal S}_j(u) \cap \overline{\Omega}$ and a constant $\lambda_0$ such that
\begin{equation} \label{lowerbd}
|\nabla V|\geq \lambda_0>0  \  \mbox{  on }  {\mathcal N}.
\end{equation}

{\sc Statement 1}$[j]$ implies that we can choose a finite covering $\{B_{r_J}(x_J):J=1, \dots, l\}$  of the compact set ${\mathcal S}_j(u) \cap \overline{\Omega}$ satisfying
\begin{equation} \label{d}
 \sum_J r_J^{n-2+D} < \varepsilon \ \mbox{ and } \ 
%\end{equation}
% By choosing $x_J \in {\mathcal S}_j(u) \cap \overline{\Omega}$ and $r_J$'s sufficiently small, we can also assume
%\begin{equation} \label{nnnn}
B_{3r_J}(x_J) \subset  {\mathcal N}.
\end{equation}
%Let
%\[
%\Omega_0 :=\Omega \backslash  \bigcup_{J=1}^l \overline{B_{r_J}(x_J)}.
%\]
 Let $\varphi_J$ be a smooth  function such that $\varphi_J \equiv 0$ on  $B_{r_J}(x_J)$, $\varphi_J\equiv 1$  on $\Omega \backslash B_{2r_J}(x_J)$, $|\nabla \varphi_J| \leq C r_J^{-1}$, $|\nabla \nabla \varphi_J| \leq C r_J^{-2}$, $|\nabla (\varphi |\nabla \varphi|^{\delta})| \leq C r_J^{-1-\delta}$ and $|\nabla \nabla (\varphi |\nabla \varphi|^{\delta})| \leq C r_J^{-2-\delta}$.    Define $\varphi$ by setting
\[
\varphi=\prod_J \varphi_J.
\]
%Thus, $\varphi \equiv 0$ in a neighborhood of ${\mathcal S}_j(u) \cap \overline{\Omega}$, $\varphi \equiv 1$ outside $\bigcup_{J=1}^l B_{2r_J}(x_J)$ and $0 \leq \varphi \leq 1$.    
%We also define $\rho_J$ to be a smooth function that is identically one on $B_{2r_J}(x_J)$ and identically zero on $\Omega \backslash B_{3r_J}(x_J)$ with $|\nabla \rho_J| \leq C r_J^{-1}$, $|\nabla \nabla \rho_J|\leq Cr_J^{-2}$ and set
%\[
%\rho= \prod_J \rho_J.
%\]
Let $\rho$ be a Lipschitz function such that $\rho \equiv 1$ on $\bigcup_J B_{2r_J}(x_J)$, $\rho \equiv 0$ outside  $\bigcup_J B_{3r_J}(x_J)$ and $|\nabla \rho| \leq \frac{1}{r_J}$ in $B_{2r_J}(x_J)$ (cf.~\cite[before (6.3)]{gromov-schoen}).

With $\varphi$ and $\rho$ now fixed, Proposition~\ref{highorderpts} implies that we can choose a finite covering $\{B_{s_J}(\xi_J):J=1, \dots, l'\}$  of
$
{\mathcal S}^{>1}(u)  \cap \overline{\Omega}_0
$
 with\enlargethispage{\baselineskip}
\begin{equation} \label{whoops}
\max\left\{ \sup_{\Omega} |\nabla \varphi|, \sup_{\Omega} |\nabla \nabla \varphi|, \sup_{\Omega} |\nabla \rho|)\right\} \sum_J s_J^{n-2+D} < \varepsilon
\end{equation}
and
\[
B_{2s_J}(x_J) \subset {\mathcal N}.
\]
Let $\phi_J$ be a smooth  function such that $\phi_J \equiv 0$ on  $B_{s_J}(\xi_J)$, $\phi_J\equiv 1$  on $\Omega_0 \backslash B_{2s_J}(\xi_J)$, $|\nabla \phi_J| \leq C s_J^{-1}$ and $|\nabla \nabla \phi_J| \leq C s_J^{-2}$.    Define $\phi$ by setting
\[
\phi=\prod_J \phi_J.
\]

Since ${\mathcal S}(u) \cap v^{-1}({\mathbb P}_0) \subset {\mathcal S}_j(u) \cup {\mathcal S}^{>1}(u)$, the set
\[
\Omega_1 := \Omega  \backslash \left(  \bigcup_{J=1}^l \overline{B_{3r_J}(x_J)}
 \cup \bigcup_{J=1}^{l'} \overline{B_{3s_J}(\xi_J)} \right)
\]
is compactly contained in $B_{\frac{x_\star}{2}}(x_{\star}) \backslash ( {\mathcal S}(u) \cap v^{-1}({\mathbb P}_0))$. With $\varphi$, $\rho$ and $\phi$ now fixed, we apply Proposition~\ref{A6} to obtain  a  smooth function $\hat \psi$ such that $0 \leq \hat \psi \leq 1$,  $\hat \psi \equiv 0$ in a neighborhood of ${\mathcal S}(u) \cap \overline{\Omega_1}$, $
  \hat \psi \equiv 1$ outside $\mathcal N$, 
  \begin{eqnarray} \label{prophat}
  \int_{\Omega_1}|\nabla u| |\nabla \hat{\psi}|d\mu < \varepsilon, \ \  \int_{\Omega_1} |\nabla u| |\nabla \hat{\psi}|^q d\mu < \varepsilon^q, \nonumber  \\   
  \sup_{\Omega_1}\{ |\nabla ( \varphi \rho \phi|\nabla \varphi|^{\delta}) |^p, |\nabla ( \varphi \rho \phi|\nabla \varphi|^{\delta}) |   \}\int_{\Omega_1} |\nabla \nabla u|| \nabla   \hat{\psi}|  d\mu <\varepsilon.
  \end{eqnarray}
  Let
\[
\psi:=\varphi^2 \phi^2 \hat{\psi}^2.
\]
By construction,   $0 \leq \psi \leq 1$,   $\psi \equiv 0$ in a neighborhood  of ${\mathcal S}(u) \cap \overline{\Omega}$,  $\psi =1$ for all $x \in \Omega   \backslash {\mathcal N}$. 
By (\ref{d}),
\begin{eqnarray*}
\int_{\Omega}  |\nabla u| |\nabla \varphi|   \ d\mu & \leq  & C
\int_{\Omega} \left| \sum_{J_0}\nabla \varphi_{J_0} \prod_{J \neq J_0} \varphi_J \right|  \ d\mu
%& \leq & C \sum_{J_0} \int_{B_{2r_{J_0}}(x_{J_0})} |\nabla \varphi_{J_0}|\\
\leq  C \sum_{J_0} r_{J_0}^{n-1}< C\varepsilon.
\end{eqnarray*}
By (\ref{whoops}), a similar estimate applies to the integral involving $\phi$ and using the inequality $n-2+D< n-q$ implied by the  fourth inequality of (\ref{choiceD}).
Combined with (\ref{prophat}), we thus conclude
\[
\int_{\Omega}|\nabla u| |\nabla \psi|  \ d\mu 
 \leq 
\int_{\Omega}   |\nabla u| |\nabla \varphi| \ d\mu + \int_{\Omega}  |\nabla u| |\nabla \phi|   \ d\mu + \int_{\Omega} |\nabla u| |\nabla \hat{\psi}|  \ d\mu\nonumber  <  C\varepsilon
\]
which proves inequality (\ref{qqeq}).  Similar computation proves  (\ref{qqqeq}).

We are left to prove (\ref{theharderone}).
We first write
\begin{eqnarray}
\int_{\Omega} |\nabla \nabla u| |\nabla \psi| d\mu 
&= & 2  \int_{\Omega} \varphi^2 \phi^2\hat{\psi} |\nabla \nabla u| |\nabla \hat{\psi}| d\mu+ 2 \int_{\Omega} \phi \varphi^2 \hat{\psi^2} |\nabla \nabla u| |\nabla \phi| d\mu \nonumber
\\& & 
+2 \int_{\Omega} \varphi \phi^2 \hat{\psi}^2 |\nabla \nabla u| |\nabla \varphi| d\mu \nonumber \\
& =: & (A)+(B)+(C). \label{ABC}
\end{eqnarray}
Applying (\ref{prophat}), we can estimate
\[
(A):=2  \int_{\Omega} \varphi^2 \phi^2\hat{\psi} |\nabla \nabla u| |\nabla \hat{\psi}| d\mu \leq  2\int_{\Omega_1}  |\nabla \nabla u| |\nabla \hat{\psi}| d\mu<C\varepsilon.
\]

We next estimate $(C)$.  Noting that the support of the function $\varphi  \phi^2 \hat{\psi}^2|\nabla \varphi|$ is contained in  ${\mathcal R}(u) \cap \bigcup_{J=1}^l B_{2r_J}(x_J)$, 
 \begin{eqnarray*}
\lefteqn{(C)  :=   2 \int_{\Omega} \varphi  \phi^2\hat{\psi}^2 |\nabla \nabla u| |\nabla \varphi| \ d\mu}\\
& \leq &
2\left( \int_{\bigcup_{J=1}^l B_{2r_J}(x_J)} 
 \varphi  \phi\hat{\psi}   |\nabla \varphi|^{\delta}|\nabla \nabla u|^2 |\nabla u|^{-1}   \ d\mu \right)^{1/2} \left( \int_{{\cup_{J=1}^l B_{2r_J}(x_J)} }  |\nabla u| |\nabla \varphi|^{2-\delta} \ d\mu  \right)^{1/2}\\
& \leq &
2\left( \int_{\bigcup_{J=1}^l B_{2r_J}(x_J)}  \varphi  \phi \hat{\psi}  |\nabla \varphi|^{\delta} |\nabla \nabla u|^2 |\nabla u|^{-1} \ d\mu \right)^{1/2} \left(  C  \sum_{J=1}^l r_J^{n-2+\delta} \right)^{1/2}\\
& \leq &
C\varepsilon^{\frac{1}{2}} \left( \int_{ \bigcup_{J=1}^l B_{2r_J}(x_J)} \varphi  \phi \hat{\psi} |\nabla \varphi|^{\delta} |\nabla \nabla u|^2 |\nabla u|^{-1} \ d\mu \right)^{1/2}
\end{eqnarray*}
where the last inequality uses (\ref{choiceD}) and (\ref{d}).
Combining the Eells-Sampson and Schoen-Yau formulae (cf.~\cite[proof of Theorem 6.4]{gromov-schoen}), we have
\[
 |\nabla \nabla u|^2 |\nabla u|^{-1} \leq C\left(  |\nabla u| +\triangle |\nabla u| \right) \ \mbox{ on }{\mathcal R}(u).
\]
We multiply $\varphi  \phi \hat{\psi} \rho|\nabla \varphi|^{\delta}$ to both sides of the above inequality to obtain
\begin{eqnarray*}
\lefteqn{\int_{ \bigcup_{J=1}^l B_{2r_J}(x_J)}\varphi  \phi \hat{\psi} |\nabla \varphi|^{\delta} |\nabla \nabla u|^2 |\nabla u|^{-1} \ d\mu}\\
& \leq &  C\int_{\cup_{J=1}^l B_{3r_J}(x_J)} \varphi  \phi \hat{\psi} \rho |\nabla \varphi|^{\delta} |\nabla u| d\mu +C\int_{\cup_{J=1}^l B_{3r_J}(x_J)}  \triangle( \varphi  \phi \hat{\psi}\rho  |\nabla \varphi|^{\delta})   |\nabla u|d\mu \\
& =:&(C_1) +(C_2).
\end{eqnarray*}
By  (\ref{d}) and since $\delta <1$,
\[
(C_1) \leq C\int_{\cup_{J=1}^l B_{3r_J}(x_J)} |\nabla \varphi|^{\delta} d\mu
\leq C\sum_{J} r_J^{n-\delta} < C\varepsilon.
\]
To estimate $(C_2)$, we write
\begin{eqnarray*}
(C_2)&= & \int_{\cup_{J=1}^l B_{3r_J}(x_J)}\triangle (\varphi \rho \phi \hat{\psi}|\nabla \varphi|^{\delta}) |\nabla u| d\mu \\
& = &  \int_{\cup_{J=1}^l B_{3r_J}(x_J)}\triangle (\varphi \rho \phi \hat{\psi}|\nabla \varphi|^{\delta})\left( |\nabla V|+ |\nabla u|-|\nabla V| \right)d\mu \\
& = &  \int_{\cup_{J=1}^l B_{3r_J}(x_J)}\triangle (\varphi \rho \phi \hat{\psi}|\nabla \varphi|^{\delta}) |\nabla V|d\mu \\
&  & \ + \int_{\cup_{J=1}^l B_{3r_J}(x_J)}\triangle (\varphi \rho \phi \hat{\psi}|\nabla \varphi|^{\delta})\left( |\nabla u|-|\nabla V| \right)d\mu \\
& = & - \int_{\Omega}   \nabla ( \varphi \rho \phi \hat{\psi} |\nabla \varphi|^{\delta} )\cdot \nabla | \nabla V|d\mu\\
& & \ 
  + \int_{\Omega}   \hat{\psi}  \triangle ( \varphi \rho \phi|\nabla \varphi|^{\delta} )   \left( |\nabla u|-|\nabla V| \right)  d\mu\\
  & & \  \ +  \int_{\Omega}  \varphi \rho \phi|\nabla \varphi|^{\delta}    \triangle  \hat{\psi}    \left( |\nabla u|-|\nabla V| \right)  d\mu  \\
   & & \ \ \  +2 \int_{\Omega} \nabla ( \varphi \rho \phi|\nabla \varphi|^{\delta} )   \cdot \nabla   \hat{\psi}    \left( |\nabla u|-|\nabla V| \right)  d\mu \\
%&  \leq & C \left(\int_{\Omega}  |\nabla  ( \varphi \rho \phi \hat{\psi} |\nabla \varphi|^{\delta} )|^{q} d\mu \right)^{\frac{1}{q}} \cdot \left( \int_{\Omega}   | \nabla \nabla V|^p d\mu \right)^{\frac{1}{p}} \\
%& &
%\ + C\int_{\Omega} \hat{\psi} \left| \triangle ( \varphi \rho \phi  |\nabla \varphi|^{\delta} ) \right| |\nabla v| d\mu \\
% & & \  \ +C\left( \int_{\Omega} |\nabla ( \varphi \rho \phi  |\nabla \varphi|^{\delta} )|^p |\nabla v|  d\mu \right)^{\frac{1}{p}}  \left( \int_{\Omega}   | \nabla \hat{\psi}|^q |\nabla v|  d\mu \right)^{\frac{1}{q}} \\
% & & \ \ \ + C\int_{\Omega}   |\nabla \varphi|^{\delta} |\nabla  \hat{\psi}|  |\nabla \nabla u|  d\mu\\
 & =: & (C_{21})+(C_{22})+(C_{23})+(C_{24}).
 \end{eqnarray*}
 Using the fact that $|\nabla \nabla V| \in L^p(\Omega)$ (cf.~Remark~\ref{regularityV}) and the fact that derivatives of $\varphi$ and $\rho$ are supported in $\cup_{J=1}^l B_{3r_J}(x_J)$ and the derivatives of $\phi$ are supported in $\cup_{J=1}^{l'} B_{3s_J}(\xi_J)$, we have
 \begin{eqnarray*}
(C_{21}) &  \leq &  C \left(\int_{\Omega}  |\nabla  ( \varphi \rho \phi \hat{\psi} |\nabla \varphi|^{\delta} )|^{q} d\mu \right)^{\frac{1}{q}} \cdot \left( \int_{\Omega}   | \nabla \nabla V|^p d\mu \right)^{\frac{1}{p}} \\
& \leq & C\left(\int_{\Omega}  |\nabla ( \varphi \rho \phi |\nabla \varphi|^{\delta} )|^{q} d\mu \right)^{\frac{1}{q}}  \\
& \leq & C\left(\sum_{J=1}^{l} r_J^{n-q-q\delta} + \sup_{\Omega} |\nabla \varphi|^{\delta q} \sum_{J=1}^{l'} s_J^{n-q} \right)^{\frac{1}{q}}\\
& \leq & C\varepsilon^{\frac{1}{q}}
\end{eqnarray*}
by (\ref{choiceD}), (\ref{d}) and (\ref{whoops}). 
Furthermore, by the order gap of $v$ (cf. Proposition~\ref{ahmresult}),
\begin{equation} \label{star}
\sup_{B_{3r_J}(x_J)} |\nabla v| \leq Cr_J^{\epsilon_0}.
\end{equation}
Next, note that by the lower bound (\ref{lowerbd}) of $|\nabla V|$ and the Lipschitz estimate of $u$, we have in $\mathcal N$ the estimates
\[
\left| |\nabla u| - |\nabla V| \right| \leq \frac{\left| |\nabla u|^2 - |\nabla V|^2 \right|}{|\nabla u| + |\nabla V|} \leq \frac{\left| |\nabla v|^2 +2<\nabla V, \nabla v>  \right|}{|\nabla u| + |\nabla V|} \leq C|\nabla v|.
\]
and
\[
\left| \nabla  \left( |\nabla u|-|\nabla V| \right) \right| \leq C\left| \nabla \nabla u \right|.
\]
Thus, by (\ref{choiceD}), (\ref{d}), (\ref{whoops}) and (\ref{prophat}),

\begin{eqnarray*}
(C_{22}) & \leq & C\int_{\Omega}   | \triangle ( \varphi \rho \phi|\nabla \varphi|^{\delta}) |   |\nabla v|  d\mu \leq C \sum_J r_J^{n-2-\delta+\epsilon_0}
\end{eqnarray*}
and
\begin{eqnarray*}
(C_{23}) & = &  \int_{\Omega}  \varphi \rho \phi |\nabla \varphi|^{\delta} \triangle  \hat{\psi} \left( |\nabla u|-|\nabla V| \right) d\mu
\\
& = & -\int_{\Omega}  \left( |\nabla u|-|\nabla V| \right) \nabla ( \varphi \rho \phi |\nabla \varphi|^{\delta}) \cdot \nabla  \hat{\psi}  d\mu\\
& & \ -\int_{\Omega} \varphi \rho \phi |\nabla \varphi|^{\delta} \nabla  \hat{\psi} \cdot \nabla  \left( |\nabla u|-|\nabla V| \right) d\mu\\
& \leq &
C\left( \int_{\Omega} |\nabla ( \varphi \rho \phi  |\nabla \varphi|^{\delta} )|^p |\nabla v|  d\mu \right)^{\frac{1}{p}}  \left( \int_{\Omega}   | \nabla \hat{\psi}|^q |\nabla v|  d\mu \right)^{\frac{1}{q}} \\
& & \ +C \int_{\Omega}  |\nabla \varphi|^{\delta}| \nabla  \hat{\psi} | |\nabla \nabla u| d\mu\\
& < & C \epsilon.
\end{eqnarray*}
%Therefore,
% Using a similar argument along with the estimate (\ref{prophat}), we also obtain
%\begin{eqnarray*}
%(C_{23}) & \leq & C\left( \int_{\Omega} |\nabla ( \varphi \rho \phi  |\nabla \varphi|^{\delta} )|^p |\nabla v|  d\mu \right)^{\frac{1}{p}} \\
%&\leq & C\left( \sum_{J=1}^l r_J^{n-p-p \delta +\epsilon_0} +  \sup_{\Omega} |\nabla \varphi|^{\delta p} \sum_{J=1}^{l'} s_J^{n-p +\epsilon_0} \right)^{\frac{1}{p}} \\
%&<& C\varepsilon^{\frac{1}{p}}
%\end{eqnarray*}
%and the same way
% \[
%(C_{22}) <   C\varepsilon^{\frac{1}{p}}.
%\]
Finally, (\ref{prophat}) also yields
 \begin{eqnarray*}
(C_{24})&  \leq &
 2 \sup_{\Omega_1} |\nabla ( \varphi \rho \phi|\nabla \varphi|^{\delta}) |    \int_{\Omega} | \nabla   \hat{\psi}|  |\nabla u| d\mu <C \varepsilon\end{eqnarray*}
Combining the estimates for $(C_1)$, $(C_{21})$, $(C_{22})$, $(C_{23})$ and $(C_{24})$, we obtain
$
(C) \leq C\varepsilon^{\frac{1}{p}}.
$

The estimate for $(B)$ is similar to the estimate for $(C)$, so we omit the details.   Indeed, to prove $(B)$ we can  repeat the argument for $(C)$ with $\delta=0$ while keeping in mind that
\begin{equation} \label{circle}
\sup_{B_{2s_J}(\xi_J)} |\nabla u| \leq Cs_J^{\epsilon_0}
\end{equation}
by  the order gap of $u$ of Proposition~\ref{ahmresult} along with the monotonicity property of $u$ (cf.~proof of \cite[Theorem 2.4]{gromov-schoen}).
Note that  the argument for $(B)$ is simpler than that for $(C)$; indeed, we can use the decay of $|\nabla u|$ in $B_{2s_J}(\zeta_J)$ by (\ref{circle}) for $(B)$  whereas   $|\nabla u|$ is only bounded in $B_{2r_J}(x_J)$ for (C).
Applying the estimates for $(A)$, $(B)$, $(C)$ into  (\ref{ABC}) 
%we obtain
%\[
%\int_{\Omega} |\nabla \nabla u| |\nabla \psi| d\mu \leq C\varepsilon^{\frac{1}{p}}
%\]
%which 
proves (\ref{theharderone}).
%The inequalities (\ref{thesimpleone'}), (\ref{thesimpleone}) and (\ref{theharderone}) show that  {\sc Statement 2}$[j]$ holds.
\end{proof}

\begin{proofRT}
Proposition~\ref{hmco2} implies that 
${\mathcal S}_j(u)$  has codimension at least 2.
Combined with Proposition~\ref{highorderpts}, this implies that 
\[
\dim_{ \mathcal H}(\hat {\mathcal S}_j(u)) \leq n-2.
\]
Now {\sc Statement 1}$[j]$ follows immediately.    Additionally, {\sc Statement 2}$[j]$ follows from Lemma~\ref{proof2j}.  
Thus, induction completes the proof of Theorem~\ref{RegularityTheorem} and Theorem~\ref{goto0}. 
\end{proofRT}

\section{Proof of the Key Technical Lemma}  \label{sec:  proofofkeylemma}
In this section, we will provide a proof of the key technical Lemma~\ref{keylemma'} by deducing it  from the iterative Lemma~\ref{induction}. We will take advantage of the fact that a harmonic map (resp. approximate harmonic map) at an order one point is closely approximated by homogeneous degree 1 maps as indicated in Lemma~\ref{linearapproxblowupsym}  (resp.  proof of Proposition~\ref{linearapproxblowupsym'}, formula (\ref{frank})).       
We employ an iterative argument which has its origin in \cite{gromov-schoen}, but with serious additional complications due to the non-local compactness and degenerating geometry of the Teichm\"uller space near its boundary.   In Section~\ref{simple g-s}, we motivate our proof of the iterative Lemma~\ref{induction} by explaining its origin in the  Gromov-Schoen regularity theorem. We do so by providing a short proof of the Lemma   for the simple case of maps into  a $k$-pod considered in Example 1 of the introduction.   The preparation of the proof of the key technical Lemma is given in
Section~\ref{sec:EC} and Section~\ref{sec:ER} where we summarize our results from \cite{daskal-meseER} needed in the proof.
 The main step in the proof of the iterative Lemma is presented in Section~\ref{sec:inductivelemma}.  Finally, the proof of the key technical Lemma~\ref{keylemma'} is given in Section~\ref{sec:theproof}.

\subsection{Simple Gromov-Schoen}\label{simple g-s} In order to motivate the proof of  the iterative Lemmas~\ref{induction'} and~\ref{induction} 
we will now sketch an argument  due to Gromov-Schoen for harmonic maps in the simple case where the target  is a finite tree as in Example 1 of the introduction.   As we will see later, iterative Lemma~\ref{induction} is a more complex version of the above argument.

\begin{quote} Let  $X$  be a $k$-pod  formed by  $k$ distinct copies $E_1, \dots,  E_k$ of the half-line  $[0,\infty)$   identified at 0 as in Example 1 of the introduction and $X_0 =E_1 \cup E_2$ be a totally geodesic subspace isometric to $\R$.  For a harmonic map $u:B_1(0)\subset \R^n  \rightarrow X$  and a homogeneous degree 1 map $l:B_1(0) \rightarrow X$  as in (\ref{alphaxb}) effectively contained in the essentially regular totally geodesic subspace $X_0 \simeq \R$,   assume  that $u(0)=l(0)$ and 
\begin{equation} \label{setup}
\sup_{x \in B_1(0)} d(u(x), l(x)) < D.
\end{equation}
\end{quote}

Given the above set up, the idea is to show that there exists  $\theta \in (0,\frac{1}{2}]$ such that if  an affine map \[
_il:B_{\theta^i}(0) \rightarrow X_0
\]
is  ``close" to $u$ in a small ball in the sense that,  
%then we can find a  new affine map
% \[
% _{i+1}l:B_{\theta^{i+1}}(0) \rightarrow X_0
% \]
%that is ``closer" to   $u$  in a smaller ball. 
%To find $_{i+1}l$, consider the  harmonic function $v:  B_{\theta^i}(0) \rightarrow X_0 \approx \R$ with boundary condition $\pi \circ u$ where $
%\pi:   X \rightarrow X_0
%$ is the projection map.  We then have the following observations:  (i) Since $X_0$ is essentially regular, $v$ has a ``good" linear approximation  $_{i+1}l$.  (ii) Since $l$ is effectively contained in $X_0$ and approximates $u$, then maps $u$ and $v$ are ``close."    From (i) and (ii), we can prove the following implication: If
\begin{equation} \label{assumptionGS}
%\mbox{given }\ \ 
 \left\{
\begin{array}{lll}
\displaystyle{\sup_{x \in B_{\theta^i}(0)} d(u(x),l(x))}& < &  \theta^i {\delta_0}
\\
 \displaystyle{\sup_{x \in B_{\theta^i}(0)} d(u(x), { _il}(x))} & < & \theta^i d_0,
\end{array}
\right.
\ \ \ \ \ \ \ \ \ 
\end{equation}
then
then there exists  a new affine map 
 \[
 _{i+1}l:B_{\theta^{i+1}}(0) \rightarrow X_0
 \]
such that
\begin{equation} \label{conclusionGS}
 %\ \ \ \ \ \ \ \ \ \  \mbox{we then have} \   \ \ \ \ \ 
 \left\{ \begin{array}{lll}
\displaystyle{\sup_{x \in B_{\theta^{i+1}}(0)} d(u(x),l(x))}& < & \theta^{i+1} \left( {\delta_0} + 2\theta^{-1}d_0 \right) 
\\
\displaystyle{\sup_{x \in B_{\theta^{i+1}}(0)} d(u(x), { _{i+1}l}(x))} & < & \theta^{i+1} \displaystyle{\frac{d_0}{2}}.
\end{array}
\right.
\end{equation}

\begin{proofofimplication}
 Since \emph{$l$ is effectively contained in $X_0$},  for $\epsilon>0$ to be chosen later,  there exists $\delta>0$ such that (cf.~(\ref{exeffct})) 
\[
%E:=
\mbox{Vol} \big(\{ x \in B_{\theta^i}(0):  {\bf B}_{\delta \theta^i}(l(x)) \cap (X \backslash X_0) \neq \emptyset \} \big) <\epsilon \theta^{in},
% \ \ \Rightarrow \ \ \mbox{Vol}(E) < \epsilon  \theta^{in}. 
\] 
thus, there exists $R \in [\frac{3}{4},1]$ such that
\[
\mbox{Vol}\big( \{ x \in \partial B_{R\theta^i}(0):  {\bf B}_{\delta \theta^i}(l(x)) \cap (X \backslash X_0) \neq \emptyset \} \big) < 4\epsilon \theta^{i(n-1)}.
\]
(Note that ${\bf B}$ denotes the ball in the target). For $x \in \partial B_{R\theta^i}(0)$, the first inequality of  assumption (\ref{assumptionGS}), implies
\[
{\bf B}_{\theta^i {\delta_0}}(l(x)) \cap (X \backslash X_0) = \emptyset 
\ \ \Rightarrow \ \ u(x) \in X_0 \ \  \Rightarrow  \ \  \pi \circ u(x) = u(x) .
\]
Thus,
\[
\mbox{Vol}\big( \{ x \in \partial B_{R\theta^i}(0):  \pi \circ u(x) \neq u(x) \} \big) < 4\epsilon \theta^{i(n-1)}.
\]
Consider the  harmonic function $v:  B_{R\theta^i}(0) \rightarrow X_0 \approx \R$ with boundary condition $\pi \circ u$ on $\partial B_{R\theta^i}(0)$.
Using the fact that the image of $_il$ is contained in $X_0$, $\pi $ is the closest point projection and  the second inequality of assumption~(\ref{assumptionGS}),
\[
d(u, v) \leq d(u, {_il}) < \theta^i d_0 \ \ \mbox{on} \ \partial B_{R\theta^i}(0).
 \]  Thus,  
%$d(u(x),v(x))=0$ whenever $x \notin E$ and 
%\[
%\int_{\frac{3}{4}}^1 \left( \int_{\partial B_r(0)} d(u,v) \, d\Sigma \right) \, dr \leq  \int_{B_{\theta^i} (0)} d(u,v) \, d\mu \leq   \int_{B_{\theta^i} (0) \cap E} d(u,v) \, d\mu
%\]
%Since $d(u,v) \leq d(u,X_0) \leq d(u,{_il})<$
% there exists $R \in [\frac{3}{4},1]$ such that 
\[
\int_{\partial {B}_{R\theta^i}(0)} d(u,v) d\Sigma<4 \epsilon \theta^{i(n-1)} \theta^id_0.
\]
Since $v$ is minimizing and $\pi$ is distance non-increasing,  $E^v \leq E^{\pi \circ u} \leq E^u$.  
Since \emph{$X_0$ is essentially regular}, there exists  $_{i+1}l$ such that (cf.~(\ref{exessreg}))
\begin{equation} \label{affinemap}
\sup_{x \in B_{\theta^{i+1}}(0)} d(v(x),{_{i+1}l}(x)) \leq C\theta^2 \sup_{x \in B_{\frac{\theta^i}{2}}(0)} d(v(x),{_il}(x))%\sup_{x \in B_{\theta^i}(0)} d(v(x),l(x)), \ \forall \sigma \in (0,1).
\end{equation}
where $C>1$  depends only on $E^u$. 
Since $u$ and $v$ are harmoninc maps,   $d(u,v)$ is subharmonic.  Thus, for a constant $c_n>0$ depending only on the domain dimension $n$, 
\[
\sup_{x \in B_{\frac{\theta^{i}}{2}}(0)} d(u(x),v(x)) <c_n \epsilon  \theta^i d_0.
\]
 The triangle inequality then implies
\begin{eqnarray*}
\sup_{x \in B_{\theta^{i+1}}(0)} d(u(x),{_{i+1}l}(x)) 
& \leq &  
\sup_{x \in B_{\theta^{i+1}}(0)} d(u(x),v(x))+ \sup_{x \in B_{\theta^{i+1}}(0)} d(v(x),{_{i+1}l}(x)))\\
& < & c_n \epsilon \theta^i d_0  + C\theta^2 \sup_{x \in B_{\frac{\theta^i}{2}}(0)} d(v(x),{_il}(x))
\\
& < & c_n \epsilon \theta^i d_0  + C\theta^2 \sup_{x \in B_{\frac{\theta^i}{2}}(0)} \Big(d(v(x),u(x))+ d(u(x),{_il}(x))\Big
)\\
& < &  c_n \epsilon \theta^i d_0 + C\theta^2(c_n\epsilon \theta^id_0+ \theta^i d_0).
\end{eqnarray*}
Thus, by choosing $\theta = \frac{1}{4C}$, $\epsilon = \frac{\theta}{4c_n}$, we obtain 
\[
\sup_{x \in B_{\theta^{i+1}}(0)} d(u(x),{_{i+1}l}(x)) < \theta^{i+1}\displaystyle{\frac{d_0}{2}}.
\]
This proves the second inequality of (\ref{conclusionGS}).

By assumption (\ref{assumptionGS})
and the triangle inequality
\[
\sup_{x \in B_{\theta^i}(0)} d(l(x),{_{i}l}(x))\leq  \sup_{x \in B_{\theta^i}(0)} (d(l(x),u(x)) +d(u(x),{_{i}l}(x)) )< \theta^i (\delta_0+d_0)
\]
Furthermore, the assumption that  $l(0)=u(0)$ implies
\[
d(l(0),{_{i}l}(0))=d(l(0),u(0))+d(u(0),{_{i}l}(0)) < \theta^i d_0
\]
By the linearity of $l$ and $_il$, we thus conclude
\begin{eqnarray*}
\sup_{x \in B_{\theta^{i+1}}(0)} d(l(x),{_{i}l}(x)) &=&\sup_{x \in B_{\theta^i}(0)} d(l(\theta x),{_{i}l}(\theta x))\\
&\leq& \theta\sup_{x \in B_{\theta^i}(0)} d(l( x),{_{i}l}( x))+(1-\theta)d(l(0),{_{i}l}(0))\\
&<&  \theta^{i+1} (\delta_0+d_0)+(1-\theta)\theta^i d_0\\
&=& \theta^{i+1} (\delta_0+\theta^{-1} d_0).
\end{eqnarray*}
Combining this with assumption~(\ref{assumptionGS}), we obtain
\begin{eqnarray*}
\sup_{x \in B_{\theta^{i+1}}(0)} d(u(x),l(x))& < & \theta^{i+1} \left( {\delta_0} + 2\theta^{-1}d_0 \right).
\end{eqnarray*}
\end{proofofimplication}

%Applying  implication (\ref{assumptionGS}) $\Rightarrow$ (\ref{conclusionGS}) inductively with $_0l=l$ and $\delta_0=D=d_0$ at the $0^{th}$ step and with $d_0=\frac{D}{2^i}$ and $\delta_0=D+\theta^{-1}\frac{D}{2^{i-1}}$ at the $i^{th}$ step, we obtain
%\[
%\left\{ \begin{array}{lll}
%\displaystyle{\sup_{x \in B_{\theta^{i+1}}(0)} d(u(x),l(x))}& < & \theta^{i+1} \left( D + 2\theta^{-1}\displaystyle{\sum_{k=0}^i \frac{D}{2^k}}\right) 
%\\
%\displaystyle{\sup_{x \in B_{\theta^{i+1}}(0)} d(u(x), { _{i+1}l}(x))} & < & \theta^{i+1} \displaystyle{\frac{D}{2^{i+1}}}.
%\end{array}
%\right.
%\]

\subsection{Effectively Contained} 
\label{sec:EC}

In preparation of the proof of the key technical Lemma,  we  recall the results from  \cite{daskal-meseER}.    First, we introduce a global coordinate system on ${\bf H}$ using symmetric geodesics (cf.~ Section~\ref{sssec:sg}). 

These new coordinates, denoted $(\varrho, \varphi)$, will depend on a given  symmetric homogeneous degree 1 map 
\[
l:B_1(0) \rightarrow {\bf H}, \ \ l(x) = \gamma(Ax^1)
\]
 where $A>0$ is the stretch of $l$ and $\gamma$ is a symmetric geodesic (cf. Definition~\ref{defstretchy} and Section~\ref{sssec:sg}).  We construct the coordinates $(\varrho,\varphi)$ so that if we write $\gamma(t)=(\gamma_{\varrho}(t), \gamma_{\varphi}(t))$ with respect to  $(\varrho, \varphi)$, then $\gamma_{\varrho}(t)=t$ and $\gamma_{\varphi}(t)=0$; i.e.
 \[
\gamma(t)=(t,0).
\] 
Thus,  $l(x)$ with respect to coordinates $(\varrho,\varphi)$ is given by
\begin{equation} \label{pout}
l(x)=(Ax^1,0).
\end{equation}
There is an advantage in using the coordinates $(\varrho,\varphi)$.  Indeed, since harmonic maps into $\overline{\bf H}$ and the singular components of maps of harmonic maps into $\overline{\mathcal T}$ at order 1 points are well approximated by symmetric  homogeneous degree 1 maps, the coordinates $(\varrho,\varphi)$ are the most convenient when analyzing the behavior of such maps. 

In the sequel, we will need to consider several symmetric homogeneous degree 1 maps at once.  Thus, we first introduce new coordinates $(s,t)$ such that we can associate a  constant $t_*$ to any symmetric homogeneous degree 1 map $l(x)$  such that with respect to  coordinates $(s,t)$
\begin{equation} \label{pouty}
l(x)=(Ax^1,t_*).  
\end{equation}
We refer to the number $t_*$
as the \emph{address of $l$}.
Once we fix a particular symmetric homogeneous degree 1 map $l(x)$, then we apply a simple translation in the $t$-coordinate which results in new coordinates $(\varrho,\varphi)$ with respect to which  $l$ is expressed by (\ref{pout}).

To construct the coordinates $(s,t)$, we foliate ${\bf H}$ by an one parameter family of symmetric geodesics.  Indeed, consider 
 \begin{equation} \label{defofc}
 c=(c_{\rho},c_{\phi}):(-\infty,\infty) \times (-\infty,\frac{3}{2})\rightarrow {\bf H}
 \end{equation}
satisfying the following:  
\begin{eqnarray}
& \bullet  & \mbox{$\varrho \mapsto c^t(\varrho)=c(\varrho,\varphi)$ is a unit speed symmetric geodesic}.  \label{unitspeed}\\
& \bullet  &   t \mapsto  c_{\rho}(0,t) \mbox{ satisfies the equation } \frac{\partial c_{\rho}}{\partial t}(0,t)=c_{\rho}^3(0,t), \label{at1}\\
& \bullet  & c_{\rho}(0,1)=1 \mbox{ and } c_{\phi}(0,t)=0 \mbox{ for all } t \in (-\infty,\frac{3}{2}), \label{normalizeit}
\end{eqnarray}

The parameters $s$ and $t$ define coordinates of ${\bf H}$ via the map
\[
(s,t) \mapsto c(s,t).
\]

Given a symmetric homogeneous degree 1 map $l(x)$ with address $t_*$ (cf.~(\ref{pouty})),
we apply a translation by $t_*$ to construct $(\varrho,\varphi)$ (see Figure 2). 
More precisely, since 
\begin{equation} \label{atstar}
l(0)=(0,t_*) \ \mbox{in the coordinates $(s,t)$},
\end{equation} 
we  define  coordinates $(\varrho,\varphi)$ by setting 
\begin{equation} \label{cov}
(\varrho, \varphi)=(s,t-t_*).
\end{equation}
This results in
\[
l(0)=(0,0) \ \mbox{in coordinates $(\varrho,\varphi)$}.
\]
Thus,  the construction of the coordinates  $(\varrho,\varphi)$ depends on $t_*$, and we will say that the coordinates $(\varrho,\varphi)$ are \emph{anchored} at $t_*$.
Using the new coordinates $(\varrho,\varphi)$, we introduce a family of  totally geodesic subspaces of ${\bf H}$ which will play a central role in the proof of the key technical Lemma. 
%\begin{figure}[h]
%    \centering
%    \includegraphics[width=\textwidth]{g2.pdf}
%    \caption{The subspace $\overline{\bf H}[\varphi_0,t_*] $ in the coordinates $(\varrho, \varphi)$} \label{pic:varrhovarphi}
%\end{figure}

\begin{definition} \label{essregset}
Let $(\varrho,\varphi)$ be the coordinates anchored at $t_*$.  For $\varphi_0>0$, define the subset 
\[
\overline{\bf H}[\varphi_0,t_*] :=\{ (\varrho, \varphi) \in {\bf H}:  |\varphi|\leq \varphi_0\}.
\] 
Furthermore, let
\begin{equation} \label{defofa}
a[\varphi_0,t_*]:= c_{\rho}(0,\varphi_0 +t_*)=\max_{\{\varphi:  |\varphi|\leq \varphi_0\}} c_{\rho}(0,\varphi+t_*).
\end{equation}
\end{definition}

In other words,   $\overline{\bf H}[\varphi_0,t_*] $ is the union of the level sets $\varphi=k$ where $-\varphi_0 \leq k \leq \varphi_0$, and the level set $\varphi=k$ is the image of symmetric geodesic 
\[
\varrho \mapsto c(\varrho,k+t_*).
\]
The boundary of $\overline{\bf H}[\varphi_0,t_*] $ consists of a pair of  level sets $\varphi=\varphi_0$ and $\varphi=-\varphi_0$, and  the set $\overline{\bf H}[\varphi_0,t_*]  $ is totally geodesic.   
Moreover, $a[\varphi_0,t_*]$ is the distance from $P_0$  of the symmetric geodesic in $\overline{\bf H}[\varphi_0,t_*] $ furthest away from $P_0$. See Figure \ref{pic:varrhovarphi}.
Define the function ${\mathcal J}(\varrho,\varphi)$ by  writing the metric $g_{\bf H}$ with respect to coordinates $(\varrho,\varphi)$   as
\begin{eqnarray} \label{wrtnew}
g_{\bf H} = d\varrho^2+ {\mathcal J}(\varrho,\varphi) d\varphi^2
\end{eqnarray}
As observed in  \cite{daskal-meseER}, this local expression of $g_{\bf H}$ with respect to $(\varrho,\varphi)$ is close to the local expression  $g_{{\bf H}}=d\rho^2+ \rho^6 d\phi^2$ with respect to $(\rho,\phi)$. More precisely, there exists a constant $C>0$ such that 
\begin{equation} \label{a2}
\varrho^3 \leq {\mathcal J}(\varrho,\varphi) \leq C(\varrho+c_{\rho}(0,\varphi+t_*))^3.
\end{equation}
In particular,
\[
\varrho^3 \leq {\mathcal J}(\varrho,\varphi) \leq C(\varrho+a[\varphi_0,t_*])^3 \mbox{ for } (\varrho,\varphi) \in \overline{\bf H}[\varphi_0,t_*].
\]

The following lemma plays the role  for a homogeneous map to be \emph{effectively contained} in a totally geodesic subspace. The proof  is contained in \cite{daskal-meseER} but since it is simple geometric argument  we include it here for the sake of completeness.

  \begin{lemma}\label{geomcor}
Fix $\theta \in (0,\frac{1}{24})$.  Given $A>0$, $\epsilon_0>0$ $D_0 \in (0,\frac{\epsilon_0}{2})$ and  $i \in \{0, 1,2,\dots\}$,  if
 \[
 _il:B_{\theta^i}(0) \rightarrow \overline{\bf H}[
\left( \frac{\theta^i \epsilon_0}{2}\right)^{-3}  \frac{\theta^iD_0}{2^i},t_*]
 \]
 and 
 \[
 v:B_{\theta^i}(0) \rightarrow \overline{\bf H}
 \]
 satisfies
\begin{equation}  \label{agcor2}
\sup_{B_{\theta^i}(0)} |v_{\varrho}-Ax^1| < \theta^i \epsilon_0
\end{equation}
and
\begin{equation} \label{geomcorassump}
\sup_{B_{\theta^i}(0)} d_{\overline{\bf H}}(v, {_il}) <\frac{\theta^i  D_0}{2^i},
\end{equation}  
then
\[
\frac{1}{2v_{n-1}}Vol \left\{ x \in B_{\theta^i}(0):  v(x) \notin \overline{\bf H}[2
\left( \frac{\theta^i \epsilon_0}{2}\right)^{-3} \frac{\theta^iD_0}{2^i},t_*]
\right\} <\theta^{in}\frac{2\epsilon_0}{A}
\]
where $Vol$ is the volume with respect to Euclidean metric and  $v_{n-1}$ denotes the Euclidean volume of the unit $(n-1)$-dimensional ball.
\end{lemma}

\begin{proof}  
We start with the following claim.\\
\\
{\sc Claim}.
For $\delta_0 <\frac{\varepsilon}{2}$,
\[
d_{\overline{\bf H}}((\varrho,\varphi),\overline{\bf H}[
\varepsilon^{-3} \delta_0,t_*]) \leq \delta_0
\ \Rightarrow \ 
 |\varrho|  \leq 2\varepsilon    \ \mbox{ or } \ (\varrho,\varphi) \in \overline{\bf H}[
2\varepsilon^{-3} \delta_0,t_*].
\]  
To prove the claim,  assume on the contrary that 
there exists $(\varrho,\varphi)$ with 
\[
d_{\overline{\bf H}}((\varrho,\varphi),\overline{\bf H}[
\varepsilon^{-3} \delta_0,t_*]) \leq \delta_0, \    \ |\varrho|  \geq 2 \varepsilon \ \  
\mbox{ and }  \  
(\varrho,\varphi) \notin \overline{\bf H}[\
2\varepsilon^{-3} \delta_0,t_*].
\]
Let $\gamma=(\gamma_{\varrho}, \gamma_{\varphi}):[0,1] \rightarrow \overline{\bf H}$ be a geodesic with
\[
 \gamma(0)= (\gamma_{\varrho}(0), \gamma_{\varphi}(0))=(\varrho,\varphi) \ \mbox{ and } \  \gamma(1) \in  \partial \overline{\bf H}[
\varepsilon^{-3} \delta_0,t_*]
\]
 where $\gamma(1)$ is the point in $\overline{\bf H}[
\varepsilon^{-3} \delta_0,t_*]$ closest to $(\varrho,\varphi)$.
  We claim
\begin{equation} \label{lowerbdonepsilonnot}
 \min_{t \in [0,1]} |\gamma_{\varrho}(t)| \geq \varepsilon.
 \end{equation}
 Indeed, assume on the contrary that $\gamma_{\varrho}(t_0) < \varepsilon$ for some $t_0 \in (0,1]$.  Then since $\gamma_{\varrho}(0)\geq 2\varepsilon$, we obtain
 \begin{eqnarray*}
\varepsilon &  < & |\gamma_{\varrho} (t_0) - \gamma_{\varrho}(0)| \\
& \leq &   \int_0^{t_0} \left|\frac{d{\gamma_{\varrho}}}{dt}\right| dt \leq  \int_0^{t_0} \left|\frac{d{\gamma}}{dt}\right| dt \\
& \leq & d_{\overline{\bf H}}((\varrho,\varphi), \overline{\bf H}[
\varepsilon^{-3} \delta_0],t_*) \leq \delta_0.
 \end{eqnarray*}
 This  contradicts the assumption that $\delta_0 <\frac{\varepsilon}{2}$ and proves (\ref{lowerbdonepsilonnot}).  
Combined with  (\ref{a2}), we conclude
\[
\varepsilon ^3 \leq  {\mathcal J}(\gamma(t)). 
\]
Therefore
  \begin{eqnarray*}
\varepsilon^3 \left||\varphi| -\varepsilon^{-3} \delta_0 \right| & \leq  & \varepsilon^3 \int_0^1 \left|\frac{d\gamma_{\varphi}}{dt}(t) \right| dt \\
& \leq  & \int_0^1 \sqrt{{\mathcal J}(\gamma(t)) \left|\frac{d\gamma_{\varphi}}{dt}(t)\right|^2} \ dt \\
& \leq  & \int_0^1 \sqrt{\left|\frac{d\gamma_{\varrho}}{dt}(t)\right|^2+{\mathcal J}(\gamma(t)) \left|\frac{d\gamma_{\varphi}}{dt}(t) \right|^2} \ dt \\
& = & \mbox{length}(\gamma)\\
& = & d_{\overline{\bf H}}((\varrho,\varphi), \gamma(1))\\
& \leq &  \delta_0
\end{eqnarray*}
which in turn implies
\[
|\varphi| \leq 2 \varepsilon^{-3} \delta_0, 
\]
In other words, 
\[
(\varrho,\varphi) \in  \overline{\bf H}[\
2\varepsilon^{-3} \delta_0,t_*].
\]
This contradiction proves the {\sc Claim}.

  Since  $_il(x)  \in\overline{\bf H}[
\left( \frac{\theta^i \epsilon_0}{2}\right)^{-3} \frac{\theta^iD_0}{2^i},t_*]$,  assumption   (\ref{geomcorassump}) implies that we have for $x \in B_{\theta^i}(0)$ 
\[
d_{\overline{\bf H}}(v(x),\overline{\bf H}[
\left( \frac{\theta^i \epsilon_0}{2}\right)^{-3} \frac{\theta^iD_0}{2^i},t_*] )\leq \sup_{B_{\theta^i}(0)} d_{\overline{\bf H}}(v, {_il}) <\frac{\theta^i  D_0}{2^i} . 
\] 
Thus,  applying the {\sc Claim} with 
\[
\varepsilon= \frac{\theta^i \epsilon_0}{2}
\ \mbox{ and } \ 
\delta_0 = \frac{\theta^i D_0}{2^i}
\]
implies that
\begin{eqnarray*}
\lefteqn{\left\{ x \in B_{\theta^i}(0):  v(x) \notin \overline{\bf H}[2
\left( \frac{\theta^i \epsilon_0}{2}\right)^{-3} \frac{\theta^iD_0}{2^i},t_*]
\right\}}\\
 & & \ \ \ \ \subset   \{ x \in B_{\theta^i}(0):  |v_{\varrho}(x)| \leq \theta^i\epsilon_0\}.
\end{eqnarray*}
Furthermore,  assumption (\ref{agcor2}) implies
\[
|v_{\varrho}(x)| \leq  \theta^i\epsilon_0
\ \ \Rightarrow \ \ 
|Ax^1| \leq |Ax^1-v_{\varrho}(x)| +|v_{\varrho}(x)|< 2\theta^i\epsilon_0 
\]
in $B_{\theta^i}(0)$.  Hence
\[
\{ x \in B_{\theta^i}(0):  |v_{\varrho}(x)|  \leq  \theta^i\epsilon_0\}
\subset
\{x \in B_{\theta^i}(0): |Ax^1|  <   2\theta^i \epsilon_0 \}.
\]
The assertion now follows from the fact that
\[
\frac{1}{2v_{n-1}}Vol \{x \in B_{\theta^i}(0): |Ax^1|  <   2\theta^i\epsilon_0 \} \leq \theta^{in} \frac{2\epsilon_0}{A}.
\]
\end{proof}

\subsection{Essentially Regular Subspaces}\label{sec:ER}
Now we turn to the notion of {\it{essentially regular}}.  We assert that 
 the  totally geodesic subspace $\overline{\bf H}[\varphi_0,t_*]$  of $\overline{\bf H}$
is  essentially regular in the  sense  that a harmonic map into $\overline{\bf H}[\varphi_0,t_*]$ is approximated  by an almost affine map.  
We first need the following
\begin{definition} \label{almostaffine}
Let $(\varrho,\varphi)$ be the coordinates anchored at $t_*$ defined in the previous section.    A map $l=(l_{\varrho},l_{\varphi}): B_1(0) \rightarrow {\bf H}$ written with respect to coordinates $(\varrho,\varphi)$, is said to be an \emph{almost affine map} if the first coordinate function $l_{\varrho}$ is an affine function; i.e.  
\[
l_{\varrho}(x)=a\cdot x +b
\]
 for $a \in \R^n$ and $b \in \R$.
\end{definition}

We have so far been  unable to prove  that these subspaces are  essentially regular in the strict sense of Gromov-Schoen \cite{gromov-schoen}.    (We remark that, as far as we  know, Euclidean spaces and buildings are the only known examples of  essentially regular sets in the strict sense of 
\cite{gromov-schoen}.)   
On the other hand,  $\overline{\bf H}[\varphi_0,t_*]$  satisfies a weaker notion of essentially regular that is sufficient  for obtaining good estimates for harmonic maps.  For convenience, we will \emph{also call this weaker notion essentially regular}.
Given that the 
the local geometry of ${\bf H}$  is very singular near $P_0$,  it is  surprising that essentially regular subspaces near the point $P_0$ exist at all.

The key is the introduction of different set of new coordinates in ${\bf H}$ that are motived by Example 2. Specifically,  we let
 \begin{equation} \label{polar}
\Upsilon:=\varrho-\frac{3}{2} \varrho^5 \varphi^2
 \ \mbox{ and } \ \Phi:= \varrho^3 \varphi.
\end{equation}
To explain the relationship  of the new  coordinates $(\Upsilon,\Phi)$ to the Euclidean coordinates $(x,y)$ in Example 2, we first  consider  $(\varrho, \varrho^2 \varphi)$  as the analogue of the polar coordinates $(r,\theta)$ of ${\bf R}^2$.  Then the coordinates
\[
(\varrho,\varphi) \mapsto (\varrho \cos \sqrt 3 \varrho^2 \varphi , \varrho \sin \sqrt 3 \varrho^2 \varphi )
\]
are the
analogues of the standard Euclidean coordinates (\ref{ec1}).  The coordinates $\Upsilon$ and $\sqrt 3 \Phi$ agree up to the first order with $\varrho \cos \sqrt  3 \varrho^2 \varphi$  and $\varrho \sin \sqrt  3 \varrho^2 \varphi$ respectively.  We then write %$u=(u_{\Upsilon}, u_{\Phi})$ in terms of coordinates $(\Upsilon,\Phi)$ 
the harmonic map equations in terms of the coordinates $(\Upsilon,\Phi)$ to obtain the regularity results needed.  
An important observation about Example~2 is the implicit use of the  assumption  $0 \leq h_{\theta}<2\pi$.  (We need this assumption in order to show that the change of variables defines a diffeomorphism away from the origin). In fact, without assuming this bound, it is unclear whether  the solutions to (\ref{he1}) are regular.  
For a harmonic $u: \Omega \rightarrow  \overline{\bf H}[\varphi_0,t_*]$, we are also assuming  an apriori bound on the ``angular'' component function.  For a harmonic $u: \Omega \rightarrow  \overline{\bf H}[\varphi_0,t_*]$, we are also assuming  an apriori bound on the ``angular'' component function   by virtue of the definition of the target set.  This bound is precisely why we are able to  use  $\overline{\bf H}[\varphi_0,t_*]$ as the analog of essentially regular sets (cf.~\cite[page 210]{gromov-schoen}) when we generalize the Gromov-Schoen argument in Section~\ref{sec:inductivelemma} below.  Indeed, the following $(1+\alpha)$-Taylor approximation of a harmonic map into $\overline{\bf H}[\varphi_0,t_*]$ is proved in  \cite{daskal-meseER}, Theorem 28. 

\begin{theorem}  \label{essreg**}
Let $R \in [\frac{1}{2},1)$,  $E_0>0$, $A_0>0$ and a normalized metric $g$ on $B_R(0)$ be given.  Then there exist $C \geq 1$ and $\alpha>0$ depending only on $E_0$, $A_0$ and $g$ with the following property:  \\
\\
For $\varphi_0>0$, $s \in (0,1]$ and  $\vartheta \in (0,1]$, if  ${\bf B}_{A_0\vartheta}(P_0)$ is a geodesic ball of radius $A_0\vartheta$ centered at $P_0$ in $\overline{\bf H}$,  if 
\[
w:(B_{\vartheta R}(0),g_s) \rightarrow \overline{\bf H}[\frac{\varphi_0}{\vartheta^2},t_*] \cap {\bf B}_{A_0\vartheta}(P_0)
\] 
is a harmonic map with
\begin{equation} \label{additionalassumption}
a[\frac{\varphi_0}{\vartheta^2},t_*]\leq \frac{\vartheta}{2}
\end{equation}
and 
\[
E^{w} \leq \vartheta^{n}E_0,
\]
then
\[
\sup_{B_{r \vartheta}(0)} d_{\overline{\bf H}}(w,\hat{l}) \leq Cr^{1+\alpha} \sup_{B_{R \vartheta}(0)} d_{\overline{\bf H}}(w,L) +C r \vartheta \varphi_0^2, \ \forall r \in (0,\frac{R}{2}]
\]
where $\hat{l}=(\hat{l}_{\varrho},\hat{l}_{\varphi}):B_1(0) \rightarrow \overline{\bf H}$ is the almost  affine map 
given by
\[
\hat{l}_{\varrho}(x)= w_{\varrho}(0) +\nabla w_{\varrho}(0) \cdot x, \ \  \ \hat{l}_{\varphi}(x)=w_{\varphi}(x)
\]
and   $L:B_1(0) \rightarrow \overline{\bf H}$    is any almost affine map.
\end{theorem}
\subsection{The statement and proof of the iterative Lemma}
\label{sec:inductivelemma}
In this subsection, we prove the iterative Lemma which allows us to go from an approximation of a harmonic map (resp.~approximate harmonic map) by an almost affine map on  one scale to an approximation on a smaller scale. 
This lemma plays a central role in the proof of the key technical Lemma~\ref{keylemma'}.

Let  $g$ be  a normalized metric on $B_1(0)$ sufficiently close to the Euclidean metric $g_0$ in the sense that if we denote by $Vol$ and  $Vol_{g}$ to be the volume with respect to $g_0$ and  $g$ respectively, then for any smooth submanifold $S$ of $B_1(0)$ 
\begin{equation} \label{volumewrtmetric'}
\frac{15}{16} Vol(S) \leq Vol_{g}(S) \leq \frac{17}{16} Vol (S).
\end{equation}
Additionally, we assume $g$ is sufficiently close to $g_0$ (in $C^2$) so that the error term $e^{c\sigma^2}$ that appears in the monotonicity formula of Theorem~\ref{orderforharmonic} is $\leq 2$ for all $\sigma \in (0,1]$.

Next, let $c_0\geq 1$ be a constant such that for any subharmonic function $f:B_1(0) \rightarrow \R$  with respect to the metric $g$, we have 
\begin{equation} \label{sybharm}
 \sup_{B_{\frac{15\vartheta R}{16}}(0)} f  \leq  \frac{c_0}{(\vartheta R)^{n-1}} \int_{\partial B_{\vartheta R}(0)} f d\Sigma.
\end{equation}

\label{sec:induct}
\begin{keylemma}\label{induction'}
Given $E_0, A>0$ and a normalized metric $g$ on $B_R(0)$,
there exist  $\theta  \in (0,\frac{1}{24})$, $\epsilon_0>0$  and  $D_0 \in (0,\frac{1}{\sqrt{8}})$ that satisfy the following statement.  \\
\\
Assume the following:  
\begin{itemize}
\item The map
\[
l:B_{\theta^i}(0) \rightarrow \overline{\bf H}, 
\ \ l(x)= (Ax^1,0)
\]
is defined  in the coordinates $(\varrho,\varphi)$ anchored at $t_* \in (-\infty,\frac{3}{2})$.

\item  The subset $\overline{\bf H}[2
\left( \frac{\theta^i \epsilon_0}{2}\right)^{-3} \frac{\theta^iD_0}{2^i},t_*]$ satisfies 
\begin{equation} \label{sat'}
a[2
\left( \frac{\theta^i \epsilon_0}{2}\right)^{-3} \frac{\theta^iD_0}{2^i},t_*]=a[\frac{16D_0}{\epsilon_0^3 \theta^{2i} 2^i},t_*
]
< \frac{\theta^i}{2}.
\end{equation}  
\item
The map 
\begin{equation} \label{uish}
u:(B_1(0),g) \rightarrow \overline{\bf H} \ \mbox{ is harmonic with } \ u(0)=P_0, \ E^u(1) \leq \frac{E_0}{2^{n+1}}.
\end{equation}

\item The map
\[
_il:B_{\theta^i}(0) \rightarrow \overline{\bf H}[
\left( \frac{\theta^i \epsilon_0}{2}\right)^{-3} \frac{\theta^iD_0}{2^i},t_*]
\
\mbox{is an almost affine map.}
\] 
\item The constant  $_i\delta>0$ is  such that
\begin{equation} \label{assumptionoflemma'}
\left\{
\begin{array}{l}
{\displaystyle \sup_{B_{^{\theta^i}}(0)}  d_{\overline{\bf H}}(u, {_il}) <\theta^{i}\frac{D_0}{2^i}}\\
{\displaystyle  \sup_{B_{^{\theta^i}}(0)} } |u_{\varrho}-Ax^1|  < \theta^{i}{_i\delta}<\theta^{i}\displaystyle{ \sum_{k=0}^{i} \frac{\theta^{-1}D_0}{2^{k-2}}}.
\end{array}
\right.
\end{equation}
\\
\end{itemize}
Then there exists an almost affine map
\[
 _{i+1}l:B_{\theta^{i+1}}(0) \rightarrow  \overline{\bf H}[
\left( \frac{\theta^{i+1} \epsilon_0}{2}\right)^{-3} \frac{\theta^{i+1}D_0}{2^{i+1}},t_*] 
\]
such that
\begin{equation} \label{conclusionoflemma'}
\left\{
\begin{array}{l}
\displaystyle{\sup_{B_{\theta^{i+1}}(0)}} d_{\overline{\bf H}}(u,{_{i+1}l}) < \theta^{i+1}\frac{D_0}{2^{i+1}}\\
\displaystyle {\sup_{B_{\theta^{i+1}}(0)}  |u_{\varrho} (x)- A x^1|}  < \displaystyle {_{i+1}\delta} \theta^{i+1}\ :=  \left( {_i\delta}  +   \frac{2D_0 \theta^{-1}}{2^i} \right) \theta^{i+1} < \displaystyle{\theta^{i+1} \sum_{k=0}^{i+1} \frac{ \theta^{-1}D_0}{2^{k-2}}} \\
{\displaystyle 
\sup_{B_{\theta^{i+1}}(0)} d_{\overline{\bf H}}(u,l) 
< \theta^{i +1}  \left( \frac{2^3\left( A+9D_0 \right)^3}{\epsilon_0^3} +10 \right)  \theta^{-1}  D_0}.
\end{array}
\right.
\end{equation}
\end{keylemma}

\begin{remark} 
\label{whatweuse}
The harmonicity of the map $u$ implies by the last part of Theorem~\ref{orderforharmonic}, the assumption on the metric $g$, the comment after (\ref{volumewrtmetric'}) and the assumption (\ref{uish}) the following energy decay estimate
\begin{equation} \label{insteadofcalF'}
\frac{E^u(\vartheta)}{\vartheta^n}  \leq e^{c} E^u(1) \leq \frac{E_0}{2^n}. 
\end{equation}
Furthermore, by \cite{korevaar-schoen1} Lemma 2.4.2, 
 for $R \in (0,\frac{7}{8}]$ and a harmonic map $w:(B_{\theta^i R}(0),g)  \rightarrow \overline{\bf H}$ with $E^w(\theta^i R) \leq E^u(\theta^i R)$,
we have for $c_0$ as in (\ref{sybharm})
 \begin{equation} \label{abc!'}
 \sup_{B_{\frac{15\theta^i R}{16}}(0)} d_{\overline{\bf H}}^2(u,w)  \leq  \frac{c_0}{(\theta^i R)^{n-1}} \int_{\partial B_{\theta^i R}(0)} d_{\overline{\bf H}}^2(u,w)d\Sigma
\end{equation}
\end{remark}

\begin{proof}
For the sake of simplicity, we denote $d=d_{\overline{\bf H}}$ throughout the proof.   
For $R= \frac{1}{2}$, $E_0>0$ as in (\ref{uish}) and the metric $g$ as above, let
 \begin{equation} \label{startofconstants'}
C \geq 1 \mbox{ and }  \alpha>0 \mbox{ be as  in Theorem~\ref{essreg**}}.
\end{equation}
Let $\theta \in (0,\min\{\frac{1}{24},\frac{1}{\sqrt{A}} \})$ sufficiently small such that
\begin{equation} \label{ceatheta'}
C\theta<1,
\end{equation}
\begin{equation} \label{th'}
C \theta^{\alpha} <   \frac{1}{2^{6}},
\end{equation}
and
\begin{equation} \label{teanot'}
 C \theta^3   < \frac{1}{2^3}.
\end{equation}
Define
 \begin{equation} \label{epth'}
\epsilon_0:=\left( \frac{A}{2^{2n+11}c_0} \right) \theta^2<1.
\end{equation}
Choose $D_0 \in (0,\frac{1}{\sqrt{8}})$ such that
\begin{equation} \label{Dnot'} 
 D_0 < \min \left\{  \frac{\epsilon_0^{6}}{2^{13}C}, \frac{A}{4},  \frac{\theta \epsilon_0}{8}   \right\}.
\end{equation}
Furthermore,  inequality (\ref{Dnot'}) implies $8 \theta^{-1}D_0 < \epsilon_0$.  Combining this with (\ref{assumptionoflemma'}) and (\ref{epth'}), we obtain 
\begin{equation} \label{Dnot''} 
 \sup_{B_{^{\theta^i}}(0)} |u_{\varrho}  -A x^1| < \theta^i 8 \theta^{-1}D_0<\theta^i\epsilon_0<\theta^iA.
 \end{equation}
Thus, the assumption (\ref{agcor2}) of Lemma~\ref{geomcor} is satisfied.  Additionally,  the assumption (\ref{geomcorassump}) of Lemma~\ref{geomcor} is implied by (\ref{assumptionoflemma'}).   Thus, Lemma~\ref{geomcor} and (\ref{volumewrtmetric'}) imply
\begin{equation} \label{psy!}
Vol_{g} \left\{ x \in B_{\theta^i}(0):  u(x) \notin \overline{\bf H}[2
\left( \frac{\theta^i \epsilon_0}{2}\right)^{-3} \frac{\theta^iD_0}{2^i},t_* ]  \right\} <\frac{17v_{n-1}}{8} \cdot \theta^{in}\frac{2\epsilon_0}{A}
\end{equation}
which in turn implies that there exists  $R_0 \in [\frac{5}{8},\frac{7}{8}]$ with the property that 
\begin{equation} \label{psy'}
Vol_{g}\left\{ x \in \partial B_{\theta^iR}(0): u(x) \notin \overline{\bf H}[2
\left( \frac{\theta^i \epsilon_0}{2}\right)^{-3} \frac{\theta^iD_0}{2^i}, t_* ]  \right\} <(\theta^i R_0)^{n-1}\frac{2^{2n+3}\epsilon_0}{A}.
\end{equation}
To see this, denote by $f(R)$ the volume appearing on the left side of (\ref{psy'}) and let $R_0 $ be such that 
\[
f(R_0)=\inf_{ R \in [\frac{5}{8},\frac{7}{8}]}f(R).
\]
Then  by (\ref{psy!})
\[
\frac{\theta^i}{4}f(R_0) \leq \int_{\frac{5\theta^i}{8}}^{\frac{7\theta^i}{8}}f(R)dR <\frac{17v_{n-1}}{8} \cdot \theta^{in}\frac{2\epsilon_0}{A}
\]
hence
\[
f(R_0) < \theta^{i(n-1)}\frac{17v_{n-1}\epsilon_0}{A} \leq (\theta^i R_0)^{(n-1)} (\frac{8}{5})^{(n-1)}\frac{17v_{n-1}\epsilon_0}{A}.
\]
Since the Euclidean volume $v_{n-1}$ of the unit $(n-1)$-dimensional ball is bounded by $6$ for all $n$ and $v_1=2$, (\ref{psy'}) follows. 
 
 Let 
\[
\pi :\overline{\bf H}  \rightarrow \overline{\bf H}[
2\left( \frac{\theta^i \epsilon_0}{2}\right)^{-3} \frac{\theta^iD_0}{2^i},t_*]
\]
be the closest point projection map and 
 \[
w:B_{\theta^i {R_0}}(0) \rightarrow \overline{\bf H}[
2\left( \frac{\theta^i \epsilon_0}{2}\right)^{-3} \frac{\theta^iD_0}{2^i},t_*]
\]
be the harmonic map with boundary value equal to $\pi \circ u$.
 By the definition of $\pi$, the fact that ${_il}(x) \in   \overline{\bf H}[ 
\left( \frac{\theta^i \epsilon_0}{2}\right)^{-3} \frac{\theta^iD_0}{2^i},t_* ]$, we conclude
\begin{equation} \label{doc'}
d(u(x),w(x)) \leq d(u(x), {_il}(x)), \  \ \forall x \in  \partial B_{\theta^{i}R}(0).  
\end{equation}   
%Since $\theta <\frac{1}{24}$, we have
%$
%\theta^{i} < \frac{1}{2^{2i+1}},
%$
%and thus (\ref{choosec'}) implies
%\begin{equation} \label{xtraterm'}
%c\theta^{3i} = c\theta^{-2} \theta^{3i+2} < \theta^{2i+2}  \frac{D_0^2}{2^{2i+9}}.
%\end{equation}
We thus     
obtain 
\begin{eqnarray} \label{bdryinfo}
\sup_{B_{\frac{15\theta^iR_0}{16}}(0)} d^2(u,w) 
& \leq & 
 \frac{c_0}{(\theta^i R_0)^{n-1}} \int_{\partial B_{\theta^iR_0}(0)} d^2(u,w)d\Sigma\ \ \ \mbox{(by (\ref{abc!'}))} \nonumber \\
& < & \frac{2^{2n+2}\epsilon_0 c_0}{A}  \sup_{\partial B_{\theta^i R_0}(0)} d^2(u,w) \ \ \ \mbox{(by (\ref{psy'})) }\nonumber \\
& \leq & \frac{2^{2n+3}\epsilon_0c_0}{A}  \sup_{\partial B_{\theta^i R_0}(0)} d^2(u,{_il})  \ \ \ \mbox{(by (\ref{doc'}))}\nonumber \\
& <&\frac{2^{2n+3}\epsilon_0c_0}{A} \cdot \theta^{2i} \frac{D_0^2}{2^{2i}}  \ \ \ \ \ \ \ \mbox{(by (\ref{assumptionoflemma'})) } \nonumber \\
&< & \theta^{2i+2}  \frac{D_0^2}{2^{2i+8}} \ \ \ \ \ \ \ \mbox{(by (\ref{epth'}))},
\end{eqnarray}
or more simply
\begin{equation} \label{uv'}
\sup_{B_{\frac{15\theta^iR_0}{16}}(0)} d(u,w)  < \theta^{i+1}  \frac{D_0}{2^{i+4}}.
\end{equation}
Combining (\ref{assumptionoflemma'}) and (\ref{uv'}), we obtain
\begin{equation} \label{wl'}
\sup_{B_{\frac{\theta^i}{2}}(0)}  d(w ,{_il}) \leq  \sup_{B_{\frac{\theta^i}{2}}(0)}  d(u,w) +  \sup_{B_{\frac{\theta^i}{2}}(0)}  d(u,{_il}) \leq   \theta^i \frac{D_0}{2^{i-1}}.
\end{equation}  
We will now check that we can apply  Theorem~\ref{essreg**}.  We fix $R=\frac{1}{2}$, $E_0$ as in (\ref{insteadofcalF'}), $A_0=6A$.  Set 
\[
\varphi_0 = 2\left(\frac{\epsilon_0}{2}\right)^{-3} \frac{D_0}{2^i}    \ \mbox{ and } \ \vartheta=\frac{\theta^i}{2}.
\]
   First, since $w\big|_{2\vartheta R_0} = \pi \circ u\big|_{2\vartheta R_0}$ and the projection into a convex set in an NPC space is distance non-increasing, we obtain 
\[
E^{w}(2\vartheta R_0) \leq E^{u}(2 \vartheta R_0).
\]
Furthermore, 
%by the monotonicity property of harmonic maps (cf. (\ref{monotonicity!})) and 
(\ref{insteadofcalF'}) implies
\[
\frac{E^{u}(2\vartheta R_0)}{(2\vartheta R_0)^n}  
%\leq e^{cR_0\theta^i} \frac{E^{u}(R_0 \theta^i)}{R_0^2\theta^{in}} 
\leq e^{c} E^u(1) \leq \frac{E_0}{2^n}.
\]
Since $R_0 \in [\frac{5}{8},\frac{7}{8}]$, we therefore conclude
\[
E^w(\vartheta) \leq \vartheta^n E_0.
\]
Next, Lemma~\ref{distequiveuc}, (\ref{uv'}), (\ref{Dnot'}) and (\ref{Dnot''}) imply that in $B_{\frac{15\theta^iR}{16}}(0)$, we have
\begin{eqnarray}\label{jsb'}
|w_{\varrho}| & \leq & |w_{\varrho}-u_{\varrho}| + |u_{\varrho}-Ax^1|+|Ax^1| \nonumber \\
& < &  \theta^{i+1}  \frac{D_0}{2^{i+4}}+\theta^iA+\theta^iA  \\ 
& \leq &  3\theta^iA = A_0 \vartheta  \nonumber.
\end{eqnarray}
Thus, $w$ maps into $\overline{\bf H}[\frac{\varphi_0}{\vartheta^2},t_*] \cap {\bf B}_{A_0\vartheta}(P_0)
$.  Finally, (\ref{sat'}) implies 
\begin{eqnarray}\label{jcb'}
a[\frac{\varphi_0}{\vartheta^2},t_*] = a[2
\left( \frac{\theta^i \epsilon_0}{2}\right)^{-3} \frac{\theta^iD_0}{2^i},t_*]
=a[\frac{16D_0}{\epsilon_0^3 \theta^{2i} 2^i},t_*]
< \frac{\theta^i}{2}=\frac{\vartheta}{2}
\end{eqnarray}
which is assumption (\ref{additionalassumption}) of Theorem~\ref{essreg**}.
In other words, we have verified all the assumptions of Theorem~\ref{essreg**}.  Thus, with  
\[
{_il}=L, \  {_{i+1}l}=\hat{l}, \ \vartheta =\theta^i, \  R= \frac{1}{2}  \mbox{ and } \ r  =\theta,
\]
Theorem~\ref{essreg**} implies with the  choice of the  constants in (\ref{startofconstants'}) that
\[
\sup_{B_{\theta^{i+1}}(0)} d(w,{_{i+1}l}) \leq C\theta^{1+\alpha} \sup_{B_{\frac{\theta^i}{2}}(0)} d(w,{_il}) +C \theta^{i+1}  \left( 2
\left( \frac{\epsilon_0}{2}\right)^{-3} \frac{D_0}{2^i} \right)^2 .
\]
%This immediately implies
% \begin{equation} \label{fpf''}
%\sup_{B_{\theta^{i+1}}(0)} d(w,{_{i+1}l}) \leq C\theta^{1+\alpha} \sup_{B_{\frac{\theta^i}{2}}(0)} d(w,{_il}) +mC \theta^{i+1}  \left( 2
%\left( \frac{\epsilon_0}{2}\right)^{-3} \frac{D_0}{2^i} \right)^2 
%\end{equation}
Hence
\begin{eqnarray*}
\sup_{B_{\theta^{i+1}}(0)} d(w,{_{i+1}l}) &  \leq & C\theta^{i+1}\theta^{\alpha}  \frac{D_0}{2^{i-1}}  +C\theta^{i+1} \frac{D_0^2}{\epsilon_0^{6} 2^{2i-8}}  \ \mbox{ (by  (\ref{wl'}))} \\
&  < &  \theta^{i+1} \frac{D_0}{2^{i+4}}  \ \mbox{ (by  (\ref{th'}) and  (\ref{Dnot'}))}.
\end{eqnarray*}
Combined with (\ref{uv'}), we obtain
\begin{eqnarray} \label{ggg'}
\sup_{B_{\theta^{i+1}}(0)} d(u,{_{i+1}l}) & \leq &  \sup_{B_{\theta^{i+1}}(0)} d(u,w)+ \sup_{B_{\theta^{i+1}}(0)} d(w,{_{i+1}l}) \nonumber \\
& < & \theta^{i+1} \frac{D_0}{2^{i+3}}.
\end{eqnarray}
This implies the first inequality of (\ref{conclusionoflemma'}). 
Furthermore, note that  ${_{i+1}l}_{\varphi}=w_{\varphi}$ by definition (cf. Theorem~\ref{essreg**}).  Since $\theta \in (0,\frac{1}{24})$, 
\begin{eqnarray}\label{terr2'}
| {_{i+1}l}_{\varphi}(x)|=| w_{\varphi}(x)| \leq  
2\left( \frac{\theta^i \epsilon_0}{2}\right)^{-3} \frac{\theta^iD_0}{2^i} \leq 
\left( \frac{\theta^{i+1} \epsilon_0}{2}\right)^{-3} \frac{\theta^{i+1}D_0}{2^{i+1}}.
\end{eqnarray}
Thus, we conclude  ${_{i+1}l}$  maps into  $\overline{\bf H}[\left( \frac{\theta^{i+1} \epsilon_0}{2}\right)^{-3} \frac{\theta^{i+1}D_0}{2^{i+1}},t_*]$. 

We now proceed with the proof of  the second inequality of  (\ref{conclusionoflemma'}).  
 Since $_{i}l_{\varrho}$ and $Ax^1$ are both  affine functions and $u(0)=P_0$, we have for every $x \in B_{\theta^i}(0)$
\begin{eqnarray*} 
|_{i}l_{\varrho}(\theta x)-A\theta x^1| & = &   |(1-\theta) {_{i}l_{\varrho}}(0)+\theta ({_{i}l_{\varrho}}( x)-A x^1)| \nonumber \\
& \leq  &   (1-\theta) {_{i}l_{\varrho}}(0)+\theta |{_{i}l_{\varrho}}( x)-A x^1|. \nonumber 
 \end{eqnarray*}
By the definition of the coordinates $(\varrho,\varphi)$,  $_il_{\varrho}(0)$ is  the distance between the point $_il(0)$ and the geodesic ray ${\mathcal L}=\{\phi=0\} \cup P_0$.   Since $u(0)=P_0$,  we have that 
 \[
 _il_{\varrho}(0)=d(_il(0),{\mathcal L})  \leq d(_il(0),u(0)).
 \]
 Thus, 
 \begin{eqnarray*} 
|_{i}l_{\varrho}(\theta x)-A\theta x^1|
& \leq &   (1-\theta) d({_{i}l}(0), u(0))+ \theta |{_{i}l}_{\varrho}( x)-A x^1| \nonumber    \\
& \leq &   (1-\theta) d({_{i}l}(0), u(0))+ \theta |{_{i}l}_{\varrho}( x)-u_{\varrho}(x)| + |u_{\varrho}(x) - A x^1| \nonumber    \\
\end{eqnarray*}
Since 
\[
|_il_{\varrho}(x) -u_{\varrho}(x)| \leq d(_il(x),u(x)).
\]
Thus, 
\begin{eqnarray*}
|_{i}l_{\varrho}(\theta x)-A\theta x^1|& \leq &   (1-\theta) d({_{i}l}(0), u(0))+ \theta d({_{i}l}( x),u( x))+  \theta |u_{\varrho}( x)-A x^1|  \nonumber    \\
 & < &  \theta^{i} \frac{D_0}{2^i} +\theta^{i+1}  {_i\delta}   \ \ \ \mbox{(by (\ref{assumptionoflemma'})}) \\
 & = & \theta^{i+1} \left( {_i\delta}  +  \frac{D_0 \theta^{-1}}{2^i} \right)
\end{eqnarray*}
which implies
\begin{equation} \label{terr'}
 \sup_{B_{\theta^{i+1}}(0)} |{_{i}l_{\varrho}} (x)-A x^1|
  \leq  \theta^{i+1} \left( {_i\delta}  +   \frac{D_0 \theta^{-1}}{2^i} \right).
\end{equation}
Thus, for $x \in B_{\theta^{i+1}}(0)$
\begin{eqnarray}\label{removed'} 
\lefteqn{|u_{\varrho}(x)-Ax^1| } \nonumber \\
& \leq & |u_{\varrho} (x)- {_{i}l}_{\varrho}(x)| +|{_{i}l}_{\varrho}(x)-Ax^1| \nonumber \\ 
& \leq & d(u(x), {_{i}l}(x))+|{_{i}l}_{\varrho}(x)-A x^1|   \ \ \ \mbox{(by Lemma~\ref{distequiveuc}}) \nonumber  \\
& < &   \theta^i \frac{D_0}{2^i} +\theta^{i+1} \left( {_i\delta}  +   \frac{D_0 \theta^{-1}}{2^i} \right)  \ \ \ \mbox{(by (\ref{assumptionoflemma'}) and (\ref{terr'})})\\
& < & \theta^{i+1} \left( {_i\delta}  +   \frac{2D_0 \theta^{-1}}{2^i} \right)  \nonumber\\
& < & \displaystyle{\theta^{i+1} \sum_{k=0}^{i+1} \frac{ \theta^{-1}D_0}{2^{k-2}}} \ \ \ \mbox{(by (\ref{assumptionoflemma'})}). \nonumber
\end{eqnarray}
This is the second inequality of (\ref{conclusionoflemma'}).  \\

Finally, we will prove the third inequality of (\ref{conclusionoflemma'}).  Since $l(x) =(Ax^1,0)$ and since  by (\ref{assumptionoflemma'})\[
{_i\delta}<\displaystyle{ \sum_{k=0}^{i} \frac{\theta^{-1}D_0}{2^{k-2}}} \leq 8\theta^{-1} D_0,
\] 
we  conclude from  (\ref{terr'}) 
that
\[
\sup_{B_{\theta^{i+1}}(0)} d( ({_{i}l_{\varrho}}(x), 0), l(x))= \sup_{B_{\theta^{i+1}}(0)} |{_{i}l}_{\varrho}(x)-Ax^1| < 9\theta^i D_0.  
\]
Thus, for $x \in B_{\theta^{i+1}}(0)$, 
\begin{eqnarray} \label{casa}
\lefteqn{ d({_{i}l}(x), ({_{i}l}_{\varrho}(x), 0))} \nonumber \\ 
&  \leq &  ({_{i}l}_{\varrho}(x))^3|_{i}l_{\varphi}(x)|  \nonumber \\
& < &  \theta^{3i}\left( A+9D_0  \right)^3\cdot 2
\left( \frac{\theta^i \epsilon_0}{2}\right)^{-3} \frac{\theta^i D_0}{2^i}   \ \ \ \mbox{(by (\ref{assumptionoflemma'})}) \\
& \leq & \theta^i \frac{2^3\left( A+ 9D_0 \right)^3}{\epsilon_0^3} D_0.  \nonumber
\end{eqnarray}
Combining the above two inequalities, we obtain
\begin{eqnarray}
d({_{i}l}(x), l(x)) < \theta^i   \left( \frac{2^3\left( A+9D_0 \right)^3}{\epsilon_0^3} +9 \right)   D_0.
\end{eqnarray}
Combined with (\ref{assumptionoflemma'}), 
\begin{eqnarray*}
\sup_{B_{\theta^{i+1}}(0)} d(u,l) & \leq &  \sup_{B_{\theta^{i+1}}(0)} d(u,{_{i}l})+ \sup_{B_{\theta^{i+1}}(0)} d({_{i}l},l) \nonumber \\
& < &\theta^i  \left( \frac{2^3\left( A+9D_0 \right)^3}{\epsilon_0^3} +10 \right)   D_0.
\end{eqnarray*}
\end{proof}

We now present the general case of the above theorem.  This generalization is needed in order to handle the case of approximate harmonic maps.   The assumptions made on  $v$ in Theorem~\ref{induction} should be compared with the properties of the harmonic map observed in Remark~\ref{whatweuse}.

\begin{keylemma}\label{induction}
Given $c_0 \geq 1$, $E_0, A^1, \dots, A^m>0$,
there exist  $\theta  \in (0,\frac{1}{24})$, $\epsilon_0>0$  and  $D_0 \in (0,\frac{1}{\sqrt{8}})$ that satisfy the following statement.  \\
\\
Assume the following:  
\begin{itemize}
\item The map
\[
l=(l^1 \circ (R^1)^{-1}, \dots, l^m \circ (R^m)^{-1}, l^{m+1}, \dots, l^{k-j}):B_{\theta^i}(0) \rightarrow \overline{\bf H}^{k-j}
\]
is such that $R^{\mu}$  is a rotation, 
\[
l^{\mu}(x)= (A^{\mu}x^1,0) \mbox{ in coordinates $(\varrho,\varphi)$ anchored at $t_*^{\mu} \in (-\infty,\frac{3}{2})$  (cf. (\ref{cov}))}
\]
for $\mu=1, \dots, m$  and 
\[
l^{\mu} \mbox{ is identically equal to }P_0
\]
for $\mu=m+1, \dots, k-j$.

\item  The subset $\overline{\bf H}[2
\left( \frac{\theta^i \epsilon_0}{2}\right)^{-3} \frac{\theta^iD_0}{2^i},t_*^{\mu}]$ satisfies 
\begin{equation} \label{sat}
a[2
\left( \frac{\theta^i \epsilon_0}{2}\right)^{-3} \frac{\theta^iD_0}{2^i},t_*^{\mu}]=a[\frac{16D_0}{\epsilon_0^3 \theta^{2i} 2^i},t_*^{\mu}
]
< \frac{\theta^i}{2}    \ \ \mbox{(cf. (\ref{defofa}))}
\end{equation}  
for $\mu=1, \dots, m$.
\item
The map 
\[
v=(v^1, \dots, v^{k-j}):(B_1(0),g) \rightarrow \overline{\bf H}^{k-j}
\]
is such that 
\begin{equation} \label{insteadofcalF}
v(0)={\mathbb P}_0, \ \ E^v(\vartheta) \leq \vartheta^nE_0 
\end{equation}
and 
\begin{quote}
\emph{for $R \in (0,\frac{7}{8}]$, a harmonic map $w:(B_{\theta^i R}(0),g)  \rightarrow \overline{\bf H}^{k-j}$ with $E^w(\theta^i R) \leq E^v(\theta^i R)$ and a constant}
\begin{equation} \label{choosec}
c=\frac{\theta^2 D_0^2}{2^8},
\end{equation}
we have
 \begin{equation} \label{abc!}
 \sup_{B_{\frac{15\theta^i R}{16}}(0)} d_h^2(v,w)  \leq  \frac{c_0}{(\theta^i R)^{n-1}} \int_{\partial B_{\theta^i R}(0)} d_h^2(v,w)d\Sigma + c \theta^{3i}.
\end{equation}
\end{quote}
\item  The metric $g$   is a normalized metric satisfying (\ref{volumewrtmetric'}) for any smooth submanifold $S$ of $B_1(0)$. 
\item The map
\[
{_il}=({_il}^1 \circ (R^1)^{-1}, \dots, {_il}^m \circ (R^1)^{-1}, {_il}^{m+1},   \dots, {_il}^{k-j}):B_{\theta^i}(0) \rightarrow \overline{\bf H}^{k-j},
\]
is such that
\[
_il^{\mu}:B_{\theta^i}(0) \rightarrow \overline{\bf H}[
\left( \frac{\theta^i \epsilon_0}{2}\right)^{-3} \frac{\theta^iD_0}{2^i},t_*^{\mu}]
\
\mbox{is an almost affine map }
\] 
for 
$\mu=1, \dots, m$ (cf. Definition~\ref{almostaffine})
 and
\[
_il^{\mu} \mbox{ is identically equal to }P_0
\]
 for $\mu=m+1, \dots, k-j$.
\item The constant  $_i\delta>0$ is  such that
\begin{equation} \label{assumptionoflemma}
\left\{
\begin{array}{l}
{\displaystyle \sup_{B_{^{\theta^i}}(0)}  d_h(v, {_il}) <\theta^{i}\frac{D_0}{2^i}}\\
{\displaystyle  \sup_{B_{^{\theta^i}}(0)} } |v_{\varrho}^{\mu} \circ R^{\mu}(x)-A^{\mu}x^1|  < \theta^{i}{_i\delta}<\theta^{i}\displaystyle{ \sum_{k=0}^{i} \frac{\theta^{-1}D_0}{2^{k-2}}}.
\end{array}
\right.
\end{equation}
\end{itemize}
Then there exists a map
\[
 _{i+1}l=( {_{i+1}l}^1 \circ (R^1)^{-1}, \dots, {_{i+1}}l^m  \circ (R^m)^{-1},  {_{i+1}}l^{m+1}, \dots,  {_{i+1}l}^{k-j}):B_{\theta^{i+1}}(0) \rightarrow  \overline{\bf H}^{k-j} 
\]
 such that 
\[
 _{i+1}l^{\mu}:B_{\theta^{i+1}}(0) \rightarrow  \overline{\bf H}[
\left( \frac{\theta^{i+1} \epsilon_0}{2}\right)^{-3} \frac{\theta^{i+1}D_0}{2^{i+1}},t_*^{\mu}] 
\
\mbox{is an almost affine map}
\] 
for 
$\mu=1, \dots, m$,
\[
_{i+1}l^{\mu} \mbox{ is identically equal to }P_0
\]
for $\mu=m+1, \dots, k-j$  and 
\begin{equation} \label{conclusionoflemma}
\left\{
\begin{array}{l}
\displaystyle{\sup_{B_{\theta^{i+1}}(0)}} d_h(v,{_{i+1}l}) < \theta^{i+1}\frac{D_0}{2^{i+1}}\\
\displaystyle {\sup_{B_{\theta^{i+1}}(0)}  |v_{\varrho}^{\mu} \circ R^{\mu}(x)- A^{\mu}x^1|}  < \displaystyle {_{i+1}\delta} \theta^{i+1}\ :=  \left( {_i\delta}  +   \frac{2D_0 \theta^{-1}}{2^i} \right) \theta^{i+1} < \displaystyle{\theta^{i+1} \sum_{k=0}^{i+1} \frac{ \theta^{-1}D_0}{2^{k-2}}} \\
{\displaystyle 
\sup_{B_{\theta^{i+1}}(0)} d_h(v,l) 
< m\theta^{i +1}  \left( \frac{2^3\left( A+9D_0 \right)^3}{\epsilon_0^3} +10 \right)  \theta^{-1}  D_0}.
\end{array}
\right.
\end{equation}
\end{keylemma}

\begin{proof}   
Let
   \begin{equation} \label{Anot1}
  A_{\min}:=\min \{A^1, \dots, A^{m}\} \ \mbox{ and }  A_{\max}:=\max \{A^1, \dots, A^{m}\}.
   \end{equation}
 %We now choose the constants $C>0$, $\theta  \in (0,\frac{1}{8})$, $\epsilon_0>0$ and  $D_0>0$ as follows.  
For $R= \frac{1}{2}$, $E_0>0$ as in (\ref{insteadofcalF}) and the metric $g$ given as in the statement of the theorem, let
 \begin{equation} \label{startofconstants}
C \geq 1 \mbox{ and }  \alpha>0 \mbox{ be as  in Theorem~\ref{essreg**}}.
\end{equation}
Let $\theta \in (0,\min\{\frac{1}{24},\frac{1}{\sqrt{A}} \})$ sufficiently small such that
(\ref{ceatheta'}), (\ref{th'}) and (\ref{teanot'}) are satisfied.
%C\theta<1,
%\end{equation}
%\begin{equation} \label{th}
%C \theta^{\alpha} <   \frac{1}{2^{6}},
%\end{equation}
%and
%\begin{equation} \label{teanot}
% C \theta^3   < \frac{1}{2^3}.
%\end{equation}
Define
 \begin{equation} \label{epth}
\epsilon_0:=\left( \frac{A_{\min}}{2^{2n+11}c_0} \right) \theta^2<1.
\end{equation}
Choose $D_0 \in (0,\frac{1}{\sqrt{8}})$ such that
\begin{equation} \label{Dnot} 
 D_0 < \min \left\{  \frac{\epsilon_0^{6}}{2^{13}mC}, \frac{A_{\min}}{4},  \frac{\theta \epsilon_0}{8}   \right\}.
\end{equation}
As in (\ref{Dnot''}), we obtain 
\begin{equation}\label{Dnot'''}
 \sup_{B_{^{\theta^i}}(0)} |v_{\varrho}^{\mu} \circ R^{\mu}(x)-A^{\mu}x^1| < \theta^i\epsilon_0< \theta^i A^{\mu}.
 \end{equation}
Thus, assumption (\ref{agcor2}) of Lemma~\ref{geomcor} is satisfied.  As in (\ref{psy'}) 
\begin{equation} \label{psy}
Vol_{g}\left\{ x \in \partial B_{\theta^iR}(0): v^{\mu}\circ R^{\mu}(x) \notin \overline{\bf H}[2
\left( \frac{\theta^i \epsilon_0}{2}\right)^{-3} \frac{\theta^iD_0}{2^i}, t_*^{\mu} ]  \right\} <(\theta^i R)^{n-1}\frac{2^{2n+2}\epsilon_0}{A^{\mu}}.
\end{equation}
 Let 
 \[
 w=(w^1, \dots, w^{k-j}):B_{\theta^i R}(0) \rightarrow \overline{\bf H}^{k-j}
\]
be the harmonic map defined as follows:  

\begin{itemize}
\item Let 
\[
\pi^{\mu} :\overline{\bf H}  \rightarrow \overline{\bf H}[
2\left( \frac{\theta^i \epsilon_0}{2}\right)^{-3} \frac{\theta^iD_0}{2^i},t_*^{\mu}] , \ \ \mu=1,\dots, m
\]
be the closest point projection map and 
 \[
w^{\mu}:B_{\theta^i R}(0) \rightarrow \overline{\bf H}[
2\left( \frac{\theta^i \epsilon_0}{2}\right)^{-3} \frac{\theta^iD_0}{2^i},t_*^{\mu}], \ \ \mu=1,\dots, m
\]
be the harmonic map with boundary value equal to $\pi^{\mu} \circ v^{\mu}$.
\item For $\mu=m+1, \dots, k-j$, 
let 
$w^{\mu}$ be identically equal to $P_0$.
\end{itemize}
 By  definition of $\pi^{\mu}$, the fact that ${_il}^{\mu}(x) \in   \overline{\bf H}[ 
\left( \frac{\theta^i \epsilon_0}{2}\right)^{-3} \frac{\theta^iD_0}{2^i},t_*^{\mu} ]$ for $\mu=1, \dots, m$ and that ${_il}^{\mu}(x) \equiv P_0$ for $\mu=m+1, \dots, k-j$, we conclude
\begin{equation} \label{doc}
d(v(x),w(x)) \leq d(v(x), {_il}(x)), \  \ \forall x \in  \partial B_{\theta^{i}R}(0).
\end{equation}   
Since $\theta <\frac{1}{24}$, we have
$
\theta^{i} < \frac{1}{2^{2i+1}},
$
and thus (\ref{choosec}) implies
\begin{equation} \label{xtraterm}
c\theta^{3i} = c\theta^{-2} \theta^{3i+2} < \theta^{2i+2}  \frac{D_0^2}{2^{2i+9}}.
\end{equation}
We thus     
obtain 
\begin{eqnarray} \label{bdryinfo}
\sup_{B_{\frac{15\theta^iR}{16}}(0)} d^2(v,w) 
& \leq & 
 \frac{c_0}{(\theta^i R)^{n-1}} \int_{\partial B_{\theta^iR}(0)} d^2(v,w)d\Sigma +c\theta^{3i}\ \ \ \mbox{(by (\ref{abc!}))} \nonumber \\
& < & \frac{2^{2n+2}\epsilon_0 c_0}{A_{\min}}  \sup_{\partial B_{\theta^i R}(0)} d^2(v,w) +c\theta^{3i}\ \ \ \mbox{(by (\ref{psy})) }\nonumber \\
& \leq & \frac{2^{2n+2}\epsilon_0c_0}{A_{\min}}  \sup_{\partial B_{\theta^i R}(0)} d^2(v,{_il}) +c\theta^{3i} \ \ \ \mbox{(by (\ref{doc}))}\nonumber \\
& <&\frac{2^{2n+2}\epsilon_0c_0}{A_{\min}} \cdot \theta^{2i} \frac{D_0^2}{2^{2i}} +c\theta^{3i} \ \ \ \ \ \ \ \mbox{(by (\ref{assumptionoflemma})) } \nonumber \\
&< & \theta^{2i+2}  \frac{D_0^2}{2^{2i+8}} \ \ \ \ \ \ \ \mbox{(by (\ref{epth}) and (\ref{xtraterm}))},
\end{eqnarray}
or more simply
\begin{equation} \label{uv}
\sup_{B_{\frac{15\theta^iR}{16}}(0)} d(v,w)  < \theta^{i+1}  \frac{D_0}{2^{i+4}}.
\end{equation}
Combining (\ref{assumptionoflemma}) and (\ref{uv}), we obtain
\begin{equation} \label{wl}
\sup_{B_{\frac{\theta^i}{2}}(0)}  d(w ,{_il}) \leq  \sup_{B_{\frac{\theta^i}{2}}(0)}  d(v,w) +  \sup_{B_{\frac{\theta^i}{2}}(0)}  d(v,{_il}) \leq   \theta^i \frac{D_0}{2^{i-1}}.
\end{equation}  
We will now check that we can apply  Theorem~\ref{essreg**}.  We fix $R=\frac{1}{2}$, $E_0$ as in (\ref{insteadofcalF}), $A_0=3A_{\max}$, $\varphi_0 = 2\left(\frac{\epsilon_0}{2}\right)^{-3} \frac{D_0}{2^i}$    and $\vartheta=\theta^i$,   First, note that  since  projection into a convex set in an NPC space is distance non-increasing, we obtain $E^{w^{\mu}}(\theta^i) \leq E^{v^{\mu}}(\theta^i) \leq \theta^{in}E_0$ by (\ref{insteadofcalF}).
In analogy with (\ref{jsb'}) we obtain
\begin{eqnarray*}
|w_{\varrho}^{\mu} \circ R^{\mu}|  \leq A_0 \vartheta.
\end{eqnarray*}
Thus, $w^{\mu}$ maps into $\overline{\bf H}[\frac{\varphi_0}{\vartheta^2},t_*] \cap {\bf B}_{A_0\vartheta}(P_0)
$.  Finally, in analogy with (\ref{jcb'})  
\[
a[\frac{\varphi_0}{\vartheta^2},t_*^{\mu}] < \frac{\theta^i}{2}=\frac{\vartheta}{2}
\]
which is assumption (\ref{additionalassumption}) of Theorem~\ref{essreg**}.  Thus, with  
\[
{_il}^{\mu}=L, \  {_{i+1}l}^{\mu}=\hat{l}, \ \vartheta =\theta^i, \  R= \frac{1}{2}  \mbox{ and } \ r  =\theta
\]
in Theorem~\ref{essreg**},
 we have by the  choice of the  constants in (\ref{startofconstants}) that
\[
\sup_{B_{\theta^{i+1}}(0)} d_{\overline{\bf H}}(w^{\mu},{_{i+1}l}^{\mu}) \leq C\theta^{1+\alpha} \sup_{B_{\frac{\theta^i}{2}}(0)} d_{\overline{\bf H}}(w^{\mu},{_il}^{\mu}) +C \theta^{i+1}  \left( 2
\left( \frac{\epsilon_0}{2}\right)^{-3} \frac{D_0}{2^i} \right)^2 .
\]
This immediately implies
\[
\sup_{B_{\theta^{i+1}}(0)} d(w,{_{i+1}l}) \leq C\theta^{1+\alpha} \sup_{B_{\frac{\theta^i}{2}}(0)} d(w,{_il}) +mC \theta^{i+1}  \left( 2
\left( \frac{\epsilon_0}{2}\right)^{-3} \frac{D_0}{2^i} \right)^2 
\]
hence
\begin{eqnarray*}
\sup_{B_{\theta^{i+1}}(0)} d(w,{_{i+1}l}) &  \leq & C\theta^{i+1}\theta^{\alpha}  \frac{D_0}{2^{i-1}}  +m C\theta^{i+1} \frac{D_0^2}{\epsilon_0^{6} 2^{2i-8}}  \ \mbox{ (by  (\ref{wl}))} \\
&  < &  \theta^{i+1} \frac{D_0}{2^{i+4}}  \ \mbox{ (by  (\ref{th'}) and  (\ref{Dnot}))}.
\end{eqnarray*}
Combined with (\ref{uv}), we obtain as in (\ref{ggg'})
\begin{eqnarray} \label{ggg}
\sup_{B_{\theta^{i+1}}(0)} d(v,{_{i+1}l}) & < & \theta^{i+1} \frac{D_0}{2^{i+3}}.
\end{eqnarray}
This implies the first inequality of (\ref{conclusionoflemma}). 
Furthermore, since ${_{i+1}l}_{\varphi}^{\mu}=w_{\varphi}^{\mu}$ by definition (cf. Theorem~\ref{essreg**}), we conclude via the analogous equation to  (\ref{terr2'}) 
that  ${_{i+1}l}^{\mu}$  maps into  $\overline{\bf H}[\left( \frac{\theta^{i+1} \epsilon_0}{2}\right)^{-3} \frac{\theta^{i+1}D_0}{2^{i+1}},t_*^{\mu}]$. 

We now proceed with the proof of  the second inequality of  (\ref{conclusionoflemma}). 
Seting $A^{\mu}=0$ for $\mu=m+1, \dots, k-j$ for simplicity, we deduce in a maner  identical to (\ref{terr'}) 
% Since $_{i}l_{\varrho}^{\mu}\circ R^{\mu}(x)$ and $A^{\mu}x^1$ are both  affine functions and $v^{\mu}(0)=P_0$, we have for every $x \in B_{\theta^i}(0)$
%\begin{eqnarray*} 
%\lefteqn{|_{i}l_{\varrho}^{\mu}\circ R^{\mu}(\theta x)-A^{\mu}\theta x^1|}\nonumber \\
% & = &   |(1-\theta) {_{i}l_{\varrho}^{\mu}}(0)+\theta ({_{i}l_{\varrho}^{\mu}}\circ R^{\mu}( x)-A^{\mu} x^1)| \nonumber \\
%& \leq &   (1-\theta) d({_{i}l_{\varrho}^{\mu}}\circ R^{\mu}(0), v^{\mu}(0))+ \theta |{_{i}l_{\varrho}^{\mu}}\circ R^{\mu}( x)-A^{\mu} x^1| \nonumber    \\
%& \leq &   (1-\theta) d({_{i}l_{\varrho}^{\mu}}\circ R^{\mu}(0), v^{\mu}(0))+ \theta d({_{i}l^{\mu}}\circ R^{\mu}( x),v^{\mu}\circ R^{\mu}( x))\nonumber    \\
%&+&  \theta |v_{\varrho}^{\mu}\circ R^{\mu}( x)-A^{\mu} x^1|  \ \ \ \mbox{(by Lemma~\ref{distequiveuc}}) \nonumber    \\
% & < &  \theta^{i} \frac{D_0}{2^i} +\theta^{i+1}  {_i\delta}   \ \ \ \mbox{(by (\ref{assumptionoflemma})}) \\
% & = & \theta^{i+1} \left( {_i\delta}  +  \frac{D_0 \theta^{-1}}{2^i} \right)
%\end{eqnarray*}
%which implies
\begin{eqnarray} \label{terr}
 \sup_{B_{\theta^{i+1}}(0)} |{_{i}l_{\varrho}^{\mu}} \circ R^{\mu}(x)-A^{\mu} x^1|
  \leq  \theta^{i+1} \left( {_i\delta}  +   \frac{D_0 \theta^{-1}}{2^i} \right).
\end{eqnarray}
Thus, for $x \in B_{\theta^{i+1}}(0)$ we obtain as in (\ref{removed'})
\[
|v_{\varrho}^{\mu} \circ R^{\mu}(x)-A^{\mu}x^1|< \displaystyle{\theta^{i+1} \sum_{k=0}^{i+1} \frac{ \theta^{-1}D_0}{2^{k-2}}}
\]
which is the second inequality of (\ref{conclusionoflemma}).

Finally, we will prove the third inequality of (\ref{conclusionoflemma}).  Since $l^{\mu}(x) =(A^{\mu}x^1,0)$ and since  by (\ref{assumptionoflemma})\[
{_i\delta}<\displaystyle{ \sum_{k=0}^{i} \frac{\theta^{-1}D_0}{2^{k-2}}} \leq 8\theta^{-1} D_0,
\] 
we  conclude from  (\ref{terr}) 
that
\[
\sup_{B_{\theta^{i+1}}(0)} d( ({_{i}l^{\mu}_{\varrho}}\circ R^{\mu}(x), 0), l^{\mu}(x))= \sup_{B_{\theta^{i+1}}(0)} |{_{i}l}_{\varrho}^{\mu}\circ R^{\mu}(x)-A^{\mu}x^1| < 9\theta^i D_0.  
\]
Thus, for $x \in B_{\theta^{i+1}}(0)$ as in (\ref{casa})
\[
d({_{i}l^{\mu}}\circ R^{\mu}(x), ({_{i}l}^{\mu}_{\varrho}\circ R^{\mu}(x), 0))
 \leq  \theta^i \frac{2^3\left( A^{\mu}+ 9D_0 \right)^3}{\epsilon_0^3} D_0.
\]
Combining the above two inequalities, we obtain
\begin{eqnarray*}
d({_{i}l}^{\mu}\circ R^{\mu}(x), l^{\mu}(x)) < \theta^i   \left( \frac{2^3\left( A^{\mu}+9D_0 \right)^3}{\epsilon_0^3} +9 \right)   D_0.
\end{eqnarray*}
Combined with (\ref{assumptionoflemma}) 
\begin{eqnarray*}
\sup_{B_{\theta^{i+1}}(0)} d(v^{\mu},l^{\mu}) & \leq &  \sup_{B_{\theta^{i+1}}(0)} d(v^{\mu},{_{i}l^{\mu}}\circ R^{\mu})+ \sup_{B_{\theta^{i+1}}(0)} d({_{i}l^{\mu}}\circ R^{\mu},l^{\mu}) \nonumber \\
& < &\theta^i  \left( \frac{2^3\left( A+9D_0 \right)^3}{\epsilon_0^3} +10 \right)   D_0.
\end{eqnarray*}
Hence, 
\begin{eqnarray*}
\sup_{B_{\theta^{i+1}}(0)} d(v,l) &  \leq  & \sup_{B_{\theta^{i+1}}(0)}  \sum_{\mu=1}^m d({_{i}l}^{\mu}\circ R^{\mu}(x), l^{\mu}(x))\\
& < &m\theta^i   \left( \frac{2^3\left( A+9D_0 \right)^3}{\epsilon_0^3} +10 \right)   D_0.
\end{eqnarray*}
\end{proof}

\subsection{Proof of the Key Technical Lemma~\ref{keylemma'}} \label{sec:theproof}

\begin{proof}
The assumption (i) that $g$ is sufficiently close to the Euclidean metric is the condition given by (\ref{volumewrtmetric'}).
Let  $\theta$, $\epsilon_0$ and $D_0$ be as in the iterative Lemma~\ref{induction} and let $c=\frac{\theta^2 D_0^2}{2^8}$.  By assumption,
\begin{equation} \label{ruff}
\sup_{B_{\frac{1}{2}}(0)} d_h(v,l) < D_0.
\end{equation}
We will also assume $v(0)={\mathbb P}_0$.
In order to arrive at a contradiction, we will apply iterative Lemma~\ref{induction} starting with $l={_0l}$ and ${_0\delta} =D_0$ (cf. assumption~(\ref{assumptionoflemma}) of the iterative Lemma~\ref{induction}).  To do so, we need to verify assumption~(\ref{sat}) of iterative Lemma~\ref{induction}; in other words,  we need to show
\[
a[2
\left( \frac{\theta^i \epsilon_0}{2}\right)^{-3} \frac{\theta^iD_0}{2^i},t_*^{\mu}]=a[\frac{16D_0}{\epsilon_0^3 \theta^{2i} 2^i},t_*^{\mu}
]
< \frac{\theta^i}{2}.
\]   For this purpose, we note the constants $\theta$ and $\epsilon_0$ are  \emph{chosen before} the constant $D_0$ in the proof of iterative Lemma~\ref{induction};  hence, there is no loss of generality in assuming that $D_0
$ is chosen sufficiently small (cf. (\ref{Dnot})) such that
\begin{equation} \label{noloss1}
\frac{8D_0}{\epsilon_0^3}<1
\end{equation}
and
\begin{equation} \label{noloss2}
m \left( \frac{2^3\left( A+9D_0 \right)^3}{\epsilon_0^3} +10 \right)  \theta^{-1} D_0<\frac{1}{\sqrt{8}}.
\end{equation}  
 For $\mu=1, \dots, m$, recall that $t^{\mu}_*$ is the address of $l^{\mu}$ (cf. (\ref{pouty})). 
Reordering if necessary, we can assume 
\begin{equation}\label{sany}
t^1_* =\max \{t^1_*, \dots, t^m_*\}.
\end{equation}  
Let $i_0$ be the  non-negative integer     such that
\begin{equation} \label{chooseinot}
\frac{\theta^{i_0+1}}{\sqrt{8}} \leq 
c_{\rho}(0,t_*^1)<\frac{\theta^{i_0}}{\sqrt{8}}.
\end{equation}
Recall by (\ref{at1}) and (\ref{normalizeit}) that 
$
t \mapsto c_{\rho}(0,t)=:f(t)
$ satisfies 
\[
f'(t)=f^3(t) \ \mbox{with} \ f(1)=1.
\]
Solving this differential equation, we obtain
\[
f(t)=\frac{1}{\sqrt{3-2t}}
\]
and 
\[
t=\frac{3}{2} -\frac{1}{2f^2(t)}.
\]
In particular, since $f(t_*^1)=c_{\rho}(0,t_*^1) < \frac{\theta^{i_0}}{\sqrt{8}}$, we have
\begin{equation} \label{ts}
-t_*^1=-\frac{3}{2} +\frac{1}{2f^2(t_*)}>-\frac{3}{2} +\frac{4}{\theta^{2i_0}}.
\end{equation}
 Therefore, 
if 
\[
i  \in \{1,2, \dots, i_0\} \ \mbox{ and } \ 
|t_*^1-t| \leq \frac{16D_0}{\epsilon_0^3 \theta^{2i} 2^i},
\]
then by (\ref{noloss1}) and (\ref{ts})
\[ 
3-2t>3-2t_*^1-\frac{32D_0}{\epsilon_0^3 \theta^{2i} 2^i}  > \frac{8}{\theta^{2i_0}} - \frac{4}{\theta^{2i} 2^i} \geq \frac{8}{\theta^{2i}} - \frac{4}{\theta^{2i}}= \frac{4}{\theta^{2i}}.
\]
In turn, this implies
\[
c_{\rho}(0,t)=f(t)=\frac{1}{\sqrt{3-2t}} < \frac{\theta^{i}}{2}.
\]
In summary, we have shown
\begin{eqnarray*}
i  \in \{1,2, \dots, i_0\}
&  \Rightarrow & 
a[\frac{16D_0}{\epsilon_0^3 \theta^{2i} 2^i},t_*^1] =\max_{\{\varphi:  |\varphi|\leq \frac{16D_0}{\epsilon_0^3 \theta^{2i} 2^i}\}} c_{\rho}(0,t_*^1+\varphi) < \frac{\theta^i}{2}.
\end{eqnarray*}
By (\ref{at1}), $t \mapsto c_{\rho}(0,t)$ is an increasing function.  Since $t_*^1 \geq t_*^{\mu}$ for $\mu=2, \dots, m$, this implies that
\begin{eqnarray*}
i  \in \{1,2, \dots, i_0\}
&  \Rightarrow & 
a[\frac{16D_0}{\epsilon_0^3 \theta^{2i} 2^i},t_*^\mu] =\max_{\{\varphi:  |\varphi|\leq \frac{16D_0}{\epsilon_0^3 \theta^{2i} 2^i}\}} c_{\rho}(0,t_*^{\mu}+\varphi) < \frac{\theta^i}{2}.
\end{eqnarray*}
In other words,   the assumption~(\ref{sat}) of iterative Lemma~\ref{induction} is satisfied for $i=0,1,2,\dots, i_0$. We can now complete the proof by applying the iterative Lemma~\ref{induction} as follows:\\

Let $_0l=l$ and ${_0\delta}=D_0$ (cf. assumption (\ref{assumptionoflemma}) of the iterative Lemma~\ref{induction}).   By (\ref{ruff}) and Lemma~\ref{distequiveuc}, 

\[
\left\{
\begin{array}{l}
{\displaystyle \sup_{B_{\theta}(0)}  d_h(v,{_0l}) < D_0}\\
{\displaystyle  \sup_{B_{\theta}(0)} }|v^{\mu}_{\varrho} \circ R^{\mu}-A^{\mu} x^1|  < {_0\delta}< 4 \theta^{-1}D_0.
\end{array}
\right.
\]
We  apply the iterative Lemma~\ref{induction}  for $i=1,2,\dots, i_0$ to obtain
\begin{equation} \label{ss}
\sup_{B_{\theta^{i_0+1}}(0)} d_h(v,l) 
< m \theta^{i_0+1}    \left( \frac{2^3\left( A+9D_0 \right)^3}{\epsilon_0^3} +10 \right) \theta^{-1}  D_0.
\end{equation}
Thus, 
\begin{eqnarray*}
\frac{\theta^{i_0+1}}{\sqrt{8}} & \leq &  c_{\rho}(0,t_*^1)  \ \ \ \ \ \mbox{(by (\ref{sany}) and (\ref{chooseinot}))} \\
& = & d_{\overline{\bf H}}(P_0, l^1 \circ (R^1)^{-1}(0))  \ \ \ \ \ \mbox{(by (\ref{atstar}))}\\
& = & d_{\overline{\bf H}}(v^1(0),l^1 \circ (R^1)^{-1}(0)) \ \mbox{(by the assumption that $v(0)={\mathbb P}_0$)} \\
& \leq &  d_h(v(0),l(0)) \\
& \leq &m \theta^{i_0+1}    \left( \frac{2^3\left( A+9D_0 \right)^3}{\epsilon_0^3} +10 \right) \theta^{-1}  D_0 \ \ \ \ \  \mbox{(by (\ref{ss}))}\\
& < & \frac{\theta^{i_0+1}}{\sqrt{8}} \ \ \ \  \mbox{(by (\ref{noloss2}))}.
\end{eqnarray*}
This contradicts our assumption that  $v(0) = {\mathbb P}_0$.
\end{proof}

\section{Two dimensional domains} \label{sec:dim2}

In this section, we prove Theorem~\ref{surf}, the regularity of harmonic maps from  two dimensional domains. We first need the following  preliminary lemma.

\begin{lemma}\label{onestratum} 
Let  $u:(\Omega,g) \rightarrow( \overline{\mathcal T}, d_{ \overline{\mathcal T}})$  be a harmonic map from  an $n$-dimensional  Lipschitz Riemannian domain,  $\Sigma$  a connected submanifold of $\Omega$ (possibly  $\Sigma= \Omega$) and ${\mathcal T}'$  a stratum of $\overline {\mathcal T}$ (possibly ${\mathcal T}'={\mathcal T}$).  If $u(\Sigma) \cap {\mathcal T}' \neq \emptyset$ and $\Sigma \subset {\mathcal R}(u)$, then $u(\Sigma) \subset  {\mathcal T}'$. Moreover,  there exists a stratum ${\mathcal T}'$ of $\overline{\mathcal T}$ such that $u({\mathcal R}(u)) \subset {\mathcal T}'$. 
\end{lemma}
 
\begin{proof}
Since   $u(\Sigma) \cap {\mathcal T}' \neq \emptyset$, we have that $W :=u^{-1}({\mathcal T}') \cap \Sigma$ is a nonempty  subset of 
 $\Sigma$. Assume on the contrary that  $u(\Sigma) \not \subset {\mathcal T}'$, and let  $x$ be a boundary point of $ W$ in $ \Sigma$. Since $\Sigma \subset {\mathcal R}(u)$, there exists $r>0$ such that $u(B_r(x))$ is contained in a single stratum. Since 
 $B_r(x) \cap W \neq \emptyset$, we conclude that  $u(B_r(x)) \subset {\mathcal T}'$ contradicting the fact that  $x$ is a boundary point of $ W$ in  $ \Sigma$.  This proves the first assertion.  
Since  ${\mathcal S}(u)$ is of Hausdorff codimension 2 the set  ${\mathcal R}(u)$ is connected. (This follows easily from \cite{bernie} Corollary 4.)  Thus, the second assertion follows from the first.
 \end{proof}  

\begin{bach}
 We first prove that if $u:\Sigma \rightarrow( \overline{\mathcal T}, d_{ \overline{\mathcal T}})$ is a harmonic map from a Riemann surface, then the set ${\mathcal S}^{>1}(u)$ of its singular points of order $>1$ is discrete.
Assume on the contrary that there exists $x_i \in {\mathcal S}^{>1}(u)$ such that $x_i \rightarrow x_0$.  By Lemma~\ref{ordergap'}, $Ord^u(x_i) \geq 1+\epsilon_0$ for some $\epsilon_0>0$.   By Theorem~\ref{orderforharmonic}, the order is a decreasing limit of continous functions and hence upper semicontinuous.  Thus,  $x_0 \in {\mathcal S}^{>1}(u)$.   Identify a neighborhood of $x_0=0$  to a disk $\D$ via normal coordinates.     By letting $\sigma_i=2|x_i|$ and taking a subsequence, if necessary, we can assume $\zeta_i = \frac{x_i}{\sigma_i} \rightarrow \zeta_*$ and the blow up maps $u_{\sigma_i}=(V_{\sigma_i},v_{\sigma_i})$ with blow-up factor  $\sqrt{\frac{I^u(\sigma_i)}{\sigma_i}}$ converge locally uniformly in the pullback sense  (cf.~Lemma~\ref{blowupcomponent}) to
\[
u_*=(V_*,v_*)=(V_*,v_*^1, \dots, v^{k-j}_*):B_1(0) \rightarrow \C^j \times Y_{1*} \times \dots \times Y_{k-j*}.
\]
  By Lemma~\ref{uscord}, $Ord^{u_*}(\zeta_*) \geq 1+\epsilon_0$, and thus the homogeneity of $u_*$ implies that $Ord^{u_*}(x) \geq 1+\epsilon_0$ for every point on the ray starting at 0 and going through $\zeta_*$. By rotating if necessary, we assume that this ray is the positive $x$-axis and $\zeta_*=(\frac{1}{2},0)$.  Thus, $V_*$ must be identically constant since otherwise $V_*$ is a harmonic map into $\C^j$ with order $\geq 1+\epsilon_0$ along the $x$-axis which is impossible.    Since $u_*$ is a non-constant map, it follows that  $v_*$ must be  non-constant. 
  
   We will now do a similar argument with $v_*$ in order to get a contradiction. From the proof of Lemma~\ref{ordergap'}, we observe that there exists a sequence of harmonic maps $w_i:\D \rightarrow \overline{\bf H}^{k-j}$ converging locally uniformly to $v_*$.  For simplicity, we will assume that $k-j=1$.  (Otherwise, pick one of the non-constant components of $v_*$ and the corresponding component  of $w_i$.)    By homogeneity, $v_*^{-1}(v_*(0))$ is a union of rays emanating from the origin in $\D$ and the connected components of $\D \backslash v_*^{-1}(v_*(0))$ are sectors of $\D$.  Furthermore,   Claim 1 in the proof of Lemma~\ref{proveitsfl} says that $v_*$  must map  every connected component of $\D \backslash v_*^{-1}(v_*(0))$ into a geodesic ray starting at $v_*(0)$.    Since $Ord^{v_*} \geq  1+\epsilon_0$ along the positive $x$-axis,  the postive $x$-axis is one of the rays in $v_*^{-1}(v_*(0))$.  We choose a sufficiently small neighborhood $\mathcal N$ of $\zeta_*=(\frac{1}{2},0)$ such that ${\mathcal N}$ intersects exactly  two sectors of $\D \backslash v_*^{-1}(v_*(0))$. Thus, $v_*(\mathcal N)$ is contained in a union of two geodesic rays.    Harmonicity of $v_*$ implies that $v_*(\mathcal N)$ is a geodesic segment.  After identifying the geodesic segment with an interval $[a,b]$ in the real line,  $v_*$ is a harmonic function in ${\mathcal N}$ with order $\geq 1+\epsilon$ along the $x$-axis, a contradiction.  Thus, we have shown that ${\mathcal S}^{>1}(u)$  is a discrete set.

Next we prove that the set ${\mathcal S}_j(u)$ (cf. (\ref{sj})) is discrete.  
Indeed, on the contrary, suppose that there exists a sequence  $x_i \in {\mathcal S}_j(u) \rightarrow x_{\star} \in {\mathcal S}_j(v)$.   Let  $u=(V,v)$ be a local representation at  $x_{\star}$. By Corollary~\ref{ptsorder1}, there exists $\epsilon_0>0$ such that $Ord^v(x_i) \geq 1+\epsilon_0$.    Identify a neighborhood of $x_0=0$ with $\D$  and take as before  $\sigma_i=2|x_i|$ and $\zeta_i = \frac{x_i}{\sigma_i} \rightarrow \zeta_*$ such that the  sequence of  blow-up maps $v_{\sigma_i}$   of $v$ at $x_{\star}$ with blow-up factor  $\sqrt{\frac{I^v(\sigma_i)}{\sigma_i}}$ is a sequence of asymptotic harmonic maps and converges locally uniformly in the pullback sense to a homogeneous harmonic map $v_0$.  Lemma~\ref{usordv} on the upper semicontinuity of order  implies $Ord^{v_0}(\zeta_{\star}) \geq 1+\epsilon_0$.   As before, the homogeneity of $v_0$ implies $Ord^{v_0} \geq 1+\epsilon_0$ along a the ray.     This contradicts Lemma~\ref{ahmresult} (cf. (\ref{em})).

 We have thus shown that the singular set of $u$ is discrete 
and hence 
given   $x \in {\mathcal S}(u)$,  there is $r>0$ such that $\overline{B_r (x)} \cap {\mathcal S}(u)= \{ x \}$. Thus $\partial B_r(x) \subset {\mathcal R}(u)$.   Applying Lemma~\ref{onestratum} for $\Sigma= \partial{B_r (x)}$, we have that $u( \partial{B_r (x)}) \subset {\mathcal T}'$ for some stratum ${\mathcal T}'$ of $\overline{\mathcal T}$.   Now recall the existence of  a convex exhaustion function $f:  {\mathcal T}' \rightarrow [0,\infty)$ (cf. \cite{wolpertJDG}).   Since $u(\partial B_r(x))$ is closed, there exists $c>0$ such that $u(\partial B_r(x)) \subset \{p \in {\mathcal T}':  f(p)\leq  c\}$.    Since sublevel sets of a convex function are convex, we conclude $u(\overline{B_r (x)}) \subset \{p \in {\mathcal T}':  f(p) \leq  c\}$, and hence $x \in {\mathcal R}(u)$.   This contradicts the assumption that  $x \in {\mathcal S}(u)$ and proves  ${\mathcal S}(u)= \emptyset$.
 \end{bach}

\section{Proof of Theorem~\ref{rigiditytheorem0} and Corollary~\ref{rigiditytheorem1}}\label{proofrigiditytheorem0}

Let $ u: \tilde M \rightarrow( \overline{\mathcal T}, d_{ \overline{\mathcal T}})$ be  a $\Gamma$-equivariant  harmonic map as in the statement of Theorem~\ref{rigiditytheorem0}. By  Lemma~\ref{onestratum}, there exists a stratum ${\mathcal T}'$ of $\overline{\mathcal T}$ such that $u({\mathcal R}(u)) \subset {\mathcal T}'$ and therefore $u(\tilde{M}) \subset \overline{\mathcal T}'$ where $\overline{\mathcal T}'$ denotes the Weil-Petersson completion of ${\mathcal T}'$.   
% Thus, $u(\tilde{M}) \subset \overline{\mathcal T}'$, and $\overline{\mathcal T}'$ must be invariant by the entire mapping class group $\Gamma$ by the equivariance of $u$.  This implies ${\mathcal T}'={\mathcal T}$; in other words, 
%\[
%u({\mathcal R}(u)) \subset {\mathcal T}.
%\]  
% Indeed, consider a minimizing sequence of finite energy $\Gamma$-equivariant maps $ {\tilde f}_n \rightarrow \tilde u$ uniformly on compact sets. By possibly going to a finite index subgroup of $\Gamma'$ (which we again denote $\Gamma'$ by an abuse of notation), the maps $ {\tilde f}_n$ induce homotopy equivalences of smooth manifolds 
%$$ f_n: M:=\tilde M \slash {\Gamma'} \rightarrow N:={\mathcal T} \slash {\Gamma'} $$
%which converge to the map
%$$ u: M \rightarrow \overline{N} := \overline {\mathcal T} \slash {\Gamma'}$$
%which lifts to $\tilde u$. On the other hand,  as in \cite{jost-yau} proof of Corollary 1, there exists a loop $\gamma$ in $N$ that cannot be homotoped to the boundary of $N$. If $h_1$ denotes a homotopy inverse of $f_1$, then $f_n \circ h_1 \circ \gamma$ represents a sequence of loops in $N$, converging
%uniformly to the loop $u \circ h_1 \circ \gamma$ which is homotopic to $ \gamma$. This implies that $u$ cannot map into the boundary of $N$, thus proving (\ref{maptoT}).
Since ${\mathcal T}'$ is isometric to a product of lower dimensional  Teichm\"uller spaces with the Weil-Petersson metric, the strong negative curvature of ${\mathcal T}'$ together with Theorem~\ref{RegularityTheorem} and Theorem~\ref{goto0} imply, as in  \cite{gromov-schoen} or \cite{daskal-meseSR}, that $u$ is pluriharmonic on the regular set ${\mathcal R}(u)$ (also cf. \cite{siu}). More precisely,  on  ${\mathcal R}(u)$,  we have that
\begin{eqnarray}\label{pluri}
D''d'u \equiv 0  \equiv D'd''u \ \ \mbox{and} \ \ \sum_{i,j,k,l}R_{ijkl}d''u_i \wedge d'u_j \wedge d'u_k \wedge d''u_l \equiv 0. 
\end{eqnarray}

Next,  applying \cite{bernie} Lemma 2, there exists a holomorphic disc $\D$ through any $x \in {\mathcal S}(u)$ such that
\begin{equation}\label{dimintersection}
{\mathcal H}^1 ( {\mathcal S}(u) \cap \D ) =0
\end{equation}
where ${\mathcal H}^1$ denotes 1-dimensional Hausdorff measure.
We next need the following 

\begin{claim} \label{rou}The restriction of $u$ to $\D$ is a harmonic map.
\end{claim}

\begin{proof}
Let $w:\D \rightarrow( \overline{\mathcal T}', d_{ \overline{\mathcal T}'})$ be a harmonic map with $w\big|_{\partial \D} = u\big|_{\partial \D}$.  
We will show $u=w$, thereby proving the claim.  Fix $\varphi \in C^{\infty}_c(\D)$ with $0 \leq \varphi \leq 1$.  For $\epsilon>0$, (\ref{dimintersection}) implies that  there exists a covering $\{B_{r_i}(x_i)\}_{i=1}^N$  of $\mbox{sup}(\varphi) \cap {\mathcal S}(u)\subset \D$ such that $\sum_{i=1}^N r_i<\epsilon$.  Let $\phi_i$ be a smooth function such $0 \leq \phi_i \leq 1$,  $\phi_i \equiv 0$ in $B_{r_i}(x_i)$, $\phi_i \equiv 1$ outside $B_{2r_i}(x_i)$ and $|\nabla \phi_i|<\frac{1}{r_i}$.  Define $\phi_{\epsilon}=\Pi_{i=1}^N \phi_i$ and 
$\phi_{\epsilon}^i = \Pi_{j \neq i} \phi_j$.   Since $u$ is pluriharmonic in ${\mathcal R}(u)$, its restriction $u\big|_{\D \backslash \bigcup_{i=1}^N B_{r_i}(x_i)}$ is a harmonic map.  Thus, $d^2(u,w)$ is weakly subharmonic in $\D \backslash \bigcup_{i=1}^N B_{r_i}(x_i)$ (cf. \cite{korevaar-schoen1} Lemma 2.4.2 and Remark 2.4.3).  Thus,  
\begin{eqnarray*}
\lefteqn{\int_{\D} \phi_{\epsilon}  \nabla \varphi \cdot \nabla d^2(u,w) dxdy + \sum_{i=1}^N \int_{B_{2r_i}(x_i)} \varphi \phi_{\epsilon}^i  \nabla \phi_i \cdot \nabla d^2(u,w) dxdy}\\
& = & \int_{\D} \nabla (\varphi \phi_{\epsilon}) \cdot \nabla d^2(u,w) dxdy  \geq 0.  \hspace{1.5in}
\end{eqnarray*}
Since $d^2(u,w)$ is a Lipschitz function in supp$(\varphi)$,  we can estimate
\[
\sum_{i=1}^N \int_{B_{2r_i}(x_i)} \left|\varphi \phi_{\epsilon}^i  \nabla \phi_i \cdot \nabla d^2(u,w) \right| dxdy \leq C \sum_{i=1}^N r_i^{-1} \int_{B_{2r_i}(x_i)}dxdy \leq C \sum_{i=1}^N r_i < C \epsilon.
\]
Letting $\epsilon \rightarrow 0$, we obtain
\[
\int_{\D} \nabla \varphi \cdot \nabla d^2(u,w) dxdy \geq 0.
\]
In other words, $d^2(u,w)$ is a weakly subharmonic function on $\D$.  Since $u$ and $w$ agree on the boundary $\partial \D$, we conclude that $u=w$ on $\D$.
\end{proof}
%\begin{claim}\label{var} 
%%If $u:\Omega \rightarrow \overline{\mathcal  T}$  is as in Theorem~\ref{RegularityTheorem} and non-constant, then  
%The singular set ${\mathcal S}(u)$ is contained in  an analytic subvariety of complex codimension at least 1.
%\end{claim}
%
%\begin{proof}
%Again, as in the proof of Theorem~\ref{surf},  we will use the fact that  the singular set ${\mathcal S}(u)$ consists of points of  order $>1$ and of points in  ${\mathcal S}_j(u)$ for some $j \in \{1, \dots, k-1\}$.  Since $u$ is pluriharmonic on ${\mathcal R}(u)$, $d'u$ and $d''u$ are   holomorphic and antiholomorphic sections of the  vector bundle $T^* ({\tilde M})\otimes 
%u^*(T {(\mathcal T}))$ restricted to ${\mathcal R}(u)$.  For $x_{\star} \in \Omega$ with order $>1$, $d'u$ and $d''u$ extend as   a holomorphic and antiholomorphic sections to $B_r(x_{\star})$ for $r>0$ by \cite{bernie}.  
%%where ${\mathcal S}^{>1}(u)$is the set of points  of  order $>1$.  
%Thus, the set ${\mathcal S}^{>1}{\mathcal S}^{>1}(u) \cap B_r(x_{\star})$ of points with order $>1$ is  the analytic subvariety defined by the equations $d'u=0=d''u$. Next, for $x_{\star} \in {\mathcal S}_j(u)$ and $r>0$ small,   write $u=(V,v)$  as in (\ref{localrep}) in $B_r(x_{\star})$.  By Lemma~\ref{A5}, we deduce that $d'v(x) =0=d''v(x)$ for $x \in  {\mathcal S}_j(u) \cap B_r(x_{\star})$.  Since the splitting $u=(V,v)$
%is holomorphic,  the same argument as above implies that $B_r(x_{\star}) \cap {\mathcal S}_j(u)$ is  the analytic subvariety defined by the equation $d'v=0=d''v$.
%\end{proof}

Now Theorem~\ref{surf} implies that the image of $u$ lies in  the  stratum ${\mathcal T}'$. From here the proof of Theorem~\ref{rigiditytheorem0} follows from the strong negativity of the curvature of ${\mathcal T}'$ as in \cite{siu}.  

We now proceed with the proof of the Corollary~\ref{rigiditytheorem1}. 
Notice that by the assumption that $\rho$ is sufficiently large,  \cite{daskal-wentworth} Corollary 1.3 implies that there exists a finite energy $\rho$-equivariant harmonic map 
\[
 u: \tilde M \rightarrow  \overline {\mathcal T}.
\]
By Theorem~\ref{rigiditytheorem0},  there exists a stratum ${\mathcal T}'$ of  $\overline {\mathcal T}$ such that $u$ is a pluriharmonic map into ${\mathcal T}'$. Since the image of $u$ is invariant under all pseudo-Anosov  transformations, 
\[
 {\mathcal T}'={\mathcal T} \ \mbox{and} \ \ u(\tilde M) \subset {\mathcal T}.
 \]
This completes the proof of the Corollary.

\section{Proof of Theorem~\ref{rigiditytheorem} and Corollary~\ref{symTeich}}\label{applications1}

\begin{george} At this point, we argue more or less as in Jost-Yau \cite{jost-yau}. We include the details here for the sake of completeness. Let $q: \hat M \rightarrow \bar M$ be a smooth resolution of singularities  with  exceptional divisor $\Sigma$ and let $M=\hat M \backslash \Sigma$. We label the connected components of $\Sigma$ by $\Sigma_j, j=1,...J$ and the irreducible components of $\Sigma_j$ by $\Sigma_j^l, l=1,...,L_j$. We can also assume that $\Sigma_j$ consists only of normal crossings. 
We endow $M$ with a  Poincare type metric (originally due to Cornabla and Griffiths  \cite{griffiths}) defined as follows: Let  $\sigma_{j,l}$ be a canonical section  of the line bundle 
 $\mathcal O(\Sigma_j^l)$ vanishing along $\Sigma_j^l$. For $\bar \omega$   a K\"ahler a form on $ M$ induced from a projective embedding of $\bar M$ we consider the metric associated to the K\"ahler form
 \begin{equation} \label{donaldsonmetric}
\omega=\sum_{j,l}  \frac{i}{2}  \partial \bar{\partial} \log \log |\sigma_{j,l}|_{h_{j,l}}^{-2}  + Cq^*(\bar{\omega})
\end{equation} 
where $h_{j,l}$ is a Hermitian metric on $\mathcal O(\Sigma_j^l)$ and $C>0$ is chosen sufficiently large such that  $\omega$ is positive. Let $g$ be the K\"ahler metric associated to $\omega.$ By \cite[Section 1]{jost-yau}, $g$ has bounded diameter and bounded Ricci curvature.

%We now proceed as in  \cite{siu-noncompact}, Theorem 6.5. Indeed, even though our argument is almost identical with the one in Siu's paper, we cannot quote his result because our spaces are slightly different than his. Therefore we will present the argument here for the sake of completeness.
%\begin{definition}
%\begin{itemize}
%\item End of a topological space X (CW complex). Space of ends $\mathcal E_X$. 
%\item If $f: X \rightarrow Y$ a proper map between topological spaces then it induces a map
%$f_\sharp: \mathcal E_X \rightarrow \mathcal E_Y$. 
%\end{itemize}
%\end{definition}
%
%\begin{lemma}\label{endsofM} 
%\begin{itemize}
%\item $M$ has an exhaustion by compact manifolds $M_s$, $s \in [0, \infty)$.
%\item $M \backslash M_s$ is a union of connected open sets each homeomorphic to  $C \times [0, \infty)$ where $C$ is compact, connected manifold.
%\item It follows that the space  $\mathcal E_M$ is in one to one correspondence with $\pi_0 (\overline M \backslash M)$.
%\end{itemize}
%\end{lemma}
%\begin{lemma}\label{endsofN} The same holds for the ends of $N$. {\bf{BAILY BOREL ????}} Furthermore, if $f: M \rightarrow N$ is proper map extending continuously to
%$\hat f: \hat M \rightarrow \overline N$ it must map  $\hat{M} \backslash M$ into $\overline{N} \backslash N$ and $f_\sharp$ is the  map induced from $\hat f$ between 
%$\pi_0(\hat{M} \backslash M) \rightarrow \pi_0(\overline{N} \backslash N)$.
%\end{lemma}

For each connected component $ \Sigma_j$ of $\hat M \backslash M$, the end $ E_j$ of $M$ corresponding to $ \Sigma_j$ can be written topologically (not metrically) as 
\begin{equation}\label{retractend}
E_j \simeq  \partial E_j \times {\R}^+.
\end{equation} 
To see this, consider the holomorphic section $\sigma_j$ vanishing on $\Sigma_j$  and Hermition metric $h_j$ defined by 
\[
\sigma_j= \sigma_{j,1}\otimes ... \otimes \sigma_{j,L_j}, \ h_j= h_{j,1}\otimes ...\otimes h_{j,L_j}
\]
and use the gradient flow of $|\sigma_j|_{h_j}^2$ to decompose into level sets.  

Retraction of each end $E_j$ to its boundary $\partial E_j$  via (\ref{retractend}) induces a deformation retraction of $M$ into its core 
\begin{equation}\label{defteract}
r_c : M \rightarrow M_c := M \backslash \bigcup_j E_j .
\end{equation}
 The same is true for $\mathcal M'$, by taking a resolution of singularities of a compactification of $\overline {\mathcal M'}$ and arguing as for $M$.

 Since  $M$ and ${\mathcal M'}$ are homotopy equivalent,  we can    induce via  (\ref{defteract}) a smooth homotopy equivalence 
 \[
 k_c: M_c \backslash \partial M_c \rightarrow \mathcal M'.
 \]
Under the codimension assumption of $\bar M \backslash M$ given in the statement of Theorem~\ref{rigiditytheorem}, the energy of the map $r_c$ 
is bounded with respect to the metric $g$ on $M$ by \cite[p.487]{jost-yau}. Hence, by the smoothness of $k_c$ and the compactness of $M_c$, we conclude that 
\[
f:=k_c \circ r_c: M \rightarrow \mathcal M'
\]
defines a smooth homotopy equivalence of finite energy. Since $\Gamma$ contains pseudo-Anosov elements associated to different measured foliations, $\Gamma$ is sufficiently large.  We thus obtain from Corollary~\ref{rigiditytheorem1} that there is a pluriharmonic map of finite energy
\[
u': M \rightarrow \mathcal M'
\]
which is also a homotopy equivalence.   

Next, consider the embedding of the  moduli space of Riemann surfaces $\mathcal M=\mathcal T/ \Gamma$ in $D/\Lambda$ where $D$ is the Siegel upper half space of degree $g$, $\Lambda$ is the Siegel modular group and let $ \overline {D/\Lambda}^{SBB}$ denote the Satake-Baily-Borel compactification of $D/\Lambda$ (cf. \cite{baborel}).  Let $ \overline {\mathcal M}^{SBB}$ denote the closure of $ {\mathcal M}$ in $\overline {D/\Lambda}^{SBB}$. Since $ \overline {\mathcal M}^{SBB} \backslash \mathcal M$ has more than one connected components (cf. \cite{jost-ji} Proposition 4.1), it follows that $\mathcal M$ has more than one ends. Since the quotient map $\mathcal M' \rightarrow \mathcal M $
is a proper surjective map, $\mathcal M'$ must have more than one ends as well.

\begin{lemma} The map $u'$ is holomorphic or conjugate holomorphic, and  its rank  is equal to  $2 \dim_{\C}  \mathcal M$.
\end{lemma}
\begin{proof}
Let $m=\dim_{\C} \mathcal M$.  We claim  that $H_{2m-1}(M, \R) \neq 0$. Assume on the contrary that $H_{2m-1}(M, \R) = 0$. Since $u'$ is a homotopy equivalence, it also implies that $ H_{2m-1}(\mathcal M', \R) = 0$ and, since ${\mathcal M}'_c$ is homotopy equivalent to ${\mathcal M}'$,  $ H_{2m-1}({\mathcal M}'_c, \R) = 0$. This contradicts the fact that ${\mathcal M}'$ has more than one ends. Indeed, since $ H_{2m}({\mathcal M}'_c, \R) \simeq H_0(({\mathcal M}'_c, \partial {\mathcal M}'_c), \R) = 0$,  the exact sequence
   \[
 H_{2m} ({\mathcal M}'_c, \R) \rightarrow  H_{2m}(({\mathcal M}'_c, \partial {\mathcal M}'_c), \R) \rightarrow H_{2m-1} ( \partial {\mathcal M}'_c, \R) \rightarrow H_{2m-1}({\mathcal M}'_c, \R)
\]
  implies
\[
H_{2m}(({\mathcal M}'_c, \partial {\mathcal M}'_c), \R)\simeq H_{2m-1} ( \partial {\mathcal M}'_c, \R)
\]
and by Poincare Lefschetz duality 
\[
H^0({\mathcal M}'_c, \R) \simeq H^0 ( \partial{\mathcal M}'_c, \R).
\]
This is a contradiction since ${\mathcal M}'_c$ is connected and $\partial {\mathcal M}'_c$ is not.
Hence $ H_{2m-1}(M, \R) \neq 0$. Since $u'$ is a homotopy equivalence it must carry a non-trivial $2m-1$ homology class to a non-trivial $2m-1$ homology class and hence it must have rank $\geq 2m-1$ somewhere.   Since $u'$ is holomorphic or conjugate holomorphic by Theorem~\ref{rigiditytheorem0}, it must have maximal rank $=2m$.
  \end{proof}

  By changing orientations if necessary we can assume $u'$  is holomorphic. Let
\[
u: M \rightarrow \mathcal M
\]
 denote the composition of the quotient map to $\mathcal M$ and $u'$, which is also holomorphic.
 By embedding  $\mathcal M$ in $D/\Lambda$, 
 we obtain a holomorphic map
 \[
 u: M \rightarrow D/\Lambda
 \]
 which by \cite{borel} extends to a  holomorphic  map 
\[
\hat u: \hat M \rightarrow \overline {D/\Lambda}^{SBB}
\]
where $\hat M$ is a smooth compactification of $M$ as before.  
 \begin{lemma}\label{isproper}
 The map 
 $\hat u$ as given above takes $\hat{M} \backslash M$ into $\overline {\mathcal M}^{SBB} \backslash \mathcal M$. In particular, $u$ and hence also $u'$ is proper.
\end{lemma}
\begin{proof} 
 Let $p \in \hat{M} \backslash M$ and    
\[
v: \D =\{ z \in \C : |z| < 1 \} \rightarrow \hat{M}, \ \ v(0)=p
\]
be such that
\[
\gamma_t:=v(\{  |z| =t \}) \ \mbox{homotopically nontrivial in }M,  \ \ \mbox{length}(\gamma_t) \rightarrow 0.
\]
Since $u: M \rightarrow  \mathcal M$ is a homotopy equivalence, $u(\gamma_t)$ is homotopically nontrivial on $ \mathcal M$ and since the domain metric has bounded Ricci curvature, we obtain by the Schwartz Lemma \cite{roy}
$ \mbox{length}(u(\gamma_t)) \rightarrow 0$. It follows that $\hat u(p) \in  \overline {D/\Lambda}^{SBB}  \backslash {D/\Lambda}$, hence $\hat u(p) \in {\mathcal M}^{SBB} \backslash \mathcal M$ which proves the lemma.
\end{proof}
 
  Since $u'$ is proper and has maximal rank, it is onto. Given $y \in \mathcal M$,  ${u'}^{-1}(y)$ is a compact subvariety of $M$ and hence, if of positive dimension, it is homologically nontrivial. Since $u'$ is a homotopy equivalence and maps ${u'}^{-1}(y)$ to $\{y\}$ which is homologically trivial, this is a contradiction. It follows that $u'$ is a covering map and since it is also a homotopy equivalence it must have degree 1.  The fact that $u'$ is a biholomorphism 
follows as in \cite[proof of Theorem 8, p.110]{siu}.
  \end{george}
  
  \begin{chika} 
 Assume on the contrary that  there exists a sufficiently large  homomorphism $\rho: \Lambda \rightarrow \Gamma$.
As in \cite[Lemma 8.1]{gromov-schoen}, we first construct a  finite energy equivariant Lipschitz map 
$f: \tilde M= G \slash K \rightarrow {\mathcal T}$. Corollary~\ref{rigiditytheorem1} implies that there exists  a   $\Lambda$-equivariant  harmonic map
$$u: \tilde M \rightarrow( \overline{\mathcal T}, d_{ \overline{\mathcal T}}).$$

By Lemma~\ref{onestratum}, there exists ${\mathcal T}' \subset \overline{\mathcal T}$ such that $u({\mathcal R}(u)) \subset {\mathcal T}'$.  
We are going to show that $u$ is constant, so with an intent of arriving at a contradiction, let's assume that $u$ is non-constant.
As in \cite{daskal-meseSR} Corollary 14 and Lemma 15, our regularity Theorem~\ref{RegularityTheorem} and Theorem~\ref{goto0} imply that $u$ is totally geodesic on the regular set ${\mathcal R}(u)$.    In other words, $u$  satisfies on ${\mathcal R}(u)$
\begin{equation}\label{totgeod}
\nabla du= 0.
\end{equation}
As in  \cite{daskal-meseSR} proof of Theorem 1, (\ref{totgeod}) combined with Theorem~\ref{RegularityTheorem} implies  that $u$ is totally geodesic on the entire $\tilde M$ in the sense that $u$ maps geodesics to geodesics. 

Since the domain is an  irreducible symmetric space, $u$ must be a totally geodesic {\it{immersion}} into a stratum ${\mathcal T}'$. This is clearly a contradiction if the symmetric space has rank
$\geq 2$. In the rank 1 case,  the contradiction follows from \cite{wu} Theorem 1.2. 
We thus conclude that $u$ is constant, hence $\rho(\Lambda)$ fixes a point in Teichm\"uller space. Since the action of the mapping class group is properly discontinuous, this implies that $\rho(\Lambda)$ is finite contradicting the fact that $\rho$ is sufficiently large. 
\end{chika}

\end{document}